\documentclass[10pt]{amsart}
\usepackage{times,amsmath,amsbsy,amssymb,amscd,mathrsfs}
\usepackage{graphicx,subfigure,epstopdf,wrapfig,chemarrow}

\usepackage{algorithm2e} 
\usepackage{multicol,multirow}
\usepackage{mathtools}
\usepackage[usenames,dvipsnames,svgnames,table]{xcolor}
\usepackage[all]{xy}
\usepackage{wrapfig}
\usepackage{tcolorbox}

\usepackage{tikz,tikz-cd}
\usepackage[utf8]{inputenc}
\usepackage{pgfplots} 
\usepackage{pgfgantt}
\usepackage{pdflscape}
\pgfplotsset{compat=newest} 
\pgfplotsset{plot coordinates/math parser=false}
\newlength\fwidth

\definecolor{myBlue}{rgb}{0.0,0.0,0.55}
\usepackage[pdftex,colorlinks=true,citecolor=myBlue,linkcolor=myBlue]{hyperref}

\usepackage[hyperpageref]{backref}

\usepackage{comment,enumerate,multicol,xspace}

  \newcounter{mnote}
  \let\oldmarginpar\marginpar
    \renewcommand\marginpar[1]{\-\oldmarginpar[\raggedleft\footnotesize #1]%
    {\raggedright\footnotesize #1}}

\usepackage{stmaryrd}
\usepackage{scalefnt}
\usepackage{graphicx}
\usepackage{listings}
\usepackage{adjustbox}

\newtheorem{theorem}{Theorem}[section]
\newtheorem{lemma}[theorem]{Lemma}
\newtheorem{corollary}[theorem]{Corollary}
\newtheorem{proposition}[theorem]{Proposition}

\newtheorem{example}[theorem]{Example}

\newtheorem{remark}[theorem]{Remark}

\newcommand{\dx}{\,{\rm d}x}
\newcommand{\dd}{\,{\rm d}}

\newcommand{\bs}{\boldsymbol}

\DeclareMathOperator*{\img}{img}
\DeclareMathOperator*{\spa}{span}

\newcommand{\curl}{{\rm curl\,}}
\renewcommand{\div}{\operatorname{div}}
\newcommand{\grad}{{\rm grad\,}}
\newcommand{\rot}{{\rm rot}}
\newcommand{\tr}{\operatorname{tr}}
\newcommand{\dev}{\operatorname{dev}}
\newcommand{\sym}{\operatorname{sym}}
\newcommand{\skw}{\operatorname{skw}}

\newcommand{\mskw}{\operatorname{mskw}}
\newcommand{\vskw}{\operatorname{vskw}}
\newcommand{\defm}{\operatorname{def}}
\newcommand{\hess}{\operatorname{hess}}
\newcommand{\inc}{\operatorname{inc}}

\newcommand{\step}[1]{\noindent\raisebox{1.5pt}[10pt][0pt]{\tiny\framebox{$#1$}}\xspace}

\newcommand{\vertiii}[1]{{\left\vert\kern-0.25ex\left\vert\kern-0.25ex\left\vert #1 
    \right\vert\kern-0.25ex\right\vert\kern-0.25ex\right\vert}}

\newcommand{\Oplus}{\ensuremath{\vcenter{\hbox{\scalebox{1.5}{$\oplus$}}}}}

\usepackage{booktabs}
\usepackage{stmaryrd}
\usepackage{scalefnt}
\usepackage{graphicx}
\allowdisplaybreaks[3]

\begin{document}
\title[Finite Element Complexes from Complexes]{Complexes from Complexes: Finite Element Complexes in Three Dimensions}

 \author{Long Chen}%
 \address{Department of Mathematics, University of California at Irvine, Irvine, CA 92697, USA}%
 \email{chenlong@math.uci.edu}%
 \author{Xuehai Huang}%
 \address{School of Mathematics, Shanghai University of Finance and Economics, Shanghai 200433, China}%
 \email{huang.xuehai@sufe.edu.cn}%

 \thanks{The first author was supported by NSF DMS-2309785 and DMS-2309777.}
 \thanks{The second author is the corresponding author. The second author was supported by the National Natural Science Foundation of China (Grant No. 12171300) and the Natural Science Foundation of Shanghai (Grant No. 21ZR1480500).}

\makeatletter
\@namedef{subjclassname@2020}{\textup{2020} Mathematics Subject Classification}
\makeatother

\begin{abstract}
In the field of solving partial differential equations (PDEs), Hilbert complexes have become highly significant. Recent advances focus on creating new complexes using the Bernstein-Gelfand-Gelfand (BGG) framework, as shown by Arnold and Hu [Complexes from complexes. {\em Found. Comput. Math.}, 2021]. This paper extends their approach to three-dimensional finite element complexes. The finite element Hessian, elasticity, and divdiv complexes are systematically derived by applying techniques such as smooth finite element de Rham complexes, the $t$-$n$ decomposition, and trace complexes, along with related two-dimensional finite element analogs. The construction includes two reduction operations and one augmentation operation to address continuity differences in the BGG diagram, ultimately resulting in a comprehensive and effective framework for constructing finite element complexes, which have various applications in PDE solving.
\end{abstract}
\maketitle

%\tableofcontents

\section{Introduction}
Hilbert complexes are essential in developing robust numerical methods for solving partial differential equations (PDEs)~\cite{ArnoldFalkWinther2006, Arnold;Falk;Winther:2010Finite, Arnold:2018Finite, ChenHuang2018}. Arnold and Hu~\cite{Arnold;Hu:2020Complexes} have recently introduced a systematic methodology for creating new complexes by applying the Bernstein-Gelfand-Gelfand (BGG) framework to well-known Hilbert complexes, including the de Rham complex. In this study, we focus on systematically constructing finite element complexes in a three-dimensional setting using the BGG approach.

%, $\defm = \sym \grad$ is the symmetric gradient operator, $\dev \boldsymbol \sigma = \boldsymbol \sigma - \tr(\boldsymbol \sigma)\boldsymbol I/3$ is the deviation operator, $  H(\mathrm{inc}, \Omega;\mathbb{S}) $ is the space of symmetric tensor $\boldsymbol  \tau$ s.t. $\inc \boldsymbol  \tau := - \curl (\curl \boldsymbol  \tau)^{\intercal}\in L^2(\Omega; \mathbb M)$,   $ H(\operatorname{div}, \Omega; \mathbb{X})$ is the space for the symmetric (for $\mathbb X = \mathbb S$) or traceless (for $\mathbb X = \mathbb T$) tensor $\boldsymbol  \sigma$ with $\div\boldsymbol  \sigma\in L^2(\Omega; \mathbb R^3)$, and ${H}(\div{\div },\Omega; \mathbb{S}):=\{\boldsymbol{\tau}\in {L}^{2}(\Omega; \mathbb{S}): \div {\div}\boldsymbol{\tau}\in L^{2}(\Omega)\}$.

%discrete is more difficult than continuous level. 

Let $\Omega$ be a domain in $\mathbb R^3$. The de Rham complex reads as
\begin{equation}\label{eq:deRham}
\mathbb R\hookrightarrow{} H^1(\Omega)\xrightarrow{\grad} H(\curl,\Omega)\xrightarrow{\curl} H(\div,\Omega)\xrightarrow{\div}L^2(\Omega) \xrightarrow{}0, 
\end{equation}
where Sobolev spaces
\begin{equation*}
\begin{aligned}
H^1(\Omega) &:= \{\phi \in L^2(\Omega) : \grad \phi \in  L^2(\Omega;\mathbb R^3) \},\\
H(\curl,\Omega) &:= \{\boldsymbol u \in  L^2(\Omega;\mathbb R^3): \curl \boldsymbol u \in  L^2(\Omega;\mathbb R^3) \},\\
 H(\div,\Omega) &:= \{\boldsymbol u \in L^2(\Omega;\mathbb R^3): \div \boldsymbol u\in L^2(\Omega) \}.
\end{aligned}
\end{equation*}
By stacking copies of de Rham complexes to form a BGG diagram, several complexes can be derived from the BGG framework~\cite{Arnold;Hu:2020Complexes}, including, but not limited to, the Hessian complex, the elasticity complex, and the divdiv complex. Both the Hessian complex and the divdiv complex are applied in solving the biharmonic equation~\cite{ChenHuang2024div-div-conforming,Chen;Hu;Huang:2018Multigrid,PaulyZulehner2020} and the Einstein-Bianchi equation~\cite{QuennevilleBelair2015}. The space $H^{-1}(\div\div,\Omega;\mathbb S)$ in the divdiv complex can also be used to address elasticity problems~\cite{PechsteinSchoeberl2011} and the Reissner-Mindlin plate model~\cite{PechsteinSchoeberl2017}. The space $H(\div,\Omega;\mathbb S)$ in the elasticity complex is crucial for modeling stress in elasticity problems~\cite{JohnsonMercier1978,ArnoldDouglasGupta1984}. The incompatibility operator ${\rm inc}$ in the elasticity complex has applications in intrinsic elasticity~\cite{CiarletGratieMardare2009}, dislocation theory~\cite{VanGoethem2010}, elastoplasticity~\cite{Amstutz;Van-Goethem:2019incompatibility}, and relativity~\cite{Christiansen:2011linearization,Li2018}. However, this paper does not present numerical methods for specific partial differential equations. Instead, it focuses on developing a framework for constructing finite element complexes, which may have applications in the areas mentioned above.

Recently, there have been significant developments in the construction of finite element Hessian complexes, elasticity complexes, and divdiv complexes. These constructions have typically been approached on a case-by-case basis in previous works~\cite{ChenHuang2020, Chen;Huang:2020Finite, Chen;Huang:2021Finite, Christiansen;Gopalakrishnan;Guzman;Hu:2020discrete, HuLiang2021, Hu;Liang;Ma:2021Finite, Hu;Liang;Ma;Zhang:2022conforming, Hu;Ma;Zhang:2020family}. Our objective is to extend the application of the BGG construction to finite element complexes, thereby unifying these previously scattered results and systematically generating new ones. This has been achieved in our recent work~\cite{Chen;Huang:2022femcomplex2d}, which focused on two-dimensional cases.

However, extending to three dimensions introduces additional challenges. One significant challenge is the construction of finite element de Rham complexes with varying degrees of smoothness in three dimensions. We have successfully tackled this issue in our recent work~\cite{Chen;Huang:2022FEMcomplex3D}, which we will briefly summarize below.

We firstly recall the smooth finite elements constructed in ~\cite{huConstructionConformingFinite2021} by Hu, Lin and Wu.  
Given an integer vector $\boldsymbol r = (r^{\texttt{v}}, r^e, r^f)^{\intercal}$ with $r^{\texttt{v}} \geq 2r^e\geq 4r^f\geq 0$ and polynomial degree $k\geq 2r^{\texttt{v}}+1$, one can construct a decomposition of the simplicial lattice and design finite elements with $C^{r^{\texttt{v}}}$ smoothness at vertices, $C^{r^e}$ smoothness on edges, and $C^{r^f}$ smoothness across faces. Therefore the finite element space is $H^{r^f +1}$-conforming. Such approach can be generalized to arbitrary dimension; see~\cite{huConstructionConformingFinite2021} or~\cite[Appendix A]{Chen;Huang:2022FEMcomplex3D}. It unifies the scattered results ~\cite{BrambleZlamal1970,Zenisek1970,ArgyrisFriedScharpf1968} in two dimensions (including the well known $C^1$ Argyris element),~\cite{Zenisek1974a,Zhang2009a,Lai;Schumaker:2007Trivariate} in three dimensions, and~\cite{Zhang2016a} in four dimensions. 
Notice that the finite element spaces constructed in~\cite{WangXu2006,WangXu2013,WuXu2019} are $H^m$ non-conforming while this paper focus on conforming discretization. 

The requirement $r^{\texttt{v}} \geq 2r^e\geq 4r^f\geq 0$ can be relaxed. We  call $\boldsymbol r$ a smoothness vector if 
$r^f \geq -1$, $r^e \geq \max\{2r^f, -1\}$, and $r^{\texttt{v}} \geq \max\{2r^e, -1\}$. Here $-1$ means no continuity and thus a lower bound $\boldsymbol r\geq -1$ is imposed. To reflect such requirement, for a smoothness vector $\boldsymbol r$,  define $\boldsymbol r\ominus 1 := \max \{\boldsymbol r-1, -1\}$. Through the utilization of a simplicial lattice decomposition, we are able to construct scalar finite element space $\mathbb V^{\grad}_{k}(\boldsymbol{r})$ which is $C^{r^f}$ conforming with $r^f\geq 0$, and space $\mathbb V^{L^2}_{k}(\boldsymbol{r})$ which admits $r^f=-1$. 
%provide a precise characterization of the polynomial space associated with element-wise bubble functions, accounting for variable smoothness at sub-simplices:
%\begin{align*}
%\mathbb B_{k}(T;\boldsymbol r):=\{u\in\mathbb P_k(T):&\, \nabla^ju \textrm{ vanishes at all vertices of $T$ for $j=0,\ldots, r^{\texttt{v}}$}, \\
%& \textrm{ $\nabla^ju$ vanishes on all edges of $T$ for $j=0,\ldots, r^{e}$}, \\
%& \textrm{ and $\nabla^ju$ vanishes on all faces of $T$ for $j=0,\ldots, r^{f}$}\},
%\end{align*}
%and of  edge bubble $\mathbb B_{k}(e;\boldsymbol r)$ and face bubble $\mathbb B_{k}(f;\boldsymbol r)$. We use these bubble spaces to construct 

%By skillfully combining $t$-$n$ decompositions across different sub-simplexes, we establish finite element descriptions for the spaces:
%Combining with appropriate $t$-$n$ decompositions on different sub-simplexes, 
%we can also characterize 
%\begin{align*}
%\mathbb B^{\curl}_{k}(T, \boldsymbol{r}) &:=\{\boldsymbol{v}\in \mathbb B^{3}_{k}(T, \boldsymbol{r}): \boldsymbol{v}\times \boldsymbol{n}|_{\partial T}=\boldsymbol{0}\},\\
% \mathbb B^{\div}_{k}(T, \boldsymbol{r}) &:=\{\boldsymbol{v}\in \mathbb B^{3}_{k}(T, \boldsymbol{r}): \boldsymbol{v}\cdot\boldsymbol{n}|_{\partial T}=0\},
%\end{align*}
%we give finite element description of the spaces 

Let $\boldsymbol r_0 \geq 0, \boldsymbol r_1 = \boldsymbol r_0 -1, \boldsymbol r_2\geq\boldsymbol r_1\ominus1, \boldsymbol r_3\geq \boldsymbol r_2\ominus1$. Introduce spaces
\begin{align*}
\mathbb V^{\curl}_{k+1}(\boldsymbol r_1, \boldsymbol r_2) &:= \{ \boldsymbol v\in \mathbb V^3_{k+1}(\boldsymbol r_1) \cap H(\curl,\Omega): \curl \boldsymbol v\in \mathbb V_{k}^{\div}(\boldsymbol r_2)\}, \\
\mathbb V^{\div}_{k}(\boldsymbol r_2, \boldsymbol r_3) &:= \{ \boldsymbol v\in \mathbb V^3_{k}(\boldsymbol r_2)\cap H(\div,\Omega): \div \boldsymbol v\in \mathbb V_{k-1}^{L^2}(\boldsymbol r_3)\}.
\end{align*}
%Assume all $\boldsymbol r_i$ are valid smoothness vectors, i.e.,
%$$
%r_1^{\texttt{v}}\geq2r_1^e+1,\; r_1^{e}\geq2r_1^f+1,\; r_2^{\texttt{v}}\geq2r_2^e,\; r_2^{e}\geq2r_2^f, \; r_3^{\texttt{v}}\geq2r_3^e,\; r_3^{e}\geq2r_3^f.
%$$
Assume $(\boldsymbol r_2, \boldsymbol r_3,k)$ is div stable, i.e., $\div\mathbb V^{\div}_{k}(\boldsymbol{r}_2,\boldsymbol{r}_3)=\mathbb V^{L^2}_{k-1}(\boldsymbol{r}_3)$, and $k\geq\max\{2 r_1^{\texttt{v}} + 1,2 r_2^{\texttt{v}} + 1,2 r_3^{\texttt{v}} + 2,1\}$. In~\cite{Chen;Huang:2022FEMcomplex3D}, we construct the finite element de Rham complex as follows:
%If we additionally assume
%$\dim\mathbb B_{k-1}(T;\boldsymbol{r}_3)\geq1$, and $\dim\mathbb B_{k}(f;\begin{pmatrix}
%r_2^{\texttt{v}} \\
%r_2^e
%\end{pmatrix})\geq1$, the finite element complex 
\begin{equation}\label{eq:intro-femderham}
\mathbb R\xrightarrow{\subset}\mathbb V^{\grad}_{k+2}(\boldsymbol{r}_0)\xrightarrow{\grad}\mathbb V^{\curl}_{k+1}(\boldsymbol{r}_1,\boldsymbol{r}_2)\xrightarrow{\curl}\mathbb V^{\div}_{k}(\boldsymbol{r}_2,\boldsymbol{r}_3)\xrightarrow{\div}\mathbb V^{L^2}_{k-1}(\boldsymbol{r}_3)\to 0.
\end{equation}
And we give finite element descriptions for spaces $\mathbb V^{\curl}_{k+1}(\boldsymbol r_1, \boldsymbol r_2)$ and $\mathbb V^{\div}_{k}(\boldsymbol r_2, \boldsymbol r_3)$ in~\eqref{eq:intro-femderham}.
When \(r_2^f\geq 0\),~\eqref{eq:intro-femderham} transforms into a finite element Stokes complex, given that the space \(\mathbb V^{\div}_{k}(\boldsymbol{r}_2,\boldsymbol{r}_3)\subset H^1(\Omega;\mathbb R^3)\). This enables the discretization of Stokes equation. 
%Notably, existing works on finite element Stokes complexes~\cite{Neilan2015} and finite element de Rham complexes~\cite{Christiansen;Hu;Hu:2018finite} exemplify instances of~\eqref{eq:intro-femderham} achieved by selecting different vectors of smoothness.

The major challenge of extending BGG to the finite element complexes emerges from the mis-match in the continuity of Sobolev spaces, specifically $H^1(\Omega)$, $H(\curl,\Omega)$, and $H(\div,\Omega)$. This discrepancy can be explained using the following diagram:
\begin{equation*}%\label{eq:intro-divdiv}
\begin{tikzcd}
%\mathbb{R}^3 \to 
{H^1(\Omega;\mathbb R^3)}\arrow{r}{{\grad}} &{H(\curl, \Omega; \mathbb{M})}  \arrow{r}{{\curl}} &{H(\div, \Omega; \mathbb{M})} \arrow{r}{{\div}} & {L^2(\Omega;\mathbb R^3)} \to \boldsymbol {0}\\
%\mathbb{R} \longrightarrow 
H^{1}(\Omega) \arrow{r}{\grad} \arrow[ur, "\iota"]& H(\curl,\Omega)  \arrow{r}{\curl} \arrow[ur, "\mskw"]& H(\div,\Omega) \arrow{r}{\div}\arrow[ur, "\mathrm{id}"] & L^2(\Omega)  \longrightarrow 0.
 \end{tikzcd}%,
\end{equation*}
%where
\begin{enumerate}[1.]
\item The mapping $\iota: H^1(\Omega)\to H(\curl, \Omega; \mathbb{M})$, where $\mathbb M = \mathbb R^{3\times 3}$ and $\iota v = v \boldsymbol I\in \mathbb M$, is well-defined. However, its right inverse $\tr: H(\curl, \Omega; \mathbb{M})\to H^1(\Omega)$ is not, as functions in $H(\curl,\Omega)$ possess only the tangential continuity.
\item The inclusion $\mskw H(\curl,\Omega) \subset H(\div,\Omega; \mathbb{M})$ is justifiable (refer to Section 2.1), yet its right inverse $\vskw: H(\div, \Omega; \mathbb{M}) \to H(\curl,\Omega)$ is not attainable.

\item The obvious inclusion $H(\div,\Omega)\subset L^2(\Omega,\mathbb R^3)$ is not surjective.
\end{enumerate}
In the work of Arnold and Hu~\cite{Arnold;Hu:2020Complexes}, the domain Sobolev spaces $H(\dd,\Omega)$, where $\dd = \curl$ or $\div$, are substituted with $H^s(\Omega)$ Sobolev spaces having matching indices $s$, as depicted in diagram~\eqref{eq:3DBGG}. For finite element spaces that solely conform to $H(\curl,\Omega)$ or $H(\div,\Omega)$, the challenge posed by the mis-match in tangential or normal continuity serves as the primary obstacle to extending the BGG framework to the discrete context.

We tackle this challenge by identifying sub-complexes within the finite element de Rham complex~\eqref{eq:intro-femderham} by imposing appropriate subspace restrictions. The $\ \widetilde{}\ $ and $\ \widehat{}\ $ operations applied to a short exact sequence are thoroughly discussed in Section~\ref{sec:bggconstruct}. By applying these two reduction operations to suitable smooth finite element de Rham complexes, we are able to construct the finite element Hessian complex, the finite element elasticity complex, and the finite element divdiv complex using the BGG framework. Additionally, we propose an augmentation operation to further extend the constructed complexes to a broader set of smoothness vectors.

Deriving finite element descriptions for these subspaces, which involve element-wise degrees of freedom (DoFs), presents a more intricate challenge and requires significant effort. We will provide DoFs for these tensor finite element spaces using three key methodologies:

\begin{enumerate}[1.]
\item Smooth finite element de Rham complexes. As previously mentioned and discussed in detail in~\cite{Chen;Huang:2022FEMcomplex3D}, these complexes offer a fundamental basis for the construction.
  
\item The $t$-$n$ decomposition approach. Introduced in~\cite{Chen;Huang:2021Geometric}, this approach plays a critical role in the construction of $H(\div)$-conforming elements.
  
\item Trace complexes and two-dimensional (2D) finite element complexes. We employ two trace complexes on each face and draw insights from 2D finite element complexes~\cite{Chen;Huang:2022femcomplex2d} to guide the construction of DoFs on edges.
\end{enumerate}

\begin{figure}[htbp]
\begin{center}
\includegraphics[width=11.75cm]{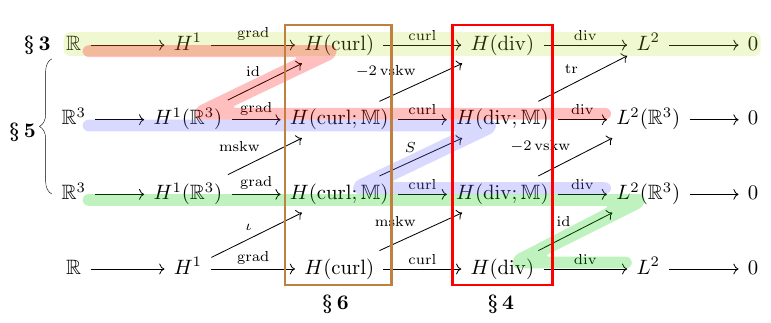}
\caption{Organization of Sections 3 - 6.}
\label{fig:sections}
\end{center}
\end{figure}

The remainder of this paper is organized as follows. Section~\ref{sec:preliminary} provides a detailed overview of the BGG framework and the two reduction operations and one augment operation, laying the groundwork for subsequent sections. Section~\ref{sec:femderhamcomplex} reviews the smooth finite element de Rham complexes. Section~\ref{sec:faceelements} focuses on the construction of face finite elements. The BGG-based construction of finite element Hessian, elasticity, and divdiv complexes is presented in Section~\ref{sec:FEMcomplexes}, followed by the construction of edge elements in Section~\ref{sec:edgeelements}.  
Appendix~\ref{sec:bubblecomplex} develops various bubble polynomial complexes. Figure~\ref{fig:sections} visually outlines the organization of the paper.
%
%Some examples of finite element Hessian complexes, elasticity complexes and divdiv complexes are provided for illustrating the BGG construction in Section~\ref{sec:examplesHessiancomplex}, Section~\ref{sec:examplesElascomplex} and Section~\ref{sec:examplesdivdivcomplex} respectively.

\section{Preliminary}\label{sec:preliminary}
In this section, we will briefly review the framework developed in~\cite{Arnold;Hu:2020Complexes} for deriving complexes using the Bernstein-Gelfand-Gelfand (BGG) construction~\cite{BernsteinGelfandGelfand1975}. Earlier contributions to this area can be found in works such as~\cite{Eastwood2000,Arnold;Winther:2002finite,ArnoldFalkWinther2006b}. Additionally, we introduce two reduction operations and one augmentation operation applied to a short exact sequence, which effectively broaden the applicability of the BGG framework. 

Throughout this paper, we assume that $\Omega$ is topologically trivial, ensuring that the de Rham complex~\eqref{eq:deRham} is exact. As a result, all derived complexes based on de Rham complexes are also exact~\cite{Arnold;Hu:2020Complexes}. For simplicity, we omit $\Omega$ in the notation for spaces. For instance, $H^s=H^s(\Omega)$ represents the standard Sobolev space with real index $s$.

\subsection{Notation}
Following \cite{Chen;Huang:2021Finite}, we define the dot product and cross product from the left, denoted as \(\boldsymbol{b} \cdot \boldsymbol{A}\) and \(\boldsymbol{b} \times \boldsymbol{A}\), respectively. These operations are applied column-wise to the matrix \(\boldsymbol{A}\). Conversely, when the vector appears on the right of the matrix, i.e., \(\boldsymbol{A} \cdot \boldsymbol{b}\) and \(\boldsymbol{A} \times \boldsymbol{b}\), the operations are defined row-wise. The order of performing row and column products is interchangeable, resulting in the associative rule for triple products:
$$
\boldsymbol{b} \times \boldsymbol{A} \times \boldsymbol{c} := (\boldsymbol{b} \times \boldsymbol{A}) \times \boldsymbol{c} = \boldsymbol{b} \times (\boldsymbol{A} \times \boldsymbol{c}).
$$
Similar rules apply for \(\boldsymbol{b} \cdot \boldsymbol{A} \cdot \boldsymbol{c}\) and \(\boldsymbol{b} \cdot \boldsymbol{A} \times \boldsymbol{c}\), allowing parentheses to be omitted.

We treat the Hamilton operator \(\nabla = (\partial_1, \partial_2, \partial_3)^{\intercal}\) as a column vector. For a scalar function \(u\), \(\nabla u = \grad u\) represents the gradient of \(u\). For a vector function \(\boldsymbol{u} = (u_1, u_2, u_3)^{\intercal}\), \(\curl \boldsymbol{u} = \nabla \times \boldsymbol{u}\) and \(\div \boldsymbol{u} = \nabla \cdot \boldsymbol{u}\) are the standard differential operations. Define \(\nabla \boldsymbol{u} = \nabla \boldsymbol{u}^{\intercal} = (\partial_i u_j)\), which can be understood as the dyadic product of the Hamilton operator \(\nabla\) and the column vector \(\boldsymbol{u}\).

By applying matrix-vector operations to the Hamilton operator \(\nabla\), we obtain column-wise differentiations \(\nabla \cdot \boldsymbol{A}\), \(\nabla \times \boldsymbol{A}\), and row-wise differentiations \(\boldsymbol{A} \cdot \nabla\), \(\boldsymbol{A} \times \nabla\). We introduce the following double differential operators:
\begin{equation*}
\inc \boldsymbol{A} := \nabla \times \boldsymbol{A} \times \nabla, \quad \div\div\boldsymbol{A} := \nabla \cdot \boldsymbol{A} \cdot \nabla.
\end{equation*}

In the literature, differential operators for matrices are typically applied row-wise to tensors. To distinguish this from the \(\nabla\) notation, we define the operators using letters:
\begin{align*}
\grad \boldsymbol{u} &:= \boldsymbol{u} \nabla^{\intercal} = (\partial_j u_i) = (\nabla \boldsymbol{u})^{\intercal},\\
\curl \boldsymbol{A} &:= - \boldsymbol{A} \times \nabla = (\nabla \times \boldsymbol{A}^{\intercal})^{\intercal},\\
\div \boldsymbol{A} &:= \boldsymbol{A} \cdot \nabla = (\nabla \cdot \boldsymbol{A}^{\intercal})^{\intercal}.
\end{align*}
Notice the sign change in $\curl \boldsymbol{A} = - \boldsymbol{A} \times \nabla$.

For a given plane \(F\) with a normal vector \(\boldsymbol{n}\), we define the projection matrix and its rotation as
$$
\Pi_F := \boldsymbol{I} - \boldsymbol{n}\boldsymbol{n}^{\intercal}, \quad \Pi_F^{\bot} := \boldsymbol{n} \times \Pi_F.
$$
We then introduce the following definitions:
$$
\nabla_F := \Pi_F \nabla, \quad
\nabla_F^{\bot} := \Pi_F^{\bot} \nabla.
$$
For a scalar function \(v\), we have:
\begin{align*}
\grad_F v = \nabla_F v &= \Pi_F (\nabla v) = - \boldsymbol{n} \times (\boldsymbol{n} \times \nabla v), \\
\curl_F v = \nabla_F^{\bot} v &= \boldsymbol{n} \times \nabla v = \boldsymbol{n} \times \nabla_F v,
\end{align*}
where \(\grad_F v\) is the surface gradient of \(v\), and \(\curl_F v\) is the surface \(\curl\).

For a vector function \(\boldsymbol{v}\), the surface divergence is defined as:
$$
\div_F \boldsymbol{v} := \nabla_F \cdot \boldsymbol{v} = \nabla_F \cdot (\Pi_F \boldsymbol{v}).
$$
Additionally, the surface rot operator is defined as:
\begin{equation*}
{\rm rot}_F \boldsymbol{v} := \nabla_F^{\bot} \cdot \boldsymbol{v} = (\boldsymbol{n} \times \nabla) \cdot \boldsymbol{v} = \boldsymbol{n} \cdot (\nabla \times \boldsymbol{v}),
\end{equation*}
which represents the normal component of \(\nabla \times \boldsymbol{v}\). 

%The tangential trace of $\nabla \times \boldsymbol v$ is
%\begin{equation}\label{eq:tangentialtrace}
%\boldsymbol  n\times (\nabla \times \boldsymbol  v) = \nabla (\boldsymbol  n\cdot \boldsymbol  v) - \partial_n \boldsymbol  v = \nabla_F (\boldsymbol  n\cdot \boldsymbol  v) - \partial_n(\Pi_F\boldsymbol  v).
%\end{equation}
%By definition,  for a vector function $\boldsymbol  v$,
%\begin{equation*}
%{\rm rot}_F \boldsymbol  v=\nabla_F^{\bot}\cdot \boldsymbol  v = - \nabla_F\cdot (\boldsymbol  n\times \boldsymbol  v), \quad\div_F\boldsymbol  v=\nabla_F\cdot \boldsymbol  v = \nabla_F^{\bot}\cdot (\boldsymbol  n\times \boldsymbol  v).
%\end{equation*}

We denote by $\mathbb M$ the space of $3\times 3$ matrices, $\mathbb S$ the subspace of symmetric matrices, $\mathbb T$ the subspace of traceless matrices, and $\mathbb K$ the subspace of skew-symmetric matrices. A matrix $\boldsymbol{\tau}$ can be decomposed into $\boldsymbol{\tau}=\sym\boldsymbol{\tau}+\skw\boldsymbol{\tau}$ with the symmetric part $\sym\boldsymbol{\tau}:=(\boldsymbol{\tau}+\boldsymbol{\tau}^{\intercal})/2$ and the skew-symmetric part $\skw\boldsymbol{\tau}:=(\boldsymbol{\tau}-\boldsymbol{\tau}^{\intercal})/2$.
%For a matrix function, differential operators $\curl, \div$ in letter are applied row-wise. 

We will use Iverson bracket  $[\cdot]$~\cite{iverson1962programming}, which extends the Kronecker delta function to a statement $P$
$$
[P]= 
\begin{cases}1 & \text { if } P \text { is true}, \\ 0 & \text { otherwise}.\end{cases}
$$

\subsection{BGG Construction}\label{sec:bggconstruct}
We follow~\cite{Arnold;Hu:2020Complexes} to briefly review the BGG construction.  A bounded Hilbert complex is a sequence of Hilbert spaces connected by a sequence of linear bounded operators satisfying the property: the composition of two consecutive operators vanishes. Assume we have two bounded Hilbert complexes $(\mathcal V_{\bullet}\otimes\mathbb E_{\bullet}, d_{\bullet})$, $(\mathcal V_{\bullet}\otimes\tilde{\mathbb E}_{\bullet}, \tilde{d}_{\bullet})$ and bounded linking maps $S_i={\rm id}\otimes s_i: \mathcal V_{i+1}\otimes\tilde{\mathbb E}_i\to \mathcal V_{i+1}\otimes\mathbb E_{i+1}$ for $i=0,\ldots, n-1$ 
\begin{equation}\label{eq:BGGgeneral}
\begin{tikzcd}[column sep=small, row sep=large]
0  \arrow{r}{} &  \mathcal V_0\otimes\mathbb E_0  \arrow{r}{d_0} & \mathcal V_1\otimes\mathbb E_1 \arrow{r}{d_1} & \;\;\;\;\,\cdots\;\;\;\;\, \arrow{r}{d_J} & \quad\,\cdots\quad\, \arrow{r}{d_{n-1}} & \mathcal V_n\otimes\mathbb E_n \arrow{r}{} &  0\\
0  \arrow{r}{} &  \mathcal V_1\otimes\tilde{\mathbb E}_0 \arrow{r}{\tilde d_0} \arrow[ur, "S_0"] & \mathcal V_2\otimes\tilde{\mathbb E}_1 \arrow{r}{\tilde d_1} \arrow[ur, "S_1"] & \;\;\;\;\,\cdots\;\;\;\;\,\arrow{r}{d_J} \arrow[ur, "S_{J}"]  & \quad\,\cdots\quad\,  \arrow{r}{\tilde d_{n-1}}\arrow[ur, "S_{n-1}"] 
& \mathcal V_{n+1}\otimes\tilde{\mathbb E}_n \arrow{r}{} & 0,
\end{tikzcd}
\end{equation}
in which $s_i: \tilde{\mathbb E}_i\to\mathbb E_{i+1}$ is a linear operator between finite-dimensional spaces. The operators in~\eqref{eq:BGGgeneral} satisfy anti-commutativity: $S_{i+1}\tilde{d}_i=-d_{i+1}S_i$, and $J$-injectivity/surjectivity condition: for some particular $J$ with $0\leq J<n$, $s_i$ is injective for $i=0,\ldots,J$ and is surjective for $i=J,\ldots,n-1$. The output complex is 
$$
0\xrightarrow{}\varUpsilon_0\xrightarrow{\mathcal D_0}\varUpsilon_1\xrightarrow{\mathcal D_1}\cdots\xrightarrow{\mathcal D_{J-1}}\varUpsilon_J\xrightarrow{\mathcal D_J}\varUpsilon_{J+1}\xrightarrow{\mathcal D_{J+1}}\cdots\xrightarrow{\mathcal D_{n-1}}\varUpsilon_n\to0,
$$
where $\varUpsilon_i$ is $\mathcal V_i\otimes(\mathbb E_i/\img(s_{i-1}))$ for $i=0,\ldots, J$ and $\mathcal V_{i+1}\otimes\ker(s_{i})$ for $i=J+1,\ldots, n$, $\mathcal D_i$ is the projection of $d_i$ onto $\varUpsilon_{i+1}$ for $i=0,\ldots, J-1$, $\tilde d_i$ for $i=J+1,\ldots,n-1$ and $\tilde d_J(S_J)^{-1}d_J$ for $i=J$.
By the BGG framework, many new complexes can be generated from old ones.

%\subsection{Derived complexes in three dimensions}
The following example is presented in~\cite{Arnold;Hu:2020Complexes}. In three dimensions, we stack copies of de Rham complexes to form the diagram
\begin{equation}\label{eq:3DBGG}
\begin{tikzcd}
\mathbb{R} \longrightarrow H^{s} \arrow{r}{\grad} &{H^{s-1}(\mathbb R^3)} \arrow{r}{\curl} &{H^{s-2}(\mathbb R^3)} \arrow{r}{{\div}} & H^{s-3} \longrightarrow 0\\
\mathbb{R}^3 \to { H^{s-1}(\mathbb R^3)}\arrow{r}{\grad} \arrow[ur, "{\mathrm{id}}"]& {H^{s-2}(\mathbb{M}) } \arrow{r}{{\curl}} \arrow[ur, "{-2\vskw}"]& {H^{s-3}(\mathbb{M})} \arrow{r}{{\div}}\arrow[ur, "{\tr}"] & {H^{s-4}(\mathbb{R}^3) } \to \bs{0}\\
\mathbb{R}^3 \to {H^{s-2}(\mathbb{R}^3)}\arrow{r}{{\grad}} \arrow[ur, "{\mskw}"]&{H^{s-3}(\mathbb{M})}  \arrow{r}{{\curl}} \arrow[ur, "{S}"]&{H^{s-4}(\mathbb{M})} \arrow{r}{{\div}}\arrow[ur, "{-2\vskw}"] & {H^{s-5}(\mathbb{R}^3)} \to \bs{0}\\
\mathbb{R} \longrightarrow H^{s-3}\arrow{r}{\grad} \arrow[ur, "\iota"]&H^{s-4}(\mathbb{R}^3)  \arrow{r}{\curl} \arrow[ur, "\mskw"]&H^{s-5}(\mathbb{R}^3) \arrow{r}{\div}\arrow[ur, "\mathrm{id}"] & H^{s-6} \longrightarrow 0,
 \end{tikzcd}%,
\end{equation}
where $H^s(\mathbb X) = H^s\otimes \mathbb X$ for $\mathbb X = \mathbb R^3$ or $\mathbb M$, and operators
\begin{align*}
S \boldsymbol  \tau &= \boldsymbol  \tau^{\intercal} - \tr(\boldsymbol  \tau) \boldsymbol  I, \quad \iota: \mathbb R\to \mathbb M,\; \iota v = v \boldsymbol I,\\
\mskw \boldsymbol  \omega &:=
\begin{pmatrix}
 0 & -\omega_3 & \omega_2 \\
\omega_3 & 0 & - \omega_1\\
-\omega_2 & \omega_1 & 0
\end{pmatrix},\quad %=-\omega\times\boldsymbol I
\vskw := \mskw^{-1}\circ \skw.
\end{align*}
Considering operators in the $\nearrow$ direction, the three operators in the diagonal are one-to-one,  the lower triangular part is injective, and the upper triangular part is surjective. By direct calculation, the parallelogram formed by the north-east diagonal  $\nearrow$ and the horizontal operators is anticommutative. 

Then applying the BGG construction, several complexes can be derived from the BGG framework including but not limited to the Hessian complex, the elasticity complex, and the divdiv complex; see the three zigzag paths in Fig.~\ref{fig:sections}. 

Recent efforts have led to the individual construction of finite element Hessian complexes, elasticity complexes, and divdiv complexes, as demonstrated in works such as~\cite{ChenHuang2020,Chen;Huang:2020Finite,Chen;Huang:2021Finite,Christiansen;Gopalakrishnan;Guzman;Hu:2020discrete,HuLiang2021,Hu;Liang;Ma:2021Finite,Hu;Liang;Ma;Zhang:2022conforming,Hu;Ma;Zhang:2020family}. However, these constructions have been carried out on a case-by-case basis, raising the question: Can the overarching BGG framework be employed to establish a unified foundation for these diverse constructions?

Applying the BGG methodology to finite element complexes is a challenging task. A significant difficulty lies in constructing finite element de Rham complexes with different degrees of smoothness. We have successfully addressed this challenge, as detailed in~\cite{Chen;Huang:2022FEMcomplex3D}, and will revisit the specific techniques in Section \ref{sec:femderhamcomplex}.

Another intricate difficulty that arises from the diagram~\eqref{eq:3DBGG} becomes evident when we shift from the Sobolev space $H^s$ to the domain spaces $H(\grad)$, $H(\curl)$, or $H(\div)$ within the diagram:
\begin{equation}\label{eq:3DBGGdomain}
\begin{tikzcd}
\mathbb{R} \longrightarrow H^{1}  \arrow{r}{\grad} & H(\curl) \arrow{r}{\curl} & H(\div) \arrow{r}{{\div}} & L^2 \longrightarrow {0}\\
\mathbb{R}^3 \to { H^1(\mathbb R^3)}\arrow{r}{\grad} \arrow[ur, "{\mathrm{id}}"]& {H(\curl; \mathbb{M}) } \arrow{r}{{\curl}} \arrow[ur, "{-2\vskw}"]& {H(\div; \mathbb{M})} \arrow{r}{{\div}}\arrow[ur, "{\tr}"] & {L^2(\mathbb R^3)} \to \bs{0}\\
\mathbb{R}^3 \to {H^1(\mathbb R^3)}\arrow{r}{{\grad}} \arrow[ur, "{\mskw}"]&{H(\curl; \mathbb{M})}  \arrow{r}{{\curl}} \arrow[ur, "{S}"]&{H(\div; \mathbb{M})} \arrow{r}{{\div}}\arrow[ur, "{-2\vskw}"] & {L^2(\mathbb R^3)} \to \boldsymbol {0}\\
\mathbb{R} \longrightarrow H^{1} \arrow{r}{\grad} \arrow[ur, "\iota"]& H(\curl)  \arrow{r}{\curl} \arrow[ur, "\mskw"]& H(\div) \arrow{r}{\div}\arrow[ur, "\mathrm{id}"] & L^2  \longrightarrow 0,
 \end{tikzcd}%,
\end{equation}
where $H(\curl; \mathbb{M}) = \mathbb R^3\otimes H(\curl)$ and $H(\div; \mathbb{M}) = \mathbb R^3\otimes H(\div)$.
The Hamilton operator $\nabla$ can be employed to represent the differential operators.
%: $\grad = \nabla$, $\curl = \nabla \times$, and $\div =\nabla \cdot$. 
Symbolically, by substituting $\nabla$ with the outward unit normal vector $\boldsymbol{n}$, these operations maintain their respective anticommutative properties. For example, 
\begin{align}
% \label{eq:anticommute1}
\notag
(2\vskw \boldsymbol \tau)\cdot\boldsymbol n & = -\tr(\boldsymbol \tau\times \boldsymbol n),\\ 
\label{eq:nS} (S \boldsymbol \tau)\boldsymbol n &= -2\vskw(\boldsymbol \tau\times \boldsymbol n).
%\notag
%\boldsymbol n\times \boldsymbol u & = 2\vskw (\boldsymbol u \boldsymbol n^{\intercal}).
\end{align}
Therefore, the operators in the $\nearrow$ direction retain their well-defined nature. To illustrate, consider the operator $S: H(\curl; \mathbb{M}) \to H(\div; \mathbb{M}), S \boldsymbol  \tau = \boldsymbol  \tau^{\intercal} - \tr(\boldsymbol  \tau) \boldsymbol  I$, positioned at the center of~\eqref{eq:3DBGGdomain}. Through the integration by parts and identity~\eqref{eq:nS}, it becomes evident that in the distributional sense, 
$$
\langle \div S \boldsymbol{\tau}, \boldsymbol{\phi} \rangle = ( 2\vskw \curl \boldsymbol{\tau}, \boldsymbol{\phi} ), \quad \boldsymbol{\phi} \in C_0^{\infty}(\Omega;\mathbb{R}^3).
$$
Given that $\boldsymbol{\tau} \in H(\curl; \mathbb{M})$, leading to ${\rm vskw\,}\curl \boldsymbol{\tau} \in L^2(\mathbb{M})$, we can conclude that $S \boldsymbol{\tau} \in H(\div; \mathbb{M})$. 

Nonetheless, the converse direction, i.e., the $\swarrow$ direction, is not well-defined due to inherent continuity mismatches. For instance, all three operators along the diagonal of diagram~\eqref{eq:3DBGGdomain} are clearly not one-to-one, which prevents a straightforward application of the BGG procedure.

In cases where finite element spaces possess sufficient smoothness, the framework corresponds to diagram~\eqref{eq:3DBGG}. However, for certain finite element spaces, the situation is better represented by~\eqref{eq:3DBGGdomain} rather than~\eqref{eq:3DBGG}. The smoothness discrepancy is the main barrier to extending the BGG construction into the discrete setting. To overcome this issue, we introduce two reduction operations for an exact sequence.

\subsection{Two reduction operations and one augment operation}
Let $V_i, i=0,1,2$ be Hilbert spaces forming an exact sequence:
\begin{equation}\label{eq:exact}
\ker(\dd_0) \xrightarrow{\subset} V_0 \xrightarrow{\dd_0} V_1 \xrightarrow{\dd_1} V_2 \to 0.
\end{equation}
The exactness implies $\ker(\dd_1)=\img(\dd_0)$ and $\dd_1 V_1 = V_2$. When they are finite-dimensional, the following dimension identity holds:
\begin{equation}\label{eq:Euler}
 \dim \ker(\dd_0) - \dim V_0 + \dim V_1  - \dim V_2 = 0.
\end{equation}

Introduce a closed subspace $\widetilde{V}_2 \subseteq V_2$, and define $\widetilde{V}_1$ as the preimage of $\widetilde{V}_2$, namely,
$$
\widetilde{V}_1 = \{ v\in V_1: \dd_1 v\in \widetilde{V}_2\} \subseteq V_1. 
$$
Since $\dd_1\dd_0 = 0$, the exact sequence remains intact through this construction:
\begin{equation}\label{eq:tilde}
\ker(\dd_0) \xrightarrow{\subset} V_0 \xrightarrow{\dd_0} \widetilde{V}_1 \xrightarrow{\dd_1} \widetilde{V}_2 \to 0.
\end{equation}
We shall refer to the process of transitioning from~\eqref{eq:exact} to~\eqref{eq:tilde} as a $\widetilde{\quad}$ (tilde) operation.

Suppose we have two additional operators, $\dd_2$ acting on $V_0$ and $s^*$ acting on $V_1$ respectively. In application $s^*$ in the $\swarrow$ direction is the adjoint operator of $s$ used in the BGG diagram. Consider a closed subspace $W$ contained in $s^*(V_1)$. Define two subspaces in $V_0$ and $V_1$:
\begin{equation}\label{eq:hatspace}
\widehat V_0 = \{ u\in V_0: \dd_2 u\in W\}, \quad \widehat V_1 = \{ v\in V_1: s^* v\in W\}. 
\end{equation}
Assume we have the following triangular commutative diagram
\begin{equation}\label{eq:sd}
\begin{tikzcd}
\widehat{V}_0
 \arrow[d,swap, "\dd_2"]
 \arrow{r}{\dd_0} 
 & \widehat{V}_1
 \arrow[dl, swap, "s^*"]
\\
%&
W & 
\end{tikzcd}.
\end{equation}
That is $s^*\dd_0 = \dd_2.$

\begin{lemma}\label{lm:hat0}
Assume we have the exact sequence~\eqref{eq:exact} and triangular commutative diagram~\eqref{eq:sd}, where $\widehat V_0$ and $\widehat V_1$ are defined by~\eqref{eq:hatspace}.
Assume $\ker(\dd_0)\subset \widehat V_0$ and $\dd_1(\widehat{V}_1)= V_2$.
Then 
\begin{equation}\label{eq:hat}
\ker(\dd_0) \xrightarrow{\subset} \widehat{V}_0 \xrightarrow{\dd_0} \widehat{V}_1 \xrightarrow{\dd_1} V_2 \to 0
\end{equation}
is also an exact sequence.
\end{lemma}
\begin{proof}
As $s^*\dd_0 = \dd_2$, $\dd_0 \widehat V_0\subseteq \widehat V_1$ and thus~\eqref{eq:hat} is indeed a complex. To show its exactness, we only have to demonstrate that $\ker(\dd_1)\cap \widehat{V}_1= \dd_0\widehat{V}_0.$ Taking a $\widehat v\in \widehat{V}_1$ and $\dd_1\widehat v = 0$, by the exactness of~\eqref{eq:exact}, there exists a $u\in V_0$ s.t. $\dd_0 u = \widehat v$. Recall that $\dd_2 = s^*\dd_0$. So $\dd_2 u = s^*\dd_0 u = s^*\widehat v\in W$, i.e., $u\in \widehat{V}_0$. 
\end{proof}

When $V_i, i=0, 1, 2,$ are finite dimensional, the condition $\dd_1(\widehat{V}_1)= V_2$ can be verified by dimension count.
\begin{lemma}\label{lm:hat}
Assume we have the exact sequence~\eqref{eq:exact} with finite-dimensional spaces and triangular commutative diagram~\eqref{eq:sd}, where $\widehat V_0$ and $\widehat V_1$ are defined by~\eqref{eq:hatspace}.
Assume $W\subseteq s^*(V_1)$ and $\ker(\dd_0)\subset \widehat V_0$.
Additionally assume that the equation $\dim V_0 - \dim \widehat{V}_0 = \dim V_1 - \dim \widehat{V}_1$, equivalently $\dim\dd_2V_0 - \dim\dd_2\widehat{V}_0 = \dim s^*V_1 - \dim s^*\widehat{V}_1$, holds. Then complex \eqref{eq:hat}
% \begin{equation}\label{eq:hat}
% \ker(\dd_0) \xrightarrow{\subset} \widehat{V}_0 \xrightarrow{\dd_0} \widehat{V}_1 \xrightarrow{\dd_1} V_2 \to 0
% \end{equation}
is an exact sequence.
\end{lemma}
\begin{proof}
% By construction, $\dd_0 \widehat V_0\subseteq \widehat V_1$ and thus~\eqref{eq:hat} is indeed a complex. To show its exactness, we start by demonstrating that $\ker(\dd_1)\cap \widehat{V}_1= \dd_0\widehat{V}_0.$ Taking a $\widehat v\in \widehat{V}_1$ and $\dd_1\widehat v = 0$, by the exactness of~\eqref{eq:exact}, there exists a $u\in V_0$ s.t. $\dd_0 u = \widehat v$. Recall that $\dd_2 = s^*\dd_0$. So $\dd_2 u = s^*\dd_0 u = s^*\widehat v\in W$, i.e., $u\in \widehat{V}_0$. 
By the proof of Lemma~\ref{lm:hat0}, \eqref{eq:hat} is a complex and $\ker(\dd_1)\cap \widehat{V}_1= \dd_0\widehat{V}_0.$

For the second part, it is evident that $\dd_1(\widehat{V}_1)\subseteq V_2$. To prove $\dd_1(\widehat{V}_1)= V_2$, we count the dimensions: 
\begin{align*}
\dim \dd_1(\widehat{V}_1) &= \dim \widehat{V}_1 - \dim \dd_0(\widehat{V}_0) = \dim \widehat{V}_1 - \dim \widehat{V}_0 + \dim \ker(\dd_0)\\
& = \dim V_1 - \dim V_0 + \dim \ker(\dd_0) = \dim V_2,
\end{align*}
where in the last step, we have used the dimension identity~\eqref{eq:Euler}. 
\end{proof}
The procedure described above, which transforms the sequence from ~\eqref{eq:exact} to the sequence presented in~\eqref{eq:hat}, is referred to as a $\widehat{\quad}$ (hat) operation. This operation specifically modifies the domain and co-domain of the operator $\dd_0$. By combining the hat and tilde operations, it becomes possible to create a further reduced exact sequence:
\begin{equation*}
\ker(\dd_0) \xrightarrow{\subset} \widehat{V}_0 \xrightarrow{\dd_0} \widehat{\widetilde{V}}_1 \xrightarrow{\dd_1} \widetilde{V}_2 \to 0.
\end{equation*}
Throughout the process, the spaces $\widehat V$, $\widetilde V$, or $\widehat{\widetilde V}$ can be renamed as needed, depending on the context and the spaces being constructed.

We further introduce an inverse hat operation to enlarge the spaces. 
\begin{lemma}
Assume we have a short exact sequence 
\begin{equation*}%\label{eq:inversehat}
\ker(\dd_0) \xrightarrow{\subset} \widehat{V}_0 \xrightarrow{\dd_0} \widehat{V}_1 \xrightarrow{\dd_1} V_2 \to 0.
\end{equation*}
Let $V_0\supseteq \widehat{V}_0$ and $V_1\supseteq \widehat{V}_1$ satisfy
$$
\dd_1 V_1\subseteq V_2, \quad \dd_0 V_0 = V_1\cap \ker(\dd_1).
$$
Then we have the exact sequence
$$
\ker(\dd_0) \xrightarrow{\subset} V_0 \xrightarrow{\dd_0} V_1 \xrightarrow{\dd_1} V_2 \to 0.
$$
\end{lemma}
\begin{proof}
 It suffices to prove $\dd_1 V_1 = V_2$ which is from the fact $\dd_1 \widehat{V}_1 =  V_2$ and $V_1\supseteq \widehat{V}_1$.
\end{proof}
 
 To illustrate these operations and their interplay, please refer to Fig.~\ref{fig:hattilde}.
\begin{figure}[htbp]
\begin{center}
\includegraphics[width=4.35in]{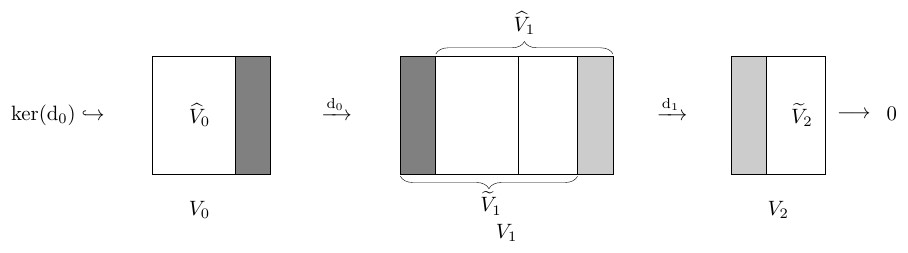}
\caption{Two reduction operations. The $\widetilde{\quad}$ operation reduces the space $V_1$ to $\widetilde V_1$ and $V_2$ to $\widetilde V_2$ by removing the sub-spaces marked by the light gray color. The $\widehat{\quad}$ operation reduces the space $V_0$ to $\widehat{V}_0$ and $V_1$ to $\widehat{V}_1$ by removing the sub-spaces marked by the dark gray color. The inverse hat operation is adding these sub-spaces back.}
\label{fig:hattilde}
\end{center}
\end{figure}

\subsection{Derived complexes}
%The previous result is for finite dimensional spaces and use the dimension count to verify the exactness. Can't apply directly here.
By diagram~\eqref{eq:3DBGG},
the surjectivity $\tr H(\div; \mathbb{M})=L^2$ follows from $\div H^1(\mathbb R^3)=L^2$, and $-2\vskw H(\curl; \mathbb{M})=H(\div)$ follows from the regular decomposition $H(\div)=H^1(\mathbb R^3)+\curl H^1(\mathbb R^3)$~\cite{HiptmaXu2007auxiliary}.
Then applying the hat operation $\widehat{\quad}$ to $V_0= H^1$ and $V_1 = H(\curl)$ with $\widehat V_0 = H^2, \widehat V_1 = H^1(\mathbb R^3)$ and $W = H^1(\mathbb R^3)$, we obtain %Modify the Sobolev spaces to fit BGG approach
\begin{equation}\label{eq:3DBGGdomainHessian}
\begin{tikzcd}
\mathbb{R} \longrightarrow H^{2}  \arrow{r}{\grad} & H^1(\mathbb R^3) \arrow{r}{\curl} & H(\div) \arrow{r}{{\div}} & L^2 \longrightarrow {0}\\
\mathbb{R}^3 \to { H^1(\mathbb R^3)}\arrow{r}{\grad} \arrow[ur, "{\mathrm{id}}"]& {H(\curl; \mathbb{M}) } \arrow{r}{{\curl}} \arrow[ur, "{-2\vskw}"]& {H(\div; \mathbb{M})} \arrow{r}{{\div}}\arrow[ur, "{\tr}"] & {L^2(\mathbb R^3)} \to \bs{0},
 \end{tikzcd}%,
\end{equation}
which leads to the Hessian complex~\cite{Arnold;Hu:2020Complexes,PaulyZulehner2020}
\begin{equation*}%\label{eq:hesscomplex}
%\resizebox{.9\hsize}{!}{$
\mathbb P_1 \xrightarrow{\subset} H^{2}\xrightarrow{\hess} H(\curl;\mathbb S)\xrightarrow{\curl}  H(\div;\mathbb T) \xrightarrow{\div}  L^2(\mathbb R^3)\xrightarrow{}\boldsymbol0,
%$}
\end{equation*}
where $H(\curl;\mathbb S)=H(\curl;\mathbb M)\cap L^2(\mathbb S)$ and $H(\div;\mathbb T)=H(\div;\mathbb M)\cap L^2(\mathbb T)$.

Now we look at the second and third rows of diagram~\eqref{eq:3DBGGdomain}. We shall apply the two reductions to construct the BGG diagram
\begin{equation}\label{eq:3DBGGdomainElasticity}
\begin{tikzcd}
\mathbb{R}^3 \to H^1(\curl) \arrow{r}{\grad} & \widehat{\widetilde{H}}(\curl;\mathbb M) \arrow{r}{\curl} & \widetilde{H}(\div;\mathbb M) \arrow{r}{\div} & L^2(\mathbb R^3) \to \boldsymbol{0} \\
\mathbb{R}^3 \to H^1(\mathbb R^3) \arrow[ur,swap,"\mskw"'] \arrow{r}{\grad} & H(\curl;\mathbb M) \arrow[ur,swap,"S"'] \arrow{r}{\curl} & H(\div;\mathbb M) \arrow[ur,swap,"-2\vskw"'] \arrow{r}{\div} \arrow[r] & L^2(\mathbb R^3)\to \boldsymbol{0}.
 \end{tikzcd}%,
\end{equation}
The bottom complex is three copies of the standard de Rham complex. 

As we mentioned before $S: H(\curl; \mathbb{M}) \to H(\div; \mathbb{M})$ is well-defined but $S^{-1}: H(\div; \mathbb{M}) \to H(\curl; \mathbb{M})$ is not. To fix it, we apply the $\widetilde{\quad}$ operation to reduce space $H(\div; \mathbb{M}) $ to $\widetilde{H}(\div;\mathbb M):=S H(\curl;\mathbb M)$. Then clearly $S: H(\curl; \mathbb{M}) \to \widetilde{H}(\div; \mathbb{M})$ is one-to-one. 

The div stability $\div \widetilde{H}(\div; \mathbb{M}) = L^2(\mathbb{R}^3)$ holds because $H^1(\mathbb{M}) \subset \widetilde{H}(\div; \mathbb{M})$ and $\div H^1(\mathbb{M}) = L^2(\mathbb{R}^3)$. Since $\div S = 2 \vskw \curl$, it follows that $\vskw: H(\div, \mathbb{M}) \to L^2(\mathbb{R}^3)$ is surjective. The injectivity of the operator $\mskw: H^1(\mathbb{R}^3) \to \widehat{\widetilde{H}}(\curl; \mathbb{M})$ can be deduced from the definition of the space $\widehat{\widetilde{H}}(\curl; \mathbb{M})$ given below.

To apply the $\widehat{\quad}$ operation, we use $2\vskw \grad \boldsymbol u= \curl \boldsymbol u$ and the triangular diagram
\begin{equation*}%\label{eq:sd}
\begin{tikzcd}
H^1(\curl)
 \arrow[d,swap, "\curl"]
 \arrow{r}{\grad} 
 & \widehat{\widetilde{H}}(\curl;\mathbb M)
 \arrow[dl, swap, "2\vskw"]
\\
%&
H^1(\mathbb R^3)  & 
\end{tikzcd},
\end{equation*}
where 
\begin{align*}
H^1(\curl)&:=\{\boldsymbol{v}\in H^1(\mathbb R^3): \curl\boldsymbol{v}\in H^1(\mathbb R^3)\}, \\
\widehat{\widetilde{H}}(\curl;\mathbb M)&:=\{\boldsymbol{\tau}\in H(\curl;\mathbb M): \vskw\boldsymbol{\tau}\in H^1(\mathbb R^3), \curl\boldsymbol{\tau}\in \widetilde{H}(\div;\mathbb M)\}.
\end{align*}
The condition $\curl \widehat{\widetilde{H}}(\curl;\mathbb M) = \widetilde{H}(\div;\mathbb M)\cap \ker(\div)$ is again due to the existence of regular potential, i.e. $\curl \widetilde{H}^1(\mathbb M) = \widetilde{H}(\div; \mathbb M)\cap \ker(\div)$ induced from the tilde operation applied to $\curl H^1(\mathbb M) = H(\div; \mathbb M)\cap \ker(\div)$ with
\begin{equation*}
\widetilde{H}^1(\mathbb M):=\{\boldsymbol{\tau}\in H^1(\mathbb M): \curl\boldsymbol{\tau}\in \widetilde{H}(\div;\mathbb M)\},
\end{equation*}
 and the fact $\widetilde{H}^1(\mathbb M)\subset \widehat{\widetilde{H}}(\curl;\mathbb M)$. By Lemma \ref{lm:hat0}, the top complex of \eqref{eq:3DBGGdomainElasticity} is exact.

Then by the BGG construction applied to \eqref{eq:3DBGGdomainElasticity}, we obtain the elasticity complex
\begin{equation}\label{eq:elasticitycomplex1inc}
{\rm RM}\xrightarrow{\subset}  H^{1} (\curl) \xrightarrow{\defm}  H(\mathrm{inc}^+;\mathbb{S}) \xrightarrow{\inc}  H(\operatorname{div}; \mathbb{S}) \xrightarrow{\div} L^{2} (\mathbb R^3)\rightarrow \boldsymbol{0},
\end{equation}
where $\defm \boldsymbol u = \sym \grad \boldsymbol u, \; \inc \boldsymbol  \tau = \nabla \times \boldsymbol  \tau \times \nabla$, ${\rm RM} =  \{\boldsymbol{a} \times \boldsymbol{x} + \boldsymbol{b}: \boldsymbol{a}, \boldsymbol{b}\in \mathbb{R}^3\} $ is the space of the linearized rigid body motion, $H(\div;\mathbb S)=H(\div;\mathbb M)\cap L^2(\mathbb S)$, and the space
$$
H(\inc^+; \mathbb S):=H(\inc; \mathbb S)\cap H(\curl; \mathbb S)=\{ \boldsymbol \tau \in H (\curl; \mathbb S) : \inc \boldsymbol \tau \in {L}^2(\mathbb S)\}
$$
with $H(\inc; \mathbb S)=\{ \boldsymbol \tau \in L^2 (\mathbb S) : \inc \boldsymbol \tau \in {L}^2(\mathbb S)\}$.
Throughout this paper, we use the script ${}^+$ to denote extra smoothness. An $L^2$ function $\boldsymbol \tau \in H(\inc; \mathbb S)$ only requires $\inc \boldsymbol \tau \in L^2$, not necessarily $\curl \boldsymbol \tau \in L^2$.
The elasticity complex~\eqref{eq:elasticitycomplex1inc} is slightly smoother than the elasticity complex~\cite{Eastwood2000,ArnoldFalkWinther2006}
\begin{equation}\label{eq:elasticitycomplex}
{\rm RM}\xrightarrow{\subset}  H^{1} (\mathbb R^3) \xrightarrow{\defm}  H(\mathrm{inc}; \mathbb{S}) \xrightarrow{\inc}  H(\operatorname{div}; \mathbb{S}) \xrightarrow{\div} L^{2} (\mathbb R^3)\rightarrow \boldsymbol{0},
\end{equation}
which can be obtained from~\eqref{eq:elasticitycomplex1inc}  by an inverse $\widehat{\quad}$ operation. The fact $\defm H^1(\mathbb R^3) = (H(\mathrm{inc}, \mathbb{S})\cap \ker(\inc))$ can be derived from the elasticity complex with  Sobolev spaces 
$$
{\rm RM}\xrightarrow{\subset}  H^{1} (\mathbb R^3) \xrightarrow{\defm}  L^2(\mathbb{S}) \xrightarrow{\inc}  H^{-2}(\mathbb{S}) \xrightarrow{\div} H^{-3} (\mathbb R^3)\rightarrow \boldsymbol{0}
$$
derived by the BGG framework in \cite{Arnold;Hu:2020Complexes}, cf. diagram \eqref{eq:3DBGG}.

Similarly by applying the two reduction operations to the third and fourth rows of diagram~\eqref{eq:3DBGGdomain}, we obtain 
\begin{equation*}%\label{eq:3DBGGdomainElasticity}
\begin{tikzcd}
\mathbb{R}^3 \to H^1(\div) \arrow{r}{\grad} & \widehat{\widetilde{H}}(\curl;\mathbb M) \arrow{r}{\curl} & \widetilde{H}(\div;\mathbb M) \arrow{r}{\div} & H(\div) \to \boldsymbol{0} \\
\mathbb{R} \to H^1 \arrow[ur,swap,"\iota"'] \arrow{r}{\grad} & H(\curl) \arrow[ur,swap,"\mskw"'] \arrow{r}{\curl} & H(\div) \arrow[ur,swap,"\mathrm{id}"'] \arrow{r}{\div} \arrow[r] & L^2\to 0,
 \end{tikzcd}%,
\end{equation*}
where
\begin{align*}
H^1(\div)&:=\{\boldsymbol{v}\in H^1(\mathbb R^3): \div\boldsymbol{v}\in H^1\}, \\
H(\div\div^+;\mathbb S)&:=\{\boldsymbol{\tau}\in H(\div;\mathbb S): \div\boldsymbol{\tau}\in H(\div)\},\\
\widetilde{H}(\div;\mathbb M)&:=H(\div\div^+;\mathbb S)\oplus\mskw H(\curl), \\
\widehat{\widetilde{H}}(\curl;\mathbb M)&:=\{\boldsymbol{\tau}\in H(\curl;\mathbb M): \tr\boldsymbol{\tau}\in H^1, \curl\boldsymbol{\tau}\in\widetilde{H}(\div;\mathbb M)\}.
\end{align*}
This leads to the $\div\div$ complex
\begin{equation*}%\label{eq:elasticitycomplex1inc}
% \resizebox{.975\hsize}{!}{$
{\rm RT}\xrightarrow{\subset} H^1(\div)\xrightarrow{\dev\grad} H(\sym\curl_+^+;\mathbb T)\xrightarrow{\sym\curl} H(\div\div^+;\mathbb S) \xrightarrow{\div{\div}} L^2\xrightarrow{}0,
% $},
\end{equation*}
where $\dev \boldsymbol \sigma = \boldsymbol \sigma - \tr(\boldsymbol \sigma)\boldsymbol I/3,$ ${\rm RT}= \{a\boldsymbol x + \boldsymbol b: a\in \mathbb R, \boldsymbol b \in \mathbb R^3\}$, and the space
\begin{equation*}	
H(\sym\curl_+^+;\mathbb T):=\{ \boldsymbol \tau \in H (\curl; \mathbb T) : \curl \boldsymbol \tau \in \widetilde{H}(\div;\mathbb M)\}
\end{equation*}
requires additional smoothness on $\boldsymbol \tau$ and $\curl\boldsymbol \tau$ comparing with the definition of space
\begin{equation*}	
H(\sym\curl;\mathbb T):=\{ \boldsymbol \tau \in L^2(\mathbb T) : \sym\curl \boldsymbol \tau \in L^2(\mathbb S)\}.
\end{equation*}
% $$
% H(\sym\curl_+^+;\mathbb T):=\{ \boldsymbol \tau \in H (\curl; \mathbb T) : \sym\curl \boldsymbol \tau \in H(\div;\mathbb S)\}.
% $$
The derived $\div\div$ complex is slightly smoother than the $\div\div$ complex~\cite{Arnold;Hu:2020Complexes,PaulyZulehner2020}
\begin{equation*}%\label{eq:divdivcomplex3d}
% \resizebox{.975\hsize}{!}{$
{\rm RT}\xrightarrow{\subset} H^1(\mathbb R^3)\xrightarrow{\dev\grad} H(\sym\curl,\mathbb T)\xrightarrow{\sym\curl} H(\div\div,\mathbb S) \xrightarrow{\div{\div}} L^2\xrightarrow{}0
% $}.
\end{equation*}
with $H(\div\div,\mathbb S)=\{\boldsymbol{\tau}\in L^2(\mathbb S): \div\div\boldsymbol{\tau}\in L^2\}$,
which again can be obtained by an inverse hat operation.  

%Lemma \ref{lm:hat0} is employed to deduce the exactness of the reduced de Rham complexes. 
%% due to the fact that these spaces are not finite-dimensional. 
%A rigorous proof would involve employing the regular decomposition of \(H(\dd)\) spaces. 
% However, since our primary focus is on the finite element spaces, we omit this proof here. 
We would like to highlight that the complexes derived from the BGG framework typically encompass spaces that possess a slightly higher degree of smoothness.

%$\bullet$ 
%$\bullet$ 
%$\bullet$ 
%where 
%$$
%H(\dd,\Omega; \mathbb X) =\{ \boldsymbol \tau \in L^2(\Omega;\mathbb X): \textrm{ each component of } \dd \boldsymbol \tau \textrm{ belongs to } L^2(\Omega)\}
%$$
%with differential operators
%\begin{align*}
%\hess u = \nabla^2 u, \; \defm \boldsymbol u = \sym \grad \boldsymbol u, \; \inc \boldsymbol  \tau = - \curl (\curl \boldsymbol  \tau)^{\intercal}, \; \dev \boldsymbol \sigma = \boldsymbol \sigma - \tr(\boldsymbol \sigma)\boldsymbol I/3,
%\end{align*}
%and 
%$$
%\mathbb X = \mathbb R, \mathbb R^3, \text{ symmetric matrix } \mathbb S,  \text{ traceless matrix }  \mathbb T, \text{ or } 3\times 3 \text{ matrix } \mathbb M,
%$$
%and  
%$\mathbb P_k$ is the space of polynomials of degree no more than $k$, , and 
%%Notice that $H(\hess,\Omega; \mathbb R)=H^{2}(\Omega)$.
%Compared with the de Rham complex~\eqref{eq:deRham}, those derived complexes involve Sobolev spaces for tensor functions which are much harder to discretize. 

%%%%%%%%%%%%%%%%%%%%%%%%%%%%
\section{Smooth Finite Element de Rham Complexes}\label{sec:femderhamcomplex}
%%%%%%%%%%%%%%%%%%%%%%%%%%%%
In this section we shall review the smooth finite element de Rham complexes developed in~\cite{Chen;Huang:2022FEMcomplex3D}. We construct finite element spaces with different smoothness at vertices, edges, and faces which is characterized by a smoothness vector. 

An integer vector $\boldsymbol r= (r^{\texttt{v}}, r^e, r^f)^{\intercal}$ is called a  smoothness vector if $r^f\geq-1$, $r^e\geq \max\{2r^f,-1\}$ and $r^{\texttt{v}}\geq \max\{2r^e,-1\}$. Its restriction $(r^{\texttt{v}}, r^e)^{\intercal}$ is a two-dimensional smoothness vector. For a smoothness vector $\boldsymbol r$ and positive integer $m$, define $\boldsymbol r\ominus m: = \max \{\boldsymbol r-m, -1\}$ and $\boldsymbol r_+ = \max\{\boldsymbol r, \boldsymbol 0\}$ where the $\max$ operator is applied component-wise. 

%and will write as $H(\dd;\Omega)$ for $\dd = \grad, \curl,$ or $\div$. 
%

\subsection{Smooth bubble functions}\label{sec:smoothbubble}
For edge $e$, let $r^{\texttt{v}}\geq 0$ and $k\geq 2r^{\texttt{v}}+1$, define edge bubble polynomial space
$$
\mathbb B_k(e; r^{\texttt{v}}) : = \{ u\in\mathbb P_k(e): \, \partial_t^ju \text{ vanishes at all vertices of } e \text{ for } j=0,\ldots, r^{\texttt{v}} \},
$$
where $\partial_t$ is the tangential derivative along $e$. This bubble space can be easily characterized as
$\mathbb B_{k}(e; r^{\texttt{v}}) =  b_e^{r^{\texttt{v}}+1} \mathbb P_{k-2(r^{\texttt{v}}+1)}(e),$ where $b_e\in \mathbb P_2(e)$ vanishes at two vertices of $e$. 

For triangle $f$ and a smoothness vector $\boldsymbol r= (r^{\texttt{v}}, r^e)^{\intercal}$, define face bubble polynomial space
\begin{align*}
\mathbb B_{k}(f; 
\begin{pmatrix}
r^{\texttt{v}}\\
r^e
\end{pmatrix}
):=\{u\in\mathbb P_k(f):&\, \nabla_f^ju \textrm{ vanishes at all vertices of $f$ for $j=0,\ldots, r^{\texttt{v}}$}, \\
& \textrm{ and $\nabla_f^ju$ vanishes on all edges of $f$ for $j=0,\ldots, r^{e}$}\},
\end{align*}
where $\nabla_f$ is the surface gradient on $f$. All polynomials defined on $e$ and $f$ can be naturally extended to the whole tetrahedron using the Bernstein basis of polynomials. 

For a tetrahedron \(T\) and a smoothness vector \(\boldsymbol{r}= (r^{\texttt{v}}, r^e, r^f)^{\intercal}\), define the bubble polynomial space as follows:
\begin{align*}
\mathbb B_{k}(T;\boldsymbol{r}) := \{u \in \mathbb P_k(T) : &\, \nabla^j u \text{ vanishes at all vertices of } T \text{ for } j=0,\ldots, r^{\texttt{v}}, \\
& \nabla^j u \text{ vanishes on all edges of } T \text{ for } j=0,\ldots, r^e, \\
& \nabla^j u \text{ vanishes on all faces of } T \text{ for } j=0,\ldots, r^f\}.
\end{align*}
When \(r^f = -1\), the bubble function may not vanish on the boundary of \(T\).

For simplicity of notation, for a three-dimensional smoothness vector \(\boldsymbol{r} = (r^{\texttt{v}}, r^e, r^f)^{\intercal}\), define \(\mathbb B_{k}(f; \boldsymbol{r}) := \mathbb B_{k}(f; (r^{\texttt{v}}, r^e)^{\intercal})\), which is the face bubble using the restriction of \(\boldsymbol{r}\) on \(f\). Similarly, \(\mathbb B_{k}(e; \boldsymbol{r}) = \mathbb B_{k}(e; r^{\texttt{v}})\).

A precise characterization of the bubble polynomial spaces \(\mathbb B_{k}(f; \boldsymbol{r})\) and \(\mathbb B_{k}(T; \boldsymbol{r})\) can be obtained through decompositions of simplicial lattice points (see~\cite{Chen;Huang:2022femcomplex2d,Chen;Huang:2022FEMcomplex3D} for more details):
\begin{align*}
\mathbb B_k(f; \boldsymbol{r}) &= 
\mathbb B_k(f; \boldsymbol{r}_+) \Oplus_{e\in \Delta_1(T)} [r^e=-1] \mathbb B_k(e; \boldsymbol{r}_+) \Oplus [r^{\texttt{v}}=-1]\mathbb P_1(f), \\
\mathbb B_k(T; \boldsymbol{r}) &= 
\mathbb B_k(T; \boldsymbol{r}_+) \Oplus_{f\in \Delta_2(T)} 
[r^f=-1] \mathbb B_k(f; \boldsymbol{r}_+) \Oplus_{e\in \Delta_1(T)} [r^e=-1] \mathbb B_k(e; \boldsymbol{r}_+) \\
&\hskip 1.95cm \Oplus [r^{\texttt{v}}=-1]\mathbb P_1(T).
\end{align*}

For a vector space \(V\), we abbreviate \(V \otimes \mathbb{R}^3\) as \(V^3\). We define the following bubble spaces:
\begin{align*}
\mathbb B^{\curl}_{k}(T; \boldsymbol{r}) := {} & \{\boldsymbol{v} \in \mathbb B^{3}_{k}(T; \boldsymbol{r}) : \boldsymbol{v} \times \boldsymbol{n}|_{\partial T} = \boldsymbol{0}\}, \\
\mathbb B^{\div}_{k}(T; \boldsymbol{r}) := {} & \{\boldsymbol{v} \in \mathbb B^{3}_{k}(T; \boldsymbol{r}) : \boldsymbol{v} \cdot \boldsymbol{n}|_{\partial T} = 0\}, \\
\mathbb B^{L^2}_{k}(T; \boldsymbol{r}) := {} & \mathbb B_{k}(T; \boldsymbol{r}) \cap L_0^2(T).
\end{align*}
Typically, \(T\) will be omitted from the notation, i.e., \( \mathbb B_{k}(\boldsymbol{r}) = \mathbb B_{k}(T; \boldsymbol{r})\). When \(r^f \geq 0\), functions in \( \mathbb B_{k}(\boldsymbol{r})\) vanish on \(\partial T\), thus \(\mathbb B^{\curl}_{k}(\boldsymbol{r}) = \mathbb B^{\div}_{k}(\boldsymbol{r}) =  \mathbb B^{3}_{k}(\boldsymbol{r})\). When \(r^f = -1\), we have \(\mathbb B^{3}_{k}(\boldsymbol{r}_+) \subset  \mathbb B^{\dd}_{k}(\boldsymbol{r}) \subset \mathbb B^{3}_{k}(\boldsymbol{r})\), for \(\dd = \curl\) or \(\div\), as only the tangential or normal components of \(\mathbb B^{\dd}_{k}(\boldsymbol{r})\) vanish, respectively.

For each edge \(e\), we choose a tangential vector \(\boldsymbol{t}_e\) and two normal vectors \(\boldsymbol{n}_1^e\) and \(\boldsymbol{n}_2^e\), abbreviated as \(\boldsymbol{t}\), \(\boldsymbol{n}_1\), and \(\boldsymbol{n}_2\). For each face \(f\), we select a normal vector \(\boldsymbol{n}_f\) and two tangential vectors \(\boldsymbol{t}_1^f\) and \(\boldsymbol{t}_2^f\), abbreviated as \(\boldsymbol{n}\), \(\boldsymbol{t}_1\), and \(\boldsymbol{t}_2\) when the face \(f\) is clear from the context. In a conforming mesh \(\mathcal{T}_h\), \(\boldsymbol{n}_1^e\), \(\boldsymbol{n}_2^e\), or \(\boldsymbol{n}_f\) depend on the edge \(e\) or the face \(f\), not the element containing them. In expressions such as \(\partial_n u\), we use the regular font \(n\) rather than the boldface \(\boldsymbol{n}\).
For $T \in \mathcal{T}_h$ and $0 \leq \ell \leq 3$, we denote by $\Delta_{\ell}(T)$ and $\Delta_{\ell}(\mathcal{T}_h)$ the sets of all $\ell$-dimensional subsimplices of $T$ and $\mathcal{T}_h$, respectively.

The bubble spaces \(\mathbb B^{\curl}_{k}(T; \boldsymbol{r})\) and \(\mathbb B^{\div}_{k}(T; \boldsymbol{r})\) have the following decomposition given in \cite{Chen;Huang:2022FEMcomplex3D,ChenChenHuangWei2023}
 %can be also found in~\cite{Chen;Huang:2022FEMcomplex3D}.
\begin{equation}
\label{eq:curlbubbledecomp}
\mathbb B_k^{\curl}(T; \boldsymbol{r}) = 
\mathbb B_k^3(T; \bs r_+) \Oplus_{f\in \Delta_{2}(T)} 
[r^f=-1]\big(\mathbb B_{k}(f; \bs r_+)\otimes{\rm span}\{\boldsymbol{n}_f\}\big),
\end{equation}
\begin{align}
\label{eq:divbubbledecomp}
\mathbb B_k^{\div}(T; \boldsymbol{r}) = {}\,
\mathbb B_k^3(T; \bs r_+) &\Oplus_{f\in \Delta_{2}(T)} 
[r^f=-1]\big(\mathbb B_{k}(f; \bs r_+)\otimes{\rm span}\{\boldsymbol{t}_1^f,\boldsymbol{t}_2^f\}\big)\\
&\Oplus_{e\in \Delta_1(T)} [r^e=-1]
\big(\mathbb B_{k}(e; \bs r_+)\otimes{\rm span}\{\boldsymbol{t}_e\}\big).
\notag 
\end{align}
Dimension of $\mathbb B_k^{\curl}(T; \boldsymbol{r})$ and $\mathbb B_k^{\div}(T; \boldsymbol{r})$ can be calculated based on \eqref{eq:dimBr} in Appendix~\ref{sec:bubblecomplex}.

For an $f\in\Delta_2(T)$ and a smooth vector $\boldsymbol{r}$, 
define bubble spaces on face $f$
\begin{align*}
\mathbb B^{\div_f}_{k}(f; \boldsymbol{r}):={}&\{\boldsymbol{v}\in \mathbb B_{k}^2(f;\boldsymbol{r}): \boldsymbol{v}\cdot\boldsymbol{n}|_{\partial f}=0\} \\
={}& \mathbb B_{k}^2(f;\boldsymbol{r}_+) \Oplus_{e\in \Delta_1(f)} [r^e=-1]
\big(\mathbb B_{k}(e; \bs r_+)\otimes{\rm span}\{\boldsymbol{t}_e\}\big),\\
\mathbb B^{\rot_f}_{k}(f; \boldsymbol{r}):={}&\{\boldsymbol{v}\in \mathbb B_{k}^2(f;\boldsymbol{r}): \boldsymbol{v}\cdot\boldsymbol{t}|_{\partial f}=0\} \\
={}& \mathbb B_{k}^2(f;\boldsymbol{r}_+) \Oplus_{e\in \Delta_1(f)} [r^e=-1]
\big(\mathbb B_{k}(e; \bs r_+)\otimes{\rm span}\{\boldsymbol{n}_e\}\big).
\end{align*}
%Let $\boldsymbol r_0 \geq 0, \boldsymbol r_1 = \boldsymbol r_0 -1, \boldsymbol r_2\geq\boldsymbol r_1\ominus1, \boldsymbol r_3\geq \boldsymbol r_2\ominus1$. 
%Assume all $\boldsymbol r_i$ are valid smoothness vectors, i.e.,
%%$$
%%r_1^{\texttt{v}}\geq2r_1^e+1,\; r_1^{e}\geq2r_1^f+1,\; r_2^{\texttt{v}}\geq2r_2^e,\; r_2^{e}\geq2r_2^f, \; r_3^{\texttt{v}}\geq2r_3^e,\; r_3^{e}\geq2r_3^f.
%%$$
%Assume $(\boldsymbol r_2, \boldsymbol r_3)$ is a div stable pair. 
%Assume $k\geq\max\{2 r_1^{\texttt{v}} + 1,2 r_2^{\texttt{v}} + 1,2 r_3^{\texttt{v}} + 2,1\} 
%$, the finite element bubble complex
%\begin{equation}\label{eq:femderhambubblecomplex3dgeneral}
%0\xrightarrow{\subset}\mathbb B^{\grad}_{k+2}(\boldsymbol{r}_0)\xrightarrow{\grad}\mathbb B^{\curl}_{k+1}(\boldsymbol{r}_1,\boldsymbol{r}_2)\xrightarrow{\curl}\mathbb B^{\div}_{k}(\boldsymbol{r}_2,\boldsymbol{r}_3)\xrightarrow{\div}\mathbb B^{L^2}_{k-1}(\boldsymbol{r}_3)\to0
%\end{equation}
%is exact.

\subsection{Bubble de Rham complexes}
The bubble spaces will form a de Rham complex. As it is not explicitly stated in~\cite{Chen;Huang:2022femcomplex2d,Chen;Huang:2022FEMcomplex3D}, we present the result below and provide a detailed proof in Appendix~\ref{sec:bubblecomplex}. 

\begin{lemma}
Let smoothness vectors $\boldsymbol{r}_1\geq-1$, $\boldsymbol{r}_0 = \boldsymbol{r}_1+ 1$, and $\boldsymbol{r}_2=\boldsymbol{r}_1\ominus1$. Let $f\in\Delta_2(T)$ and $k\geq\max\{2r_1^{\texttt{v}}-1,0\}$.
Then the bubble de Rham complexes 
\begin{equation*}%\label{eq:femderhambubblecomplex2d}
0\xrightarrow{\subset}\mathbb B_{k+2}(f;\boldsymbol{r}_0)\xrightarrow{\curl_f}\mathbb B^{\div_f}_{k+1}(f; \boldsymbol{r}_1)\xrightarrow{\div_f}\mathbb B_{k}(f;\boldsymbol{r}_2)/\mathbb R\to0,
\end{equation*}
% and
\begin{equation*}%\label{eq:rotfemderhambubblecomplex2d}
0\xrightarrow{\subset}\mathbb B_{k+2}(f;\boldsymbol{r}_0)\xrightarrow{\grad_f}\mathbb B^{\rot_f}_{k+1}(f; \boldsymbol{r}_1)\xrightarrow{\rot_f}\mathbb B_{k}(f;\boldsymbol{r}_2)/\mathbb R\to0
\end{equation*}
are exact.
\end{lemma}

When move to three dimensions, we require the following condition on a smoothness vector $\boldsymbol r = (r^{\texttt{v}}, r^e, r^f)^{\intercal}$:
\begin{equation}\label{eq:boundr2fordivbubble}
\begin{cases}
 r^f\geq 0, & \,r^e\geq2r^f+1\geq 1, \quad r^{\texttt{v}}\geq 2r^e\geq 2,\\
 r^f = -1, & 
\begin{cases}
 r^e \geq 1, & r^{\texttt{v}}\geq 2r^e\geq 2,\\
 r^e \in \{0, -1\}, & r^{\texttt{v}}\geq 2r^e + 1.
\end{cases}
\end{cases}
\end{equation}
\begin{lemma}[Theorem 4.5 in~\cite{Chen;Huang:2022FEMcomplex3D}]\label{lem:divbubbleonto}
Let $\boldsymbol r$ be a smoothness vector satisfying \eqref{eq:boundr2fordivbubble}.
Assume $k\geq\max\{2r^{\texttt{v}}, 1\}$.
Then we have the div stability
\begin{equation}\label{eq:divbubbleonto}
\div\mathbb B^{\div}_{k}(\boldsymbol{r})=\mathbb B_{k-1}(\boldsymbol{r}\ominus1)/\mathbb R.
\end{equation}
\end{lemma}

%Notice that as only the bubble spaces are considered, the degree of polynomial  $k\geq\max\{2r^{\texttt{v}}, 1\}$ is lower than the requirement in the div stability \eqref{eq:divonto3dsimple}. 

We now present the bubble de Rham complexes in three dimensions.
\begin{lemma}
Let $\boldsymbol r_0 \geq 0, \bs r_1 = \bs r_0 - 1, \boldsymbol r_2=\boldsymbol r_1\ominus1, \boldsymbol r_3=\boldsymbol r_2\ominus1$ be smoothness vectors.
Assume $\boldsymbol r_2$ satisfies \eqref{eq:boundr2fordivbubble}, and $k\geq\max\{2r_{2}^{\texttt{v}}, 1\}$.
Then the bubble de Rham complex 
\begin{equation*}%\label{eq:femderhambubblecomplex}
0\xrightarrow{\subset}\mathbb B_{k+2}(\boldsymbol{r}_0)\xrightarrow{\grad}\mathbb B^{\curl}_{k+1}(\boldsymbol{r}_1)\xrightarrow{\curl}\mathbb B^{\div}_{k}(\boldsymbol{r}_2)\xrightarrow{\div}\mathbb B_{k-1}(\boldsymbol{r}_3)/\mathbb R\to0
\end{equation*}
is exact.
\end{lemma}

\subsection{Smooth scalar finite elements}

% \begin{theorem}\label{thm:Cr3dfemunisolvence}
 Let $\boldsymbol r= (r^{\texttt{v}}, r^e, r^f)^{\intercal}$ be a smoothness vector, and nonnegative integer $k\geq 2r^{\texttt{v}}+1$. The shape function space $\mathbb P_{k}(T)$ is determined by the DoFs
\begin{subequations}\label{eq:Cr3D}
\begin{align}
\label{eq:C13d0}
\nabla^j u (\texttt{v}), & \quad j=0,1,\ldots,r^{\texttt{v}}, \texttt{v}\in \Delta_0(T), \\
\label{eq:C13d1}
\int_e \frac{\partial^{j} u}{\partial n_1^{i}\partial n_2^{j-i}} \, q \dd s, & \quad q \in \mathbb B_{k-j}(e; r^{\texttt{v}} - j), 0\leq i\leq j\leq r^{e}, e\in \Delta_1(T), \\
\label{eq:C13d2}
\int_f \frac{\partial^{j} u}{\partial n_f^{j}} \, q \dd S, & \quad q \in \mathbb B_{k-j}(f;\boldsymbol r-j), 0\leq j\leq r^{f},  f\in \Delta_2(T), \\
\label{eq:C13d3}
\int_T u \, q \dx, & \quad q \in \mathbb B_k(T;\boldsymbol r).
\end{align} 
\end{subequations}
As $b_e\geq 0$, the test function space in~\eqref{eq:C13d1} can be changed to $q\in \mathbb P_{k - 2(r^{\texttt{v}}+1) + j}(e)$. 

For the sake of simplifying notation, we use \( \text{DoF}_k(\boldsymbol{r}) \) to represent the set of DoFs as defined in~\eqref{eq:Cr3D}, and \( \text{DoF}_k(s; \boldsymbol{r}) \) for the subset corresponding to the sub-simplex \( s \). The unisolvence can be expressed as:
$$
\mathbb P_k(T) \text{ is uniquely determined by } \text{DoF}_k(\boldsymbol{r}).
$$
%Namely a function \( u \in \mathbb P_k(T) \) is uniquely determined by the associated set of DoFs \( \text{DoF}_k(\boldsymbol{r}) \).

When considering a mesh \( \mathcal T_h \), the DoFs~\eqref{eq:Cr3D} define the global \( C^{r^f} \)-continuous finite element space as follows:
\begin{align*}
\mathbb V_k(\mathcal T_h; \boldsymbol{r}) = \{u\in C^{r^f}(\Omega): & \, u|_T\in\mathbb P_k(T)\textrm{ for all } T\in\mathcal T_h, \\
&\qquad\textrm{ and all the DoFs~\eqref{eq:Cr3D} are single-valued}\}.
\end{align*}
In cases where \( \mathbb V_k(\mathcal T_h; \boldsymbol{r}) \) is used as a subspace of \( H^1(\Omega) \) or \( L^2(\Omega) \), notation \( \mathbb V_k^{\grad}(\mathcal T_h; \boldsymbol{r}) \) or \( \mathbb V_k^{L^2}(\mathcal T_h; \boldsymbol{r}) \) are employed respectively. The reference to the mesh \( \mathcal T_h \) will subsequently be omitted in the notation to emphasize the dependence on the smoothness vector $\boldsymbol r$.

% \end{theorem}
%We have the dimension identity
%\begin{align*}
%&\dim\mathbb V_{k}(\mathcal T_h; \boldsymbol{r})=\dim\mathbb V_{k}(\mathcal T_h; (\boldsymbol{r})_+) - \chi(r^{\texttt{v}}=-1)|\Delta_0(\mathcal T_h)| + 4\chi(r^{\texttt{v}}=-1)|\Delta_3(\mathcal T_h)| \\
%&\quad - \chi(r^e=-1)(k - 2r_+^{\texttt{v}}-1)|\Delta_1(\mathcal T_h)|- \chi(r^f=-1)\dim\mathbb B_{k} (f;\begin{pmatrix}
%r^{\texttt{v}} \\
%r^e
%\end{pmatrix}_+)|\Delta_2(\mathcal T_h)| \\
%&\quad+\big(6\chi(r^e=-1)(k - 2r_+^{\texttt{v}}-1) + 4\chi(r^f=-1)\dim\mathbb B_{k} (f;\begin{pmatrix}
%r^{\texttt{v}} \\
%r^e
%\end{pmatrix}_+)\big)|\Delta_3(\mathcal T_h)|.
%\end{align*}

\subsection{$H(\div)$-conforming finite elements}
Let $\boldsymbol{r}_3\geq \boldsymbol r_2\ominus 1$ and positive integer $k\geq \max \{ 2r_2^{\texttt{v}}+1, 2r_3^{\texttt{v}}+2\}$. We define the space 
$$
\mathbb V^{\div}_{k}(\boldsymbol r_2, \boldsymbol r_3) := \{ \boldsymbol v\in \mathbb V^3_{k}(\boldsymbol r_2)\cap H(\div,\Omega):  \div \boldsymbol v\in \mathbb V_{k-1}^{L^2}(\boldsymbol r_3)\},
$$
%and give finite element description of $\mathbb V^{\div}_{k}(\boldsymbol r_2, \boldsymbol r_3)$ in~\cite[Section 4.4]{Chen;Huang:2022FEMcomplex3D}.
%
and abbreviate $\mathbb V_{k}^{\div}(\boldsymbol r_2, \boldsymbol r_2\ominus1)= \mathbb V^3_{k}(\boldsymbol r_2)\cap H(\div,\Omega)$ as $\mathbb V_{k}^{\div}(\boldsymbol r_2)$. We construct $\mathbb V_{k}^{\div}(\boldsymbol r_2)$ by using the $t$-$n$ decomposition approach developed in \cite{Chen;Huang:2021Geometric}.

In order to have a stable discretization of Stokes equations, it is crucial to have the surjectivity of the div operator in view of Babu\v{s}ka-Brezzi condition~\cite{BoffiBrezziFortinothers2013finite}, which is thus called the div stability; see for example~\cite{Gunzbu1996Navier-Stokes}.
The div stability \(\div\mathbb V^{\div}_{k}(\boldsymbol{r}_2, \boldsymbol{r}_3) = \mathbb V^{L^2}_{k-1}(\boldsymbol r_3)\) is established under certain restrictions on \((\boldsymbol r_2, \boldsymbol r_3)\) and for sufficiently large values of \(k\).

\begin{lemma}[Theorem 4.10 in~\cite{Chen;Huang:2022FEMcomplex3D}] \label{thm:divonto}
% Let $\tilde{\boldsymbol r}_2, \boldsymbol r_2$ and $\boldsymbol r_3$ be three smoothness vectors. 
Let $\boldsymbol r_2$ and $\boldsymbol r_3$ be two smoothness vectors. 
Assume $\boldsymbol r_2$ satisfies \eqref{eq:boundr2fordivbubble}
%\begin{equation*}%\label{eq:boundr2fordivbubble}
%\begin{cases}
% r_2^f\geq 0, & \,r_{2}^e\geq2r_{2}^f+1, \quad r_{2}^{\texttt{v}}\geq 2r_{2}^e,\\
% r_2^f = -1, & 
%\begin{cases}
% r_2^e \geq 1, & r_{2}^{\texttt{v}}\geq 2r_{2}^e,\\
% r_2^e \in \{0, -1\}, & r_{2}^{\texttt{v}}\geq 2r_{2}^e + 1,
%\end{cases}
%\end{cases}
%%r_{2}^{\texttt{v}}\geq\max\{2r_{2}^e, \ceil{(3r_{2}^e+1)/2}\}, \quad r_{2}^e\geq2r_{2}^f+1, \quad r_{2}^f\geq-1, 
%\end{equation*}
and $\boldsymbol r_3 \geq \boldsymbol r_2\ominus 1$. Assume $k\geq\max\{2r_{2}^{\texttt{v}}+1, r_{2}^{\texttt{v}}+2, 3(r_2^e+1), 2r_{3}^{\texttt{v}}+2, 4r_3^f+5, (r_3^e+r_3^f+5)[r_{3}^{\texttt{v}}=0]\}.$ 
It holds that
\begin{equation}\label{eq:divonto3dsimple}
\div\mathbb V^{\div}_{k}(\boldsymbol{r}_2,\boldsymbol{r}_3)=\mathbb V^{L^2}_{k-1}(\boldsymbol{r}_3).   
\end{equation}
\end{lemma}
%\begin{proof}
%{The div stability $\div\mathbb V^{\div}_{k}(\boldsymbol{r}_2,\boldsymbol{r}_3)=\mathbb V^{L^2}_{k-1}(\boldsymbol{r}_3) 
%$ is proved in Theorem 4.10 in~\cite{Chen;Huang:2022FEMcomplex3D}. 
%For a smoothness vector $\tilde{\boldsymbol r}_2\leq \boldsymbol r_2$, \eqref{eq:divonto3dsimple} follows from $\mathbb V^{\div}_{k}(\boldsymbol{r}_2,\boldsymbol{r}_3)\subseteq \mathbb V^{\div}_{k}(\tilde{\boldsymbol{r}}_2,\boldsymbol{r}_3)$ as the smoothness is relaxed.}
%\end{proof}

When~\eqref{eq:divonto3dsimple} holds, we shall call $(\boldsymbol r_2, \boldsymbol r_3, k)$ div stable and we give finite element description of such space $\mathbb V^{\div}_{k}(\boldsymbol r_2, \boldsymbol r_3)$ in~\cite[Section 4.4]{Chen;Huang:2022FEMcomplex3D}. We emphasize that for continuous div element, i.e., $r^f\geq 0$, the minimal $\boldsymbol r_2$ satisfying \eqref{eq:boundr2fordivbubble} is $\boldsymbol{r}_2=(2,1,0)^{\intercal}$. Consequently $\boldsymbol{r}_3=(1,0,-1)^{\intercal}$ and $k\geq 6$. The corresponding Stokes element $\mathbb V^{\div}_{k}( (2,1,0)^{\intercal})\times \mathbb V^{L^2}_{k-1}((1,0,-1)^{\intercal}) $ is firstly constructed by Neilan~\cite{Neilan2015}. 
%{Using the inequality constraint, we can further relax to $\tilde{\boldsymbol r}_2 = (0,0,0)^{\intercal}$ and $\boldsymbol r_3 = (1, 0, 0)^{\intercal}$. The space for the pressure is a Hermite element and the velocity is a subspace of the standard Lagrange element such that $\div \boldsymbol u\in \mathbb V_{k-1}((1, 0, 0)^{\intercal})$.}

%\begin{remark}\rm 
%{
%% So the div stability holds by changing $\boldsymbol r_2$ to $\tilde{\boldsymbol r}_2$ in \eqref{eq:divonto3dsimple}. If this is true, we have a lot of freedom on the div stability. For example, we can relax Neilan's element to Lagrange element with additional DoFs to impose the smoothness of div u.
% }
%\end{remark}

\subsection{$H(\curl)$-conforming finite elements}\label{sec:Hcurl}
Let $\boldsymbol{r}_2\geq \boldsymbol r_1\ominus 1$ be two smoothness vectors. Next we introduce
$$
\mathbb V^{\curl}_{k+1}(\boldsymbol r_1, \boldsymbol r_2) := \{ \boldsymbol v\in \mathbb V^3_{k+1}(\boldsymbol r_1)\cap H(\curl,\Omega): \curl \boldsymbol v\in \mathbb V_{k}^{\div}(\boldsymbol r_2)\}, 
$$
and will abbreviate $\mathbb V_{k+1}^{\curl}(\boldsymbol r_1, \boldsymbol r_1\ominus1)$ as $\mathbb V_{k+1}^{\curl}(\boldsymbol r_1)$. We give finite element description of  $\mathbb V_{k+1}^{\curl}(\boldsymbol r_1)$, i.e., local DoFs for the shape function space $\mathbb P_{k+1}(T;\mathbb R^3)$ in~\cite[Section 5.3]{Chen;Huang:2022FEMcomplex3D}. %\mnote{ any implicit assumption on the smoothness vector $\bs r_1, \bs r_2$?}

\subsection{Finite element de Rham complexes in three dimensions}

\begin{theorem}[Theorem 5.9 in~\cite{Chen;Huang:2022FEMcomplex3D}]\label{thm:femderhamcomplex3dgeneral}
Let $\boldsymbol r_0 \geq 0, \boldsymbol r_1 = \boldsymbol r_0 -1, \boldsymbol r_2\geq\boldsymbol r_1\ominus1, \boldsymbol r_3\geq \boldsymbol r_2\ominus1$ be smoothness vectors. %Assume all $\boldsymbol r_i$ are valid smoothness vectors, i.e.,
%$$
%r_1^{\texttt{v}}\geq2r_1^e+1,\; r_1^{e}\geq2r_1^f+1,\; r_2^{\texttt{v}}\geq2r_2^e,\; r_2^{e}\geq2r_2^f, \; r_3^{\texttt{v}}\geq2r_3^e,\; r_3^{e}\geq2r_3^f.
%$$
Assume $(\boldsymbol r_2, \boldsymbol r_3,k)$ is div stable. Assume $k\geq\max\{2r_{1}^{\texttt{v}}+1, 2r_{2}^{\texttt{v}}+1, r_{2}^{\texttt{v}}+2, 3(r_2^e+1), 2r_{3}^{\texttt{v}}+2, 4r_3^f+5, (r_3^e+r_3^f+5)[r_{3}^{\texttt{v}}=0]\}$. Then the finite element de Rham complex
\begin{equation}\label{eq:femderhamcomplex3dgeneral}
\mathbb R\xrightarrow{\subset}\mathbb V^{\grad}_{k+2}(\boldsymbol{r}_0)\xrightarrow{\grad}\mathbb V^{\curl}_{k+1}(\boldsymbol{r}_1,\boldsymbol{r}_2)\xrightarrow{\curl}\mathbb V^{\div}_{k}(\boldsymbol{r}_2,\boldsymbol{r}_3)\xrightarrow{\div}\mathbb V^{L^2}_{k-1}(\boldsymbol{r}_3)\to0
\end{equation}
is exact. 
%And the polynomial bubble complex 
%\begin{equation*}%\label{eq:femderhambubblecomplex3dgeneral}
%0\xrightarrow{\subset}\mathbb B^{\grad}_{k+2}(\boldsymbol{r}_0)\xrightarrow{\grad}\mathbb B^{\curl}_{k+1}(\boldsymbol{r}_1,\boldsymbol{r}_2)\xrightarrow{\curl}\mathbb B^{\div}_{k}(\boldsymbol{r}_2,\boldsymbol{r}_3)\xrightarrow{\div}\mathbb B^{L^2}_{k-1}(\boldsymbol{r}_3)\to0
%\end{equation*}
%is also exact.
 \end{theorem}
 
We refer to a parameter sequence $(\boldsymbol r_0, \boldsymbol r_1, \boldsymbol r_2, \boldsymbol r_3, k)$ as a valid de Rham parameter sequence if~\eqref{eq:femderhamcomplex3dgeneral} holds with exactness in~\cite{Chen;Huang:2022FEMcomplex3D}. We also provide the finite element description of the space $\mathbb V^{\curl}_{k+1}(\boldsymbol{r}_1, \boldsymbol{r}_2)$ in \eqref{eq:femderhamcomplex3dgeneral}.

When $r_2^f \geq 0$,~\eqref{eq:femderhamcomplex3dgeneral} transforms into a finite element Stokes complex, as the space $\mathbb V^{\div}_{k}(\boldsymbol{r}_2, \boldsymbol{r}_3) \subset H^1(\Omega; \mathbb R^3)$. This  allows for the discretization of the Stokes equation. Notably, existing works on finite element Stokes complexes~\cite{Neilan2015} and finite element de Rham complexes~\cite{Christiansen;Hu;Hu:2018finite} are specific instances of~\eqref{eq:femderhamcomplex3dgeneral}, depending on the selection of different smoothness vectors.

\section{Face Elements}\label{sec:faceelements}
In this section, our objective is to construct finite elements conforming to \(H(\div; \mathbb X)\) for either \(\mathbb X = \mathbb S\) or \(\mathbb T\). We will use the $t$-$n$ decomposition approach introduced in~\cite{Chen;Huang:2021Geometric} to construct the finite elements, and we will subsequently leverage the BGG framework to establish the divergence stability property.

\subsection{Smooth $H(\div; \mathbb T)$ and $H(\div; \mathbb S)$ finite elements}\label{sec:divX}
Given a smoothness vector \(\boldsymbol r\) with \(r^{\texttt{v}}\geq 0\), we examine the space associated with \(\boldsymbol r_+\geq 0\), leading to the unisolvence condition:
\begin{equation}\label{eq:r+unisolvence}
\mathbb P_k(T) \otimes \mathbb X \text{ is uniquely determined by } {\rm DoF}_k(\boldsymbol r_+)\otimes \mathbb X,
\end{equation}
and its global extension, \(\mathbb V_k(\boldsymbol r_+) \otimes \mathbb X\). For $\bs\sigma \in \mathbb P_k(T) \otimes \mathbb X$, as $\tr^{\div}\boldsymbol \sigma = \boldsymbol \sigma\boldsymbol n$, to be $H(\div)$-conforming, $\boldsymbol \sigma\boldsymbol n$ must be continuous across the faces of the triangulation. In cases where \(r^f = -1\) or \(r^e = -1\), we need to adjust the continuity of the finite element corresponding to \(\boldsymbol r_+\) by transferring the tangential component into the bubble space.
The key lies in an appropriate $t$-$n$ decomposition of the tensor \(\mathbb X\) at a sub-simplex \(s\):
\begin{equation*}%\label{eq:tnX}
\mathbb X = \mathscr T^s(\mathbb X) \oplus \mathscr N^s(\mathbb X),
\end{equation*}
where \(\mathscr T^s\) and \(\mathscr N^s\) represent the tangential and normal planes of \(s\), respectively.
%The terms ``tangential" and ``normal" refer to the second component in the tensor product \(\boldsymbol u\otimes \boldsymbol v\) representation of a matrix. 
%Due to the constraints imposed on \(\mathbb X\) (e.g., symmetry or tracelessness), the tensor product structure of \(\mathbb M\) is lost, rendering the decomposition~\eqref{eq:tnX} nontrivial. Nonetheless, a systematic decomposition has been provided in~\cite{Chen;Huang:2021Geometric}. 

For a sub-simplex $s$, we use the following decomposition of DoFs associated to $s$
\begin{equation}\label{eq:DoFdec}
\resizebox{0.92\textwidth}{!}{$
{\rm DoF}(s; \boldsymbol r_+)\otimes \mathbb X = 
\begin{cases}
{\rm DoF}(s; \boldsymbol r)\otimes \mathbb X, \qquad\qquad\qquad\qquad\qquad\;\; \text{ when } r^s\geq 0, \\
\medskip 
\begin{cases}
{\rm DoF}(s; \boldsymbol r_+) \otimes \mathscr N^s(\mathbb X) & \text{normal trace},\\
\oplus  & \qquad\qquad\qquad \text{ when } r^s = -1. \\
{\rm DoF}(s; \boldsymbol r_+) \otimes \mathscr T^s(\mathbb X) & \text{div bubbles},
\end{cases}
\end{cases}
$}
\end{equation}
We relocate the tangential component into the $\div$ bubble space and introduce
\begin{align} \label{eq:divbubbler} 
\mathbb B_{k}^{\div}(\boldsymbol{r}; \mathbb X) :={} (\mathbb B_{k}(\boldsymbol{r}_+) \otimes \mathbb X) &\oplus[r^e=-1]\Oplus_{e\in\Delta_1(T)}\mathbb B_k(e; r^{\texttt{v}}_+) \otimes \mathscr T^e(\mathbb X)\\
&\oplus[r^f=-1]\Oplus_{f\in\Delta_2(T)}\mathbb B_{k} (f;\boldsymbol{r}_+) \otimes \mathscr T^f(\mathbb X) \notag.
\end{align}
Let $\mathbb B_{k}(\boldsymbol{r}; \mathbb X):=\mathbb B_{k}(\boldsymbol{r})\otimes\mathbb X$ for $\mathbb X=\mathbb R^3, \mathbb M, \mathbb S$, or $\mathbb T$.

Different frames will be employed for distinct sub-simplices. On an edge \(e\), one possible frame is \(\{\boldsymbol n_{f_1}, \boldsymbol n_{f_2}, \boldsymbol t_e\}\), where \(f_1\) and \(f_2\) denote the two faces containing \(e\). Another option is \(\{\boldsymbol n_{1}^e, \boldsymbol n_{2}^e, \boldsymbol t_e\}\), where \(\boldsymbol n_{1}^e\) and \(\boldsymbol n_{2}^e\) represent two orthogonal normal vectors of \(e\). Notably, \(\boldsymbol n_{i}^e\) depends only on \(e\), while \(\boldsymbol n_{f_i}\) is contingent on face \(f_i\). On a face \(f\), an orthonormal frame \(\{\boldsymbol t_{1}^f, \boldsymbol t_{2}^f, \boldsymbol n^f\}\) is utilized, comprising two tangential vectors \(\boldsymbol t_i^f\) and a face normal \(\boldsymbol n^f\), both of which depend solely on face \(f\). 

%Here are the explicit decompositions for $\mathbb X = \mathbb T$ or $\mathbb S$. 
Decompositions for traceless matrix $\mathbb T$ on edge $e$ and face $f$ are given below and illustrated in Fig. \ref{fig:tnT}
\begin{align*}
\mathscr T^e(\mathbb T) &:= \textrm{span}\big\{\boldsymbol n_i^e\otimes \boldsymbol t_e, i=1,2\big\},\\
\mathscr N^e(\mathbb T) &:= \textrm{span}\big\{ \boldsymbol t_e\otimes \boldsymbol n_{f_i}, \; \boldsymbol n_{f_i}\otimes \boldsymbol n_{f_j} - (\boldsymbol n_{f_i}\cdot\boldsymbol n_{f_j})\boldsymbol t_e\otimes \boldsymbol t_e,  i,j=1,2 \big\},\\
\mathscr T^f(\mathbb T) &:= \textrm{span}\big\{\boldsymbol n_f\otimes \boldsymbol t_i^f,   i=1,2,\boldsymbol t_2^f\otimes \boldsymbol t_1^f, \boldsymbol t_1^f\otimes \boldsymbol t_2^f, \boldsymbol t_2^f\otimes \boldsymbol t_2^f - \boldsymbol t_1^f\otimes \boldsymbol t_1^f \big\},\\
\mathscr N^f(\mathbb T) &:= \textrm{span}\big\{ \boldsymbol n^f\otimes \boldsymbol n^f - \boldsymbol t_1^f\otimes \boldsymbol t_1^f, \boldsymbol t_i^f\otimes \boldsymbol n^f, i=1,2 \big\}.
\end{align*}

\begin{figure}[htbp]
\subfigure[Decomposition on an edge.]{
\begin{minipage}[t]{0.5\linewidth}
\centering
\includegraphics*[width=3.125cm]{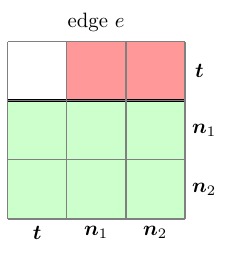}
\end{minipage}}%%
\subfigure[Decomposition on a face.]
{\begin{minipage}[t]{0.5\linewidth}
\centering
\includegraphics*[width=3.125cm]{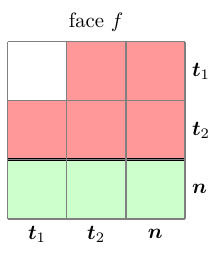}
\end{minipage}}
\caption{The $t$-$n$ decompositions of $\mathbb T$ on edges and faces. Red blocks are associated to bubbles and green blocks for the normal traces which are redistributed to faces.}
\label{fig:tnT}
\end{figure}

The tangential component will be integrated into the div bubble space. As an example, consider a function \(b_e \boldsymbol n^e_1\otimes \boldsymbol t_e\in \mathbb B_k(e; r^{\texttt{v}}) \otimes \mathscr T^e(\mathbb T)\). For two faces \(f\) that include the edge \(e\), \(\boldsymbol t_e \cdot \boldsymbol n_f = 0\). For the other two faces \(f\) that do not contain \(e\), the quadratic edge bubble function \(b_e\) vanishes on \(f\), i.e., \(b_e|_f = 0\). Consequently, \((b_e \boldsymbol n^e_1\otimes \boldsymbol t_e)\boldsymbol n |_{\partial T} = \boldsymbol 0\), which falls within \(\mathbb B_{k}^{\div}(\boldsymbol{r}; \mathbb T)\). A less apparent fact is that \(\mathbb B_{k}^{\div}(\boldsymbol{r}; \mathbb T)\) defined in~\eqref{eq:divbubbler} encompasses all div bubble polynomials $\mathbb B_{k}(\boldsymbol{r}; \mathbb T)\cap \ker(\tr^{\div})$, which was proved in~\cite{Chen;Huang:2021Geometric} for $\boldsymbol{r} = (0, -1, - 1)^{\intercal}$.

The normal component can be reallocated to each face to enforce the desired normal continuity. Further details will be elucidated in the proof of Lemma~\ref{lm:divT}.

Take $\mathbb P_{k}(T;\mathbb T)$ as the space of shape functions. When $r^f\geq 0$, DoFs are simply tensor product of \( \text{DoF}_k(\boldsymbol{r}) \) in~\eqref{eq:Cr3D} and $\mathbb T$. We thus focus on the case $r^f = -1$. The DoFs are 
\begin{subequations}\label{eq:divTdof}
\begin{align}
\nabla^i\boldsymbol{\tau}(\texttt{v}), & \quad i=0,\ldots, r^{\texttt{v}}, \texttt{v}\in \Delta_{0}(T), \label{eq:3dCrmodidivTfemdofV}\\
\int_e \frac{\partial^{j}\boldsymbol{\tau}}{\partial n_1^{i}\partial n_2^{j-i}}:\boldsymbol{q} \dd s, &\quad \boldsymbol{q}\in \mathbb B_{k-j}(e; r^{\texttt{v}}-j)\otimes \mathbb T, 0\leq i\leq j\leq r^e, e\in \Delta_{1}(T), \label{eq:3dCrmodidivTfemdofE}\\
%
% \int_f (\boldsymbol{\tau}\boldsymbol{n})\cdot\boldsymbol{q} \dd S, &\quad \boldsymbol{q}\in \mathbb B_{k}^3 (f;\boldsymbol r), f\in \Delta_{2}(T), \label{eq:3dCrmodidivTfemdofF1}\\
%
\int_f (\Pi_f\boldsymbol{\tau}\boldsymbol{n})\cdot\boldsymbol{q} \dd S, &\quad \boldsymbol{q}\in (\mathbb B_{k}^2 (f;\boldsymbol r)/{\rm RT}(f)) \oplus {\rm RT}(f), f\in \Delta_{2}(T), \label{eq:3dCrmodidivTfemdofF1}\\
\int_f (\boldsymbol{n}^{\intercal}\boldsymbol{\tau}\boldsymbol{n})\,{q} \dd S, &\quad {q}\in (\mathbb B_{k}(f;\boldsymbol r)/\mathbb P_0(f))\oplus\mathbb P_0(f), f\in \Delta_{2}(T), \label{eq:3dCrmodidivTfemdofF2}\\
\int_T \boldsymbol{\tau}:\boldsymbol{q} \dx, &\quad \boldsymbol{q}\in \mathbb B_{k}^{\div}(\boldsymbol{r};\mathbb T). \label{eq:3dCrmodidivTfemdofT}
\end{align}
\end{subequations}

\begin{lemma}\label{lm:divT}
Let $\boldsymbol r$ be a smoothness vector with $r^f = -1, r^{\texttt{v}}\geq 0$, and let $k\geq 2r^{\texttt{v}}+1$. DoFs~\eqref{eq:divTdof} are unisolvent for $\mathbb P_{k}(T;\mathbb T)$. 
Given a triangulation $\mathcal T_h$ of $\Omega$, define
\begin{align*}  
\mathbb V^{\div}_{k}(\boldsymbol{r};\mathbb T):=\{\boldsymbol{\tau}\in {L}^2(\Omega;\mathbb T): & \boldsymbol{\tau}|_T\in\mathbb P_{k}(T;\mathbb T) \textrm{ for all } T\in\mathcal T_h, \\
&\textrm{ and all the DoFs~\eqref{eq:divTdof} are single-valued}\}.
\end{align*} 
Then  $\mathbb V^{\div}_{k}(\boldsymbol{r};\mathbb T)\subset H(\div,\Omega;\mathbb T)$. 
\end{lemma}
\begin{proof}
First consider the case $r^e \geq 0$. The continuous element $\mathbb V_k(\boldsymbol r_+)\otimes \mathbb T$ is determined by DoFs~\eqref{eq:3dCrmodidivTfemdofV}-\eqref{eq:3dCrmodidivTfemdofE} plus 
\begin{align}
\label{eq:divdofproof2}
\int_f \boldsymbol  \tau: \boldsymbol q \dd S, & \quad \boldsymbol{q}\in  \mathbb B_{k}(f; 
\begin{pmatrix}
 r^{\texttt{v}}\\
 r^{e}
\end{pmatrix}
) \otimes \mathbb T, f\in \Delta_2(T), \\
\label{eq:divdofproof3}
\int_T \boldsymbol \tau : \boldsymbol q \dx, &\quad \boldsymbol q\in \mathbb B_k(T; \boldsymbol r_+)\otimes \mathbb T.
\end{align}
For~\eqref{eq:divdofproof2}, we decompose $\mathbb T = \mathscr T^s(\mathbb X) \oplus \mathscr N^s(\mathbb X)$ and move $ \mathbb B_{k} (f; \boldsymbol r)\otimes \mathscr T^f(\mathbb T)$ into the volume DoFs~\eqref{eq:3dCrmodidivTfemdofT} by utilizing $\boldsymbol q\in \mathbb B_{k}^{\div}(\boldsymbol{r};\mathbb T)$. For the normal component, we employ the idea of Petrov-Galerkin method. The function $\boldsymbol \tau$ is in the trial space containing basis $\boldsymbol n^f\otimes \boldsymbol n^f - \boldsymbol t_1^f\otimes \boldsymbol t_1^f$ for which the test function could be just $\boldsymbol n^f\otimes \boldsymbol n^f$ as $(\boldsymbol n^f\otimes \boldsymbol n^f - \boldsymbol t_1^f\otimes \boldsymbol t_1^f)\boldsymbol n^f = \boldsymbol n^f$, corresponding to DoF~\eqref{eq:3dCrmodidivTfemdofF2}. We then combine this with the other two components $\boldsymbol t_i^f\otimes \boldsymbol n$, i.e. DoF~\eqref{eq:3dCrmodidivTfemdofF1}, to determine the vector $\boldsymbol \tau \boldsymbol n$. 

The test function space is further decomposed, e.g. $\mathbb B_{k}(f;\boldsymbol r) = (\mathbb B_{k}(f;\boldsymbol r)/\mathbb P_0(f))\oplus\mathbb P_0(f)$ so that the moment $\int_f \boldsymbol{n}^{\intercal}\boldsymbol{\tau}\boldsymbol{n}\dd S$ is included in DoF, which is crucial for the div stability. Similar modification is applied in \eqref{eq:3dCrmodidivTfemdofF1} to include ${\rm RT}(f)$ in the test function space.

Consequently,~\eqref{eq:divdofproof2}-\eqref{eq:divdofproof3} are rearranged as
\begin{align*}
\int_f (\Pi_f\boldsymbol{\tau}\boldsymbol{n})\cdot\boldsymbol{q} \dd S, &\quad \boldsymbol{q}\in (\mathbb B_{k}^2 (f;\boldsymbol r)/{\rm RT}(f)) \oplus {\rm RT}(f), f\in \Delta_{2}(T), \\
\int_f (\boldsymbol{n}^{\intercal}\boldsymbol{\tau}\boldsymbol{n})\,{q} \dd S, &\quad {q}\in (\mathbb B_{k}(f;\boldsymbol r)/\mathbb P_0(f))\oplus\mathbb P_0(f), f\in \Delta_{2}(T), \\
\int_f \boldsymbol{\tau}:\boldsymbol{q} \dd S, &\quad \boldsymbol{q}\in \mathbb B_{k}(f;\boldsymbol r)\otimes\mathscr T^f(\mathbb T), f\in \Delta_{2}(T), \\
\int_T \boldsymbol{\tau}:\boldsymbol{q} \dx, &\quad \boldsymbol{q}\in \mathbb B_k(T; \boldsymbol r_+)\otimes \mathbb T,
\end{align*}
which are equivalent to~\eqref{eq:3dCrmodidivTfemdofF1}-\eqref{eq:3dCrmodidivTfemdofT}. The unisolvence then follows from that for tensor product spaces; see~\eqref{eq:r+unisolvence}. 

Now let us turn our attention to the case where \(r^{\texttt{v}}\geq 0\), \(r^e= -1\), \(r^f = -1\), and thus \(\boldsymbol r_+ =  (r^{\texttt{v}}, 0, 0)^{\intercal}\). The set of DoFs \({\rm DoF}_k(\boldsymbol r_+)\otimes \mathbb T\) includes vertex DoF~\eqref{eq:3dCrmodidivTfemdofV}, volume DoF~\eqref{eq:divdofproof3}, as well as the following edge and face DoFs:
\begin{align}
\label{eq:edgeDoFdivT}
\int_e \boldsymbol \tau:\boldsymbol q \dd s, & \quad \boldsymbol q\in  \mathbb B_{k}(e; 
r^{\texttt{v}})\otimes \mathbb T, \quad e\in \Delta_1(T), \\
\label{eq:faceDoFdivT}
\int_f \boldsymbol \tau : \boldsymbol q \dd S, & \quad \boldsymbol q\in  \mathbb B_{k}(f; 
\begin{pmatrix}
 r^{\texttt{v}}\\
0
\end{pmatrix}
)\otimes \mathbb T, \quad f\in \Delta_2(T).
\end{align}

As previously mentioned, on each edge \(e\), we employ the frame \(\{\boldsymbol n_{f_1}, \boldsymbol n_{f_2}, \boldsymbol t_e\}\), where \(f_1, f_2\) are two faces containing \(e\). The tangential component \(\mathbb B_{k}(e; r^{\texttt{v}}) \otimes \mathscr T^e(\mathbb T)\) is moved into the bubble space \(\mathbb B_k^{\div}(\boldsymbol r; \mathbb T)\). The normal components will be redistributed to the two faces \(f_i, i=1,2,\) containing \(e\). More precisely, we can first modify the DoF \eqref{eq:edgeDoFdivT} with $\boldsymbol q\in \mathbb B_{k}(e; r^{\texttt{v}}) \otimes \mathscr N^e(\mathbb T)$ to 
$$
\int_e (\boldsymbol \tau \boldsymbol n_{f_i})|_e \cdot \boldsymbol q  \quad  \text{ for } \boldsymbol q\in \mathbb B_{k}^3(e; r^{\texttt{v}}), \quad i= 1,2. 
$$
Then redistribute this edge DoF to the face $f_i, i=1,2$ containing $e$:
\[
\int_e (\boldsymbol \tau \boldsymbol n_{f_i})|_e \cdot \boldsymbol q \to \int_{e} (\boldsymbol \tau \boldsymbol n_{f_i})|_{f_i} \cdot \boldsymbol q \quad \text{ for } \boldsymbol q\in \mathbb B_{k}^3(e; r^{\texttt{v}}), \quad i= 1,2,
\]
so that in \eqref{eq:faceDoFdivT}
\[
\mathbb B_{k}(f; 
\begin{pmatrix}
 r^{\texttt{v}}\\
0
\end{pmatrix}
) \oplus_{e\in \Delta_1(f)} \mathbb B_{k}(e; 
r^{\texttt{v}}) = 
\mathbb B_{k}(f; 
\begin{pmatrix}
 r^{\texttt{v}}\\
-1
\end{pmatrix}
)
\]
for \(r^{\texttt{v}}\geq 0, r^e = -1\), which leads to \eqref{eq:3dCrmodidivTfemdofF1}-\eqref{eq:3dCrmodidivTfemdofF2}. Thus, the unisolvence is proven.

The conclusion $\mathbb V^{\div}_{k}(\boldsymbol{r};\mathbb T)\subset H(\div,\Omega;\mathbb T)$ is obvious as $\boldsymbol \tau\boldsymbol n$ is continuous on each face $f$ due to the single-valued DoFs~\eqref{eq:3dCrmodidivTfemdofV}-\eqref{eq:3dCrmodidivTfemdofF2}. 
\end{proof}

%Comment on $r^{\texttt{v}}\geq 0$ cannot be relaxed.
When \(r^{\texttt{v}}= 0\) and \(r^e = r^f = -1\), the element \(\mathbb V^{\div}_{k}(\boldsymbol{r};\mathbb T)\) exhibits continuity at vertices, which is reminiscent of the Stenberg element~\cite{stenbergNonstandardMixedFinite2010} designed for \(H(\div)\)-conforming vector functions. We refer to \cite{Chen;Huang:2021Geometric} for an illustration of Stenberg element and the corresponding inf-sup condition.
%If \(\boldsymbol \tau \in \mathbb M\) at vertices, the vertex DoFs can be further reorganized across each face, yielding the second family of N\'ed\'elec face elements~\cite{Nedelec1986} or Brezzi-Douglas-Marini (BDM) face element~\cite{BrezziDouglasMarini1986}.

\begin{figure}[htbp]
\subfigure[Decomposition on an edge]{
\begin{minipage}[t]{0.5\linewidth}
\centering
\includegraphics*[width=3.25cm]{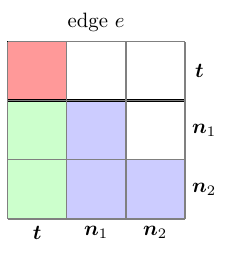}
\end{minipage}}%%
\subfigure[Decomposition on a face]
{\begin{minipage}[t]{0.5\linewidth}
\centering
\includegraphics*[width=3.25cm]{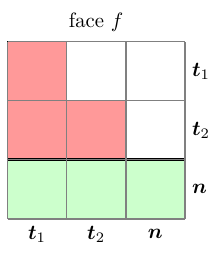}
\end{minipage}}
\caption{The $t$-$n$ decomposition of $\mathbb S$ on edges and faces. Red blocks are associated to bubbles, and green and blue blocks for the normal traces. The green blocks can be redistributed to faces while the blue blocks introduces stronger continuity on the normal plane \(\mathscr N^e(\mathbb S)\).}
\label{fig:tnS}
\end{figure}

The construction of an \(H(\div; \mathbb S)\)-conforming element follows a similar approach, albeit with additional complexities introduced by \(\mathscr N^e(\mathbb S)\). Decompositions on edge $e$ and $f$ are
\begin{align*}
\mathscr T^e(\mathbb S) &:= \textrm{span}\big\{\boldsymbol t_e\otimes \boldsymbol t_e \big\},\\
\mathscr N^e(\mathbb S) &:= \textrm{span}\big\{\sym(\boldsymbol t_e \otimes \boldsymbol n_{f_i}),   i=1,2 \big\} \oplus\textrm{span}\big\{\sym(\boldsymbol n_i^e\otimes \boldsymbol n_j^e),  1\leq i\leq j\leq 2 \big\},\\
\mathscr T^f(\mathbb S) &:= \textrm{span}\big\{ \sym(\boldsymbol t_i^f\otimes \boldsymbol t_j^f),  1\leq i\leq j\leq 2 \big\},\\
\mathscr N^f(\mathbb S) &:= \textrm{span}\big\{\boldsymbol n^f\otimes \boldsymbol n^f, \sym(\boldsymbol t_i^f\otimes \boldsymbol n^f), i=1,2 \big\}.
\end{align*}

Once again, the tangential component will be incorporated into the bubble space. However, the redistribution of certain normal components to faces might be constrained by symmetry conditions. On an edge \(e\), for instance, the symmetry constraint demands that the normal plane of \(e\) must obey \(\textrm{span}\big\{\sym(\boldsymbol n_i^e\otimes \boldsymbol n_j^e),  1\leq i\leq j\leq 2 \big\}\), which is a global requirement, indicating that the two normal vectors \(\{\boldsymbol n_1^e, \boldsymbol n_2^e\}\) are independent of the elements containing \(e\). We refer to the blue blocks in Fig.~\ref{fig:tnS} for clarification. Conversely, in \(\mathscr N^e(\mathbb T)\), all the components can be effectively redistributed to faces, as demonstrated by the green blocks in Fig.~\ref{fig:tnS}.

Take $\mathbb P_{k}(T;\mathbb S)$ as the space of shape functions. Again we focus on the case $r^f = -1$.
The DoFs are 
\begin{subequations}\label{eq:divSdof}
\begin{align}
\nabla^i\boldsymbol{\tau}(\texttt{v}), & \quad i=0,\ldots, r^{\texttt{v}}, \texttt{v}\in \Delta_{0}(T), \label{eq:3dCrdivSfemdofV}\\
% e: D_n
\int_e \frac{\partial^{j}\boldsymbol{\tau}}{\partial n_1^{i}\partial n_2^{j-i}}:\boldsymbol{q} \dd s, &\quad \boldsymbol{q}\in \mathbb B_{k-j}(e; r^{\texttt{v}}-j) \otimes \mathbb S, 0\leq i\leq j\leq r^e, e\in \Delta_{1}(T), \label{eq:3dCrdivSfemdofE1}\\
% e: ni-nj
\int_e (\boldsymbol{n}_i^{\intercal}\boldsymbol{\tau}\boldsymbol{n}_j)\,q \dd s, &\quad q\in \mathbb B_{k}(e; r^{\texttt{v}}), 1\leq i\leq j\leq 2, e\in \Delta_{1}(T), \textrm{ if } r^{e}=-1, \label{eq:3dCrdivSfemdofE2}\\
% f: Pi_f - n
\int_f (\Pi_f \boldsymbol{\tau}\boldsymbol{n})\cdot\boldsymbol{q} \dd S, &\quad \boldsymbol{q}\in (\mathbb B_{k}^{\div} (f;\boldsymbol r)/{\rm RM}(f)) \oplus {\rm RM}(f), f\in \Delta_{2}(T), \label{eq:3dCrdivSfemdofF1}\\
% f: n-n
\int_f (\boldsymbol n^{\intercal}\boldsymbol{\tau}\boldsymbol{n}) \, q \dd S, &\quad q\in (\mathbb B_{k}(f;\boldsymbol{r}_+)/\mathbb P_1(f))\oplus \mathbb P_1(f), f\in \Delta_{2}(T), \label{eq:3dCrdivSfemdofF2}\\
\int_T \boldsymbol{\tau}:\boldsymbol{q} \dx, &\quad \boldsymbol{q}\in \mathbb B_{k}^{\div}(\boldsymbol{r};\mathbb S). \label{eq:3dCrdivSfemdofT}
\end{align}
\end{subequations}

\begin{lemma}\label{lm:divS}
Let $\boldsymbol r$ be a valid smoothness vector with $r^f = -1, r^{\texttt{v}}\geq 0$, and let $k\geq 2r^{\texttt{v}}+1$. DoFs~\eqref{eq:divSdof} are unisolvent for $\mathbb P_{k}(T;\mathbb S)$. 
Given a triangulation $\mathcal T_h$ of $\Omega$, define
\begin{align*}  
\mathbb V^{\div}_{k}(\boldsymbol{r};\mathbb S):=\{\boldsymbol{\tau}\in {L}^2(\Omega;\mathbb S): & \,\boldsymbol{\tau}|_T\in\mathbb P_{k}(T;\mathbb S) \textrm{ for all } T\in\mathcal T_h, \\
&\textrm{ and all the DoFs~\eqref{eq:divSdof} are single-valued}\}.
\end{align*} 
Then  $\mathbb V^{\div}_{k}(\boldsymbol{r};\mathbb S)\subset H(\div,\Omega;\mathbb S)$. 
\end{lemma}
\begin{proof}
The core approach of the proof aligns with that of Lemma~\ref{lm:divT}. We will highlight the differences here. The case where \(r^e \geq 0\) remains unchanged. When \(r^e = -1\), the components \(\boldsymbol t_e \otimes \boldsymbol n_{f_i}\) can be redistributed to faces, resulting in the expression:
\[
\mathbb B_{k}^{\div} (f;\boldsymbol{r}) = \mathbb B_{k}^2(f;\boldsymbol{r}_+) \Oplus \oplus_{e\in \Delta_1(f)} \mathbb B_k(e, r^{\texttt{v}}) \boldsymbol t_e,
\]
which leads to the form in~\eqref{eq:3dCrdivSfemdofF1}. The components \(\boldsymbol n_i^e\otimes \boldsymbol n_j^e\) cannot be redistributed to faces and are preserved in~\eqref{eq:3dCrdivSfemdofE2}. Therefore, in~\eqref{eq:3dCrdivSfemdofF2}, the notation \(\boldsymbol r_+\) is still retained in $\mathbb B_{k}(f;\boldsymbol{r}_+)$, while in \eqref{eq:3dCrmodidivTfemdofF2} for $H(\div,\mathbb T)$ elements, $\mathbb B_{k}(f;\boldsymbol{r})$ is used.%The resulting finite element space is still $H(\div)$-conforming as~\eqref{eq:3dCrdivSfemdofV},~\eqref{eq:3dCrdivSfemdofE2}, and~\eqref{eq:3dCrdivSfemdofF2} will ensure the continuity of $\boldsymbol n^{\intercal}\boldsymbol{\tau}\boldsymbol{n}$. 
\end{proof}

\begin{remark}\label{rm:HZdof}\rm 
When $r^{\texttt{v}}=0, r^e=r^f = -1$, if we do not redistribute the tangential-normal component  of edge DoFs, cf. \cite{Chen;Huang:2021divFinite,Chen;Huang:2021Geometric} for detailed explanation, we can recover the Hu-Zhang element~\cite{HuZhang2015}.
%\begin{subequations}\label{eq:HZdivSdof}
%\begin{align}
%\nabla^i\boldsymbol{\tau}(\texttt{v}), & \quad i=0,\ldots, r^{\texttt{v}}, \texttt{v}\in \Delta_{0}(T), \label{eq:3dHZdivSfemdofV}\\
%% e: ni-nj
%\int_e (\boldsymbol{n}_i^{\intercal}\boldsymbol{\tau}\boldsymbol{n}_j)\,q \dd s, &\quad q\in \mathbb B_{k}(e; r^{\texttt{v}}), 1\leq i\leq j\leq 2, e\in \Delta_{1}(T), \label{eq:3dHZdivSfemdofE1}\\
%% e: t-n_j
%\int_e (\boldsymbol{t}^{\intercal}\boldsymbol{\tau}\boldsymbol{n}_j)\,q \dd s, &\quad q\in \mathbb B_{k}(e; r^{\texttt{v}}), j = 1, 2, e\in \Delta_{1}(T), \label{eq:3dHZdivSfemdofE2}\\
%% f: Pi_f - n
%\int_f (\Pi_f \boldsymbol{\tau}\boldsymbol{n})\cdot\boldsymbol{q} \dd S, &\quad \boldsymbol{q}\in \mathbb B_{k}^2(f;\boldsymbol{r}_+), f\in \Delta_{2}(T), \label{eq:3dHZdivSfemdofF1}\\
%% f: n-n
%\int_f (\boldsymbol n^{\intercal}\boldsymbol{\tau}\boldsymbol{n}) \, q \dd S, &\quad q\in \mathbb B_{k}(f;\boldsymbol{r}_+), f\in \Delta_{2}(T), \label{eq:3dHZdivSfemdofF2}\\
%%
%\int_T \boldsymbol{\tau}:\boldsymbol{q} \dx, &\quad \boldsymbol{q}\in \mathbb B_{k}^{\div}(\boldsymbol{r};\mathbb S). \label{eq:3dHZdivSfemdofT}
%\end{align}
%\end{subequations}
%The unisolvence can be proved similarly as DoFs~\eqref{eq:3dHZdivSfemdofE2}-\eqref{eq:3dHZdivSfemdofF1} are just rearrangement of the tangential-normal component~\eqref{eq:3dCrdivSfemdofF1}. 
We prefer the redistribution of tangential-normal DoF~\eqref{eq:3dCrdivSfemdofF1} as it is more close to the vector face elements. $\qed$
\end{remark}

\begin{remark}\rm 
The continuity at vertices is enforced due to the constraints -- tracelessness conditions in \(\mathbb T\) or symmetry conditions in \(\mathbb S\). This constraint-driven continuity at vertices cannot be relaxed. To elucidate, let \(\texttt{v}_0,\texttt{v}_1,\texttt{v}_2,\texttt{v}_3\) be the four vertices of a tetrahedron \(T\), with corresponding barycentric coordinates \(\lambda_0,\lambda_1,\lambda_2,\lambda_3\). Selecting \(\texttt{v}_0\) as the origin, we define \(\boldsymbol{t}_{0i}:=\texttt{v}_i-\texttt{v}_0\) for \(i=1,2,3\), which serve as three basis vectors. For a smooth traceless tensor \(\boldsymbol{\tau}\), due to the duality between \(\{\boldsymbol{t}_{01},\boldsymbol{t}_{02}, \boldsymbol{t}_{03}\}\) and \(\{\nabla\lambda_1,\nabla\lambda_2, \nabla\lambda_3\}\), we can represent
\[
\boldsymbol{\tau}(\texttt{v}_0)=\sum_{i=0}^3(\boldsymbol{\tau}\nabla\lambda_i)|_{f_i}(\texttt{v}_0)\boldsymbol{t}_{0i}^{\intercal}.
\]
The traceless property of \(\boldsymbol{\tau}\) implies
\[
\sum_{i=0}^3\boldsymbol{t}_{0i}^{\intercal}(\boldsymbol{\tau}\nabla\lambda_i)|_{f_i}(\texttt{v}_0)=0,
\]
indicating that \((\boldsymbol{\tau}\boldsymbol{n}_1)|_{f_1}\), \((\boldsymbol{\tau}\boldsymbol{n}_2)|_{f_2}\), and \((\boldsymbol{\tau}\boldsymbol{n}_3)|_{f_3}\) are linearly dependent at vertex \(\texttt{v}_0\). Consequently, the vertex DoFs in equation~\eqref{eq:3dCrmodidivTfemdofV} cannot be reallocated to the faces, which underscores the inalterable nature of the constraint \(r^{\texttt{v}}\geq 0\). To relax the continuity at vertices, we can use the barycentric refinement in \cite{ChenHuang2025} or the distributional finite element in \cite{ChenHuangZhang2023}.
$\qed$
\end{remark}

\subsection{Div stability}
Due to the similarity, we use $H(\div; \mathbb S)$-conforming finite element to illustrate the BGG procedure. Consider the diagram
\begin{equation*}%\label{eq:BGGdivS}
\begin{tikzcd}[column sep=small, row sep=normal]
% \mathbb R \arrow{r}{\subset}
% &
% \mathbb V^{\grad}_{k+2}(\boldsymbol{r}_0)
% \arrow{r}{\grad}
%  &
% \mathbb V^{\curl}_{k+1}(\boldsymbol{r}_1+1)
% \arrow{r}{\curl}
%  &
&
\mathbb V^{\div}_{k+1}(\boldsymbol{r}+1; \mathbb M)
% \arrow[dl,swap,"\mskw "']  
  \arrow{r}{\div}
 &
\mathbb V^{L^2}_{k}(\boldsymbol{r}; \mathbb R^3)
\arrow{r}{}
& \boldsymbol{0} \\
% \\
% \mathbb V \arrow{r}{\subset}
% &
% \mathbb V^{\grad}_{k+1}(\boldsymbol{r}_1+1;\mathbb R^3)
%  \arrow[ur,swap,"{\rm id}"'] \arrow{r}{\grad}
%  & 
{\mathbb V}^{\curl}_{k+1}(\boldsymbol{r}+1;\mathbb M)
 \arrow[ur,swap,"S"'] \arrow{r}{\curl}
 & 
\mathbb V^{\div}_{k}(\boldsymbol{r};\mathbb M)
 \arrow[ur,swap,"- 2\vskw "'] \arrow{r}{\div}
 & 
\mathbb V^{L^2}_{k-1}(\boldsymbol{r}\ominus 1;\mathbb R^3)
\arrow[r] 
&\boldsymbol{0},
\end{tikzcd}
\end{equation*}
where $\mathbb V^{\div}_{k}(\boldsymbol{r};\mathbb M)=\mathbb R^3\otimes\mathbb V^{\div}_{k}(\boldsymbol{r})$ and $\mathbb V^{\curl}_{k+1}(\boldsymbol{r}+1;\mathbb M)=\mathbb R^3\otimes\mathbb V^{\curl}_{k+1}(\boldsymbol{r}+1)$.
%There are three constraints on the choice of parameters $(\boldsymbol r_1, \boldsymbol r_2, \boldsymbol r_3)$: 
%\begin{enumerate}
% \item $r_1^f\geq 0$ so that $\mskw (\mathbb V^{\div}_{k+1}(\boldsymbol{r}_1,\boldsymbol{r}_2))\subset {\mathbb V}^{\curl}_{k+1}(\boldsymbol{r}_1,\boldsymbol{r}_2;\mathbb M)$. When $r_1^f = -1$, these two spaces will have different continuity and cannot trace back. 
%
% \item $(\boldsymbol r_1, \boldsymbol r_2, k+1)$ is $\div$ stable. 
%
% \item $(\boldsymbol r_2, \boldsymbol r_3, k)$ is $\div$ stable. 
%\end{enumerate}
%Constraints (1) and (2) imply that the minimal vector for $\boldsymbol r_1$  is $(2, 1, 0)^{\intercal}$.
We require that both $(\boldsymbol r + 1, \boldsymbol r, k+1)$  and $(\boldsymbol r, \boldsymbol r\ominus 1, k)$ are div stable. Then 
\begin{equation*}%\label{eq:lowerbound}
\boldsymbol r +1 \geq \begin{pmatrix}
 2\\
 1\\
 0
\end{pmatrix}, 
\quad
\boldsymbol r \geq 
\begin{pmatrix}
 1\\
 0\\
 -1
\end{pmatrix},
\quad
\boldsymbol r \ominus 1 \geq
\begin{pmatrix}
 0\\
 -1\\
 -1
\end{pmatrix}.
\end{equation*}
In particular $r^e \geq 0$. 

\begin{lemma}\label{lm:divSstability}
Let $\boldsymbol r$ be a smoothness vector and $k$ large enough satisfying: both $(\boldsymbol r + 1, \boldsymbol r, k+1)$  and $(\boldsymbol r, \boldsymbol r\ominus 1, k)$ are div stable. Then we have the $(\div; \mathbb S)$ stability:
\begin{equation*}%\label{eq:divSstability}
\div \mathbb V^{\div}_{k}(\boldsymbol{r};\mathbb S) = 
\mathbb V^{L^2}_{k-1}(\boldsymbol{r}\ominus 1;\mathbb R^3).
\end{equation*}
\end{lemma}
\begin{proof}
As \(\boldsymbol{r}+1\geq 0\), we have \({\mathbb V}^{\div}_{k+1}(\boldsymbol{r}+1;\mathbb M) = {\mathbb V}^{\curl}_{k+1}(\boldsymbol{r}+1;\mathbb M) = {\mathbb V}^{\grad}_{k+1}(\boldsymbol{r} +1)\otimes \mathbb M\). Therefore, \(S\) is one-to-one.

Given \(\boldsymbol{u} \in \mathbb V^{L^2}_{k}(\boldsymbol{r}; \mathbb R^3)\), since \((\boldsymbol{r} + 1, \boldsymbol{r}, k+1)\) is div stable, we can find \(\boldsymbol{\sigma} \in {\mathbb V}^{\div}_{k+1}(\boldsymbol{r}+1;\mathbb M)\) such that \(\div \boldsymbol{\sigma} = \boldsymbol{u}\). Then, by defining \(\boldsymbol{\tau}:=\curl S^{-1}\boldsymbol{\sigma}\), we have \(\boldsymbol{\tau} \in \mathbb V^{\div}_{k}(\boldsymbol{r};\mathbb M)\) and
\[
2\vskw \boldsymbol{\tau} = 2\vskw \curl (S^{-1}\boldsymbol{\sigma}) = \div S(S^{-1}\boldsymbol{\sigma}) = \boldsymbol{u}.
\]
Thus, \(\vskw: \mathbb V^{\div}_{k}(\boldsymbol{r};\mathbb M)\to  \mathbb V^{L^2}_{k}(\boldsymbol{r}; \mathbb R^3)\) is surjective.

We can apply the BGG construction to conclude that \(\div: \mathbb V^{\div}_{k}(\boldsymbol{r};\mathbb M)\cap \ker(\vskw)\to \mathbb V^{L^2}_{k-1}(\boldsymbol{r}\ominus 1;\mathbb R^3)\) is surjective. Our next step is to establish the relationship
\begin{equation}\label{eq:VdivS=kervskw}
\mathbb V^{\div}_{k}(\boldsymbol{r};\mathbb S) = \mathbb V^{\div}_{k}(\boldsymbol{r};\mathbb M)\cap \ker(\vskw).
\end{equation}
Namely, we need to show that the subspace \(\mathbb V^{\div}_{k}(\boldsymbol{r};\mathbb M)\cap \ker(\vskw)\) derived via BGG corresponds to the finite element space \(\mathbb V^{\div}_{k}(\boldsymbol{r};\mathbb S)\) defined by DoFs~\eqref{eq:divSdof}.

The inclusion \(\mathbb V^{\div}_{k}(\boldsymbol{r};\mathbb S)\subseteq \mathbb V^{\div}_{k}(\boldsymbol{r};\mathbb M)\cap \ker(\vskw)\) is evident. To establish their equality, it suffices to demonstrate that
\[
\dim \mathbb V^{\div}_{k}(\boldsymbol{r};\mathbb S) = \dim (\mathbb V^{\div}_{k}(\boldsymbol{r};\mathbb M)\cap \ker(\vskw)),
\]
which is equivalent to showing
\begin{equation}\label{eq:dimchange}
\dim \mathbb V^{\div}_{k}(\boldsymbol{r};\mathbb M) -  \dim \mathbb V^{\div}_{k}(\boldsymbol{r};\mathbb S) = \dim \mathbb V^{L^2}_{k}(\boldsymbol{r}; \mathbb R^3),
\end{equation}
since we have proved that \(\vskw: \mathbb V^{\div}_{k}(\boldsymbol{r};\mathbb M)\to  \mathbb V^{L^2}_{k}(\boldsymbol{r}; \mathbb R^3)\) is surjective.

In the case where \(r^f \geq 0\), we have \(\mathbb V^{\div}_{k}(\boldsymbol{r};\mathbb X) = \mathbb V_{k}(\boldsymbol{r}) \otimes \mathbb X\) for \(\mathbb X = \mathbb M, \mathbb S\), or \(\mathbb R^3\). Consequently,~\eqref{eq:dimchange} trivially holds. Let us now consider the case where \(r^f = -1\) and \(r^e\geq 0\). For the vertex and edge DoFs~\eqref{eq:3dCrdivSfemdofV}-\eqref{eq:3dCrdivSfemdofE1}, we find that \(\dim \mathbb M - \dim \mathbb S = \dim \mathbb R^3\). The face DoFs~\eqref{eq:3dCrdivSfemdofF1}-\eqref{eq:3dCrdivSfemdofF2} remain the same. The only remaining dimension change is within the bubble spaces, and this can be computed as follows:
\begin{align*}
&\dim \mathbb B_k^{\div}(\boldsymbol r; \mathbb M) - \dim \mathbb B_k^{\div}(\boldsymbol r; \mathbb S) \\
&\quad = \dim \mathbb B_k(\boldsymbol r_+; \mathbb M) - \dim \mathbb B_k(\boldsymbol r_+; \mathbb S)\\
&\quad\quad +  4\dim \mathbb B_k(f; \boldsymbol r)\otimes \mathscr T^f(\mathbb M) -  4\dim \mathbb B_k(f; \boldsymbol r)\otimes \mathscr T^f(\mathbb S)\\
&\quad = 3(\dim \mathbb B_k(\boldsymbol r_+) + 4\dim \mathbb B_k(f; \boldsymbol r))\\
&\quad = \dim \mathbb B_k(\boldsymbol r; \mathbb R^3),
\end{align*}
where $\mathbb B_k^{\div}(\boldsymbol r; \mathbb M)=\mathbb R^3\otimes\mathbb B_k^{\div}(\boldsymbol r)$.
Hence,~\eqref{eq:dimchange} holds, and consequently,~\eqref{eq:VdivS=kervskw} is confirmed.
\end{proof}
As 
$
\mathbb V_{k}^{\div}(
(r^{\texttt{v}}, 0, -1)^{\intercal} ;\mathbb S) 
\subset
\mathbb V_{k}^{\div}(
(r^{\texttt{v}}, -1, -1)^{\intercal};\mathbb S),
$
we also observe the $(\div; \mathbb S)$ stability for the pair 
$
( r^{\texttt{v}},  -1,  -1)^{\intercal}
- ( r^{\texttt{v}}-1,  -1,  -1)^{\intercal}
$
when \(r^{\texttt{v}}\geq 1\). The $(\div; \mathbb S)$ stability for the case with the lowest level of smoothness, i.e., 
$
( 0,  -1,  -1)^{\intercal} - ( -1,  -1,  -1)^{\intercal},
$
has been established in~\cite{HuZhang2015}, and it appears to be challenging to obtain this result through the BGG construction. Again, cases that can be handled by the BGG construction are slightly smoother.

For the situation where \(r^e=-1\), we encounter different variants of $\mathbb V^{\div}_k(\boldsymbol r; \mathbb S)$ elements depending on whether the tangential-normal component is redistributed to faces or not, as discussed in Remark~\ref{rm:HZdof}. Despite these variations, the $(\div; \mathbb S)$ stability still holds.

Discussion on $(\div; \mathbb T)$ stability is similar. 
\begin{lemma}
Let $\boldsymbol r$ be a smoothness vector and $k$ large enough satisfying: both $(\boldsymbol r + 1, \boldsymbol r, k+1)$  and $(\boldsymbol r, \boldsymbol r\ominus 1, k)$ are div stable. Then we have the $(\div; \mathbb T)$ stability:
\begin{equation*}%\label{eq:divSstability}
\div \mathbb V^{\div}_{k}(\boldsymbol{r};\mathbb T) = 
\mathbb V^{L^2}_{k-1}(\boldsymbol{r}\ominus 1;\mathbb R^3).
\end{equation*}
\end{lemma}
\begin{proof}
The dimension identity for traceless matrices 
\begin{equation}\label{eq:divTdimdiff}
 \dim \mathbb V^{\div}_{k-1}(\boldsymbol r; \mathbb M) - \dim \mathbb V^{\div}_{k-1}(\boldsymbol r; \mathbb T) = \dim \mathbb V^{L^2}_{k}(\boldsymbol{r})
\end{equation}
holds for smoothness vector $\boldsymbol r$ with $r^{\texttt{v}}\geq 0$ but without requirement $r^e\geq 0$ as all normal components can be redistributed to faces.
\end{proof}

We summarize the result below by treating $\mathbb S$ and $\mathbb T$ together. 
\begin{theorem}\label{thm:divX}
Assume the smoothness vector $\boldsymbol r$ and polynomial degree $k$ satisfy: 
\begin{enumerate}
\item Case $r^{\texttt{v}}\geq 1$ and $r^{e}\geq 0$: both $(\boldsymbol r + 1, \boldsymbol r, k+1)$ and $(\boldsymbol r, \boldsymbol r\ominus 1, k)$ are div stable;
\item Case $\bs r = (r^{\texttt{v}}, -1, -1)$ with $r^{\texttt{v}} \geq 1$: $k\geq \max\{2r^{\texttt{v}}+1,4\}$;
\item Case $\bs r = (0, -1, -1)^{\intercal}$: $k\geq \begin{cases}
4, & \textrm{ for }\, \mathbb X = \mathbb S,\\
2, & \textrm{ for }\, \mathbb X = \mathbb T.
 \end{cases}$
\end{enumerate}
% \begin{enumerate}
% \item $r^{\texttt{v}}\geq 1$, $r^{e}\geq 0$, both $(\boldsymbol r + 1, \boldsymbol r, k+1)$ and $(\boldsymbol r, \boldsymbol r\ominus 1, k)$ are div stable, or
% \item $r^{\texttt{v}}\geq 0$, $r^{e}=r^{f}=-1$, 
% $$
% k\geq \max\{2r^{\texttt{v}}+1,4\} \text{ for }\, \mathbb S, 
% \quad
% \begin{cases}
%   k\geq \max\{2r^{\texttt{v}}+1,4\}, & \textrm{ if } r^{\texttt{v}}\geq1,\\
% k\geq2, & \textrm{ if } r^{\texttt{v}}=0
%  \end{cases}
% \textrm{ for }\, \mathbb T.$$
% \end{enumerate}
Then we have the $(\div; \mathbb X)$ stability, for $\mathbb X = \mathbb S$ or $\mathbb T$,
\begin{equation}\label{eq:divXstability}
\div \mathbb V^{\div}_{k}(\boldsymbol{r};\mathbb X)
= \mathbb V^{L^2}_{k-1}(\boldsymbol{r}\ominus 1;\mathbb R^3).
\end{equation}
\end{theorem}
When~\eqref{eq:divXstability} is satisfied, we will refer to the triple $( \boldsymbol r, \boldsymbol r\ominus 1, k)$ as being $(\div; \mathbb X)$ stable. The conditions presented in Theorem~\ref{thm:divX} are sufficient to establish this stability, although they might not be necessary in all cases. It is important to note that due to the redistribution of edge DoFs to faces, when \(r^{\texttt{v}}=0\), for $\mathbb X = \mathbb T$, the face DoFs~\eqref{eq:3dCrmodidivTfemdofF1}-\eqref{eq:3dCrmodidivTfemdofF2} include $\mathbb P_1^3(f)$ for \(k\geq 2\), which is necessary to prove the div stability. On the other hand, for $\mathbb X = \mathbb S$, \(k\geq 4\) is required, since the face DoF~\eqref{eq:3dCrdivSfemdofF2} contains a smaller face bubble $\mathbb B_{k}(f;\boldsymbol{r}_+)$ that demands higher degree of polynomial.
% {Do we need to reformulate the face dof to $\mathbb B/\mathbb P_1^3(f)\oplus \mathbb P_1^3(f)$?} 
%\mnote{ What is the difference in the proof for $\mathbb X = \mathbb T$ and $\mathbb X = \mathbb S$? Why $k=2$ is allowed for $\mathbb X = \mathbb T$. \\
%{Ans: The normal DoF on faces covers $\mathbb P_1^3(f)$ ($RM, RT\subset\mathbb P_1^3(f)$). Thus $k\geq2$ for $\mathbb T$ and $k\geq4$ for $\mathbb S$.}   }

By the same proof, we also have the div stability for the bubble spaces and include the proof in Appendix: Lemma \ref{lem:divbubbleontoT} and Lemma \ref{lem:divbubbleontoS}. 
\begin{lemma}\label{lem:divbubbleontoST}
Assume the polynomial degree $k\geq \max\{2r^{\texttt{v}}+1,2\}$, and the smoothness vector $\boldsymbol r$  satisfies either: 
\begin{enumerate}
\item $r^{\texttt{v}}\geq 2r^e+1$ and $r^{e}\geq 2(r^f+1)$,  or
% \item $r^{\texttt{v}}\geq 1$, $r^{e}\geq 0$, both $(\boldsymbol r + 1, \boldsymbol r, k+1)$ and $(\boldsymbol r, \boldsymbol r\ominus 1, k)$ satisfy the condition in Lemma \ref{lem:divbubbleonto}, or
\item $r^{\texttt{v}}\geq 0$ and $r^{e}=r^{f}=-1$.
% $$
% \begin{cases}
%   k\geq 2r^{\texttt{v}}+2, & \textrm{ if } r^{\texttt{v}}\geq1,\\
% k\geq2r^{\texttt{v}}+3, & \textrm{ if } r^{\texttt{v}}=0.
%  \end{cases}
% $$
\end{enumerate}
Then we have the $(\div; \mathbb X)$ stability, for $\mathbb X = \mathbb S$ or $\mathbb T$,
\begin{equation*}%\label{eq:divbubbleontoS}
\div\mathbb B^{\div}_{k}(\boldsymbol{r};\mathbb X)=\mathbb B_{k-1}(\boldsymbol{r}\ominus1;\mathbb R^3)/{\rm RX}.
\end{equation*}
Here ${\rm RX} = {\rm RM}$ for $\mathbb X = \mathbb S$, and ${\rm RX} = {\rm RT}$ for $\mathbb X = \mathbb T$.
\end{lemma}

\subsection{Inequality constraints}\label{sec:r2r3}
We can apply one $\widetilde{\quad}$ operation to get the div stability with an inequality constraint on the smoothness vectors. 
\begin{corollary}
Let $( \boldsymbol r_2, \boldsymbol r_2\ominus 1, k)$ be $(\div; \mathbb X)$ stable and $\boldsymbol r_3 \geq \boldsymbol r_2\ominus 1$. Define $$\mathbb V^{\div}_{k}(\boldsymbol{r}_2, \boldsymbol r_3;\mathbb X) = \{ \boldsymbol \tau \in \mathbb V^{\div}_{k}(\boldsymbol{r}_2;\mathbb X): \div \boldsymbol \tau \in \mathbb V^{L^2}_{k-1}(\boldsymbol{r}_3;\mathbb R^3)\}.$$ 
Then we have the $(\div; \mathbb X)$ stability 
\begin{equation*}%\label{eq:divXstabilityinequality}
\div \mathbb V^{\div}_{k}(\boldsymbol{r}_2, \boldsymbol r_3;\mathbb X) = 
\mathbb V^{L^2}_{k-1}(\boldsymbol{r}_3;\mathbb R^3).
\end{equation*}
\end{corollary}

The subspace $\mathbb V^{\div}_{k}(\boldsymbol{r}_2, \boldsymbol r_3;\mathbb X)$ always exists as $\mathbb V^{L^2}_{k-1}(\boldsymbol{r}_3;\mathbb R^3)\subseteq \mathbb V^{L^2}_{k-1}(\boldsymbol r_2\ominus 1;\mathbb R^3)$. However, the challenge lies in formulating local DoFs for this subspace. In this pursuit, we draw insights from our recent work, as outlined in~\cite[Section 4.4]{Chen;Huang:2022FEMcomplex3D}. We add DoFs to determine $\div \boldsymbol \tau$ first but remove non-free index (white blocks in Fig.~\ref{fig:tnT} and Fig.~\ref{fig:tnS}) in the $t$-$n$ decomposition. For example, for face DoFs, we remove component $\partial_n^j(\boldsymbol t_1^{\intercal}\boldsymbol \tau \boldsymbol t_1)$ from vector $\partial_n^j(\boldsymbol \tau \boldsymbol t_1)$ as $\boldsymbol t_1\otimes \boldsymbol t_1\not\in \mathscr T^f(\mathbb T)$. Similarly remove $\boldsymbol t^{\intercal}\boldsymbol \tau\boldsymbol t$ from the edge DoF. 

%Construction of $\mathbb V^{\div}_{k}(\boldsymbol{r}_2, \boldsymbol r_3;\mathbb S)$ is similar. 
To save space, we only write out DoFs for $\mathbb V^{\div}_{k}(\boldsymbol{r}_2, \boldsymbol r_3;\mathbb S)$.
Take $\mathbb P_{k}(T;\mathbb S)$ as the space of shape functions with $k\geq\max\{2 r_2^{\texttt{v}} + 1, 2 r_3^{\texttt{v}} + 2\}$. Assume $\boldsymbol r_3 \geq \boldsymbol r_2\ominus 1, r_2^{\texttt{v}}\geq 0$ and $( \boldsymbol r_2, \boldsymbol r_2\ominus 1, k)$ is $(\div; \mathbb S)$ stable. The DoFs are 
\begin{subequations}\label{eq:divSr2r3dof}
\begin{align}
\nabla^i\boldsymbol{\tau}(\texttt{v}), & \quad i=0,\ldots, r_2^{\texttt{v}}, \label{eq:3dCr2divSr2r3V1}\\
\nabla^j\div\boldsymbol{\tau}(\texttt{v}),  & \quad j=r_2^{\texttt{v}},\ldots, r_3^{\texttt{v}}, \label{eq:3dCr2divSr2r3V2}\\
% e: ni-nj
\int_e (\boldsymbol{n}_i^{\intercal}\boldsymbol{\tau}\boldsymbol{n}_j)\,q \dd s, &\quad q\in \mathbb B_{k}(e; r_2^{\texttt{v}}), 1\leq i\leq j\leq 2, \textrm{ if } r_2^{e}=-1, \label{eq:3dCrdivSr2r3E2}\\
% e: t-t
\int_e \frac{\partial^{j}(\boldsymbol t^{\intercal}\boldsymbol{\tau}\boldsymbol{t})}{\partial n_1^{i}\partial n_2^{j-i}}q \dd s, &\quad q\in \mathbb B_{k-j}(e; r_2^{\texttt{v}} - j), 0\leq i\leq j\leq r_2^e,\label{eq:3dCr2divSr2r3E4}\\
% e: t-n1
\int_e \frac{\partial^{j}(\boldsymbol t^{\intercal}\boldsymbol{\tau}\boldsymbol{n}_{1})}{\partial n_1^{i}\partial n_2^{j-i}} q \dd s, &\quad q\in \mathbb B_{k-j}(e; r_2^{\texttt{v}} - j), 0\leq i\leq j\leq r_2^e, \label{eq:3dCr2divSr2r3E5}\\
% e: n1-n1
\int_e \frac{\partial^{j}(\boldsymbol n_1^{\intercal}\boldsymbol{\tau}\boldsymbol{n}_{1})}{\partial n_1^{i}\partial n_2^{j-i}} q \dd s, &\quad q\in \mathbb B_{k-j}(e; r_2^{\texttt{v}} - j), 0\leq i\leq j\leq r_2^e, \label{eq:3dCr2divSr2r3E6}\\
% e: n_2
\int_e \partial_{n_1}^j(\boldsymbol{\tau}\boldsymbol{n}_2)\cdot \boldsymbol q \dd s, &\quad \boldsymbol{q}\in \mathbb B_{k-j}^3(e; r_2^{\texttt{v}} - j), 0\leq j\leq r_2^e, \label{eq:3dCr2divSr2r3E1}\\
% e: div
\int_e \frac{\partial^{j}(\div\boldsymbol{\tau})}{\partial n_1^{i}\partial n_2^{j-i}} \cdot\boldsymbol{q} \dd s, &\quad \boldsymbol{q}\in \mathbb B_{k-1-j}^3(e; r_3^{\texttt{v}} - j), 0\leq i\leq j\leq r_3^{e}, \label{eq:3dCr2divSr2r3E7}\\
% f: Pif n
\int_f (\Pi_f\boldsymbol{\tau}\boldsymbol{n})\cdot\boldsymbol{q} \dd S, &\quad \boldsymbol{q}\in (\mathbb B_{k}^{\div} (f; \boldsymbol r_2)/{\rm RM}(f)) \oplus {\rm RM}(f), 
\label{eq:3dCrdivSr2r3F1}\\
% f: n tau n
\int_f (\boldsymbol n^{\intercal}\boldsymbol{\tau}\boldsymbol{n})\,q \dd S, &\quad q\in (\mathbb B_{k}(f;(\boldsymbol{r}_2)_+)/\mathbb P_1(f))\oplus \mathbb P_1(f),
\label{eq:3dCrdivSr2r3F2}\\
% f Pif tau Pif
\int_f \partial_n^j(\Pi_f\boldsymbol{\tau}\Pi_f): \boldsymbol{q} \dd S, &\quad \boldsymbol{q}\in \mathbb B_{k - j}(f; \boldsymbol r_2-j)\otimes \mathbb S(f),  0\leq j\leq r_2^{f}, \label{eq:3dCr2divSr2r3F2}\\
%% t_2 t_1
%\int_f \partial_n^j(\boldsymbol{t}_{1}^{\intercal}\boldsymbol{\tau}\boldsymbol{t}_{2})\,q \dd S, &\quad q\in \mathbb B_{k - j} (f; \boldsymbol r_2-j),  0\leq j\leq r_2^{f}, \label{eq:3dCr2divSr2r3F3}\\
%% n t_1
%\int_f \partial_n^j(\boldsymbol{n}^{\intercal}\boldsymbol{\tau}\boldsymbol{t}_{1})\,q \dd S, &\quad q\in \mathbb B_{k - j} (f;\begin{pmatrix}
%r_2^{\texttt{v}} \\
%r_2^e
%\end{pmatrix}-j),  0\leq j\leq r_2^{f}, \label{eq:3dCr2divTfemdofF4}\\
% f div
\int_f \partial_n^j(\div\boldsymbol{\tau})\cdot\boldsymbol{q} \dd S, &\quad \boldsymbol{q}\in \mathbb B_{k-1 - j}^3 (f; \boldsymbol r_3-j),  0\leq j\leq r_3^{f}, \label{eq:3dCr2divTfemdofF5}\\
\int_T (\div\boldsymbol{\tau})\cdot\boldsymbol{q}\dx, &\quad \boldsymbol{q}\in \mathbb B_{k-1}^3(\boldsymbol{r}_3)/{\rm RM}, \label{eq:3dCr2divSr2r3T1}\\
\int_T \boldsymbol{\tau}:\boldsymbol{q} \dx, &\quad \boldsymbol{q}\in \mathbb B_{k}^{\div}(\boldsymbol{r}_2;\mathbb S)\cap\ker(\div) \label{eq:3dCr2divSSr2r3T2}
\end{align}
\end{subequations}
for each $\texttt{v}\in \Delta_{0}(T)$, $e\in \Delta_{1}(T)$ and $f\in \Delta_{2}(T)$.

\begin{lemma}\label{lem:PkSr2r3unisolvence}
Let $( \boldsymbol r_2, \boldsymbol r_2\ominus 1, k)$ be $(\div; \mathbb S)$ stable and $\boldsymbol r_3 \geq \boldsymbol r_2\ominus 1$. The DoFs~\eqref{eq:divSr2r3dof} are uni-solvent for $\mathbb P_{k}(T;\mathbb S)$.
\end{lemma}
\begin{proof}
The introduced DoFs given by~\eqref{eq:3dCr2divSr2r3V2},~\eqref{eq:3dCr2divSr2r3E7}, and~\eqref{eq:3dCr2divTfemdofF5}-\eqref{eq:3dCr2divSr2r3T1} play a crucial role in characterizing the divergence of $\boldsymbol \tau$. The total number of these DoFs, along with~\eqref{eq:3dCr2divSr2r3V1}, is independent of $\boldsymbol r_3$, specifically given by the expression:
$$
6 { r_2^{\texttt{v}} + 3 \choose 3} + \dim \mathbb P_{k-1}^3(T) - \dim {\rm RM} - 3 { r_2^{\texttt{v}} + 2 \choose 3}.
$$
This count remains unaffected by variations in $\boldsymbol{r}_3$. For convenience, we can proceed with $\boldsymbol{r}_3=\boldsymbol{r}_2\ominus1$. Then by making comparisons with~\eqref{eq:divSdof}, we deduce that the number of DoFs~\eqref{eq:divSr2r3dof} is equal to $\dim\mathbb P_{k}(T;\mathbb S)$.

Suppose we have $\boldsymbol{\tau}\in\mathbb P_{k}(T;\mathbb S)$ satisfying the vanishing conditions for all DoFs~\eqref{eq:divSr2r3dof}. It then follows that $\boldsymbol \tau\boldsymbol n|_{\partial T} = \boldsymbol{0}$. By applying integration by parts and utilizing the vanishing DoFs~\eqref{eq:3dCrdivSr2r3F1}-\eqref{eq:3dCrdivSr2r3F2}, we deduce the critical relation:
$$
\int_T (\div\boldsymbol{\tau})\cdot\boldsymbol{q}\dx=0, \quad \boldsymbol{q}\in {\rm RM}.
$$
This result, combined with DoFs~\eqref{eq:3dCr2divSr2r3V2},~\eqref{eq:3dCr2divSr2r3E7}, and~\eqref{eq:3dCr2divTfemdofF5}-\eqref{eq:3dCr2divSr2r3T1}, which pertain to the divergence of $\boldsymbol \tau$, leads to the conclusion that $\div\boldsymbol{\tau}=\boldsymbol{0}$.

The situation is somewhat analogous when dealing with edges. By expressing $\div$ in the frame $\{\boldsymbol t, \boldsymbol n_1, \boldsymbol n_2 \}$, we uncover a representation of $\div\boldsymbol{\tau}$ that encompasses the partial derivatives along tangential ($\boldsymbol t$) and normal ($\boldsymbol n_1$, $\boldsymbol n_2$) directions:
$$
\div\boldsymbol{\tau}=\partial_{t}(\boldsymbol{\tau}\boldsymbol{t})+\partial_{n_1}(\boldsymbol{\tau}\boldsymbol{n_1})+\partial_{n_2}(\boldsymbol{\tau}\boldsymbol{n_2}).
$$
Taking into account~\eqref{eq:3dCr2divSr2r3V1}-\eqref{eq:3dCr2divSr2r3E7}, it becomes evident that $\nabla^j\boldsymbol{\tau}$ vanishes along edges, where $0\leq j\leq r_2^e$. A similar reasoning applies to faces, where the decomposition into $\{\boldsymbol n, \boldsymbol t_1, \boldsymbol t_2 \}$ aids in expressing 
$$\div\boldsymbol{\tau}=\partial_{n}(\boldsymbol{\tau}\boldsymbol{n})+\partial_{t_1}(\boldsymbol{\tau}\boldsymbol{t_1})+\partial_{t_2}(\boldsymbol{\tau}\boldsymbol{t_2}),$$ 
 and highlighting that $\nabla^j\boldsymbol{\tau}$ vanishes along faces for $0\leq j\leq r_2^f$. This collective analysis demonstrates that $\boldsymbol{\tau}\in\mathbb B_{k}^{\div}(\boldsymbol{r}_2;\mathbb S)\cap\ker(\div)$. This observation, coupled with DoF~\eqref{eq:3dCr2divSSr2r3T2}, solidifies the conclusion that $\boldsymbol{\tau}=\boldsymbol{0}$.\end{proof}

\begin{example}\rm 
The space $\mathbb V_k^{\div}(\boldsymbol r_2, \boldsymbol r_3; \mathbb S)$ for $\boldsymbol r_2 = \boldsymbol r_3 =
( 0,  -1,  -1)^{\intercal}$ and $k\geq 4$ has been constructed recently in~\cite{Hu;Liang;Ma;Zhang:2022conforming}.
\end{example}

\subsection{Smooth $H(\div\div^+;\mathbb S)$ and $H(\div\div;\mathbb S)$ elements}
In this subsection, we proceed to construct various $H(\div\div)$-conforming finite elements characterized by a smoothness vector $\boldsymbol{r}$. 
Define the spaces:
$$
 H(\div\div;\mathbb S):=\{\boldsymbol{\tau}\in {L}^2(\Omega;\mathbb S): \div\div\boldsymbol{\tau}\in L^2(\Omega)\},
$$
$$
 H(\div\div^+;\mathbb S):=\{\boldsymbol{\tau}\in {L}^2(\Omega;\mathbb S): \div\boldsymbol{\tau}\in H(\div,\Omega)\}.
$$
It is evident that the inclusion $ H(\div\div^+;\mathbb S)\subset H(\div\div;\mathbb S)$ holds.% We will construct $ H(\div\div^+)$-conforming finite elements in this subsection.

We will now proceed to construct finite elements that are $H(\div\div^+;\mathbb S)$-conforming. Specifically, when $r^f\geq1$, we define $\mathbb V^{\div\div^+}_{k}(\boldsymbol{r};\mathbb S)$ as $\mathbb V_k(\boldsymbol{r})\otimes\mathbb S$. For the cases where $r^f=-1$ or $r^f=0$, we will make use of a recent approach presented in~\cite{Hu;Ma;Zhang:2020family} and~\cite{Chen;Huang:2021divFinite}.  The space of shape functions is still $\mathbb P_k(T;\mathbb S)$. By modifying the DoFs~\eqref{eq:divSdof}, which are originally designed for $H(\div;\mathbb S)$-conforming finite elements, we ensure that $\div \boldsymbol \tau$ belongs to $\mathbb V^{\div}_{k-1}(\boldsymbol r\ominus 1)$, satisfying the $H(\div;\mathbb S)$-conforming condition. 

Take $\mathbb P_{k}(T;\mathbb S)$ as the shape function space. The DoFs are
\begin{subequations}\label{eq:divdiv+dof}
\begin{align}
% v
\label{eq:divdiv+SdV}
\nabla^j\boldsymbol{\tau} (\texttt{v}), & \quad j=0,1,\ldots,r^{\texttt{v}}, \\
% e: Dn
\label{eq:divdiv+SdE1}
\int_e \frac{\partial^{j}\boldsymbol{\tau}}{\partial n_1^{i}\partial n_2^{j-i}}:\boldsymbol{q} \dd s, & \quad \boldsymbol{q} \in \mathbb B_{k-j}(e; r^{\texttt{v}}-j)\otimes\mathbb S, 0\leq i\leq j\leq r^{e}, \\
% e: ni nj
\int_e (\boldsymbol{n}_i^{\intercal}\boldsymbol{\tau}\boldsymbol{n}_j)\,q \dd s, &\quad q\in \mathbb B_{k}(e; r^{\texttt{v}}), 1\leq i\leq j\leq 2, \textrm{ if } r^{e}=-1, \label{eq:divdiv+SdE2}\\
%%
%\int_e (\boldsymbol{t}^{\intercal}\boldsymbol{\tau}\boldsymbol{n}_f)\,q \dd s, &\quad q\in \mathbb B_{k}(e; r^{\texttt{v}}), e\subset\partial f, \textrm{ if } r^{e}=-1, \label{eq:divdiv+SdF0}\\
% f: tau n
\label{eq:divdiv+SdF1}
\int_f \boldsymbol{\tau} :\boldsymbol{q} \dd S, & \quad  \boldsymbol{q}\in \mathbb B_{k}(f;\boldsymbol r)\otimes \mathbb S, \textrm{ if } r^f=0,\\
\label{eq:divdiv+SdF2}
\int_f (\Pi_f\boldsymbol{\tau}\boldsymbol n) \cdot \boldsymbol{q} \dd S, & \quad  \boldsymbol{q}\in \mathbb B_{k}^{\div}(f;\boldsymbol r)/{\rm RM}(f) \oplus {\rm RM}(f), \textrm{ if } r^f=-1,\\
% f: n tau n
\label{eq:divdiv+SdF22}
\int_f (\boldsymbol n^{\intercal}\boldsymbol{\tau}\boldsymbol n) \ q \dd S, & \quad  q\in \mathbb B_{k}(f;\boldsymbol{r}_+)/\mathbb P_1(f)\oplus \mathbb P_1(f),\textrm{ if } r^f=-1,\\
% f: n div
\label{eq:divdiv+SdF3}
\int_f \boldsymbol{n}^{\intercal}\div\boldsymbol{\tau} \, q \dd S, & \quad {q}\in \mathbb B_{k-1}(f;\boldsymbol r\ominus1),\\
\label{eq:divdiv+SdT1}
\int_T (\div\boldsymbol{\tau})\cdot\boldsymbol{q}\dx, & \quad \boldsymbol{q}\in \mathbb B_{k-1}^{\div}(\boldsymbol r\ominus1)/{\rm RM}, \\
\label{eq:divdiv+SdT2}
\int_T \boldsymbol{\tau}:\boldsymbol{q} \dx, & \quad \boldsymbol{q} \in \mathbb B_k^{\div}(\boldsymbol r;\mathbb S)\cap\ker(\div),
\end{align}
\end{subequations}
for each $\texttt{v}\in \Delta_{0}(T)$, $e\in \Delta_{1}(T)$ and $f\in \Delta_{2}(T)$.

\begin{lemma}
Let $r^f=-1,0$.
The DoFs~\eqref{eq:divdiv+dof} are uni-solvent for $\mathbb P_{k}(T;\mathbb S)$.
\end{lemma}
\begin{proof}
We first consider the case when $r^f=-1$. If we compare the DoFs~\eqref{eq:divSdof} for constructing $\mathbb V_k^{\div}(\boldsymbol r; \mathbb S)$ with the DoFs required for $H(\div\div^+;\mathbb S)$-conforming finite elements, the primary distinction lies in the volume DoF~\eqref{eq:3dCrdivSfemdofT} for $\mathbb B_{k}^{\div}(\boldsymbol{r};\mathbb S)$. In the new context, this particular DoF is replaced by three alternative DoFs:~\eqref{eq:divdiv+SdF3},~\eqref{eq:divdiv+SdT1}, and~\eqref{eq:divdiv+SdT2}.

Let $E_0 = \mathbb B_k^{\div}(\boldsymbol r;\mathbb S)\cap\ker(\div)$, and let $E_0^{\bot}$ denote its $L^2$-orthogonal complement of $\mathbb B_k^{\div}(\boldsymbol r;\mathbb S)$. By performing a decomposition of the dual space, we arrive at:
$$
(\mathbb B_k^{\div}(\boldsymbol r;\mathbb S))' = ( E_0)' \oplus (E_0^{\bot})'.
$$ 
DoF~\eqref{eq:divdiv+SdT2} is exactly a basis of $(E_0)'$. 
The subspace $\div E_0^{\bot}$ can be uniquely determined through the DoFs~\eqref{eq:divdiv+SdF3} and~\eqref{eq:divdiv+SdT1}, both of which are consistent with the requirements for constructing $H(\div)$-conforming elements. 
%When $r^f=0$, the structure of the space $\mathbb V^{\div}_k(\boldsymbol r; \mathbb S)$ remains similar, with the primary change being that $\mathbb B_k^{\div}(\boldsymbol r;\mathbb S) = \mathbb B_k(\boldsymbol r) \otimes \mathbb S$.

As we count the dimensions, it is essential to note that $\div \mathbb B_k^{\div}(\boldsymbol r;\mathbb S) = \mathbb B^3_{k-1}(\boldsymbol r\ominus 1)/{\rm RM}$. The difference in the number of DoFs between~\eqref{eq:3dCrdivSfemdofT} and the newly introduced DoFs~\eqref{eq:divdiv+SdF3}-\eqref{eq:divdiv+SdT2} is given by:
$$
\dim\mathbb B_{k-1}^3(T;\boldsymbol r\ominus1)-\dim\mathbb B_{k-1}^{\div}(T;\boldsymbol r\ominus1) - 4\dim\mathbb B_{k-1}(f;\boldsymbol r\ominus1)=0.
$$
So the sum of number of DoFs~\eqref{eq:divdiv+dof} is equal to $\mathbb P_{k}(T;\mathbb S)$.

Now, let $\boldsymbol{\tau}\in\mathbb P_k(T;\mathbb S)$ and assume that all the DoFs~\eqref{eq:divdiv+SdV}-\eqref{eq:divdiv+SdT2} vanish. Due to the vanishing DoFs~\eqref{eq:divdiv+SdV}-\eqref{eq:divdiv+SdE1} and~\eqref{eq:divdiv+SdF3}, we can infer that $\div\boldsymbol{\tau}\in\mathbb B_{k-1}^{\div}(\boldsymbol r\ominus1)$. 
By considering the vanishing DoFs~\eqref{eq:divdiv+SdF1}-\eqref{eq:divdiv+SdF22} and the integration by parts, we deduce that
\begin{equation*}
\int_T (\div\boldsymbol{\tau})\cdot\boldsymbol{q}\dx=0  \quad \forall~\boldsymbol{q}\in {\rm RM},
\end{equation*}
which means $\div\boldsymbol{\tau}\in\mathbb B_{k-1}^{\div}(\boldsymbol r\ominus1)/{\rm RM}$. This together with the vanishing DoF~\eqref{eq:divdiv+SdT1} yields $\div\boldsymbol{\tau}=\boldsymbol{0}$. Finally, utilizing the uniqueness of the DoFs~\eqref{eq:divSdof} for $H(\div;\mathbb S)$-conforming finite elements, we conclude that $\boldsymbol{\tau}=\boldsymbol{0}$.

This completes the explanation of the construction for $H(\div\div^+;\mathbb S)$-conforming finite elements for the case $r^f=-1$. The case for $r^f=0$ follows a similar logic, with the primary difference being in the structure of the bubble space $\mathbb B_k^{\div}(\boldsymbol r;\mathbb S) = \mathbb B_k(\boldsymbol r) \otimes \mathbb S$.
\end{proof}

When $r^f \geq 1$, $\mathbb V^{\div\div^+}_{k}(\boldsymbol{r};\mathbb S) = \mathbb V_k(\boldsymbol{r})\otimes \mathbb S$. 
When $r^f=-1,0$,
define
\begin{align*}
\mathbb V^{\div\div^+}_{k}(\boldsymbol{r};\mathbb S) &= \{\boldsymbol{\tau}\in {L}^2(\Omega;\mathbb S): \boldsymbol{\tau}|_T\in\mathbb P_k(T;\mathbb S)\textrm{ for all } T\in\mathcal T_h, \\
&\qquad\qquad\qquad\qquad\textrm{ and all the DoFs~\eqref{eq:divdiv+dof} are single-valued}\}. 
\end{align*}
%when $r_2^f=-1,0$.
Due to DoFs~\eqref{eq:divdiv+SdF1}-\eqref{eq:divdiv+SdF3},
 $\mathbb V^{\div\div^+}_{k}(\boldsymbol{r};\mathbb S)\subset H(\div\div^+;\mathbb S)$.

Next we use the following BGG diagram
\begin{equation*}%\label{eq:divdiv+}
\begin{tikzcd}
&
\mathbb V_{k}^{\div\div^+} (
\boldsymbol r; \mathbb S
)
% \arrow[dl,swap,"{\rm mskw}"]   
 \arrow{r}{\div}
 &
\mathbb V^{\div}_{k-1}(
\boldsymbol r\ominus 1)
 \arrow{r}{}
 & \boldsymbol{0} \\
 &
\mathbb V_{k-1}^{\div}(
\boldsymbol r\ominus 1) 
 \arrow[ur,"{\rm id}"] 
 \arrow{r}{\div}
 & 
\mathbb V^{L^2}_{k-2}(
\boldsymbol r\ominus 2) 
\arrow[r] 
 & 0 
\end{tikzcd}
\end{equation*}
 to prove the divdiv stability.

\begin{lemma}\label{lem:divdivSr2onto}
Assume $(\boldsymbol r, \boldsymbol r\ominus 1, k)$ is $(\div; \mathbb S)$ stable, and $r^f=-1,0$. It holds that
\begin{equation}\label{eq:divdiv+Sr2onto}
\div\mathbb V^{\div\div^+}_{k}(\boldsymbol{r};\mathbb S)=\mathbb V^{\div}_{k-1}(\boldsymbol{r}\ominus1).
\end{equation}
\end{lemma}
\begin{proof}
Clearly $\div\mathbb V^{\div\div^+}_{k}(\boldsymbol{r};\mathbb S)\subseteq\mathbb V^{\div}_{k-1}(\boldsymbol{r}\ominus1)$, then it suffices to count the dimensions.
Both $\dim\mathbb V^{\div}_{k}(\boldsymbol{r};\mathbb S)-\dim\mathbb V^{\div\div^+}_{k}(\boldsymbol{r};\mathbb S)$ and $\dim\mathbb V^{L^2}_{k-1}(\boldsymbol{r}\ominus1;\mathbb R^3) - \dim\mathbb V^{\div}_{k-1}(\boldsymbol{r}\ominus1)$ equal
$$
\big(4|\Delta_3(\mathcal T_h)| - |\Delta_2(\mathcal T_h)|\big)\dim \mathbb B_{k-1}(f;\boldsymbol r\ominus1),
$$
that is
$$
\dim\mathbb V^{\div}_{k}(\boldsymbol{r};\mathbb S)-\dim\mathbb V^{\div\div^+}_{k}(\boldsymbol{r};\mathbb S) = \dim\mathbb V^{L^2}_{k-1}(\boldsymbol{r}\ominus1;\mathbb R^3) - \dim\mathbb V^{\div}_{k-1}(\boldsymbol{r}\ominus1).
$$
As $\div(\mathbb V^{\div}_{k}(\boldsymbol{r};\mathbb S)) = \mathbb V^{L^2}_{k-1}(\boldsymbol{r}\ominus1;\mathbb R^3)$ and the modification will not change $\ker(\div)$, we get $\dim\div\mathbb V^{\div\div^+}_{k}(\boldsymbol{r};\mathbb S) = \dim\mathbb V^{\div}_{k-1}(\boldsymbol{r}\ominus1)$ and~\eqref{eq:divdiv+Sr2onto} follows. 
%Then by~\eqref{eq:divSr2onto},
%\begin{align*}
%&\quad\,\dim\div\mathbb V^{\div\div^+}_{k}(\boldsymbol{r}_2;\mathbb S)-\dim\mathbb V^{\div}_{k-1}(\boldsymbol{r}_2\ominus1) \\
%&=\dim\mathbb V^{\div\div^+}_{k}(\boldsymbol{r}_2;\mathbb S)-\dim\big(\mathbb V^{\div}_{k}(\boldsymbol{r}_2;\mathbb S)\cap\ker(\div)\big)-\dim\mathbb V^{\div}_{k-1}(\boldsymbol{r}_2\ominus1) \\
%&=\dim\mathbb V^{\div\div^+}_{k}(\boldsymbol{r}_2;\mathbb S)-\dim\mathbb V^{\div}_{k}(\boldsymbol{r}_2;\mathbb S) \\
%&\quad\,+ \dim\mathbb V^{L^2}_{k-1}(\boldsymbol{r}_2\ominus1;\mathbb R^3)-\dim\mathbb V^{\div}_{k-1}(\boldsymbol{r}_2\ominus1) =0,
%\end{align*}
%which ends the proof.
\end{proof}
Combined the div stability for $(\boldsymbol r\ominus 1, \boldsymbol r\ominus 2, k-1)$, we conclude the divdiv stability. 
\begin{corollary}\label{cor:divdiv+}
Assume $(\boldsymbol r, \boldsymbol r\ominus 1, k)$ is $(\div; \mathbb S)$ stable and  $(\boldsymbol r\ominus 1, \boldsymbol r\ominus 2, k-1)$ is div stable. Then it holds that
\begin{equation*}%\label{eq:divdiv+SL2onto}
\div\div\mathbb V^{\div\div^+}_{k}(\boldsymbol{r};\mathbb S)=\mathbb V^{L^2}_{k-2}(\boldsymbol{r}\ominus 2).
\end{equation*}
\end{corollary}

Next we modify the DoFs~\eqref{eq:divdiv+dof} slightly to get an $H(\div\div; \mathbb S)$-conforming element. Take $\mathbb P_k(T;\mathbb S)$ as the space of shape functions. 
When $r^f\geq1$, define $\mathbb V^{\div\div}_{k}(\boldsymbol{r};\mathbb S):=\mathbb V^{\div\div^+}_{k}(\boldsymbol{r};\mathbb S)=\mathbb V_k(\boldsymbol{r})\otimes\mathbb S$. 
For $r^f=-1,0$, the degrees of freedom are
\begin{subequations}
\begin{align}
% v
\label{eq:divdivSdV}
\nabla^j\boldsymbol{\tau} (\texttt{v}), & \quad j=0,1,\ldots,r^{\texttt{v}}, \texttt{v}\in \Delta_0(T),\\
% e: Dn
\label{eq:divdivSdE1}
\int_e \frac{\partial^{j}\boldsymbol{\tau}}{\partial n_1^{i}\partial n_2^{j-i}}:\boldsymbol{q} \dd s, & \quad \boldsymbol{q} \in \mathbb P_{k - 2(r^{\texttt{v}}+1) + j}(e;\mathbb S), 0\leq i\leq j\leq r^{e}, e\in \Delta_1(T), \\
% e: ni nj
\int_e (\boldsymbol{n}_i^{\intercal}\boldsymbol{\tau}\boldsymbol{n}_j)\,q \dd s, &\quad q\in \mathbb B_{k}(e; r^{\texttt{v}}), 1\leq i\leq j\leq 2, e\in \Delta_1(T), \textrm{ if } r^{e}=-1, \label{eq:divdivSdE2}\\
%%
%\int_e (\boldsymbol{t}^{\intercal}\boldsymbol{\tau}\boldsymbol{n}_f)\,q \dd s, &\quad q\in \mathbb B_{k}(e; r^{\texttt{v}}), e\subset\partial f, \textrm{ if } r^{e}=-1, \label{eq:divdivSdF0}\\
% f: tau n
\label{eq:divdivSdF1}
\int_f \boldsymbol{\tau} :\boldsymbol{q} \dd S, & \quad  \boldsymbol{q}\in \mathbb B_{k}(f;\boldsymbol r)\otimes \mathbb S, f\in \Delta_2(T), \textrm{ if } r^f=0,\\
% f: n tau n
\label{eq:divdivSdF22}
\int_f (\boldsymbol n^{\intercal}\boldsymbol{\tau}\boldsymbol n) \ q \dd S, & \quad  q\in \mathbb B_{k}(f;\boldsymbol{r}_+), f\in \Delta_2(T), \textrm{ if } r^f=-1,\\
% f: n div
\label{eq:divdivSdF3}
\int_f \tr_2^{\div\div}(\boldsymbol \tau) \, q \dd S, & \quad  q\in \mathbb B_{k-1}(f;\boldsymbol r\ominus1), f\in \Delta_2(T),\\
\label{eq:divdivSdT1}
\int_T (\div\boldsymbol{\tau})\cdot\boldsymbol{q}\dx, & \quad \boldsymbol{q}\in \mathbb B_{k-1}^{\div}(\boldsymbol r\ominus1)/{\rm RM}, \\
\label{eq:divdivSdT2}
\int_T \boldsymbol{\tau}:\boldsymbol{q} \dx, & \quad \boldsymbol{q} \in \mathbb B_k^{\div}(\boldsymbol r;\mathbb S)\cap\ker(\div),\\
% f
\label{eq:divdivSdF2}
\int_f (\Pi_f\boldsymbol{\tau}\boldsymbol n) \cdot \boldsymbol{q} \dd S, & \quad  \boldsymbol{q}\in \mathbb B_{k}^{\div}(f;\boldsymbol r), f\in \Delta_2(T), \textrm{ if } r^f=-1.
\end{align}
\end{subequations}
The modification is introduced in~\eqref{eq:divdivSdF3}, where we now enforce the continuity condition:
\begin{equation*}%\label{eq:tr2divdiv}
\tr_2^{\div\div}(\boldsymbol \tau) = \boldsymbol n^{\intercal}\div \boldsymbol \tau +  \div_f(\Pi_f\boldsymbol\tau \boldsymbol n)
\end{equation*}
instead of enforcing continuity for both $\boldsymbol n^{\intercal}\div \boldsymbol \tau$ and $\Pi_f\boldsymbol\tau \boldsymbol n$. Notably, the face DoF~\eqref{eq:divdivSdF2} associated with $\Pi_f\boldsymbol\tau \boldsymbol n$ has been moved to the end to signify its role as a local DoF that contributes to the divdiv bubble space. This implies that~\eqref{eq:divdivSdF2} can take different values in different elements containing the shared face $f$.

Defining the space
\begin{align*}
\mathbb V^{\div\div}_{k}(\boldsymbol{r};\mathbb S) &= \{\boldsymbol{\tau}\in {L}^2(\Omega;\mathbb S): \boldsymbol{\tau}|_T\in\mathbb P_k(T;\mathbb S)\textrm{ for all } T\in\mathcal T_h, \\
&\qquad\textrm{ and all the DoFs~\eqref{eq:divdivSdV}-\eqref{eq:divdivSdT2} are single-valued}\},
\end{align*}
we find that for the case $r^f = 0$, $\mathbb V^{\div\div}_{k}(\boldsymbol{r};\mathbb S)$ and $\mathbb V^{\div\div^+}_{k}(\boldsymbol{r};\mathbb S)$ are indistinguishable. However, in the scenario where $r^f = -1$, again due to \eqref{eq:divdivSdF2}, we can deduce that:
$$
\mathbb V^{\div\div^+}_{k}(\boldsymbol{r};\mathbb S)\subset \mathbb V^{\div\div}_{k}(\boldsymbol{r};\mathbb S).
$$
Consequently, this inclusion relationship ensures the divdiv stability.
\begin{corollary}%\label{cor:divdiv}
Assume $(\boldsymbol r, \boldsymbol r\ominus 1, k)$ is $(\div; \mathbb S)$ stable and  $(\boldsymbol r\ominus 1, \boldsymbol r\ominus 2, k-1)$ is div stable. Then it holds that
\begin{equation*}%\label{eq:divdivSL2onto}
\div\div\mathbb V^{\div\div}_{k}(\boldsymbol{r};\mathbb S)=\mathbb V^{L^2}_{k-2}(\boldsymbol{r}\ominus 2).
\end{equation*}
\end{corollary}

In a similar vein to the inequality constraint that ensures $(\div; \mathbb S)$ stability, we also have a comparable flexibility when it comes to divdiv elements. Consider the scenario where $(\boldsymbol r_2, \boldsymbol r_3, k)$ are $(\div; \mathbb S)$ stable, where $\boldsymbol r_3\geq \boldsymbol r_2\ominus 1$, and further assume that $(\boldsymbol r_3, \boldsymbol r_3 \ominus 1, k-1)$ exhibit div stability. In such cases, following the sequence
\begin{equation*}
\mathbb V^{\div}_k(\boldsymbol r_2, \boldsymbol r_3; \mathbb S) \xrightarrow{\div} \mathbb V^{\div}_{k-1}(\boldsymbol r_3) \xrightarrow{\div} \mathbb V^{L^2}_{k-2}(\boldsymbol r_3\ominus 1) \to 0,
\end{equation*}
and the approach in Section~\ref{sec:r2r3}, we can similarly define the spaces $\mathbb V^{\div\div^+}_{k}(\boldsymbol{r}_2, \boldsymbol r_3\ominus 1;\mathbb S)$ and $\mathbb V^{\div\div}_{k}(\boldsymbol{r}_2, \boldsymbol r_3\ominus 1;\mathbb S)$, with the constraint $\boldsymbol r_3 \geq \boldsymbol r_2\ominus 1$.

Regarding the changes in DoFs, when $r_3^f \geq 0$, $\mathbb V^{\div}_{k-1}(\boldsymbol r_3) = \mathbb V^{\grad}_{k-1}(\boldsymbol r_3; \mathbb R^3)$. This implies that $\mathbb V^{\div\div^+}_{k}(\boldsymbol{r}_2, \boldsymbol r_3\ominus 1;\mathbb S) = \mathbb V^{\div}_k(\boldsymbol r_2, \boldsymbol r_3; \mathbb S)$, and therefore, there is no need to modify the DoFs~\eqref{eq:divSr2r3dof}. On the other hand, when $r_3^f = -1$, we partition the DoF~\eqref{eq:3dCr2divSr2r3T1} associated with the bubble component of $\div \boldsymbol \tau$ into:
\begin{subequations}
\begin{align}
 \label{eq:divdiv+r2r3SdF3}
\int_f \boldsymbol{n}^{\intercal}\div\boldsymbol{\tau} \, q \dd S, & \quad {q}\in \mathbb B_{k-1}(f; \boldsymbol r_3), f\in \Delta_2(T),\\
\label{eq:divdiv+r2r3SdT1}
\int_T (\div\boldsymbol{\tau})\cdot\boldsymbol{q}\dx, & \quad \boldsymbol{q}\in \mathbb B_{k-1}^{\div}(\boldsymbol r_3)/{\rm RM},
\end{align}
\end{subequations}
while retaining all other DoFs from~\eqref{eq:divSr2r3dof}. This gives DoFs for $\mathbb V^{\div\div^+}_{k}(\boldsymbol{r}_2, \boldsymbol r_3\ominus 1;\mathbb S)$. 

To construct $\mathbb V^{\div\div}_{k}(\boldsymbol{r}_2, \boldsymbol r_3\ominus 1;\mathbb S)$, we further replace DoF~\eqref{eq:divdiv+r2r3SdF3} by 
$$
\int_f \tr_2^{\div\div}(\boldsymbol{\tau}) \, q \dd S, \quad {q}\in \mathbb B_{k-1}(f; \boldsymbol r_3), f\in \Delta_2(T),
$$
and treat~\eqref{eq:3dCrdivSr2r3F1} local. The procedure is the same as before and thus the details are skipped. 

\begin{corollary}\label{cor:divdiv}
Let $(\boldsymbol r_2, \boldsymbol r_3, k)$ be $(\div; \mathbb S)$ stable with $\boldsymbol r_3\geq \boldsymbol r_2\ominus 1$ and $(\boldsymbol r_3, \boldsymbol r_3 \ominus 1, k-1)$ be div stable. Then it holds that
\begin{equation*}%\label{eq:divdivr2r3SL2onto}
\div\div\mathbb V^{\div\div^+}_{k}(\boldsymbol{r}_2, \boldsymbol r_3\ominus 1;\mathbb S)=\div\div\mathbb V^{\div\div}_{k}(\boldsymbol{r}_2, \boldsymbol r_3\ominus 1;\mathbb S) = \mathbb V^{L^2}_{k-2}(\boldsymbol r_3\ominus 1).
\end{equation*}
\end{corollary}

\begin{example}\rm 
The space $\mathbb V_k^{\div\div^+}(\boldsymbol r_2, \boldsymbol r_3\ominus 1; \mathbb S)$ for $\boldsymbol r_2 = \boldsymbol r_3 =
( 0,  -1,  -1)^{\intercal}$ and $k\geq 4$ has been constructed recently in~\cite{Hu;Liang;Ma;Zhang:2022conforming}.
 
\end{example}

%
%Define 
%\begin{align*}
%\mathbb B^{\div\div^+}_{k}(\boldsymbol{r};\mathbb S) &= \{\boldsymbol{\tau}\in\mathbb P_k(T;\mathbb M): \textrm{ all the DoFs~\eqref{eq:divdiv+SdV}-\eqref{eq:divdiv+SdF3} vanish}\}, \\
%\mathbb B^{\div\div}_{k}(\boldsymbol{r};\mathbb S) &= \{\boldsymbol{\tau}\in\mathbb P_k(T;\mathbb M): \textrm{ all the DoFs~\eqref{eq:divdivSdV}-\eqref{eq:divdivSdF3} vanish}\}.
%\end{align*}
%Then $\mathbb B^{\div\div^+}_{k}(\boldsymbol{r};\mathbb S)\subseteq \mathbb B^{\div\div}_{k}(\boldsymbol{r};\mathbb S)$ and when $r^f = -1$
%$$
%\dim \mathbb B^{\div\div}_{k}(\boldsymbol{r};\mathbb S) - \dim \mathbb B^{\div\div^+}_{k}(\boldsymbol{r};\mathbb S) = 4|\Delta_3(\mathcal T_h)| \dim \mathbb B_{k}^{\div}(f;\begin{pmatrix}
%r^{\texttt{v}} \\
%r^e
%\end{pmatrix}).
%$$

For convenience, we introduce the following notation:
\begin{align*}
\mathbb B_{k}^{\div\div^+}(\boldsymbol{r};\mathbb S) &= \{  \boldsymbol{\tau}\in\mathbb P_k(T;\mathbb S) :  \text{ all the DoFs~\eqref{eq:divdiv+SdV}-\eqref{eq:divdiv+SdF3} vanish } \},\\
\mathbb B_{k}^{\div\div}(\boldsymbol{r};\mathbb S) &= \{  \boldsymbol{\tau}\in\mathbb P_k(T;\mathbb S) :  \text{ all the DoFs~\eqref{eq:divdivSdV}-\eqref{eq:divdivSdF3} vanish } \}.
\end{align*}
By definition, we have the inclusion relationship:
$$
\mathbb B^{\div\div^+}_{k}(\boldsymbol{r};\mathbb S)\subseteq \mathbb B^{\div\div}_{k}(\boldsymbol{r};\mathbb S),
$$
since the face DoF \eqref{eq:divdivSdF2} associated with $\Pi_f\boldsymbol\tau \boldsymbol n$ is zero in the bubble space $\mathbb B^{\div\div^+}_{k}(\boldsymbol{r};\mathbb S)$ and non-zero in $\mathbb B^{\div\div}_{k}(\boldsymbol{r};\mathbb S)$.

By the same proof with BGG diagram for the bubble spaces, we conclude the divdiv stability for these bubble spaces and refer to  Lemma~\ref{lem:fembubbledivdiv+complex} for detailed proof.
\begin{lemma}
Assume $\boldsymbol r$ satisfies the condition in Lemma \ref{lem:divbubbleontoST}, $\boldsymbol r\ominus1$ satisfies the condition in Lemma \ref{lem:divbubbleonto}, and $k\geq \max\{2r^{\texttt{v}}+1,3\}$.
%  Assume $(\boldsymbol r_2, \boldsymbol r_2\ominus 1, k)$ is $(\div;\mathbb S)$ stable and $(\boldsymbol r_2\ominus 1, \boldsymbol r_3, k-1)$ is div stable \mnote{ $k$ is smaller for bubble div stability.}. 
Then it holds that
\begin{equation*}%\label{eq:divdivr2r3SL2onto}
\div\div\mathbb B^{\div\div^+}_{k}(\boldsymbol{r};\mathbb S)=\div\div\mathbb B^{\div\div}_{k}(\boldsymbol{r};\mathbb S) = \mathbb B^{L^2}_{k-2}(\boldsymbol r\ominus 2)/\mathbb P_1(T).
\end{equation*}
\end{lemma}
% \begin{lemma}
%  Assume $(\boldsymbol r_2, \boldsymbol r_2\ominus 1, k)$ is $(\div;\mathbb S)$ stable and $(\boldsymbol r_2\ominus 1, \boldsymbol r_3, k-1)$ is div stable \mnote{ $k$ is smaller for bubble div stability.}. Then it holds that
% \begin{equation*}%\label{eq:divdivr2r3SL2onto}
% \div\div\mathbb B^{\div\div^+}_{k}(\boldsymbol{r}_2, \boldsymbol r_3\ominus 1;\mathbb S)=\div\div\mathbb B^{\div\div}_{k}(\boldsymbol{r}_2, \boldsymbol r_3\ominus 1;\mathbb S) = \mathbb B^{L^2}_{k-2}(\boldsymbol r_3\ominus 1)/\mathbb P_1(T).
% \end{equation*}
% \end{lemma}

\section{Finite Element Complexes}\label{sec:FEMcomplexes}
In the preceding section, our focus was primarily on the last two columns of the diagram~\eqref{eq:3DBGGdomain}, where we established the $(\div; \mathbb X)$ stability. In this section, we turn our attention to the first three columns of the diagram~\eqref{eq:3DBGGdomain} to derive finite element complexes.

\subsection{Finite element Hessian complexes}\label{sec:femHessian}
Let 
$$
\boldsymbol r_0 \geq (4, 2, 1)^{\intercal}, \quad \boldsymbol r_1 = \boldsymbol r_0 -2, \quad \boldsymbol r_2=\boldsymbol r_1\ominus1, \quad \boldsymbol r_3= \boldsymbol r_2\ominus1, \quad k+2\geq 2r_0^{\texttt v} + 1.
$$
Assume both $(\boldsymbol r_0, \boldsymbol r_0-1, \boldsymbol r_1, \boldsymbol r_2)$ and $(\boldsymbol r_0-1, \boldsymbol r_1, \boldsymbol r_2, \boldsymbol r_3)$ are valid de Rham smoothness sequences. As a discretization of \eqref{eq:3DBGGdomainHessian}, we will justify the following BGG diagram
\begin{equation}\label{eq:hessBGG}
%\adjustbox{scale=0.9,center}{% \begin{equation*}%\label{eq:BGGelasticity}
\begin{tikzcd}
\mathbb V_{k+2}^{\grad} ( \boldsymbol r_0 )
\arrow{r}{\grad}
&
\mathbb V_{k+1}^{\curl} (\boldsymbol r_0 - 1)
% \arrow[dl,swap,"{\rm id}"]   
 \arrow{r}{\curl}
 &
\mathbb V^{\div}_k(\boldsymbol r_1)\cap \ker(\div)
 \arrow{r}{\div}
 & 0
 \\
\mathbb V_{k+1}^{\grad}(\boldsymbol r_0-1; \mathbb R^3)
 \arrow[ur,swap,"{\rm id}"'] 
 \arrow{r}{\grad}
 & 
\widetilde{\mathbb V}_{k}^{\curl}(
\boldsymbol r_1
;\mathbb M) 
 \arrow[ur,"- 2\vskw"] \arrow{r}{\curl}
 & 
\mathbb V^{\div}_{k-1}(
\boldsymbol r_2
; \mathbb T) 
\cap \ker(\div)
 \arrow[ur,"\tr"] \arrow{r}{\div} \arrow[r] 
 &\boldsymbol{0}, 
\end{tikzcd}
%}
\end{equation}
where $\widetilde{\mathbb V}_{k}^{\curl}(
\boldsymbol r_1
;\mathbb M)  = \{ \boldsymbol \tau \in \mathbb V_{k}^{\curl}(\boldsymbol r_1;\mathbb M) : \curl \boldsymbol \tau \in  \mathbb V^{\div}_{k-1}(
\boldsymbol r_2
; \mathbb T) \cap \ker(\div)
\}. $

\begin{lemma}
The mapping $\vskw: \widetilde{\mathbb V}_{k}^{\curl}(
\boldsymbol r_1
;\mathbb M) 
\to \mathbb V^{\div}_{k}(
\boldsymbol r_1) 
\cap \ker(\div)
$ is well-defined and surjective.  
\end{lemma}
\begin{proof}
Recall that the parallelogram formed by the north-east diagonal  $\nearrow$ and the horizontal operators is anticommutative. By substituting the differential operators with the face normal vector, we derive the following relationship:
\begin{equation*}%\label{eq:vskwn}
2(\vskw\boldsymbol\tau)\cdot\boldsymbol{n} = -\tr(\boldsymbol\tau\times\boldsymbol{n}).
\end{equation*}
Since $\boldsymbol\tau\times\boldsymbol{n}$ remains continuous for an $H(\curl)$ function $\boldsymbol \tau$, it follows that $(\vskw\boldsymbol\tau)\boldsymbol{n}$ also maintains continuity across each face. In other words, $\vskw\boldsymbol\tau\in H(\div,\Omega)$. Moreover, due to the traceless property of $\curl \boldsymbol \tau$:
\begin{equation*}
\div 2(\vskw\boldsymbol\tau) = \tr \curl \boldsymbol \tau = 0.
\end{equation*}
This implies that $\vskw(\widetilde{\mathbb V}_{k}^{\curl}(\boldsymbol r_1;\mathbb M)) \subseteq \mathbb V^{\div}_{k}(\boldsymbol r_1) \cap \ker(\div)$.

For the surjectivity proof, we select a function $\boldsymbol u\in \mathbb V^{\div}_{k}(\boldsymbol r_1) \cap \ker(\div)$. Since $\div \boldsymbol u=0$, we can find a $\boldsymbol v\in \mathbb V_{k+1}^{\curl} (\boldsymbol r_0 - 1)$ such that $\curl \boldsymbol v = \boldsymbol u$. Consequently, $\grad \boldsymbol v \in \widetilde{\mathbb V}_{k}^{\curl}(\boldsymbol r_1 ;\mathbb M)$ and $2\vskw \grad \boldsymbol v = \curl \boldsymbol v = \boldsymbol u$. This completes the argument for surjectivity.
\end{proof}

Define
\begin{equation}\label{eq:VcurlSBGG}
\mathbb V^{\curl}_{k}(\boldsymbol{r}_1;\mathbb S): = \widetilde{\mathbb V}_{k}^{\curl}(
\boldsymbol r_1
;\mathbb M) 
\cap \ker(\vskw). 
\end{equation}
With this space in place, we can apply the BGG construction to~\eqref{eq:hessBGG} and subsequently deduce the finite element Hessian complex.

\begin{theorem}
Let $\boldsymbol r_0 \geq 1, \boldsymbol r_1 = \boldsymbol r_0 -2, \boldsymbol r_2=\boldsymbol r_1\ominus1, \boldsymbol r_3= \boldsymbol r_2\ominus1$ be valid smoothness vectors. Assume $(\boldsymbol r_2, \boldsymbol r_3, k-1)$ is $(\div; \mathbb T)$ stable and $k + 2\geq 2r_0^{\texttt{v}}+1$. 
Then the finite element Hessian complex
\begin{equation*}%\label{eq:femhesscomplex}
%\resizebox{.9\hsize}{!}{$
\mathbb P_1\xrightarrow{\subset} \mathbb V^{\grad}_{k+2}(\boldsymbol{r}_0)\xrightarrow{\hess}\mathbb V^{\curl}_{k}(\boldsymbol{r}_1;\mathbb S)\xrightarrow{\curl} \mathbb V^{\div}_{k-1}(\boldsymbol{r}_2;\mathbb T) \xrightarrow{\div} \mathbb V^{L^2}_{k-2}(\boldsymbol{r}_3;\mathbb R^3)\xrightarrow{}\boldsymbol0
%$}
\end{equation*}
is exact. 
%The bubble complex
%\begin{equation}\label{eq:femhessbubblecomplex}
%%\resizebox{.9\hsize}{!}{$
%0\to \mathbb B^{\grad}_{k+2}(\boldsymbol{r}_1+2)\xrightarrow{\hess}\mathbb B^{\curl}_{k}(\boldsymbol{r}_1;\mathbb S)\xrightarrow{\curl} \mathbb B^{\div}_{k-1}(\boldsymbol{r}_2;\mathbb T) \xrightarrow{\div} \mathbb B_{k-2}(\boldsymbol{r}_3;\mathbb R^3)/{\rm RT}\xrightarrow{}\boldsymbol0,
%%$}
%\end{equation}
%is exact.
\end{theorem}
While we have established finite element descriptions for several spaces, including $\mathbb V_{k+2}^{\grad}(\boldsymbol{r}_0)$, $\mathbb V_{k-1}^{\div}(\boldsymbol{r}_2;\mathbb{T})$, and $\mathbb V_{k-2}^{L^2}(\boldsymbol{r}_3;\mathbb{R}^3)$, we currently face a challenge in providing a finite element description for $\mathbb V_{k}^{\curl}(\boldsymbol{r}_1;\mathbb{S})$. The DoFs for this space are not easily derived using the BGG construction and will be discussed in Section \ref{sec:edgeelements}.

Throughout this paper, we will further simplify the notation in examples by presenting only the smoothness vectors and omitting the space notation, which should be clear from the differential operator attached to the space.

\begin{example}\rm 
Taking $\boldsymbol{r}_0=(4,2,1)^{\intercal}$, $\boldsymbol{r}_1=\boldsymbol{r}_0-2$, $\boldsymbol{r}_2=\boldsymbol{r}_1\ominus1$, $\boldsymbol{r}_3=\boldsymbol{r}_2\ominus1$, and $k\geq7$, we obtain the first finite element Hessian complex constructed in~\cite{HuLiang2021}
$$
\mathbb P_1 \hookrightarrow
\begin{pmatrix}
	 4\\
	 2\\
	 1
\end{pmatrix}
\xrightarrow{\hess}
\begin{pmatrix}
 2\\
 0\\
 -1
\end{pmatrix}
\xrightarrow{\curl}
\begin{pmatrix}
	1\\
	-1\\
	-1
\end{pmatrix}
\xrightarrow{\div} 
\begin{pmatrix}
	 0\\
	 -1\\
	 -1
\end{pmatrix} \to \boldsymbol{0}.
$$
\end{example}

% Include our discrete Hessian complex and mention the difference. The smoothness vector for $C^1$ element is $(2, 1, 1)$. ... 
\begin{remark}\rm
The macro-element Hessian complex based on the Alfeld split and virtual element Hessian complex with $k\geq3$ developed in \cite{Chen;Huang:2020Discrete} correspond, within our notation, to the following sequence:
$$
\mathbb P_1 \hookrightarrow
\begin{pmatrix}
	 2\\
	 1\\
	 1
\end{pmatrix}
\xrightarrow{\hess}
\begin{pmatrix}
 0\\
 -1\\
 -1
\end{pmatrix}
\xrightarrow{\curl}
\begin{pmatrix}
	0\\
	-1\\
	-1
\end{pmatrix}
\xrightarrow{\div} 
\begin{pmatrix}
	 -1\\
	 -1\\
	 -1
\end{pmatrix} \to \boldsymbol{0}.
$$
This particular complex cannot be derived using our framework due to the fact that $\boldsymbol r_0 = (2, 1, 1)^{\intercal}$ does not constitute a valid smoothness vector for a $C^1$-element. 
%In~\cite{Christiansen;Gopalakrishnan;Guzman;Hu:2020discrete}, the space $\mathbb V^{\grad}((2, 1, 1)^{\intercal}; \mathbb R^3)$ is constructed on Alfeld splits of tetrahedra.
\qed
\end{remark}

%\subsection{Finite element Hessian complexes with inequality constraint} 
%copy to here. Will proof read later.
%There are several variants of finite element spaces and Hessian complexes. 
By modifying the smoothness constraint for $\mathbb V_{k}^{\curl}(\boldsymbol r_1; \mathbb S)$ to $\boldsymbol r_2\geq\boldsymbol r_1\ominus1$ and then introducing the space $\mathbb V_{k}^{\curl}(\boldsymbol r_1, \boldsymbol r_2; \mathbb S) = \{\boldsymbol \tau\in\mathbb V_{k}^{\curl}(\boldsymbol r_1; \mathbb S): \curl \boldsymbol \tau \in \mathbb V^{\div}_{k-1}(\boldsymbol{r}_2;\mathbb T)\}$, we can derive finite element Hessian complexes with inequality constraints of the smoothness vectors. 

\begin{corollary}
Let $\boldsymbol r_0 \geq 1, \boldsymbol r_1 = \boldsymbol r_0 -2, \boldsymbol r_2\geq\boldsymbol r_1\ominus1, \boldsymbol r_3\geq \boldsymbol r_2\ominus1$. 
Assume $(\boldsymbol r_2, \boldsymbol r_3, k-1)$ is $(\div; \mathbb T)$ stable and $k \geq \max \{2r_0^{\texttt{v}} - 1, 2r_2^{\texttt{v}} + 2, 2r_3^{\texttt{v}} +3 \}$. 
Then the finite element Hessian complex
\begin{equation*}%\label{eq:femhesscomplexgeneral}
%\resizebox{.9\hsize}{!}{$
\mathbb P_1\xrightarrow{\subset} \mathbb V^{\grad}_{k+2}(\boldsymbol{r}_0)\xrightarrow{\hess}\mathbb V^{\curl}_{k}(\boldsymbol{r}_1,\boldsymbol{r}_2;\mathbb S)\xrightarrow{\curl} \mathbb V^{\div}_{k-1}(\boldsymbol{r}_2, \boldsymbol{r}_3;\mathbb T) \xrightarrow{\div} \mathbb V^{L^2}_{k-2}(\boldsymbol{r}_3;\mathbb R^3)\xrightarrow{}\boldsymbol0
%$}
\end{equation*}
is exact.
\end{corollary}

By a similar proof, we can obtain the bubble Hessian complex and refer to appendix Section~\ref{sec:bubblehesscomplex} for the detailed proof.
\begin{proposition}
Let $\boldsymbol r_0 \geq (4,2,1)^{\intercal}, \boldsymbol r_1 = \boldsymbol r_0 -2, \boldsymbol r_2=\boldsymbol r_1\ominus1, \boldsymbol r_3= \boldsymbol r_2\ominus1$. 
Assume $\boldsymbol r_2$ satisfy the condition in Lemma \ref{lem:divbubbleontoST}, and $k\geq 2r_2^{\texttt{v}}+2$.
Then the bubble Hessian complex
\begin{equation*}%\label{eq:fembubblehessiancomplex}
\resizebox{.925\hsize}{!}{$
0\xrightarrow{\subset} \mathbb B_{k+2}(\boldsymbol{r}_0)\xrightarrow{\hess}\mathbb B^{\curl}_{k}(\boldsymbol{r}_1;\mathbb S)\xrightarrow{\curl} \mathbb B^{\div}_{k-1}(\boldsymbol{r}_2;\mathbb T) \xrightarrow{\div} \mathbb B_{k-2}(\boldsymbol{r}_3;\mathbb R^3)/{\rm RT}\xrightarrow{}0
$}
\end{equation*}
is exact, where $\mathbb B^{\curl}_{k}(\boldsymbol{r}_1;\mathbb S):=\{\boldsymbol{\tau}\in\mathbb B_{k}(\boldsymbol{r}_1;\mathbb S): \boldsymbol{\tau}\times\boldsymbol{n}|_{\partial T}=0\}$.
\end{proposition}

%%%%%%%%%%%%%%%%%%%%%%%%%
\subsection{Finite element elasticity complexes}
%%%%%%%%%%%%%%%%%%%%%%%%%
%Let $\boldsymbol{r}_1\geq\begin{pmatrix}
%1\\
%0 \\
%-1
%\end{pmatrix}$ be a valid smoothness vector, and set 
Let 
\begin{equation}\label{eq:relasticity}
\boldsymbol{r}_0 \geq (2, 1, 0)^{\intercal}, 
\;
\boldsymbol r_1 =\boldsymbol{r}_0-1, 
%\geq\begin{pmatrix}
%1\\
%0 \\
%-1
%\end{pmatrix},
\;
 \boldsymbol{r}_2=\max\{\boldsymbol{r}_1\ominus2, (0, -1, -1)^{\intercal}\},
\;
 \boldsymbol{r}_3=\boldsymbol{r}_2\ominus1.
\end{equation}
Assume $(\boldsymbol r_1 \ominus 1, \boldsymbol r_2, k)$ is div-stable, and that $(\boldsymbol r_2, \boldsymbol r_3, k-1)$ exhibits $(\div;\mathbb S)$ stability. 

We first consider a slightly smoother case
\begin{equation}\label{eq:smoothr0}
\bs r_0 \geq (4, 2, 1)^{\intercal}, \quad \bs r_1 = \bs r_0-1, \quad \bs r_2 = \bs r_1\ominus 2, \quad \bs r_3 = \bs r_2\ominus 1.
\end{equation}
%The requirement $r_2^{\texttt{v}}\geq 0$ in~\eqref{eq:relasticity} is essential for the $(\div;\mathbb S)$ stability.
%{We will construct the BGG diagram for $\boldsymbol{r}_0\geq (3, 1, 0)^{\intercal}$}
%\begin{equation}\label{eq:BGGelasticity}
%\adjustbox{scale=0.96,center}{
%\begin{tikzcd}
%\mathbb V_{k+2}^{\curl} (\boldsymbol{r}_0, (\boldsymbol{r}_1)_+)
%\arrow{r}{\grad}
%&
%\widehat{\widetilde{\mathbb V}}_{k+1}^{\curl} (\boldsymbol{r}_1; \mathbb M
%)
% \arrow{r}{\curl}
% &
%\widetilde{\mathbb V}^{\div}_k(
%\boldsymbol{r}_1\ominus1; \mathbb M
%)
% \arrow{r}{\div}
% & 0
% \\
% %
%\mathbb V_{k+1}^{\grad}((\boldsymbol{r}_1)_+; \mathbb R^3)
% \arrow[ur,swap,"\mskw"'] \arrow{r}{\grad}
% & 
%\widetilde{\mathbb V}_{k}^{\curl}(
%\boldsymbol{r}_1\ominus1
%;\mathbb M) 
% \arrow[ur,swap,"S"'] \arrow{r}{\curl}
% & 
%\mathbb V^{\div}_{k-1}(
%\boldsymbol{r}_2
%; \mathbb S) 
% \arrow[ur,swap,"-2\vskw"'] \arrow{r}{\div} \arrow[r] 
% & \mathbb V^{L^2}_{k-2}(
%\boldsymbol{r}_3
%; \mathbb R^3).
%\end{tikzcd}
%}
%\end{equation}
%Our endeavor initiates with the bottom complex. 
Evidently $\boldsymbol r_2 = (\boldsymbol r_1\ominus 1)\ominus 1$. Moreover, by our assumptions, $(\boldsymbol r_2, \boldsymbol r_3, k-1)$ is div-stable. Consequently, the sequence $(\boldsymbol r_1, \boldsymbol r_1\ominus 1, \boldsymbol r_2, \boldsymbol r_3)$ forms a valid de Rham smoothness sequence:
\begin{equation*}
\mathbb R^3\xrightarrow{\subset}\mathbb V_{k+1}^{\grad} (\boldsymbol{r}_1; \mathbb R^3)\xrightarrow{\grad}\mathbb V_{k}^{\curl} (\boldsymbol{r}_1 \ominus 1; \mathbb M
) \xrightarrow{\curl} \mathbb V_{k-1}^{\div}(\boldsymbol r_2;\mathbb M)\cap\ker(\div) \xrightarrow{\div}  \boldsymbol{0}.
\end{equation*}
Employing a $\widetilde{\quad}$ operation to transfer from $\mathbb M$ to $\mathbb S$, we then obtain the exact sequence
\begin{equation*}
\mathbb R^3\xrightarrow{\subset}\mathbb V_{k+1}^{\grad} (\boldsymbol{r}_1; \mathbb R^3)\xrightarrow{\grad}\widetilde{\mathbb V}_{k}^{\curl} (\boldsymbol{r}_1 \ominus 1; \mathbb M
) \xrightarrow{\curl} \mathbb V_{k-1}^{\div}(\boldsymbol r_2;\mathbb S)\cap\ker(\div) \xrightarrow{\div}  \boldsymbol{0},
\end{equation*}
%
%the resultant exact sequence can be presented as:
%\begin{equation}\label{eq:elasticitybottomcomplex1}
%\mathbb R^3\xrightarrow{\subset}\mathbb V_{k+1}^{\grad} (
%(\boldsymbol{r}_1)_+; \mathbb R^3)\xrightarrow{\grad}\widetilde{\mathbb V}_{k}^{\curl} (\boldsymbol{r}_1 \ominus 1; \mathbb M
%) \xrightarrow{\curl} \mathbb V_{k-1}^{\div}(\boldsymbol r_2;\mathbb S)\cap\ker(\div) \xrightarrow{\div}  \boldsymbol{0},
%\end{equation}
where 
$$
\widetilde{\mathbb V}_{k}^{\curl} (\boldsymbol{r}_1 \ominus 1; \mathbb M
) :=\{\boldsymbol{\tau}\in\mathbb V_{k}^{\curl}(\boldsymbol r_1\ominus1;\mathbb M):\curl\boldsymbol{\tau}\in\mathbb V^{\div}_{k-1}(
\boldsymbol{r}_2
; \mathbb S)\}.
$$

%We then proceed to address the top complex of~\eqref{eq:BGGelasticity}.  
\begin{lemma}\label{lem:div0Mfem1}
Define
\begin{equation}\label{eq:elasticityVtilde}
\widetilde{\mathbb V}_{k}^{\div}(\boldsymbol r_1\ominus1;\mathbb M):=S(\widetilde{\mathbb V}_{k}^{\curl}(\boldsymbol r_1\ominus1;\mathbb M)).
\end{equation} It holds
$$
\widetilde{\mathbb V}_{k}^{\div}(\boldsymbol r_1\ominus1;\mathbb M)\subseteq \mathbb V_{k}^{\div}(\boldsymbol r_1\ominus1;\mathbb M)\cap\ker(\div).
$$
\end{lemma}
\begin{proof}
By $(S\boldsymbol\tau)\boldsymbol{n} = -2\vskw(\boldsymbol\tau\times\boldsymbol{n})$, $\widetilde{\mathbb V}_{k}^{\div}(\boldsymbol r_1\ominus1;\mathbb M)\subseteq \mathbb V_{k}^{\div}(\boldsymbol r_1\ominus1;\mathbb M)$.
And $\div\widetilde{\mathbb V}^{\div}_k(
\boldsymbol{r}_1\ominus1; \mathbb M
)=\boldsymbol{0}$ follows from $\div(S\boldsymbol\tau) = 2\vskw(\curl\boldsymbol\tau).$
\end{proof}

It is straightforward to demonstrate that the sequence $(\boldsymbol r_0, \boldsymbol r_1, \boldsymbol r_1\ominus 1, \boldsymbol r_2)$ is also a valid de Rham smoothness sequence. With the application of a $\widetilde{\quad}$ operation to this finite element de Rham complex, an exact sequence holds as follows:
\begin{equation*}%\label{eq:elasticitybottomcomplex2}
\mathbb R^3\xrightarrow{\subset}\mathbb V_{k+2}^{\grad} (
\boldsymbol{r}_0; \mathbb R^3) \xrightarrow{\grad}\widetilde{\mathbb V}_{k+1}^{\curl} (\boldsymbol{r}_1; \mathbb M
) \xrightarrow{\curl} \widetilde{\mathbb V}^{\div}_k(
\boldsymbol{r}_1\ominus1; \mathbb M
) \xrightarrow{\div}  \boldsymbol{0},
\end{equation*}
where $\widetilde{\mathbb V}^{\div}_k(\boldsymbol{r}_1\ominus1; \mathbb M)$ is defined in \eqref{eq:elasticityVtilde} and
$$
\widetilde{\mathbb V}_{k+1}^{\curl}(\boldsymbol r_1;\mathbb M):=\{\boldsymbol{\tau}\in\mathbb V_{k+1}^{\curl}(\boldsymbol r_1;\mathbb M): \curl\boldsymbol{\tau}\in\widetilde{\mathbb V}_{k}^{\div}(\boldsymbol r_1\ominus1;\mathbb M)\}.
$$
%and in light of the fact that $\boldsymbol r_0\geq 0$, $\mathbb V_{k+2}^{\curl} (
%\boldsymbol{r}_0) = \mathbb V_{k+2}^{\grad} (
%\boldsymbol{r}_0; \mathbb R^3)$. 

\begin{theorem}\label{th:elasticitysmooth}
Let $(\boldsymbol r_0, \boldsymbol r_1, \boldsymbol r_2, \boldsymbol r_3)$ be given by~\eqref{eq:smoothr0}. Assume $(\boldsymbol r_1 \ominus 1, \boldsymbol r_2, k)$ is div stable, and $(\boldsymbol r_2, \boldsymbol r_3, k-1)$ is $(\div;\mathbb S)$ stable. Let $k+2\geq 2r_0^{\texttt{v}}+1$. We have the BGG diagram
\begin{equation*}
\adjustbox{scale=0.96,center}{
\begin{tikzcd}
%\mathbb V_{k+2}^{\curl} (\boldsymbol{r}_0)
\mathbb V_{k+2}^{\grad} (
\boldsymbol{r}_0; \mathbb R^3)
\arrow{r}{\grad}
&
\widetilde{\mathbb V}_{k+1}^{\curl} (\boldsymbol{r}_1; \mathbb M
)
 \arrow{r}{\curl}
 &
\widetilde{\mathbb V}^{\div}_k(
\boldsymbol{r}_1\ominus1; \mathbb M
)
 \arrow{r}{\div}
 & 0
 \\
\mathbb V_{k+1}^{\grad}(\boldsymbol{r}_1; \mathbb R^3)
 \arrow[ur,swap,"\mskw"'] \arrow{r}{\grad}
 & 
\widetilde{\mathbb V}_{k}^{\curl}(
\boldsymbol{r}_1\ominus1
;\mathbb M) 
 \arrow[ur,swap,"S"'] \arrow{r}{\curl}
 & 
\mathbb V^{\div}_{k-1}(
\boldsymbol{r}_2
; \mathbb S) 
 \arrow[ur,swap,"-2\vskw"'] \arrow{r}{\div} \arrow[r] 
 & \mathbb V^{L^2}_{k-2}(
\boldsymbol{r}_3
; \mathbb R^3), 
\end{tikzcd}
}
\end{equation*}
which leads to the finite element elasticity complex
\begin{align*}%\label{eq:femelasticitycomplex3d+421}
{\rm RM}\xrightarrow{\subset}\mathbb V_{k+2}^{\grad} (\boldsymbol{r}_0;\mathbb R^3) \xrightarrow{\defm} \mathbb V_{k+1}^{\inc^+}(\boldsymbol{r}_1;\mathbb{S}) \xrightarrow{\inc}  \mathbb V_{k-1}^{\div}(\boldsymbol{r}_2;\mathbb{S}) \xrightarrow{\div} \mathbb V_{k-2}^{L^2} (\boldsymbol{r}_3; \mathbb R^3)\rightarrow \boldsymbol{0},
\end{align*}
where $\mathbb V_{k+1}^{\inc^+}(\boldsymbol{r}_1;\mathbb{S}):=\{\boldsymbol{\tau}\in \mathbb V_{k+1}^{\curl} (\boldsymbol{r}_1; \mathbb M): \boldsymbol \tau \in \mathbb S, \curl\boldsymbol{\tau}\in\widetilde{\mathbb V}_{k}^{\div}(\boldsymbol r_1\ominus1;\mathbb M)\}$.
\end{theorem}
\begin{proof}
The bijectiveness of the mapping $S:\widetilde{\mathbb V}_{k}^{\curl}(\boldsymbol r_1\ominus1;\mathbb M)\to\widetilde{\mathbb V}_{k}^{\div}(\boldsymbol r_1\ominus1;\mathbb M)$ is a direct outcome of its definition. 
As $r_1^f\geq 0$, $\vskw \widetilde{\mathbb V}_{k+1}^{\curl} (\boldsymbol{r}_1; \mathbb M) \subseteq \mathbb V_{k+1}^{\grad}(\boldsymbol{r}_1; \mathbb R^3)$. Therefore $\widetilde{\mathbb V}_{k+1}^{\curl} (\boldsymbol{r}_1; \mathbb M) /\mskw \mathbb V_{k+1}^{\grad}(\boldsymbol{r}_1; \mathbb R^3) = \mathbb V_{k+1}^{\inc^+}(\boldsymbol{r}_1;\mathbb{S})$. 
We arrive at our desired conclusion by applying the BGG framework.
\end{proof}

%In the notation $\inc^+$, the super-script ${}^+$ means $\boldsymbol \tau \in H(\curl)$
%
%Notice that by definition $\boldsymbol \tau\in \mathbb V_{k+1}^{\inc^+}(\boldsymbol{r}_1;\mathbb{S})$ implies $\boldsymbol \tau \in \mathbb S$ and 
%$\curl \boldsymbol \tau\in \widetilde{\mathbb V}^{\div}_k(
%\boldsymbol{r}_1\ominus1; \mathbb M
%)
%$, then 
%$\inc \boldsymbol \tau \in {L}^2(\Omega; \mathbb S)$. The ${}^{+}$ means $\curl \boldsymbol \tau \in {L}^2(\Omega; \mathbb M)$ 
%
As $r_1^f \geq 0$, $\mathbb V_{k+1}^{\inc^+}(\boldsymbol{r}_1;\mathbb{S}) \subset H^1(\Omega;\mathbb S)$ and $\mathbb V_{k+2}^{\grad} (\boldsymbol{r}_0;\mathbb R^3)\subset H^{2} (\Omega;\mathbb R^3)$.
The corresponding continuous version is the elasticity complex that initiates with $H^2(\Omega;\mathbb R^3)$:
\begin{equation}\label{eq:smoothelasticitycomplex}
{\rm RM}\xrightarrow{\subset}  H^{2} (\Omega;\mathbb R^3) \xrightarrow{\defm}  H^1(\mathrm{inc}, \Omega;\mathbb{S}) \xrightarrow{\inc}  H(\operatorname{div}, \Omega; \mathbb{S}) \xrightarrow{\div}  L^{2} (\Omega; \mathbb R^3)\rightarrow \boldsymbol{0}.
\end{equation}
Here, the space 
\[
H^1(\mathrm{inc}, \Omega;\mathbb{S}) := \{\boldsymbol \tau \in H^1(\Omega; \mathbb S) : \inc \boldsymbol \tau \in L^2(\Omega; \mathbb S)\}.
\]

\begin{example}\rm 
%Taking $\boldsymbol{r}_0=(4,2,1)^{\intercal}$, $\boldsymbol{r}_1=\boldsymbol{r}_0-1$, $\boldsymbol{r}_2=\boldsymbol{r}_1\ominus 2$, $\boldsymbol{r}_3=\boldsymbol{r}_2\ominus1$ and 
For the case of $k+1\geq 8$ and $\boldsymbol{r}_0=(4,2,1)^{\intercal}$, we arrive at a discrete elasticity complex~\eqref{eq:smoothelasticitycomplex} originating from a subspace of $H^2(\Omega)$:
\begin{equation*}
{\rm RM}\xrightarrow{\subset} 
\begin{pmatrix}
	 4\\
	 2\\
	 1
\end{pmatrix}
\xrightarrow{\defm}
\begin{pmatrix}
 3\\
 1\\
 0
\end{pmatrix}
\xrightarrow{\inc}
\begin{pmatrix}
	1\\
	-1\\
	-1
\end{pmatrix}
\xrightarrow{\div} 
\begin{pmatrix}
	 0\\
	 -1\\
	 -1
\end{pmatrix} \rightarrow \boldsymbol{0}.
\end{equation*}
\qed
\end{example}

%%%%%%%%%%%%%%%%%%%%%%%%%%%%%
\begin{remark}\rm
It is worth noting that a finite element elasticity complex has been recently established for the Alfeld split of a tetrahedron~\cite{Christiansen;Gopalakrishnan;Guzman;Hu:2020discrete}, presenting another discrete counterpart of~\eqref{eq:smoothelasticitycomplex}. The discrete elasticity complex introduced in~\cite{Christiansen;Gopalakrishnan;Guzman;Hu:2020discrete} corresponds, within our notation, to the following sequence:
\begin{equation*}
{\rm RM}\xrightarrow{\subset} 
\begin{pmatrix}
	 2\\
	 1\\
	 1
\end{pmatrix}
\xrightarrow{\defm}
\begin{pmatrix}
 1\\
 0\\
 0
\end{pmatrix}
\xrightarrow{\inc}
\begin{pmatrix}
	0\\
	-1\\
	-1
\end{pmatrix}
\xrightarrow{\div} 
\begin{pmatrix}
	 -1\\
	 -1\\
	 -1
\end{pmatrix} \rightarrow \boldsymbol{0}.
\end{equation*}
This particular complex cannot be derived using our framework due to the fact that $\boldsymbol r_0 = (2, 1, 1)^{\intercal}$ does not constitute a valid smoothness vector for a $C^1$-element. In~\cite{Christiansen;Gopalakrishnan;Guzman;Hu:2020discrete}, the space $\mathbb V^{\grad}((2, 1, 1)^{\intercal}; \mathbb R^3)$ is constructed on Alfeld splits of tetrahedra.
%\mnote{ can we change the first one to VEM space? \\
%{Ans: Yes. I have constructed the VEM space with $\boldsymbol{r}_0=\begin{pmatrix}1 \\ 1 \\ 0\end{pmatrix}$ for analysis of the adaptive mixed fem of elasticity problem.}}
%We can apply one $\widetilde{\quad}$ operation to the discrete elasticity complex constructed in~\cite{Christiansen;Gopalakrishnan;Guzman;Hu:2020discrete}, i.e., use a subspace of $\mathbb V_{k+2}^{\grad}((2, 1, 1)^{\intercal}; \mathbb R^3; {\rm Alfel})$ s.t. $\defm \boldsymbol u \in \mathbb V_{k+1}^{\inc} ((1,0,0)^{\intercal}; \mathbb S)$. 
%
%Another approach is to use the VEM space $\mathbb V_{k+2}^{\grad}((2, 1, 1)^{\intercal}; \mathbb R^3; {\rm VEM})$ constructed in~\cite{Chen;Huang:2020Discrete}. The advantage is the degree of polynomial can be reduced to $k+2 \geq 5$ but with additional $2k(k-1)$ non-polynomials shape functions. The minimal degree of $\inc$-conforming element can be reduced to $4$ and VEM space for $H(\div;\mathbb S)$ space is used. 
%{Take an example and use BGG for VEM space.}
\qed
\end{remark}

For $\boldsymbol{\tau}\in\mathbb V_{k+1}^{\inc^+}(\boldsymbol{r}_1;\mathbb{S})$, the super-script ${}^+$ means some additional smoothness more than $\boldsymbol \tau \in H(\inc)$.
To relax the smoothness of $\mathbb V_{k+1}^{\inc^+}(\boldsymbol{r}_1;\mathbb{S})$, we introduce
\begin{equation}\label{eq:VincS}
\mathbb V_{k+1}^{\inc}(\boldsymbol r_1;\mathbb S):=\{\boldsymbol{\tau}\in\mathbb V_{k+1}(\boldsymbol r_1)\otimes \mathbb S: \inc\boldsymbol{\tau}\in\mathbb V^{\div}_{k-1}(\boldsymbol{r}_2; \mathbb S)\}.
\end{equation}
Obviously $\mathbb V_{k+1}^{\inc^+}(\boldsymbol r_1;\mathbb S)\subseteq \mathbb V_{k+1}^{\inc}(\boldsymbol r_1;\mathbb S)$. By applying an inverse hat operation, we can obtain the finite element elasticity complex for $\bs r_0 = (r_0^{\texttt{v}}, r_0^e, 0)^{\intercal}$ with $r_0^{\texttt{v}}\geq4$ and $r_0^e\geq1$.

\begin{corollary}\label{cor:femelasticityinc410}
Let $(\boldsymbol r_0, \boldsymbol r_1, \boldsymbol r_2, \boldsymbol r_3)$ be given by~\eqref{eq:relasticity} and $\bs r_0 = (r_0^{\texttt{v}}, r_0^e, 0)^{\intercal}$ with $r_0^{\texttt{v}}\geq4$ and $r_0^e\geq1$. Assume $(\boldsymbol r_1 \ominus 1, \boldsymbol r_2, k)$ is div stable, $(\boldsymbol r_2, \boldsymbol r_3, k-1)$ is $(\div;\mathbb S)$ stable, and $k+2\geq 2r_0^{\texttt{v}}+1$. We have the finite element elasticity complex
 \begin{equation*}
% \label{eq:femelasticitycomplex3d410}
{\rm RM}\xrightarrow{\subset}\mathbb V_{k+2}^{\grad} (\boldsymbol{r}_0;\mathbb R^3) \xrightarrow{\defm} \mathbb V_{k+1}^{\inc}(\boldsymbol{r}_1;\mathbb{S}) \xrightarrow{\inc}  \mathbb V_{k-1}^{\div}(\boldsymbol{r}_2;\mathbb{S}) \xrightarrow{\div} \mathbb V_{k-2}^{L^2} (\boldsymbol{r}_3; \mathbb R^3)\rightarrow \boldsymbol{0}.
\end{equation*}
\end{corollary}
\begin{proof}
Let $\hat{\bs r}_0 = (r_0^{\texttt{v}}, \max\{r_0^e, 2\}, 1)^{\intercal}\geq(4,2,1)^{\intercal}$ and $\hat{\bs r}_1 = (r_1^{\texttt{v}}, \max\{r_1^e, 1\}, 0)^{\intercal}$, then $\hat{\bs r}_1\ominus 2=\bs r_2$. So we can apply the inverse hat operation. To verify $\defm \mathbb V_{k+2}^{\grad} (\boldsymbol{r}_0;\mathbb R^3) = \mathbb V_{k+1}^{\inc}(\boldsymbol{r}_1;\mathbb{S})\cap \ker(\inc)$, we use the exactness of the elasticity complex \eqref{eq:elasticitycomplex} and the fact if $\defm \bs u \in \mathbb V_{k+1}(\bs r_1;\mathbb S)$, then $\bs u\in \mathbb V_{k+2}(\bs r_0;\mathbb R^3)$. 
% For $\bs r_0 = (4, 2, 0)^{\intercal}$ or $(4, 1, 0)^{\intercal}$, $\bs r_2 = (1, -1, -1)^{\intercal}$ which is equal to $\hat{\bs r}_1\ominus 2$ with $\hat{\bs r}_1 = (3, 1, 0)^{\intercal}$ and $\hat{\bs r}_0 = (4, 2, 1)^{\intercal}$. So we can apply the inverse hat operation. To verify $\defm \mathbb V_{k+2}^{\grad} (\boldsymbol{r}_0;\mathbb R^3) = \mathbb V_{k+1}^{\inc}(\boldsymbol{r}_1;\mathbb{S})\cap \ker(\inc)$, we use the exactness of the elasticity complex \eqref{eq:elasticitycomplex} and the fact if $\defm \bs u \in \mathbb V_{k+1}(\bs r_1;\mathbb S)$, then $\bs u\in \mathbb V_{k+2}(\bs r_0;\mathbb R^3)$. 
\end{proof}
%\begin{proof}
%Apparently, $(3, 0, -1)^{\intercal}\leq \boldsymbol{r}_1\leq (3, 1, 0)^{\intercal}$ and $\boldsymbol{r}_2=(1, -1, -1)^{\intercal}$. 
%By
%\begin{equation*}
%\inc\mathbb V_{k+1}^{\inc}((3, 1, 0)^{\intercal};\mathbb{S})\subseteq
%\inc\mathbb V_{k+1}^{\inc}(\boldsymbol{r}_1;\mathbb{S})\subseteq \mathbb V_{k-1}^{\div}(\boldsymbol{r}_2;\mathbb{S})\cap\ker(\div),
%\end{equation*}
%and $\inc\mathbb V_{k+1}^{\inc}((3, 1, 0)^{\intercal};\mathbb{S})=\mathbb V_{k-1}^{\div}(\boldsymbol{r}_2;\mathbb{S})\cap\ker(\div)$ from the exactness of complex~\eqref{eq:femelasticitycomplex3d310}, 
%we have $\inc\mathbb V_{k+1}^{\inc}(\boldsymbol{r}_1;\mathbb{S})=\mathbb V_{k-1}^{\div}(\boldsymbol{r}_2;\mathbb{S})\cap\ker(\div)$.
%This combined with $\mathbb V_{k+1}^{\inc}(\boldsymbol{r}_1;\mathbb{S})\cap\ker(\inc)=\defm\mathbb V_{k+2}^{\grad} (\boldsymbol{r}_0;\mathbb R^3)$ ends the proof.
%\end{proof}

Next consider $\boldsymbol{r}_0= (3, 1, 0)^{\intercal}$. %\mnote{ where  is the condition $(3, 1, 0)$ instead of $(2, 1, 0)$ needed? \\
% The space $\mathbb V_{k+2}^{\curl} (
% \boldsymbol{r}_0, (\boldsymbol{r}_1)_+
% )$ in complex \eqref{eq:femcomplex3dTensor0310} will be the problem for $\boldsymbol{r}_0= (2, 1, 0)^{\intercal}$.
%When $\boldsymbol{r}_0= (3, 1, 0)^{\intercal}$, $\boldsymbol{r}_2=\boldsymbol{r}_1\ominus2= (0, -1, -1)^{\intercal}$.\\
% When $\boldsymbol{r}_0= (2, 1, 0)^{\intercal}$, $\boldsymbol{r}_1\ominus2=-1< \boldsymbol{r}_2=(0, -1, -1)^{\intercal}$.

%Again employing a $\widetilde{\quad}$ operation to effectuate the transition from $\mathbb M$ to $\mathbb S$, we then obtain the bottom sequence of diagram \eqref{eq:BGGelasticity},
%where 
%$$
%\widetilde{\mathbb V}_{k}^{\curl} (\boldsymbol{r}_1 \ominus 1; \mathbb M
%) :=\{\boldsymbol{\tau}\in\mathbb V_{k}^{\curl}(\boldsymbol r_1\ominus1;\mathbb M): \tr\boldsymbol{\tau}=0, \curl\boldsymbol{\tau}\in\mathbb V^{\div}_{k-1}(
%\boldsymbol{r}_2
%; \mathbb S)\}.
%$$

\begin{lemma}
Let $\boldsymbol{r}_0= (3, 1, 0)^{\intercal}$, $\boldsymbol{r}_1= (2, 0, -1)^{\intercal}$ and $\boldsymbol{r}_2= (0, -1, -1)^{\intercal}$.
The de Rham complex 
\begin{equation}\label{eq:femcomplex3dTensor1}
\mathbb R^3\xrightarrow{\subset}\widehat{\mathbb V}_{k+1}^{\div}((\boldsymbol{r}_1)_+) \xrightarrow{\grad}\widetilde{\mathbb V}_{k}^{\curl} (\boldsymbol{r}_1 \ominus 1; \mathbb M
) \xrightarrow{\curl} \mathbb V^{\div}_{k-1}(
\boldsymbol{r}_2
; \mathbb S)\cap\ker(\div) \to 0
\end{equation}
is exact, where $\widehat{\mathbb V}_{k+1}^{\div}((\boldsymbol{r}_1)_+)=\mathbb V_{k+1}^{\grad}((\boldsymbol{r}_1)_+;\mathbb R^3)\cap\ker(\div)$ and $\widetilde{\mathbb V}_{k}^{\curl} (\boldsymbol{r}_1 \ominus 1; \mathbb M
) =\{\boldsymbol{\tau}\in\mathbb V_{k}^{\curl}(\boldsymbol r_1\ominus1;\mathbb M): \tr\boldsymbol{\tau}=0, \curl\boldsymbol{\tau}\in\mathbb V^{\div}_{k-1}(
\boldsymbol{r}_2
; \mathbb S)\}$.
\end{lemma}
\begin{proof}
We introduce auxiliary smoothness vector $\hat{\bs r}_1 = (2, 1, 0)^{\intercal} \geq (\bs r_1)_+$ and $\hat{\bs r}_1 \ominus 2 = \bs r_2$. 
We have the finite element de Rham complex
%\begin{equation*}
%\mathbb R^3\xrightarrow{\subset}\mathbb V_{k+1}^{\grad} (\hat{\boldsymbol{r}}_1;\mathbb R^3) \xrightarrow{\grad}\mathbb V_{k}^{\curl} (\hat{\boldsymbol{r}}_1-1; \mathbb M) \xrightarrow{\curl} \mathbb V^{\div}_{k-1}(
%\boldsymbol{r}_2; \mathbb M)\cap\ker(\div) \xrightarrow{\div}  \boldsymbol{0}.
%\end{equation*}
and the triangular diagram
$$
\begin{tikzcd}
\mathbb V_{k+1}^{\grad} (\hat{\boldsymbol{r}}_1;\mathbb R^3)
\arrow[d,swap,"\div"] 
\arrow{r}{\grad}
&
\mathbb V_{k}^{\curl} (\hat{\boldsymbol{r}}_1-1; \mathbb M)
 \arrow[dl,swap," \tr "] 
 \arrow{r}{\curl}
 &
\mathbb V^{\div}_{k-1}(
\boldsymbol{r}_2; \mathbb M)\cap\ker(\div)
% \arrow[dl,swap,"S^{-1}"] 
 \arrow{r}{\div}
 & 0
 \\
\mathbb V_{k}^{L^2} (\hat{\boldsymbol{r}}_1-1)
% \arrow[ur,swap,"\mskw"] \arrow{r}{\grad}
 & 
\end{tikzcd}.
$$
%Since
%\begin{equation*}
%\tr\mathbb V_{k}^{\curl} (\widetilde{\boldsymbol{r}}_1-1; \mathbb M)=\div\mathbb V_{k+1}^{\grad} (\widetilde{\boldsymbol{r}}_1;\mathbb R^3) =\mathbb V_{k}^{L^2} (\widetilde{\boldsymbol{r}}_1-1),
%\end{equation*}
Notice that $(\hat{\bs r}_1, \hat{\bs r}_1 - 1)$ is div stable, and thus $\tr\mathbb V_{k}^{\curl} (\hat{\bs r}_1-1; \mathbb M)=\mathbb V_{k}^{L^2} (\hat{\bs r}_1-1)$. 
Applying one hat operation (with $W = \{0\}$) will induce the exact sequence
\begin{equation*}
\mathbb R^3\xrightarrow{\subset} \widehat{\mathbb V}_{k+1}^{\div}(\hat{\boldsymbol{r}}_1) \xrightarrow{\grad}\mathbb V_{k}^{\curl} (\hat{\boldsymbol{r}}_1-1; \mathbb M)\cap\ker(\tr) \xrightarrow{\curl} \mathbb V^{\div}_{k-1}(
\boldsymbol{r}_2; \mathbb M)\cap\ker(\div).
\end{equation*}
As $\hat{\bs r}_1\geq (\bs r_1)_+$, $\hat{\bs r}_1 - 1 \geq \bs r_1\ominus 1$, and $(\bs r_1)_+ - 1 = \bs r_1\ominus 1$, applying one inverse hat operation, we get the following exact sequence
%Due to the fact $\mathbb V_{k}^{\curl} (\widetilde{\boldsymbol{r}}_1-1; \mathbb M)\subseteq\mathbb V_{k}^{\curl} (\boldsymbol{r}_1 \ominus 1; \mathbb M)$, we have the following exact sequence
\begin{equation*}
\mathbb R^3\xrightarrow{\subset} \widehat{\mathbb V}_{k+1}^{\div}((\boldsymbol{r}_1)_+) \xrightarrow{\grad}\mathbb V_{k}^{\curl} (\boldsymbol{r}_1 \ominus 1; \mathbb M)\cap\ker(\tr) \xrightarrow{\curl} \mathbb V^{\div}_{k-1}(
\boldsymbol{r}_2; \mathbb M)\cap\ker(\div).
\end{equation*}
Apply one more tilde operation $\widetilde{\quad}$ to change $\mathbb M$ to $\mathbb S$ to get complex \eqref{eq:femcomplex3dTensor1}.
\end{proof}

%We further apply a $\widehat{\quad}$ operation to the exact sequence~\eqref{eq:elasticitybottomcomplex2}. 

%\begin{equation}\label{eq:elasticitycomplex0}
%\mathbb R^3\xrightarrow{\subset}\mathbb V_{k+2}^{\grad} (
%\boldsymbol{r}_0)\xrightarrow{\grad}\mathbb V_{k+1}^{\curl} (\boldsymbol{r}_1; \mathbb M
%) \xrightarrow{\curl} \mathbb V_{k}^{\div}(\boldsymbol r_1\ominus1;\mathbb M)\cap\ker(\div) \xrightarrow{\div}  \boldsymbol{0}
%\end{equation}

\begin{lemma}
Let $\boldsymbol{r}_0= (3, 1, 0)^{\intercal}$, $\boldsymbol{r}_1= (2, 0, -1)^{\intercal}$ and $\boldsymbol{r}_2= (0, -1, -1)^{\intercal}$.
The de Rham complex 
\begin{equation}\label{eq:femcomplex3dTensor0310}
\mathbb R^3\xrightarrow{\subset}\mathbb V_{k+2}^{\curl} (
\boldsymbol{r}_0, (\boldsymbol{r}_1)_+
) \xrightarrow{\grad}\widehat{\widetilde{\mathbb V}}^{\curl}_{k+1} (\boldsymbol{r}_1; \mathbb M
) \xrightarrow{\curl} \widetilde{\mathbb V}_{k}^{\div}(\boldsymbol r_1\ominus1;\mathbb M) \xrightarrow{\div}  \boldsymbol{0}
\end{equation}
is exact, 
where $\widetilde{\mathbb V}_{k}^{\div}(\boldsymbol r_1\ominus1;\mathbb M)=S(\widetilde{\mathbb V}_{k}^{\curl}(\boldsymbol r_1\ominus1;\mathbb M))$, and
\begin{align*}
\widehat{\widetilde{\mathbb V}}_{k+1}^{\curl}(\boldsymbol r_1;\mathbb M)&:=\{\boldsymbol{\tau}\in\mathbb V_{k+1}^{\curl}(\boldsymbol r_1;\mathbb M):\vskw\boldsymbol{\tau}\in\widehat{\mathbb V}_{k+1}^{\div}((\boldsymbol{r}_1)_+), \\
&\qquad\qquad\qquad\qquad\qquad\;\;\;\curl\boldsymbol{\tau}\in\widetilde{\mathbb V}_{k}^{\div}(\boldsymbol r_1\ominus1;\mathbb M)\}.
\end{align*}
\end{lemma}
\begin{proof}
Consider the diagram
%All operators in $\nearrow$ are injective and anticommutativity still holds. To be able to trace back, we need to put restriction on the top complex. 
$$
%\adjustbox{scale=0.875,center}{
\begin{tikzcd}
\mathbb V_{k+2}^{\curl} (\boldsymbol r_0)
\arrow[d,swap,"\curl"] 
\arrow{r}{\grad}
&
\widetilde{\mathbb V}_{k+1}^{\curl} (
\boldsymbol r_1; \mathbb M
)
 \arrow[dl,swap," 2\vskw "] 
% \arrow{r}{\curl}
% &
%\widetilde{\mathbb V}^{\div}_k(
%\boldsymbol{r}_1\ominus1; \mathbb M
%)
%% \arrow[dl,swap,"S^{-1}"] 
% \arrow{r}{\div}
% & 0
 \\
\mathbb V_{k+1}^{\div}(\boldsymbol{r}_1)\cap\ker(\div)
% \arrow[ur,swap,"\mskw"] \arrow{r}{\grad}
 & 
\end{tikzcd},
\begin{tikzcd}
\mathbb V_{k+2}^{\curl} (
\boldsymbol r_0, (\boldsymbol r_1)_+
)
\arrow[d,swap,"\curl"] 
\arrow{r}{\grad}
&
\widehat{\widetilde{\mathbb V}}_{k+1}^{\curl} (
\boldsymbol r_1; \mathbb M
)
 \arrow[dl,swap," 2\vskw "] 
% \arrow{r}{\curl}
% &
%\widetilde{\mathbb V}^{\div}_k(
%\boldsymbol{r}_1\ominus1; \mathbb M
%)
%% \arrow[dl,swap,"S^{-1}"] 
% \arrow{r}{\div}
% & 0
 \\
\mathbb V_{k+1}^{\div}((\boldsymbol{r}_1)_+)\cap\ker(\div)
% \arrow[ur,swap,"\mskw"] \arrow{r}{\grad}
 & 
\end{tikzcd}.
%}
$$  
%and
%%All operators in $\nearrow$ are injective and anticommutativity still holds. To be able to trace back, we need to put restriction on the top complex. 
%$$
%%\adjustbox{scale=0.875,center}{
%\begin{tikzcd}
%\mathbb V_{k+2}^{\curl} (
%\boldsymbol r_0, (\boldsymbol r_1)_+
%)
%\arrow[d,swap,"\curl"] 
%\arrow{r}{\grad}
%&
%\widehat{\widetilde{\mathbb V}}_{k+1}^{\curl} (
%\boldsymbol r_1; \mathbb M
%)
% \arrow[dl,swap," 2\vskw "] 
% \arrow{r}{\curl}
% &
%\widetilde{\mathbb V}^{\div}_k(
%\boldsymbol{r}_1\ominus1; \mathbb M
%)
%% \arrow[dl,swap,"S^{-1}"] 
% \arrow{r}{\div}
% & 0
% \\
% %
%\widehat{\mathbb V}_{k+1}^{\div}((\boldsymbol{r}_1)_+)
%% \arrow[ur,swap,"\mskw"] \arrow{r}{\grad}
% & 
% &
%\end{tikzcd}.
%%}
%$$
As $\bs r_0+1 = (4, 2, 1)^{\intercal}$ and $(\bs r_1, \bs r_1\ominus 1)$ is div stable, we have the finite element de Rham complex \eqref{eq:femderhamcomplex3dgeneral} with valid de Rham smoothness vectors $(\bs r_0+1, \bs r_0, \bs r_1, \bs r_1\ominus 1)$, and consequently
$\curl\mathbb V_{k+2}^{\curl} (\boldsymbol r_0)=\mathbb V_{k+1}^{\div}(\boldsymbol{r}_1)\cap\ker(\div)$.

Then applying one tilde operation $\widetilde{\quad}$, it follows from $\mathbb V_{k+1}^{\div}((\boldsymbol{r}_1)_+)\cap\ker(\div)\subseteq \mathbb V_{k+1}^{\div}(\boldsymbol{r}_1)\cap\ker(\div)$ that $\curl\mathbb V_{k+2}^{\curl}(\boldsymbol r_0, (\boldsymbol r_1)_+)=\mathbb V_{k+1}^{\div}((\boldsymbol{r}_1)_+)\cap\ker(\div)$.
As the curl operators in the diagram are surjective %\mnote{ $(\bs r_1)_+$ is not div stable, the curl operator on the right diagram may not be surjective. Do we need $\bs r_0 - \curl - (\bs r_1)_+ - \div - (\bs r_1)_+ -1 $ is a valid de Rham sequence?}
%$\curl \mathbb V_{k+2}^{\curl} (\boldsymbol{r}_0) =\mathbb V_{k+1}^{\div}(\boldsymbol{r}_1)\cap\ker(\div)$ and $\curl \mathbb V_{k+2}^{\curl} (\boldsymbol{r}_0, (\boldsymbol r_1)_+) =\mathbb V_{k+1}^{\div}((\boldsymbol{r}_1)_+)\cap\ker(\div)$, 
we can apply one hat operation $\widehat{\quad}$, cf. 
%\begin{equation*}
%\vskw\widetilde{\mathbb V}_{k+1}^{\curl} (\boldsymbol{r}_1; \mathbb M
%)=\curl\mathbb V_{k+2}^{\curl} (
%\boldsymbol{r}_0) =\mathbb V_{k+1}^{\div}(\boldsymbol{r}_1)\cap\ker(\div),
%\end{equation*}
Lemma~\ref{lm:hat} to acquire complex \eqref{eq:femcomplex3dTensor0310}.

We use $\bs r_0 = (3, 1, 0)^{\intercal}$ to ensure $(\bs r_0+1, \bs r_0, \bs r_1, \bs r_1\ominus 1)$ is a de Rham parameter sequence, which is not true for $\bs r_0 = (2, 1, 0)^{\intercal}$ as $\bs r_0 + 1 = (3, 2, 1)^{\intercal}$ is not a valid smoothness vector.
\end{proof}

We further give characterization of the space $\widehat{\widetilde{\mathbb V}}_{k+1}^{\curl} (\boldsymbol{r}_1; \mathbb M)$.
\begin{lemma}
Let $(\boldsymbol r_0, \boldsymbol r_1, \boldsymbol r_2, \boldsymbol r_3)$ be given by~\eqref{eq:relasticity} and $\boldsymbol{r}_0= (3, 1, 0)^{\intercal}$.
We have
\begin{equation}\label{eq:VcurlMtildeDecomp}
\widehat{\widetilde{\mathbb V}}_{k+1}^{\curl} (\boldsymbol{r}_1; \mathbb M)=\mathbb V_{k+1}^{\inc^+}(\boldsymbol{r}_1;\mathbb{S}) \oplus\mskw\widehat{\mathbb V}_{k+1}^{\div}((\boldsymbol{r}_1)_+),
\end{equation}
% \begin{equation}\label{eq:VcurlVdef}
% \defm\mathbb V_{k+2}^{\curl} (\boldsymbol{r}_0, (\boldsymbol{r}_1)_+)=\defm\mathbb V_{k+2}^{\grad} (\boldsymbol{r}_0),
% \end{equation}
where $\mathbb V_{k+1}^{\inc^+}(\boldsymbol{r}_1;\mathbb{S}):=\{\boldsymbol{\tau}\in \mathbb V_{k+1}^{\curl} (\boldsymbol{r}_1; \mathbb M): \boldsymbol \tau \in \mathbb S, \curl\boldsymbol{\tau}\in\widetilde{\mathbb V}_{k}^{\div}(\boldsymbol r_1\ominus1;\mathbb M)\}$.
\end{lemma}
\begin{proof}
For $\boldsymbol{\tau}=\mskw\boldsymbol{v}\in\mskw\widehat{\mathbb V}_{k+1}^{\div}((\boldsymbol{r}_1)_+)$, it follows
$$
\curl\boldsymbol{\tau}=\curl\mskw\boldsymbol{v}=-S(\grad\boldsymbol{v})\in\widetilde{\mathbb V}_{k}^{\div}(\boldsymbol r_1\ominus1;\mathbb M),
$$
and 
$\vskw \boldsymbol \tau = \boldsymbol v\in \widehat{\mathbb V}_{k+1}^{\div}((\boldsymbol{r}_1)_+).$ 
We thus have proved 
$$
\mskw\widehat{\mathbb V}_{k+1}^{\div}((\boldsymbol{r}_1)_+)\subseteq\widehat{\widetilde{\mathbb V}}_{k+1}^{\curl} (\boldsymbol{r}_1; \mathbb M).
$$
Then decomposition~\eqref{eq:VcurlMtildeDecomp} holds from $$\boldsymbol{\tau}=\sym\boldsymbol{\tau}+\skw\boldsymbol{\tau} = \sym \boldsymbol \tau + \mskw (\vskw \boldsymbol \tau).$$
\end{proof}

\begin{theorem}
Let $(\boldsymbol r_0, \boldsymbol r_1, \boldsymbol r_2, \boldsymbol r_3)$ be given by~\eqref{eq:relasticity} and $\boldsymbol{r}_0= (3, 1, 0)^{\intercal}$. Let $k+2\geq 2r_0^{\texttt{v}}+1$. We have the BGG diagram
\begin{equation*}
\adjustbox{scale=0.96,center}{
\begin{tikzcd}
\mathbb V_{k+2}^{\curl} (\boldsymbol{r}_0, (\boldsymbol{r}_1)_+)
\arrow{r}{\grad}
&
\widehat{\widetilde{\mathbb V}}_{k+1}^{\curl} (\boldsymbol{r}_1; \mathbb M
)
 \arrow{r}{\curl}
 &
\widetilde{\mathbb V}^{\div}_k(
\boldsymbol{r}_1\ominus1; \mathbb M
)
 \arrow{r}{\div}
 & 0
 \\
\widehat{\mathbb V}_{k+1}^{\div}((\boldsymbol{r}_1)_+)
 \arrow[ur,swap,"\mskw"'] \arrow{r}{\grad}
 & 
\widetilde{\mathbb V}_{k}^{\curl}(
\boldsymbol{r}_1\ominus1
;\mathbb M) 
 \arrow[ur,swap,"S"'] \arrow{r}{\curl}
 & 
\mathbb V^{\div}_{k-1}(
\boldsymbol{r}_2
; \mathbb S) 
 \arrow[ur,swap,"-2\vskw"'] \arrow{r}{\div} \arrow[r] 
 & \mathbb V^{L^2}_{k-2}(
\boldsymbol{r}_3
; \mathbb R^3), 
\end{tikzcd}
}
\end{equation*}
which leads to the finite element elasticity complex
\begin{align*}%\label{eq:femelasticitycomplex3d+}
{\rm RM}\xrightarrow{\subset}\mathbb V_{k+2}^{\curl} (\boldsymbol{r}_0, (\boldsymbol{r}_1)_+) &\xrightarrow{\defm} \mathbb V_{k+1}^{\inc^+}(\boldsymbol{r}_1;\mathbb{S}) \\
&\xrightarrow{\inc}  \mathbb V_{k-1}^{\div}(\boldsymbol{r}_2;\mathbb{S}) \xrightarrow{\div} \mathbb V_{k-2}^{L^2} (\boldsymbol{r}_3; \mathbb R^3)\rightarrow \boldsymbol{0}. \notag
\end{align*}
Consequently, by applying one inverse hat operation, we have
 \begin{equation}
\label{eq:femelasticitycomplex3d310}
{\rm RM}\xrightarrow{\subset}\mathbb V_{k+2}^{\grad} (\boldsymbol{r}_0;\mathbb R^3) \xrightarrow{\defm} \mathbb V_{k+1}^{\inc}(\boldsymbol{r}_1;\mathbb{S}) \xrightarrow{\inc}  \mathbb V_{k-1}^{\div}(\boldsymbol{r}_2;\mathbb{S}) \xrightarrow{\div} \mathbb V_{k-2}^{L^2} (\boldsymbol{r}_3; \mathbb R^3)\rightarrow \boldsymbol{0}.
\end{equation}
%where $\mathbb V_{k+1}^{\inc^+}(\boldsymbol{r}_1;\mathbb{S}):=\{\boldsymbol{\tau}\in \widehat{\widetilde{\mathbb V}}_{k+1}^{\curl} (\boldsymbol{r}_1; \mathbb M):\vskw\boldsymbol{\tau}=\boldsymbol{0}\}$.
\end{theorem}
\begin{proof}
The bijectiveness of the mapping $S:\widetilde{\mathbb V}_{k}^{\curl}(\boldsymbol r_1\ominus1;\mathbb M)\to\widetilde{\mathbb V}_{k}^{\div}(\boldsymbol r_1\ominus1;\mathbb M)$ is a direct outcome of its definition. With complex~\eqref{eq:femcomplex3dTensor0310}, and the decomposition~\eqref{eq:VcurlMtildeDecomp} at our disposal, we arrive at our desired conclusion by applying the BGG framework.
\end{proof}

After we obtain the complex \eqref{eq:femelasticitycomplex3d310} for the case $\boldsymbol{r}_0= (3, 1, 0)^{\intercal}$, we can extend to the case $\boldsymbol{r}_0= (2, 1, 0)^{\intercal}$ by applying one inverse hat operation, as they share the same $\bs r_2 = (0, -1, -1)^{\intercal}$.

Combining all the cases, we conclude the construction of finite element elasticity complexes.

\begin{theorem}\label{cor:femelasticityinc421}
Let $$
\boldsymbol{r}_0 \geq (2, 1, 0)^{\intercal}, 
\boldsymbol r_1 =\boldsymbol{r}_0-1, 
%\geq\begin{pmatrix}
%1\\
%0 \\
%-1
%\end{pmatrix},
\;
 \boldsymbol{r}_2=\max\{\boldsymbol{r}_1\ominus2, (0, -1, -1)^{\intercal}\},
\;
 \boldsymbol{r}_3=\boldsymbol{r}_2\ominus1.$$ 
Assume $(\boldsymbol r_1 \ominus 1, \boldsymbol r_2, k)$ is div stable, and $(\boldsymbol r_2, \boldsymbol r_3, k-1)$ is $(\div;\mathbb S)$ stable. Let $k+2\geq 2r_0^{\texttt{v}}+1$.  We have the finite element elasticity complex
 \begin{equation}
\label{eq:femelasticitycomplex3d}
{\rm RM}\xrightarrow{\subset}\mathbb V_{k+2}^{\grad} (\boldsymbol{r}_0;\mathbb R^3) \xrightarrow{\defm} \mathbb V_{k+1}^{\inc}(\boldsymbol{r}_1;\mathbb{S}) \xrightarrow{\inc}  \mathbb V_{k-1}^{\div}(\boldsymbol{r}_2;\mathbb{S}) \xrightarrow{\div} \mathbb V_{k-2}^{L^2} (\boldsymbol{r}_3; \mathbb R^3)\rightarrow \boldsymbol{0}.
\end{equation}
\end{theorem}

By the exactness of complex~\eqref{eq:femelasticitycomplex3d}, 
% \eqref{eq:femelasticitycomplex3d410} and \eqref{eq:femelasticitycomplex3d210}, 
we have the dimension identity
\begin{align}\label{eq:incSfemdim}  
  - 6 + \dim \mathbb V_{k+2}^{\grad} (\boldsymbol{r}_0; \mathbb R^3) - & \dim\mathbb V_{k+1}^{\inc}(\boldsymbol{r}_1;\mathbb{S}) \\
+& \dim \mathbb V_{k-1}^{\div}(\boldsymbol{r}_2;\mathbb{S}) -  \dim \mathbb V_{k-2}^{L^2} (\boldsymbol{r}_3; \mathbb R^3) = 0. \notag
\end{align}
So far we have finite element descriptions for spaces $\mathbb V_{k+2}^{\grad} (\boldsymbol{r}_0; \mathbb R^3), \mathbb V_{k-1}^{\div}(\boldsymbol{r}_2;\mathbb{S}), $ and $\mathbb V_{k-2}^{L^2} (\boldsymbol{r}_3; \mathbb R^3)$ but not for $\mathbb V_{k+1}^{\inc}(\boldsymbol{r}_1;\mathbb{S})$ which will be given in Section \ref{sec:edgeelements}. 

\begin{example}\rm 
Taking $\boldsymbol{r}_0=(3,1,0)^{\intercal}$, $\boldsymbol{r}_1=\boldsymbol{r}_0-1$, $\boldsymbol{r}_2=(0,-1,-1)^{\intercal}$, $\boldsymbol{r}_3=\boldsymbol{r}_2\ominus1$, we have 
\begin{align*}
{\rm RM}\xrightarrow{\subset} 
\mathbb V_{k+2}^{\curl}
(\begin{pmatrix}
3\\
1 \\
0
\end{pmatrix}, 
\begin{pmatrix}
2\\
0 \\
0
\end{pmatrix}
)
& \xrightarrow{\defm} 
\mathbb V_{k+1}^{\inc^+}(
\begin{pmatrix}
2\\
0 \\
-1
\end{pmatrix};\mathbb S)\\
& \xrightarrow{\inc}
\mathbb V^{\div}_{k-1}(
\begin{pmatrix}
0\\
-1 \\
-1
\end{pmatrix}; \mathbb S) 
\xrightarrow{\div} 
\mathbb V_{k-2}^{L^2}(
\begin{pmatrix}
-1\\
-1 \\
-1
\end{pmatrix}; \mathbb R^3)\rightarrow \boldsymbol{0},
\end{align*}
which is the finite element discretization of the elasticity complex~\eqref{eq:elasticitycomplex1inc}.
% The continuous version is the elasticity complex
% \begin{equation*}%\label{eq:gradcurlelasticitycomplex}
% {\rm RM}\xrightarrow{\subset}  H^1(\curl, \Omega; \mathbb R^3) \xrightarrow{\defm}  H(\inc^+, \Omega; \mathbb S) \xrightarrow{\inc}  H(\operatorname{div}, \Omega; \mathbb{S}) \xrightarrow{\div}  L^{2} (\Omega; \mathbb R^3)\rightarrow \boldsymbol{0},
% \end{equation*}
% where the spaces
% \begin{align*}
% H^1(\curl, \Omega; \mathbb R^3) &:= \{ \boldsymbol u\in H^1(\Omega; \mathbb R^3), \curl \boldsymbol u \in  H^1(\Omega; \mathbb R^3)\}, \\
% H(\inc^+, \Omega; \mathbb S) &:= \{ \boldsymbol \tau \in H (\curl, \Omega ; \mathbb S) : \inc \boldsymbol \tau \in {L}^2(\Omega; \mathbb S)\}. 
% \end{align*}
\qed
\end{example}

\begin{example}\rm 
Consider the choice $\boldsymbol{r}_0=(2,1,0)^{\intercal}$, $\boldsymbol{r}_1=\boldsymbol{r}_0-1$, $\boldsymbol{r}_2=(0,-1,-1)^{\intercal}$, $\boldsymbol{r}_3=\boldsymbol{r}_2\ominus1$ and $k+1 \geq 6$. From the finite element elasticity complex~\eqref{eq:femelasticitycomplex3d} we get a finite element discretization of the elasticity complex \eqref{eq:elasticitycomplex}
\begin{equation*}
{\rm RM}\xrightarrow{\subset} \begin{pmatrix}
2\\
1 \\
0
\end{pmatrix} \xrightarrow{\defm} \begin{pmatrix}
1\\
0 \\
-1
\end{pmatrix} \xrightarrow{\inc}\begin{pmatrix}
0\\
-1 \\
-1
\end{pmatrix} \xrightarrow{\div} \begin{pmatrix}
-1\\
-1 \\
-1
\end{pmatrix}\rightarrow \boldsymbol{0}.
\end{equation*} 
This sequence presents a variation of the finite element elasticity complex in the work by Chen and Huang~\cite{Chen;Huang:2021Finite}, where the Hu-Zhang $H(\div;\mathbb S)$-conforming element is used. 
\qed
\end{example}

The bubble elasticity complex is shown in Proposition~\ref{prp:bubbleelascomplex}, whose proof is given in Section~\ref{sec:bubbleelascomplex}.
\begin{proposition}\label{prp:bubbleelascomplex}
Let 
\begin{equation*}%\label{eq:relasticity}
\boldsymbol{r}_0 \geq (2, 1, 0)^{\intercal},
\; 
\boldsymbol r_1 =\boldsymbol{r}_0-1, 
%\geq\begin{pmatrix}
%1\\
%0 \\
%-1
%\end{pmatrix},
\;
 \boldsymbol{r}_2=\max\{\boldsymbol{r}_1\ominus2, (0, -1, -1)^{\intercal}\},
\;
 \boldsymbol{r}_3=\boldsymbol{r}_2\ominus1.
\end{equation*}
Assume $\boldsymbol r_2$ satisfies the condition in Lemma \ref{lem:divbubbleontoST}, and $k\geq\max\{2r_2^{\texttt{v}}+2,3\}$. %Let $k\geq 2r_1^{\texttt{v}}-1$. 
Then the bubble elasticity complex 
\begin{equation*}%\label{eq:fembubbleelasticitycomplex}
\resizebox{.925\hsize}{!}{$
0\xrightarrow{\subset} \mathbb B_{k+2}(\boldsymbol{r}_0;\mathbb R^3)\xrightarrow{\defm}\mathbb B^{\inc}_{k+1}(\boldsymbol{r}_1;\mathbb S)\xrightarrow{\inc} \mathbb B^{\div}_{k-1}(\boldsymbol{r}_2;\mathbb S) \xrightarrow{\div} \mathbb B_{k-2}(\boldsymbol{r}_3;\mathbb R^3)/{\rm RM}\xrightarrow{}0
$}
\end{equation*}
is exact.
\end{proposition}

\subsection{Finite element divdiv complexes}\label{sec:femdivdiv}
Let 
\begin{equation}\label{eq:divdivrsequence}
\boldsymbol{r}_0\geq (1, 0, 0)^{\intercal}, \quad 
\boldsymbol{r}_1=\boldsymbol{r}_0-1,\quad 
\boldsymbol{r}_2= \max\{\boldsymbol{r}_1\ominus1, (0, -1, -1)^{\intercal}\},\quad 
\boldsymbol{r}_3=\boldsymbol{r}_2\ominus2.
\end{equation}
Assume both $(\boldsymbol r_2, \boldsymbol r_2\ominus 1, k)$ and $(\boldsymbol r_2\ominus 1, \boldsymbol r_3, k-1)$ are div stable. Consequently both $(\boldsymbol r_0, \boldsymbol r_1, \boldsymbol r_2, \boldsymbol r_2\ominus 1)$ and $((\boldsymbol r_1)_+, \boldsymbol r_1\ominus 1, \boldsymbol r_2\ominus 1, \boldsymbol r_3)$ 
are valid de Rham smoothness sequences. 

\begin{lemma}
Let $(\boldsymbol r_0, \boldsymbol r_1, \boldsymbol r_2, \boldsymbol r_3)$ be given by~\eqref{eq:divdivrsequence} and $\boldsymbol{r}_0\geq (2, 1, 0)^{\intercal}$. 
Assume  $(\boldsymbol r_2, \boldsymbol r_2\ominus 1, k)$ is  div stable. 
The complex 
\begin{equation}\label{eq:femcomplex3dTensor2}
\mathbb R^3\xrightarrow{\subset}\mathbb V_{k+2}^{\div} (\boldsymbol{r}_0, (\boldsymbol r_1)_+) \xrightarrow{\grad}\widehat{\widetilde{\mathbb V}}_{k+1}^{\curl} (\boldsymbol r_1;\mathbb M) \xrightarrow{\curl} \widetilde{\mathbb V}_{k}^{\div}(\boldsymbol r_2;\mathbb M) \xrightarrow{\div}  \mathbb V^{\div}_{k-1}(\boldsymbol{r}_2\ominus1)\to\boldsymbol{0}
\end{equation}
is exact, where
\begin{align*}
\widetilde{\mathbb V}_{k}^{\div}(\boldsymbol r_2;\mathbb M)&:=\mathbb V^{\div\div^+}_{k}(\boldsymbol{r}_2;\mathbb S)\oplus\mskw\mathbb V^{\curl}_{k}(\boldsymbol{r}_1\ominus1), \\
\widetilde{\mathbb{V}}_{k+1}^{\curl}(\boldsymbol r_1;\mathbb M)&:=\{\boldsymbol{\tau}\in\mathbb V_{k+1}^{\curl}(\boldsymbol r_1;\mathbb M): \curl\boldsymbol{\tau}\in\widetilde{\mathbb V}_{k}^{\div}(\boldsymbol r_2;\mathbb M)\},\\
\widehat{\widetilde{\mathbb{V}}}_{k+1}^{\curl}(\boldsymbol r_1;\mathbb M)&:=\{\boldsymbol{\tau}\in\widetilde{\mathbb V}_{k+1}^{\curl}(\boldsymbol r_1;\mathbb M): \tr\boldsymbol{\tau}\in\mathbb V_{k+1}^{\grad}((\boldsymbol r_1)_+)\}.
\end{align*}
\end{lemma}
\begin{proof}
By definition, we have the de Rham complex
\begin{equation}\label{eq:femcomplex3dTensor3}
\mathbb R^3\xrightarrow{\subset}\mathbb V_{k+2}^{\grad} (\boldsymbol{r}_0; \mathbb R^3) \xrightarrow{\grad}\widetilde{\mathbb V}_{k+1}^{\curl} (\boldsymbol r_1;\mathbb M) \xrightarrow{\curl} \widetilde{\mathbb V}_{k}^{\div}(\boldsymbol r_2;\mathbb M) \xrightarrow{\div}  \mathbb V^{\div}_{k-1}(\boldsymbol{r}_2\ominus1)\to\boldsymbol{0}.
\end{equation}

When $r_1^f\geq 0$, $(\boldsymbol r_1)_+ = \boldsymbol r_1 = \bs r_0 - 1$, and complex \eqref{eq:femcomplex3dTensor2} is exactly complex \eqref{eq:femcomplex3dTensor3}.

Then consider case $r_1^f=-1$ and thus $r_0^f = 0$. We have the diagram
$$
\begin{tikzcd}
\mathbb V_{k+2}^{\div} (\boldsymbol r_0)
\arrow[d,swap,"\div"] 
\arrow{r}{\grad}
&
\widetilde{\mathbb V}_{k+1}^{\curl} (\boldsymbol r_1; \mathbb M)
 \arrow[dl,swap," \tr "] 
% \arrow{r}{\curl}
% &
%\widetilde{\mathbb V}_k^{\div}(\boldsymbol r_2; \mathbb M)\cap \ker(\div)
%% \arrow[dl,swap,"S^{-1}"] 
% \arrow{r}{\div}
% & 0
 \\
\mathbb V_{k+1}^{L^2}(\boldsymbol r_1)
% \arrow[ur,swap,"\mskw"] \arrow{r}{\grad}
 & 
\end{tikzcd},
\begin{tikzcd}
\mathbb V_{k+2}^{\div} (\boldsymbol r_0, (\boldsymbol r_1)_+)
\arrow[d,swap,"\div"] 
\arrow{r}{\grad}
&
\widehat{\widetilde{\mathbb V}}_{k+1}^{\curl} (\boldsymbol r_1; \mathbb M)
 \arrow[dl,swap," \tr "] 
% \arrow{r}{\curl}
% &
%\widetilde{\mathbb V}_k^{\div}(\boldsymbol r_2; \mathbb M)\cap \ker(\div)
%% \arrow[dl,swap,"S^{-1}"] 
% \arrow{r}{\div}
% & 0
 \\
\mathbb V_{k+1}^{\grad}((\boldsymbol r_1)_+)
% \arrow[ur,swap,"\mskw"] \arrow{r}{\grad}
 & 
\end{tikzcd}.
%}
$$

Both div operators are surjective. Then apply Lemma \ref{lm:hat} to conclude the sequence~\eqref{eq:femcomplex3dTensor2} is exact. %\mnote{ That's okay to have $\div \mathbb V_{k+2}^{\div} (\boldsymbol r_0, \boldsymbol r_1) \neq \mathbb V_{k+1}^{L^2}(\boldsymbol r_1)$. Only count the difference. The trace of $\widetilde{\mathbb{V}}_{k+1}^{\curl}(\boldsymbol r_1;\mathbb M)$ is not the whole space $\mathbb V_{k+1}^{L^2}(\boldsymbol r_1)$ neither. Only count the additional constraint added to be in $\mathbb V_{k+1}^{\grad}((\boldsymbol r_1)_+)$.}
We do need the div stability $\div \mathbb V_{k+2}^{\div} (\boldsymbol r_0)
 = \mathbb V_{k+1}^{L^2}(\boldsymbol r_1)$. Otherwise the range of div operator is only a sub-space, i.e., $\div \mathbb V_{k+2}^{\div} (\boldsymbol r_0)
\subset \mathbb V_{k+1}^{L^2}(\boldsymbol r_1)$ and how the dimension change in the hat operation is unclear.
\end{proof}

We construct a finite element divdiv complex by the BGG procedure.
\begin{theorem}
Let $(\boldsymbol r_0, \boldsymbol r_1, \boldsymbol r_2, \boldsymbol r_3)$ be given by~\eqref{eq:divdivrsequence} and $\boldsymbol{r}_0\geq (2, 1, 0)^{\intercal}$. 
Assume  $(\boldsymbol r_2, \boldsymbol r_2\ominus 1, k)$ and $(\boldsymbol r_2\ominus 1, \boldsymbol r_3, k-1)$ are div stable. Then we have the BGG diagram
\begin{equation*}%\label{eq:BGGdivdiv}
\adjustbox{scale=0.96,center}{
\begin{tikzcd}
\mathbb V_{k+2}^{\div} (\boldsymbol{r}_0, (\boldsymbol r_1)_+)
\arrow{r}{\grad}
& \widehat{\widetilde{\mathbb V}}_{k+1}^{\curl} (\boldsymbol{r}_1; \mathbb M)
 \arrow{r}{\curl}
 &
\widetilde{\mathbb V}^{\div}_k(\boldsymbol{r}_2; \mathbb M)
 \arrow{r}{\div}
 & \mathbb V^{\div}_{k-1}(\boldsymbol{r}_2\ominus1)\to \boldsymbol{0}
 \\
\mathbb V_{k+1}^{\grad}((\boldsymbol r_1)_+)
 \arrow[ur,swap,"\iota"'] \arrow{r}{\grad}
 & 
\mathbb V_{k}^{\curl}(\boldsymbol{r}_1\ominus1) 
 \arrow[ur,swap,"\mskw"'] \arrow{r}{\curl}
 & 
\mathbb V^{\div}_{k-1}(\boldsymbol{r}_2\ominus1) 
 \arrow[ur,swap,"\mathrm{id}"'] \arrow{r}{\div} \arrow[r] 
 & \mathbb V^{L^2}_{k-2}(\boldsymbol{r}_3) 
\to 0,
% & 0
\end{tikzcd}
}
\end{equation*}
which leads to the exact finite element divdiv complex
\begin{align}\label{eq:femdivdivcomplex3d+}
{\rm RT}\xrightarrow{\subset}\mathbb V_{k+2}^{\div} (\boldsymbol{r}_0, (\boldsymbol r_1)_+) &\xrightarrow{\dev\grad} \mathbb V_{k+1}^{\sym\curl_+^+}(\boldsymbol{r}_1;\mathbb{T}) \\
&\xrightarrow{\sym\curl}  \mathbb V_{k}^{\div\div^+}(\boldsymbol{r}_2;\mathbb{S}) \xrightarrow{\div\div} \mathbb V_{k-2}^{L^2} (\boldsymbol{r}_3)\rightarrow 0, \notag
\end{align}
where 
\begin{align*}
\mathbb V_{k+1}^{\sym\curl_+^+}(\boldsymbol{r}_1;\mathbb{T})&:=\{\boldsymbol{\tau}\in\mathbb V_{k+1}^{\curl}(\boldsymbol r_1;\mathbb M): \tr\boldsymbol{\tau}=0, \curl\boldsymbol{\tau}\in\widetilde{\mathbb V}_{k}^{\div}(\boldsymbol r_2;\mathbb M)\}.
\end{align*}
% \begin{align*}
% \mathbb V_{k+1}^{\sym\curl_+^+}(\boldsymbol{r}_1;\mathbb{T})&:=\{\boldsymbol{\tau}\in\mathbb V_{k+1}^{\curl}(\boldsymbol r_1;\mathbb M): \tr\boldsymbol{\tau}=0, \sym\curl\boldsymbol{\tau}\in{\mathbb V}_{k}^{\div\div^+}(\boldsymbol r_2;\mathbb S)\}.
% \end{align*}
\end{theorem}
\begin{proof}
By definition of spaces, both $\iota$ and $\mskw$ are injective. In addition, we have the decomposition
\begin{align*}
\widehat{\widetilde{\mathbb V}}_{k+1}^{\curl} (\boldsymbol{r}_1; \mathbb M) &=\mathbb V_{k+1}^{\sym\curl_+^+}(\boldsymbol{r}_1;\mathbb{T})\oplus\iota\mathbb V_{k+1}^{\grad}((\boldsymbol r_1)_+),\\
\widetilde{\mathbb V}_{k}^{\div}(\boldsymbol r_2;\mathbb M)&=\mathbb V^{\div\div^+}_{k}(\boldsymbol{r}_2;\mathbb S)\oplus\mskw\mathbb V^{\curl}_{k}(\boldsymbol{r}_1\ominus1).
\end{align*}
We conclude the result by employing the BGG framework. 
\end{proof}

Again, the complexes derived from the BGG framework typically include spaces that have a slightly higher degree of smoothness. We can relax the smoothness
% of $\boldsymbol \tau \in\mathbb V_{k+1}^{\curl}(\boldsymbol r_1;\mathbb M)$ 
and define
\begin{align*}
% \mathbb V_{k+1}^{\sym\curl_+}(\boldsymbol{r}_1;\mathbb{T})&:=\{\boldsymbol{\tau}\in\mathbb V_{k+1}(\boldsymbol r_1)\otimes \mathbb T: \sym\curl\boldsymbol{\tau}\in{\mathbb V}_{k}^{\div\div^+}(\boldsymbol r_2;\mathbb S)\}, \\
\mathbb V_{k+1}^{\sym\curl_+}(\boldsymbol{r}_1;\mathbb{T})&:=\{\boldsymbol{\tau}\in\mathbb V_{k+1}(\boldsymbol r_1)\otimes \mathbb T: \sym\curl\boldsymbol{\tau}\in{\mathbb V}_{k}^{\div\div^+}(\boldsymbol r_2;\mathbb S)\}.
\end{align*}
\begin{corollary}
Let $(\boldsymbol r_0, \boldsymbol r_1, \boldsymbol r_2, \boldsymbol r_3)$ be given by~\eqref{eq:divdivrsequence}. 
Assume both $(\boldsymbol r_2, \boldsymbol r_2\ominus 1, k)$ and $(\boldsymbol r_2\ominus 1, \boldsymbol r_3, k-1)$ are div stable. We have the exact finite element divdiv complex
\begin{align}
\label{eq:femdivdivcomplex3dvariant1}
{\rm RT}\xrightarrow{\subset}\mathbb V_{k+2}^{\grad} (\boldsymbol{r}_0;\mathbb R^3) &\xrightarrow{\dev\grad} \mathbb V_{k+1}^{\sym\curl_+}(\boldsymbol{r}_1;\mathbb{T}) \\
&\xrightarrow{\sym\curl}  \mathbb V_{k}^{\div\div^+}(\boldsymbol{r}_2;\mathbb{S}) \xrightarrow{\div\div} \mathbb V_{k-2}^{L^2} (\boldsymbol{r}_3)\rightarrow 0. \notag
\end{align}
Consequently
\begin{align}
\label{eq:symcurlTfemdim}
- 4 + \dim\mathbb V_{k+2}^{\grad} (\boldsymbol{r}_0;\mathbb R^3)  & - \dim\mathbb V_{k+1}^{\sym\curl_+}(\boldsymbol{r}_1;\mathbb{T})\\
& +\dim \mathbb V_{k}^{\div\div^+}(\boldsymbol{r}_2;\mathbb{S}) -  \dim \mathbb V_{k-2}^{L^2} (\boldsymbol{r}_3) = 0.\notag
\end{align}
\end{corollary}
\begin{proof}
We apply an inverse hat operation to obtain complex \eqref{eq:femdivdivcomplex3dvariant1} for $\bs r_0\geq (2, 1, 0)^{\intercal}$ and then apply one more inverse hat operation to relax to $\bs r_0\geq (1, 0, 0)^{\intercal}$ since they share the same $\bs r_2$.
% for $\boldsymbol{r}_0\geq (2, 1, 0)^{\intercal}$.
%
% Consider case $(1, 0, 0)^{\intercal}\leq \boldsymbol{r}_0\leq (2, 1, 0)^{\intercal}$. Then $(0, -1, -1)^{\intercal}\leq \boldsymbol{r}_1\leq (1, 0, -1)^{\intercal}$ and $\boldsymbol{r}_2=(0, -1, -1)^{\intercal}$. 
% By
% \begin{equation*}
% \mathbb V_{k+1}^{\sym\curl_+}((1, 0, -1)^{\intercal};\mathbb{T})\subseteq \mathbb V_{k+1}^{\sym\curl_+}(\boldsymbol{r}_1;\mathbb{T}),
% \end{equation*}
% \begin{equation*}
% \sym\curl\mathbb V_{k+1}^{\sym\curl_+}(\boldsymbol{r}_1;\mathbb{T})\subseteq \mathbb V_{k}^{\div\div^+}(\boldsymbol{r}_2;\mathbb{S})\cap\ker(\div\div),
% \end{equation*}
% we have $\sym\curl\mathbb V_{k+1}^{\sym\curl_+}(\boldsymbol{r}_1;\mathbb{T})=\mathbb V_{k}^{\div\div^+}(\boldsymbol{r}_2;\mathbb{S})\cap\ker(\div\div)$.
% This combined with $\mathbb V_{k+1}^{\sym\curl_+}(\boldsymbol{r}_1;\mathbb{T})\cap\ker(\sym\curl)=\dev\grad\mathbb V_{k+2}^{\grad} (\boldsymbol{r}_0;\mathbb R^3)$ ends the proof.
\end{proof}

We can further enlarge the space $\mathbb V_{k}^{\div\div^+}(\boldsymbol{r}_2;\mathbb{S}) $ to $\mathbb V_{k}^{\div\div}(\boldsymbol{r}_2;\mathbb{S}) $
and define
$$
\mathbb V_{k+1}^{\sym\curl}(\boldsymbol{r}_1):=\{ \boldsymbol{\tau}\in\mathbb V_{k+1}(\boldsymbol r_1)\otimes \mathbb T: \sym\curl\boldsymbol{\tau}\in\mathbb V_{k}^{\div\div}(\boldsymbol{r}_2;\mathbb{S})\}
$$
to get the finite element divdiv complex
\begin{align}
\label{eq:femdivdivcomplex3d}
{\rm RT}\xrightarrow{\subset}\mathbb V_{k+2}^{\grad} (\boldsymbol{r}_0;\mathbb R^3) &\xrightarrow{\dev\grad} \mathbb V_{k+1}^{\sym\curl}(\boldsymbol{r}_1;\mathbb{T}) \\
&\xrightarrow{\sym\curl}  \mathbb V_{k}^{\div\div}(\boldsymbol{r}_2;\mathbb{S}) \xrightarrow{\div\div} \mathbb V_{k-2}^{L^2} (\boldsymbol{r}_3)\rightarrow 0. \notag
\end{align}
%\begin{proof}
%Complex~\eqref{eq:femdivdivcomplex3d} is exactly complex~\eqref{eq:femdivdivcomplex3dvariant1} for $r_2^f\geq0$. But the case $r_2^f=-1$ is not easy.
%\end{proof}
%
We can also define
\begin{align*}
% \mathbb V_{k+1}^{\sym\curl_+}(\boldsymbol{r}_1;\mathbb{T})&:=\{\boldsymbol{\tau}\in\mathbb V_{k+1}(\boldsymbol r_1)\otimes \mathbb T: \sym\curl\boldsymbol{\tau}\in{\mathbb V}_{k}^{\div\div^+}(\boldsymbol r_2;\mathbb S)\}, \\
\mathbb V_{k+1}^{\sym\curl^+}(\boldsymbol{r}_1;\mathbb{T})&:=\{\boldsymbol{\tau}\in\mathbb V_{k+1}^{\curl}(\boldsymbol r_1;\mathbb M): \tr\boldsymbol{\tau}=0, \sym\curl\boldsymbol{\tau}\in{\mathbb V}_{k}^{\div\div}(\boldsymbol r_2;\mathbb S)\},
\end{align*}
and obtain another finite element divdiv complex
% \begin{align}
% \label{eq:femdivdivcomplex3dvariant1}
% {\rm RT}\xrightarrow{\subset}\mathbb V_{k+2}^{\grad} (\boldsymbol{r}_0) &\xrightarrow{\dev\grad} \mathbb V_{k+1}^{\sym\curl_+}(\boldsymbol{r}_1;\mathbb{T}) \\
% &\xrightarrow{\sym\curl}  \mathbb V_{k}^{\div\div^+}(\boldsymbol{r}_2;\mathbb{S}) \xrightarrow{\div\div} \mathbb V_{k-2}^{L^2} (\boldsymbol{r}_3)\rightarrow \boldsymbol{0}. \notag
% \end{align}
\begin{align}\label{eq:femdivdivcomplex3dvariant2}
{\rm RT}\xrightarrow{\subset}\mathbb V_{k+2}^{\div} (\boldsymbol{r}_0, (\boldsymbol r_1)_+) &\xrightarrow{\dev\grad} \mathbb V_{k+1}^{\sym\curl ^+}(\boldsymbol{r}_1;\mathbb{T}) \\
&\xrightarrow{\sym\curl}  \mathbb V_{k}^{\div\div}(\boldsymbol{r}_2;\mathbb{S}) \xrightarrow{\div\div} \mathbb V_{k-2}^{L^2} (\boldsymbol{r}_3)\rightarrow 0. \notag
\end{align}
%which can be derived from~\eqref{eq:femdivdivcomplex3d} by applying one $\widetilde{\quad}$ operation or one $\widehat{\quad}$ operation, respectively. And~\eqref{eq:femdivdivcomplex3d+} is from one $\widetilde{\quad}$ operation and one $\widehat{\quad}$ operation. 
%Probably we should formulate in a more general form. The later space can be $\mathbb V^{\div\div}(\boldsymbol r_2, \boldsymbol r_3)$. Again the BGG can merge two together.
%and using the diagram
%\begin{equation*}%\label{eq:BGGelasticity}
%\begin{tikzcd}
%\mathbb V_{k+2}^{\div} (\boldsymbol{r}_0, (\boldsymbol{r}_1)_+)
%\arrow{r}{\grad}
%& \widehat{\widetilde{\mathbb V}}_{k+1}^{\curl} (\boldsymbol{r}_1; \mathbb M)
% \arrow{r}{\curl}
% &
%\mathbb V^{\div\div^+}_{k}(\boldsymbol{r}_2;\mathbb S)\oplus\mskw\mathbb V^{\curl}_{k}(\boldsymbol{r}_2)
% \arrow{r}{\div\div}
% &  \mathbb V^{L^2}_{k-2}(\boldsymbol{r}_3) \to 0
% \\
% %
%\mathbb V_{k+1}^{\grad}((\boldsymbol{r}_1)_+)
% \arrow[ur,swap,"\iota"'] \arrow{r}{\grad}
% & 
%\mathbb V_{k}^{\curl}(\boldsymbol{r}_2) 
% \arrow[ur,swap,"\mskw"'] 
%%\arrow{r}{\curl}
% & 
%%
% & 
%% & 0
%\end{tikzcd},
%\end{equation*}
\begin{figure}[htbp]
\begin{center}
\includegraphics[width=4.35in]{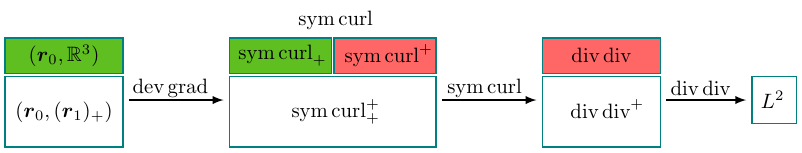}
\caption{Diverse configurations of finite element divdiv complexes. The fundamental finite element divdiv complex~\eqref{eq:femdivdivcomplex3d+}, depicted by white blocks, is formulated through the BGG framework, with additional smoothness. Then add green blocks to get ~\eqref{eq:femdivdivcomplex3dvariant1}, add red blocks to get~\eqref{eq:femdivdivcomplex3dvariant2}, and add both to get~\eqref{eq:femdivdivcomplex3d}.}
\label{fig:divdiv}
\end{center}
\end{figure}
Those variants are summarized in Fig.~\ref{fig:divdiv}. 

Verification of the exactness for the complexes~\eqref{eq:femdivdivcomplex3d} and~\eqref{eq:femdivdivcomplex3dvariant2} becomes intricate when the space ${\mathbb V}_{k}^{\div\div}(\boldsymbol r_2;\mathbb S)$ is introduced. A rigorous proof will be presented subsequently, following the constructive characterization of these spaces. 

\begin{example}\rm 
%Taking
%$\boldsymbol{r}_0=\begin{pmatrix}
%2\\
%0 \\
%0
%\end{pmatrix}$, $\boldsymbol{r}_1=\begin{pmatrix}
%1\\
%-1 \\
%-1
%\end{pmatrix}$, $\boldsymbol{r}_2=\begin{pmatrix}
%0\\
%-1 \\
%-1
%\end{pmatrix}$, $\boldsymbol{r}_3=-1$, 
We recover the finite element divdiv complex in~\cite{Hu;Liang;Ma:2021Finite}
\begin{equation*}
{\rm RT}\xrightarrow{\subset} \begin{pmatrix}
2\\
0 \\
0
\end{pmatrix} \xrightarrow{\dev\grad} \begin{pmatrix}
1\\
-1 \\
-1
\end{pmatrix} \xrightarrow{\sym\curl}\begin{pmatrix}
0\\
-1 \\
-1
\end{pmatrix} \xrightarrow{\div\div} \begin{pmatrix}
-1\\
-1 \\
-1
\end{pmatrix}\rightarrow 0.
\end{equation*} 
% In view of \eqref{eq:divonto3dsimple}, $((2, 0, 0)^{\intercal}, (1, 0, 0)^{\intercal}, k+2)$ for $k\geq 4$ is div stable.
\qed
\end{example}

\begin{example}\rm 
%Taking
%$\boldsymbol{r}_0=\begin{pmatrix}
%1\\
%0 \\
%0
%\end{pmatrix}$, $\boldsymbol{r}_1=\begin{pmatrix}
%0\\
%-1 \\
%-1
%\end{pmatrix}$, $\boldsymbol{r}_2=\begin{pmatrix}
%0\\
%-1 \\
%-1
%\end{pmatrix}$, $\boldsymbol{r}_3=-1$, 
%
We recover the finite element divdiv complex in~\cite{Chen;Huang:2020Finite}
\begin{equation*}
{\rm RT}\xrightarrow{\subset} \begin{pmatrix}
1\\
0 \\
0
\end{pmatrix} \xrightarrow{\dev\grad} \begin{pmatrix}
0\\
-1 \\
-1
\end{pmatrix} \xrightarrow{\sym\curl}\begin{pmatrix}
0\\
-1 \\
-1
\end{pmatrix} \xrightarrow{\div\div} \begin{pmatrix}
-1\\
-1 \\
-1
\end{pmatrix}\rightarrow 0.
\end{equation*}
\qed
\end{example}

\begin{remark}\rm 
Recently in~\cite{ChenHuang2024div-div-conforming} we have constructed $H(\div\div)$-conforming element without vertex continuity and the corresponding divdiv complex:
 \begin{equation*}
{\rm RT}\xrightarrow{\subset} \begin{pmatrix}
1\\
0 \\
0
\end{pmatrix} \xrightarrow{\dev\grad} \begin{pmatrix}
0\\
-1 \\
-1
\end{pmatrix} \xrightarrow{\sym\curl}\begin{pmatrix}
-1\\
-1 \\
-1
\end{pmatrix} \xrightarrow{\div\div} \begin{pmatrix}
-1\\
-1 \\
-1
\end{pmatrix}\rightarrow 0.
\end{equation*}
The construction of $\mathbb V_{k}^{\div\div}(\bs{-1};\mathbb{S})$ and the $\div\div$ stability requires careful redistribution of degrees of freedom (DoFs) and appears challenging to derive from the BGG construction. The last space can be further relaxed to the generalized $H^2$ non-conforming Morley-Wang-Xu elements~\cite{WangXu2006,WangXu2013} when $\div\div$ is understood in the distribution sense; see~\cite[Section 5.2]{ChenHuang2024div-div-conforming}.
 \qed
\end{remark}

%This additional step is essential for establishing the exactness of these extended complexes, ensuring their reliability and applicability.

% By a similar proof, we can obtain the bubble complex.
The bubble divdiv complex is shown in Proposition~\ref{prp:bubbledivdivcomplex}, whose proof is given in Section~\ref{sec:bubbledivdivcomplex}.
\begin{proposition}\label{prp:bubbledivdivcomplex}
Let %\mnote{ can we enlarge to $r_0 = (1, 0, 0)$?}
\begin{equation*}%\label{eq:divdivrsequence}
\boldsymbol{r}_0\geq (1, 0, 0)^{\intercal}, \quad 
\boldsymbol{r}_1=\boldsymbol{r}_0-1,\quad 
\boldsymbol{r}_2= \max\{\boldsymbol{r}_1\ominus1, (0, -1, -1)^{\intercal}\},\quad 
\boldsymbol{r}_3=\boldsymbol{r}_2\ominus2.
\end{equation*}
Assume $\boldsymbol r_2$ satisfies the condition in Lemma \ref{lem:divbubbleontoST}, $\boldsymbol r_2\ominus1$ satisfies \eqref{eq:boundr2fordivbubble}, and $k\geq \max\{2r_2^{\texttt{v}}+1,3\}$.
% Assume $(\boldsymbol r_2, \boldsymbol r_2\ominus 1, k)$ is $(\div;\mathbb S)$ stable and $(\boldsymbol r_2\ominus 1, \boldsymbol r_3, k-1)$ is div stable.
We have the following exact bubble divdiv complex   
\begin{align*}
0\xrightarrow{\subset}\mathbb B_{k+2} (\boldsymbol{r}_0;\mathbb R^3) &\xrightarrow{\dev\grad} \mathbb B_{k+1}^{\sym\curl}(\boldsymbol{r}_1;\mathbb{T}) \\
&\xrightarrow{\sym\curl}  \mathbb B_{k}^{\div\div}(\boldsymbol{r}_2;\mathbb{S}) \xrightarrow{\div\div} \mathbb B_{k-2}(\boldsymbol{r}_3)/\mathbb P_1(T)\rightarrow 0. 
\end{align*}
\end{proposition}

\section{Edge Elements}\label{sec:edgeelements}
In this section we shall construct finite element spaces for $H(\curl; \mathbb S),$ $H(\inc; \mathbb S)$, and $H(\sym\curl; \mathbb T)$ spaces. For $H(\curl; \mathbb S),$ we use the $t$-$n$ decomposition approach, and for the other two, we use the trace bubble complexes and the knowledge for div elements to determine the face DoFs and the edge traces to determine the edge DoFs. 

\subsection{$H(\curl; \mathbb S)$-conforming elements}
When $r^f\geq 0$, it is simply the tensor product: $$\mathbb V_k^{\curl}(\boldsymbol r,\mathbb S): = \mathbb V_k(\boldsymbol r)\otimes \mathbb S \quad  \text{ when }\boldsymbol r\geq 0.$$ 

Recall that the Hessian complex starts with an $H^2$-conforming element $\mathbb V^{\hess}_{k+2}(\boldsymbol r+2)$, where $\boldsymbol r+2 \geq (4, 2, 1)^{\intercal}$, and consequently, $r^e \geq 0$. Hence, for the remainder of this subsection, we will focus exclusively on the cases where $r^f = -1$ and $r^e \geq 0$. 

Since we have $r^{\texttt{v}} \geq 2r^e \geq 0$, the vertex and edge DoFs take on a tensor product structure: ${\rm DoF}_k(s; \boldsymbol r) \otimes \mathbb S$, where $s=\texttt{v}$ or $e$. However, determining the face DoFs requires a different approach. We continue to use the diagram~\eqref{eq:DoFdec}, but now with a modified $t$-$n$ decomposition, as the trace of the $\curl$ operator $(\cdot)\times \boldsymbol n$ contains only the tangential component. Consequently, the normal component $\mathbb B_{k} (f;\boldsymbol r) \otimes \text{span} \{\boldsymbol n \otimes \boldsymbol n\}$ contributes to the bubble space. 

To clarify this process, we define
\begin{align*} 
\mathbb B_{k}^{\curl}(\boldsymbol{r}; \mathbb S):={}& \mathbb B_{k}(\boldsymbol{r}_+;\mathbb S) \oplus[r^f=-1]\Oplus_{f\in\Delta_2(T)}(\mathbb B_{k} (f;\boldsymbol r) \otimes \text{span} \{\boldsymbol n \otimes \boldsymbol n\}),
\end{align*}
and refer to Fig.~\ref{fig:tnScurl} for an illustration.

\begin{figure}[htbp]
\begin{center}
\includegraphics[width=3.5cm]{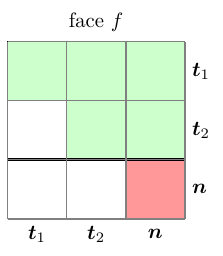}
\caption{A $t$-$n$ decomposition of $\mathbb S$ on a face for $H(\curl; \mathbb S)$ element. The normal component $\mathbb B_{k} (f;\boldsymbol r) \otimes \spa \{\boldsymbol n\otimes \boldsymbol n\}$ (red bloack) will contribute to the bubble space.}
\label{fig:tnScurl}
\end{center}
\end{figure}

Take $\mathbb P_{k}(T;\mathbb S)$ as the space of shape functions. Let $\boldsymbol r$ be a smoothness vector with $r^f=-1$ and $r^e\geq 0$. The degrees of freedom are 
\begin{subequations}\label{eq:curlSdof}
\begin{align}
% v
\nabla^i\boldsymbol{\tau}(\texttt{v}), & \quad i=0,\ldots, r^{\texttt{v}}, \texttt{v}\in \Delta_{0}(T), \label{eq:3dCrmodicurlSfemdofV}\\
% e
\int_e \frac{\partial^{j}\boldsymbol{\tau}}{\partial n_1^{i}\partial n_2^{j-i}}:\boldsymbol{q} \dd s, &\quad \boldsymbol{q}\in \mathbb B_{k-j}(e; r^{\texttt{v}} - j)\otimes \mathbb S, 0\leq i\leq j\leq r^e, e\in \Delta_{1}(T), \label{eq:3dCrmodicurlSfemdofE1}\\
% f: n x tau x n
\int_f (\boldsymbol n\times \boldsymbol{\tau}\times \boldsymbol n):\boldsymbol{q} \dd S, &\quad \boldsymbol{q}\in \mathbb B_{k}(f; \boldsymbol r)\otimes \mathbb S(f), f\in \Delta_{2}(T), \label{eq:3dCrmodicurlSfemdofF1}\\
%% f: t2-t2
%\int_f (\boldsymbol t_2^{\intercal} \boldsymbol{\tau}\boldsymbol{t}_2)\, q \dd S, &\quad q \in \mathbb B_{k}(f; \boldsymbol r), f\in \Delta_{2}(T), \label{eq:3dCrmodicurlSfemdofF1}\\
% f: n-t2
\int_f (\boldsymbol n^{\intercal} \boldsymbol{\tau}\Pi_f )\cdot \boldsymbol q \dd S, &\quad \boldsymbol q \in \mathbb B_{k}^2(f;\boldsymbol r), f\in \Delta_{2}(T), \label{eq:3dCrmodicurlSfemdofF2}\\
% bubble
\int_T \boldsymbol{\tau}:\boldsymbol{q} \dx, &\quad \boldsymbol{q}\in \mathbb B_{k}^{\curl}(\boldsymbol{r};\mathbb S), \label{eq:3dCrmodicurlSfemdofT}
\end{align}
\end{subequations}
where $\mathbb S(f):=\mathscr T^f(\mathbb S)= \textrm{span}\big\{ \sym(\boldsymbol t_i^f\otimes \boldsymbol t_j^f),  1\leq i\leq j\leq 2 \big\}$.
\begin{lemma}\label{lm:curlS}
Let $\boldsymbol r$ be a smoothness vector with $r^f = -1, r^{e}\geq 0$, and let $k\geq 2r^{\texttt{v}}+1$. DoFs~\eqref{eq:curlSdof} are unisolvent for $\mathbb P_{k}(T;\mathbb S)$. 
Given a triangulation $\mathcal T_h$ of $\Omega$, define
\begin{align}  
\label{eq:VcurlS}
\mathbb V^{\curl}_{k}(\boldsymbol{r};\mathbb S):=\{\boldsymbol{\tau}\in L^2(\Omega;\mathbb S): &\, \boldsymbol{\tau}|_T\in\mathbb P_{k}(T;\mathbb S) \textrm{ for all } T\in\mathcal T_h, \\
&\textrm{ and all the DoFs~\eqref{eq:curlSdof} are single-valued}\}. \notag
\end{align}  
Then  $\mathbb V^{\curl}_{k}(\boldsymbol{r};\mathbb S)\subset H(\curl,\Omega;\mathbb S)$. 
\end{lemma}
\begin{proof}
The unisolvence property can be established by leveraging the unisolvence of the tensor product space $\mathbb V_k(\boldsymbol r_+)\otimes \mathbb S$ and relocating the face DoF $\int_f \boldsymbol n^{\intercal}\boldsymbol \tau \boldsymbol n \, q$ to $\mathbb B_{k}^{\curl}(\boldsymbol{r};\mathbb S)$. The $H(\curl)$-conformity is a direct consequence of the continuity exhibited by $\boldsymbol \tau\times \boldsymbol n$, which arises from the single-valued DoFs~\eqref{eq:3dCrmodicurlSfemdofV}-\eqref{eq:3dCrmodicurlSfemdofF2}.
\end{proof}

Recall that in Section~\ref{sec:femHessian}, we have defined a space $\mathbb V^{\curl}_{k}(\boldsymbol{r}_1;\mathbb S)=\widetilde{\mathbb V}_{k}^{\curl}(\boldsymbol r_1;\mathbb M) \cap \ker(\vskw)$ from the BGG construction. Next we show they are equal.
\begin{lemma}
Let $\bs r_1\geq (2, 0, -1)^{\intercal}$. The space defined by~\eqref{eq:VcurlSBGG} is equal to the space $\mathbb V^{\curl}_{k}(\boldsymbol{r}_1;\mathbb S)$ defined by~\eqref{eq:VcurlS} when $r_1^f=-1$, and to $\mathbb V_k^{\curl}(\boldsymbol r_1,\mathbb S): = \mathbb V_k(\boldsymbol r_1)\otimes \mathbb S$ when $r_1^f\geq 0$.
\end{lemma}
\begin{proof}
Certainly, we have the inclusion $\mathbb V^{\curl}_{k}(\boldsymbol{r}_1;\mathbb S)\subseteq \widetilde{\mathbb V}_{k}^{\curl}(\boldsymbol r_1;\mathbb M) \cap \ker(\vskw)$. Thus, it is enough to demonstrate the dimension equality:
\begin{equation}\label{eq:dimcurlS}
\dim\mathbb V^{\curl}_{k}(\boldsymbol{r}_1;\mathbb S) = \dim (\widetilde{\mathbb V}^{\curl}_{k}(\boldsymbol{r}_1;\mathbb M)\cap \ker(\vskw)).
\end{equation}
Let us compare this with $\dim \mathbb V^{\curl}_{k}(\boldsymbol{r}_1;\mathbb M)$. Using the definition of the $\widetilde{\quad}$ operation and~\eqref{eq:divTdimdiff}, we can express:
\begin{align*}
\dim \widetilde{\mathbb V}^{\curl}_{k}(\boldsymbol{r}_1;\mathbb M) &= \dim \mathbb V^{\curl}_{k}(\boldsymbol{r}_1;\mathbb M) - (\dim \mathbb V^{\div}_{k-1}(\boldsymbol r_2; \mathbb M) - \dim \mathbb V^{\div}_{k-1}(\boldsymbol r_2; \mathbb T))\\
& = \dim \mathbb V^{\curl}_{k}(\boldsymbol{r}_1;\mathbb M)  - \dim \mathbb V_{k-1}^{L^2}(\boldsymbol r_2).
\end{align*}
Also, observe that $\dim \mathbb V^{\div}_k(\boldsymbol r_1)\cap \ker(\div) = \dim \mathbb V^{\div}_k(\boldsymbol r_1) - \dim \mathbb V_{k-1}^{L^2}(\boldsymbol r_2)$. With the surjectiveness of $\vskw$, we can write:
\begin{equation}\label{eq:dimtildeVcurl}
\dim (\widetilde{\mathbb V}^{\curl}_{k}(\boldsymbol{r}_1;\mathbb M)\cap \ker(\vskw)) = \dim \mathbb V^{\curl}_{k}(\boldsymbol{r}_1;\mathbb M) - \dim \mathbb V^{\div}_k(\boldsymbol r_1). 
\end{equation}
We now proceed to the construction of $\mathbb{V}^{\curl}_{k}(\boldsymbol{r}_1; \mathbb{S})$ and examine the associated dimension reduction in transitioning from $\mathbb{M}$ to $\mathbb{S}$.

When $r_1^f\geq0$, it is evident that $\dim\mathbb V^{\curl}_{k}(\boldsymbol{r}_1;\mathbb S) = \dim \mathbb V^{\curl}_{k}(\boldsymbol{r}_1;\mathbb M) - \dim \mathbb V^{\div}_k(\boldsymbol r_1)$. For the scenario where $r_1^f=-1$, a more intricate analysis is needed. On vertices and edges, since $r_1^{\texttt{v}}\geq 2 r_1^e\geq 0$, the change is a net decrease of $3$ DoFs from $\mathbb M$ to $\mathbb S$. In the case of faces, referring to Fig.~\ref{fig:tnScurl}, the component $\boldsymbol t_1\boldsymbol t_2^{\intercal}$ goes missing. In terms of the bubble space, transitioning from $\mathbb B_k(\boldsymbol r_1; \mathbb M)$ to $\mathbb B_k(\boldsymbol r_1; \mathbb S)$ incurs a reduction of $3$ DoFs. Additionally, for the face bubbles, we observe a reduction of $2$ DoFs (namely, the components $\boldsymbol t_1\boldsymbol n^{\intercal}$ and $\boldsymbol t_2\boldsymbol n^{\intercal}$). This reduction perfectly aligns with the dimension decrease required to define $\mathbb V^{\div}_k(\boldsymbol r_1)$. Therefore, we have successfully demonstrated that $\dim\mathbb V^{\curl}_{k}(\boldsymbol{r}_1;\mathbb S) = \dim \mathbb V^{\curl}_{k}(\boldsymbol{r}_1;\mathbb M) - \dim \mathbb V^{\div}_k(\boldsymbol r_1)$, which implies the validity of~\eqref{eq:dimcurlS} by referring back to~\eqref{eq:dimtildeVcurl}.%The bubble complex~\eqref{eq:femhessbubblecomplex} can be proved similarly and the proof is left for readers. 
\end{proof}

A finite element characterization of $\mathbb{V}^{\curl}_{k}(\boldsymbol{r}_1, \boldsymbol{r}_2; \mathbb{S})$ can be derived by incorporating DoFs associated with $\curl \boldsymbol{\tau} \in \mathbb{V}^{\div}_{k-1}(\boldsymbol{r}_2; \mathbb{T})$, followed by systematically eliminating redundancies in the DoFs related to derivatives. The detailed derivation is omitted for brevity.

\subsection{$H(\inc; \mathbb S)$-conforming elements}
We proceed to provide a detailed and explicit characterization of the $H(\inc)$-conforming element space $\mathbb V_{k+1}^{\inc}(\boldsymbol{r}_1;\mathbb{S})$, as defined in~\eqref{eq:VincS}. Our focus will be primarily on scenarios where $r_1^f=0$ or $r_1^f=-1$, as the cases with $r_1^f \geq 1$  are simply tensor product $\mathbb V_{k+1}^{\inc}(\boldsymbol{r}_1;\mathbb{S}) = \mathbb V_{k+1}(\boldsymbol{r}_1)\otimes\mathbb{S}$. As $\boldsymbol r_0\geq (2, 1, 0)^{\intercal}$ in the finite element elasticity complex, our subsequent analysis is restricted to $r_1^{\texttt{v}}\geq 1$ and $r_1^e \geq 0$.

%For a vector $\boldsymbol  v\in \mathbb R^3$, define the tangential trace and the normal trace as
%$$
%{\rm tr}_1(\boldsymbol  v) := \boldsymbol  v\times \boldsymbol  n, \quad {\rm tr}_2(\boldsymbol  v) := \boldsymbol  v \cdot \boldsymbol  n.
%$$
%For a smooth and symmetric tensor $\boldsymbol  \sigma \in \boldsymbol  H(\div, T; \mathbb S)$, define the normal-normal trace and the normal-tangential trace as
%$$
%{\rm tr}_1^{\div}(\boldsymbol  \sigma) := \boldsymbol  n\cdot \boldsymbol  \sigma \cdot \boldsymbol  n, \quad {\rm tr}_2^{\div}(\boldsymbol  \sigma) := \boldsymbol  n\times \boldsymbol  \sigma\cdot  \boldsymbol  n.
%$$
To motivate the edge DoFs, we first recall the trace complexes. For a smooth and symmetric tensor $\boldsymbol  \tau \in H(\inc, T; \mathbb S)$, define two trace operators as
\begin{align*}
% \label{eq:tr1inc}
{\rm tr}_1^{\inc}(\boldsymbol  \tau) = {}& \boldsymbol  n\times\boldsymbol  \tau \times \boldsymbol  n, \\
% \label{eq:tr2inc}
{\rm tr}_2^{\inc}(\boldsymbol  \tau) = {}&\boldsymbol  n\times(\curl\bs\tau)^{\intercal}\Pi_f + \grad_f(\Pi_f\boldsymbol  \tau\cdot \boldsymbol  n ).
\end{align*}
In~\cite[Section 4.2]{Chen;Huang:2021Finite} we have obtained the following trace complexes
\begin{equation*}%\label{eq:tracecomplex1}
\begin{array}{c}
\xymatrix{
\boldsymbol  a\times \boldsymbol  x + \boldsymbol  b \ar[r]^-{\subset} & \boldsymbol  v \ar[d]^{} \ar[r]^-{\defm}
                & \boldsymbol  \tau \ar[d]^{}   \ar[r]^-{\inc} & \ar[d]^{}\boldsymbol  \sigma \ar[r]^{\div} & \boldsymbol{p} \\
a_f \boldsymbol  x_f + \boldsymbol  b_f \ar[r]^{\subset} & \boldsymbol  v\times \boldsymbol  n \ar[r]^{\sym \curl_f}
                & \boldsymbol  n\times \boldsymbol  \tau \times \boldsymbol  n   \ar[r]^{\mathrm{div}_f{\mathrm{div}}_f} &  \boldsymbol  n\cdot \boldsymbol  \sigma \cdot \boldsymbol  n \ar[r]^{}& 0,    }
\end{array}
\end{equation*}
and
\begin{equation*}%\label{eq:tracecomplex2}
\begin{array}{c}
\xymatrix{
\boldsymbol  a\times \boldsymbol  x + \boldsymbol  b \ar[r]^-{\subset} & \boldsymbol  v \ar[d]^{} \ar[r]^-{\defm}
                & \boldsymbol  \tau \ar[d]^{}   \ar[r]^-{\inc} & \ar[d]^{}\boldsymbol  \sigma \ar[r]^{\div} & \boldsymbol{p} \\
\boldsymbol  a_f \cdot \boldsymbol  x_f + b_f \ar[r]^{\subset} & \boldsymbol  v\cdot \boldsymbol  n \ar[r]^{\nabla_f^2}
                & {\rm tr}^{\inc}_2(\boldsymbol  \tau)   \ar[r]^{\nabla_f^{\bot}\cdot} &  \boldsymbol  n\cdot \boldsymbol  \sigma \times \boldsymbol  n \ar[r]^{}& \boldsymbol{0}.    }
\end{array}
\end{equation*}

Take $\mathbb P_{k+1}(T; \mathbb S)$ as the shape function space. Let $(\boldsymbol r_0, \boldsymbol r_1, \boldsymbol r_2, \boldsymbol r_3)$ be given by~\eqref{eq:relasticity}.  
The degrees of freedom are 
\begin{subequations}\label{eq:incdof}
\begin{align}
% v
\nabla^i\boldsymbol{\tau}(\texttt{v}), & \quad i=0,\ldots, r_1^{\texttt{v}}, \label{eq:3dCrincSfemdofV1}\\
% v: inc
\inc\boldsymbol{\tau}(\texttt{v}), & \quad \textrm{ if }r_1^{\texttt{v}}=1, \label{eq:3dCrincSfemdofV2}\\
% e: Dn
\int_e \frac{\partial^{j}\boldsymbol{\tau}}{\partial n_1^{i}\partial n_2^{j-i}}:\boldsymbol{q} \dd s, &\quad \boldsymbol{q}\in \mathbb B_{k+1 -j }(e; r_1^{\texttt{v}}- j)\otimes \mathbb S, 0\leq i\leq j\leq r_1^e, \label{eq:3dCrincSfemdofE1}\\
% e: curl tau 
\int_e (\curl\boldsymbol{\tau})^{\intercal}\boldsymbol{t} \cdot\boldsymbol{q} \dd s, &\quad \boldsymbol{q}\in \mathbb B_{k}^3(e; r_1^{\texttt{v}}-1), \textrm{ if } r_1^{e}=0, \label{eq:3dCrincSfemdofE2}\\
% e: ni inc nj
\int_e (\boldsymbol{n}_i^{\intercal}(\inc\boldsymbol{\tau})\boldsymbol{n}_j)\,q \dd s, &\quad q\in \mathbb B_{k-1}(e; r_2^{\texttt{v}}), 1\leq i \leq j\leq 2, \textrm{ if } r_2^{e}=-1, \label{eq:3dCrincSfemdofE3}\\
%
% \int_e (\boldsymbol{t}^{\intercal}(\inc\boldsymbol{\tau})\boldsymbol{n}_f)\,q \dd s, &\quad q\in \mathbb P_{k-1 - 2(r_2^{\texttt{v}} +1)}(e), e\subset\partial f, \textrm{ if } r_2^{e}=-1, \label{eq:3dCrincSfemdofF0}\\
% f: tr1 tau
\int_f (\boldsymbol n\times \boldsymbol \tau\times \boldsymbol n): \boldsymbol{q} \dd S, &\quad \boldsymbol{q}\in \sym \curl_f \mathbb B_{k+2}^2(f; \boldsymbol r_0), \label{eq:3dCrincSfemdofF1}\\
% f: tr2 tau
\int_f \tr_2^{\inc}(\boldsymbol \tau): \boldsymbol q \dd S, &\quad \boldsymbol{q}\in \nabla_f^2\mathbb B_{k+2}(f; \boldsymbol r_0), \label{eq:3dCrincSfemdofF2}\\
% f: tau n
\int_f (\boldsymbol{\tau}\boldsymbol{n})\cdot \boldsymbol q \dd S, &\quad \boldsymbol{q}\in \mathbb B_{k+1}^3(f;\boldsymbol r_1), \textrm{ if } r_1^{f}=0, \label{eq:3dCrincSfemdofF3}\\
% f: pi inc n
\int_f \Pi_f(\inc\boldsymbol{\tau})\boldsymbol{n}\cdot\boldsymbol{q} \dd S, &\quad \boldsymbol{q}\in \mathbb B_{k-1}^{\div}(f; \boldsymbol r_2)/{\rm RT}(f), \label{eq:3dCrincSfemdofF4}\\
% f: n inc n
\int_f \boldsymbol{n}^{\intercal}(\inc\boldsymbol{\tau})\boldsymbol{n}\,q \dd S, &\quad q\in \mathbb B_{k-1}(f; (\boldsymbol r_2)_+)/\mathbb P_1(f), \label{eq:3dCrincSfemdofF5}\\
%
% \int_T \boldsymbol{\tau}\cdot\boldsymbol{q} \dx, &\quad \boldsymbol{q}\in \nabla^2\mathbb B_{k+3}(\boldsymbol{r}_1+2), \label{eq:3dCrincSfemdofT1} \\
% %
% \int_T \curl(\div\boldsymbol{\tau})\cdot\boldsymbol{q} \dx, &\quad \boldsymbol{q}\in \mathbb B_{k+1}^{3}((\boldsymbol{r}_1)_+), \label{eq:3dCrincSfemdofT2} \\
%
\int_T \boldsymbol{\tau}:\boldsymbol{q} \dx, &\quad \boldsymbol{q}\in \defm(\mathbb B_{k+2}^{3}(\boldsymbol{r}_0)), \label{eq:3dCrincSfemdofT1} \\
\int_T (\inc\boldsymbol{\tau}):\boldsymbol{q} \dx, &\quad \boldsymbol{q}\in \mathbb B_{k-1}^{\div}(\boldsymbol{r}_2;\mathbb S)\cap\ker(\div) \label{eq:3dCrincSfemdofT2}
\end{align}
\end{subequations}
for each $\texttt{v}\in \Delta_{0}(T)$, $e\in \Delta_{1}(T)$ and $f\in \Delta_{2}(T)$.

The motivation behind incorporating DoFs such as~\eqref{eq:3dCrincSfemdofV2},\eqref{eq:3dCrincSfemdofE3},\eqref{eq:3dCrincSfemdofF4},\eqref{eq:3dCrincSfemdofF5}, and\eqref{eq:3dCrincSfemdofT2} lies in their role in enforcing the condition $\inc \boldsymbol \tau \in \mathbb V_{k-1}^{\div}(\boldsymbol r_2; \mathbb S)$, which mirrors the purpose of DoFs~\eqref{eq:3dCrdivSfemdofV}-\eqref{eq:3dCrdivSfemdofT}. The inclusion of DoFs~\eqref{eq:3dCrincSfemdofT1}-\eqref{eq:3dCrincSfemdofT2} serves the distinct purpose of determining the bubble component. 
By the trace complexes, the face bubble complexes would be
\begin{align*}
\tr_1: \quad &\boldsymbol{0}\to \mathbb B^2_{k+2}(f; \boldsymbol r_0)\xrightarrow{\sym \curl_f} \mathbb B_{k+1}^{\div_f\div_f}(f; \boldsymbol r_1)\xrightarrow{\div_f\div_f} \mathbb B_{k-1}(f; \boldsymbol r_2)/\mathbb P_1(f) \to 0,\\
\tr_2: \quad &0\to \mathbb B_{k+2}(f; \boldsymbol r_0)\xrightarrow{\hess_f} \mathbb B_{k}^{\rot _f}(f; \boldsymbol r_1\ominus 1, \mathbb S)\xrightarrow{\rot_f} \mathbb B_{k-1}^2(f; \boldsymbol r_2)/{\rm RM}(f) \to \boldsymbol{0}. 
\end{align*}
However the face DoFs~\eqref{eq:3dCrincSfemdofF4} and~\eqref{eq:3dCrincSfemdofF5} imply the face bubble complexes are
\begin{align*}
% \label{eq:tracecomplex1}
\tr_1: \;  \boldsymbol{0}\to \mathbb B^2_{k+2}(f; \boldsymbol r_0)\xrightarrow{\sym \curl_f}  &\mathbb B_{k+1}^{\div_f\div_f}(f; \boldsymbol r_1, (\boldsymbol r_2)_+)\\
\notag \xrightarrow{\div_f\div_f} &\mathbb B_{k-1}(f; (\boldsymbol r_2)_+)/\mathbb P_1(f) \to 0,
\end{align*}
\begin{equation*}
% \label{eq:tracecomplex2}
\tr_2: \; 0\to \mathbb B_{k+2}(f; \boldsymbol r_0)\xrightarrow{\hess_f} \widetilde{\mathbb B}_{k}^{\rot _f}(f; \boldsymbol r_1\ominus 1, \mathbb S)\xrightarrow{\rot_f} \mathbb B_{k-1}^{\rot_f}(f; \boldsymbol r_2)/{\rm RM}(f)\to \boldsymbol{0}.
\end{equation*}
%where \mnote{ define these bubble spaces}
%$$
%\mathbb B_{k-1}^{\rot_f}(f; \boldsymbol r_2), ... 
%$$
These modifications have been accounted in the face DoFs for $\inc \boldsymbol \tau$, without affecting the components stemming from $\mathbb B^2_{k+2}(f; \boldsymbol r_0)$ and $\mathbb B_{k+2}(f; \boldsymbol r_0)$, as described by DoFs~\eqref{eq:3dCrincSfemdofF1}-\eqref{eq:3dCrincSfemdofF2}. For further insight into the specifics of these two-dimensional bubble polynomial spaces and finite element complexes, we refer to our recent work~\cite{Chen;Huang:2022femcomplex2d}.

In the event that $r_1^e=0$, the inclusion of~\eqref{eq:3dCrincSfemdofE2} is rooted in the aim of enforcing $(\curl\boldsymbol \tau)^{\intercal} \in \mathbb V_{k}^{\curl}(\boldsymbol r_1 \ominus 1; \mathbb M)$. Conversely, if $r_1^e\geq 1$, the same condition is inherently encompassed by~\eqref{eq:3dCrincSfemdofE1}, given that $\boldsymbol \tau$ exhibits $C^1$ continuity across edges. A similar rationale underpins the inclusion of~\eqref{eq:3dCrincSfemdofE3}, which is exclusively required when $r_2^e = -1$. Importantly, all traces of $\boldsymbol \tau$ are confined to its tangential component. In instances where $r_1^f = 0$, the introduction of~\eqref{eq:3dCrincSfemdofF3} becomes crucial to ensure the continuous nature of the normal component $\boldsymbol \tau\boldsymbol n$.

%Check if the proof holds for the inequality constraint.

%We can also increase $\boldsymbol r_2\geq \boldsymbol r_1 \ominus 2$ and change/add DoFs on $\inc \boldsymbol \tau$ only. In view of bubble complexes, the previous one doesn't change. The finite element elasticity complex can be also changed.
We present the following lemma for the ease of the dimension count. 
\begin{lemma}
The polynomial elasticity complex
\begin{equation}\label{eq:polyelasticity}
{\rm RM} \xrightarrow{\subset}\mathbb P_{k+1}(T; \mathbb R^3)\xrightarrow{\defm} \mathbb P_{k}(T; \mathbb S)\xrightarrow{\inc}\mathbb P_{k-2}(T; \mathbb S)\xrightarrow{\div} \mathbb P_{k-3}(T; \mathbb R^3) \to \boldsymbol{0}
\end{equation}
is exact for $k\geq 3$.
For integer $k\geq 0$,
\begin{equation}\label{eq:polydimelas}
-6 + 3{k+4\choose3}-6{k+3\choose3}+6{k+1\choose3}-3{k\choose3}=0.	
\end{equation}
\end{lemma}
\begin{proof}
The polynomial elasticity complex is comprehensively presented in~\cite[(2.6)]{ArnoldAwanouWinther2008}, or can be systematically derived via polynomial de Rham complexes within the BGG framework; see~\cite[Section 2.2]{Chen;Huang:2021Finite}.

Upon investigating cases where $k\geq 3$, the exactness of~\eqref{eq:polyelasticity} implies that~\eqref{eq:polydimelas} can be deduced by taking the alternating sum of the dimensions of the involved spaces. For the instances of $k=0,1,2$, the validity of~\eqref{eq:polydimelas} remains intact and can be readily confirmed through direct verification.
\end{proof}

\begin{lemma}\label{lem:incSDoFsnumber}
%Let $\boldsymbol r_0\geq (2, 1, 0)^{\intercal}$ and $\boldsymbol r_1 = \boldsymbol r_0 - 1$. Let $\boldsymbol r_2\geq \boldsymbol r_1 \ominus 2$ and $(\boldsymbol r_2, \boldsymbol r_3, k)$ is $(\div, \mathbb S)$ stable. 
For $r_1^f = 0, -1$, the sum of the number of DoFs~\eqref{eq:incdof} equals $\dim\mathbb P_{k+1}(T;\mathbb S)$.
\end{lemma}
\begin{proof}
At each vertex,  when $r_1^{\texttt{v}}\geq 2$, only DoFs~\eqref{eq:3dCrincSfemdofV1} exist with dimension $6 {r_1^{\texttt{v}}+3\choose3}$. By ~\eqref{eq:polydimelas}, we have
\begin{equation}\label{eq:incvertexdofnumber}
|{\rm DoF}_{k+1}^{\inc}(\texttt{v}; \boldsymbol r_1) |= -6 +3{r_0^{\texttt{v}}+3\choose3}+6{r_2^{\texttt{v}}+3\choose3}-3{r_3^{\texttt{v}}+3\choose3}.
\end{equation}
When $r_1^{\texttt{v}}=1$, $( r_0^{\texttt{v}}, r_1^{\texttt{v}}, r_2^{\texttt{v}}, r_3^{\texttt{v}}) = (2,1,0,-1)$, and $6$ DoFs~\eqref{eq:3dCrincSfemdofV2} are added. But the sum of number of DoFs~\eqref{eq:3dCrincSfemdofV1} -\eqref{eq:3dCrincSfemdofV2} is still equal to~\eqref{eq:incvertexdofnumber} by direct calculation. 

On tetrahedron $T$, the number of DoFs~\eqref{eq:3dCrincSfemdofT1}-\eqref{eq:3dCrincSfemdofT2} is
\begin{equation}\label{eq:incvolumedofnumber}
|{\rm DoF}_{k+1}^{\inc}(T; \boldsymbol r_1) | = 6 +\dim\mathbb B_{k+2}^{3}(\boldsymbol{r}_0) + \dim\mathbb B_{k-1}^{\div}(\boldsymbol{r}_2;\mathbb S) - \dim\mathbb B_{k-2}^3(\boldsymbol{r}_3),
\end{equation}
as $\div \mathbb B_{k-1}^{\div}(\boldsymbol{r}_2;\mathbb S) = \mathbb B_{k-2}^3(\boldsymbol{r}_3)/{\rm RM}$. 

On each face $f$, we first consider the case $r^f_1 = -1$. The number of DoFs~\eqref{eq:3dCrincSfemdofF1}-\eqref{eq:3dCrincSfemdofF5} is
\begin{align*}
&\dim \mathbb B_{k+2}^2(f;\boldsymbol r_0) + \dim \mathbb B_{k-1}(f; (\boldsymbol r_2)_+) - 3\\
+& \dim \mathbb B_{k+2}(f;\boldsymbol r_0) + \dim \mathbb B_{k-1}^{\rot}(f; \boldsymbol r_2) - 3\\
= & - 6 + 3|{\rm DoF}_{k+2}^{\grad}(f; \boldsymbol r_0)| + |{\rm DoF}_{k-1}^{\rm div}(f; \boldsymbol r_2, \mathbb S)|.
\end{align*}
No DoFs for ${\rm DoF}_{k-2}^{\grad}(f; \boldsymbol r_3)$ as $r_3^f = -1$. When $r_1^f = 0$, $r_0^f = 1$, we have one more layer in ${\rm DoF}_{k+2}^{\grad}(f; \boldsymbol r_0)$ for $\partial_n \boldsymbol v$: $\mathbb B_{k+2-1}^3(f; \boldsymbol r_0 - 1)$ which matches the number of DoF~\eqref{eq:3dCrincSfemdofF3} added for $r_1^f = 0$. So we conclude the 
the number of face DoFs~\eqref{eq:3dCrincSfemdofF1}-\eqref{eq:3dCrincSfemdofF5} satisfies
\begin{equation}\label{eq:incfacedofnumber}
{\rm DoF}_{k+1}^{\inc}(f; \boldsymbol r_1) = - 6 + 3|{\rm DoF}_{k+2}^{\grad}(f; \boldsymbol r_0)|  +   |{\rm DoF}_{k-1}^{\div}(f; \boldsymbol r_2, \mathbb S)|.
\end{equation}

%
%\begin{align*}  
%&-6 +3\dim\mathbb B_{k+2}(f; \boldsymbol r_0) + 3[r_1^f=0]\dim\mathbb B_{k+1}(f; \boldsymbol r_0-1) \\
%&\quad\quad  + \dim\mathbb B_{k-1}(f; (\boldsymbol r_2)_+),
%\end{align*}

On each edge $e$, when $r^e_1\geq 2$, 
% and consequently $(r^e_0, r^e_1, r^e_2,r^e_3) = (r^e_1+1, r^e_1, r^e_1-2,r^e_1-3)$
only the~\eqref{eq:3dCrincSfemdofE1} exists and its number satisfies 
\begin{align}  
\notag 
&6 + 3\sum_{j=0}^{r_0^e}(j+1)(k -2r_1^{\texttt{v}}-1+j) - 6\sum_{j=0}^{r_1^e}(j+1)(k-2r_1^{\texttt{v}}+j) \\
&\;\; + 6\sum_{j=0}^{r_2^e}(j+1)(k-2r_1^{\texttt{v}}+2+j) 
-3\sum_{j=0}^{r_3^e}(j+1)(k-2r_1^{\texttt{v}}+3+j).
\label{eq:incedgedofnumber0}
\end{align}
Let $m = k -2r_1^{\texttt{v}} - 1\geq 0$. We split~\eqref{eq:incedgedofnumber0} into terms containing $m$ or not:
\begin{align*}
{\rm I}_1 &= m\left [ 3{r_1^e+3\choose2}-6{r_1^e+2\choose2}+6{r_1^e \choose2} - 3{r_1^e-1\choose 2} \right ],\\
{\rm I}_2 &= 6 + 3\sum_{j=0}^{r_1^e+1}(j+1)j - 6\sum_{j=0}^{r_1^e}(j+1)^2+6\sum_{j=0}^{r_1^e-2}(j+1)(j+3) - 3\sum_{j=0}^{r_1^e-3}(j+1)(j+4).
\end{align*}
By symbolical calculation, ${\rm I}_1 = 0$ and ${\rm I}_2 = 0$ for all integers $r^e_1\geq 2$ even for the boundary case $(r_0^{e}, r_1^{e}, r_2^{e}, r_3^{e}) = (3,2,0,-1)$. 

%\begin{align} \label{eq:edgedofnumber}
%6 +3\sum_{j=0}^{r_0^e}(j+1)\dim\mathbb B_{k+2-j}(e; \boldsymbol r_0-j) -6\sum_{j=0}^{r_1^e}(j+1)\dim\mathbb B_{k+1-j}(e; \boldsymbol r_1-j)  + & \\
%+6\sum_{j=0}^{r_2^e}(j+1)\dim\mathbb B_{k-1-j}(e; \boldsymbol r_2-j) \notag 
%-3\sum_{j=0}^{r_3^e}(j+1)\dim\mathbb B_{k-2-j}(e; \boldsymbol r_3-j) & = 0. 
%\end{align}
%which follows from~\eqref{eq:dimelas1} for the smooth finite element elasticity complex for $r^e_1\geq 3$ and consequently $r_3^e\geq 0$. 
%The formulae~\eqref{eq:edgedofnumber} can be further extrapolated to $r^e_1= 2$ by direct calculation. 

When $r^e_1= 1,r^e_2= -1$,~\eqref{eq:3dCrincSfemdofE3} is added. The added number of DoFs~\eqref{eq:3dCrincSfemdofE3} matches that of~\eqref{eq:3dCrdivSfemdofE2} for ${\rm DoF}_{k-1}^{\div}(e;\mathbb S)$. When $r^e_1= 0$,~\eqref{eq:3dCrincSfemdofE2} is further added. Recall that  $m = k -2r_1^{\texttt{v}} - 1$. By direct calculation, $ 3|{\rm DoF}_{k+2}^{\grad}(e; \boldsymbol r_0)| = 9m + 6$. Sum of number of DoFs~\eqref{eq:3dCrincSfemdofE1}-\eqref{eq:3dCrincSfemdofE2} is 
$$ 
6\dim\mathbb B_{k+1}(e; \boldsymbol r_1) + 3 \dim\mathbb B_{k}(e; \boldsymbol r_1 -1 ) =  9m+12 = 6+  3|{\rm DoF}_{k+2}^{\grad}(e; \boldsymbol r_0)| .  
$$ 
So for all cases $r^e_1\geq 0$
the sum of number of edge DoFs~\eqref{eq:3dCrincSfemdofE1}-\eqref{eq:3dCrincSfemdofE3} satisfies
\begin{equation}\label{eq:incedgedofnumber}
|{\rm DoF}_{k}^{\inc}(e; \boldsymbol r_1)| = 6 + 3|{\rm DoF}_{k+2}^{\grad}(e; \boldsymbol r_0)|  +   |{\rm DoF}_{k-1}^{\div}(e; \boldsymbol r_2,\mathbb S)| - 3|{\rm DoF}_{k-2}^{L^2}(e; \boldsymbol r_3)|.
\end{equation}

%\begin{align*}  
%&6 +3\sum_{j=0}^{r_0^e}(j+1)\dim\mathbb B_{k+2-j}(e; \boldsymbol r_0-j) +6\sum_{j=0}^{r_2^e}(j+1)\dim\mathbb B_{k-1-j}(e; \boldsymbol r_2-j) \\
%&+3[r_2^e=-1]\dim\mathbb B_{k-1}(e; \boldsymbol r_2)-3\sum_{j=0}^{r_3^e}(j+1)\dim\mathbb B_{k-2-j}(e; \boldsymbol r_3-j),
%\end{align*}
% \begin{align*}  
% &6 +3\sum_{j=0}^{r_1^e+1}(j+1)(k-2r_1^{\texttt{v}}-1+j) +6\sum_{j=0}^{r_2^e}(j+1)(k-2r_2^{\texttt{v}}-2+j) \\
% &+3\chi(r_2^e=-1)(k-2r_2^{\texttt{v}}-2)-3\sum_{j=0}^{r_2^e\ominus1}(j+1)(k-2r_2^{\texttt{v}}-1+j).
% \end{align*}
 %the exactness of smooth finite element elasticity complexes 
%for $r_1^{e}\geq3$. 
%And we can verify it directly for $r_1^{e}=0,1,2$.

%In summary, we have extended the identities~\eqref{eq:dimelas} for $r_1^f\geq 1$ to the case $r_1^f=0, -1$. 
Then by the DoFs~\eqref{eq:C13d0}-\eqref{eq:C13d3} of spaces $\mathbb V_{k+2}(\boldsymbol{r}_0;\mathbb R^3)$ and $\mathbb V^{L^2}_{k-2}(\boldsymbol{r}_3;\mathbb R^3)$, Euler's formulae $|\Delta_0(T)| - |\Delta_1(T)| + |\Delta_2(T)| - |\Delta_3(T)| = 1$, and the DoFs~\eqref{eq:3dCrdivSfemdofV} -~\eqref{eq:3dCrdivSfemdofT} of space $\mathbb V^{\div}_{k-1}(\boldsymbol{r}_2;\mathbb S)$, and identities~\eqref{eq:incvertexdofnumber},\eqref{eq:incedgedofnumber},~\eqref{eq:incfacedofnumber}, and~\eqref{eq:incvolumedofnumber}, the number of DoFs~\eqref{eq:incdof} is
$$
-6 + \dim\mathbb P_{k+2}(T;\mathbb R^3)+\dim\mathbb P_{k-1}(T;\mathbb S)-\dim\mathbb P_{k-2}(T;\mathbb R^3),
$$
which equals $\dim\mathbb P_{k+1}(T;\mathbb S)$ in view of the polynomial elasticity complex~\eqref{eq:polyelasticity}.
\end{proof}

\begin{lemma}
The DoFs~\eqref{eq:incdof} are uni-solvent for $\mathbb P_{k+1}(T;\mathbb S)$.
\end{lemma}
\begin{proof}
By leveraging Lemma~\ref{lem:incSDoFsnumber}, our objective is to establish $\boldsymbol{\tau}=\boldsymbol{0}$ when $\boldsymbol{\tau}\in\mathbb P_{k+1}(T;\mathbb S)$ adheres to the condition that all the specified DoFs~\eqref{eq:incdof} vanish.

To commence, the vanishing~\eqref{eq:3dCrincSfemdofV1} implies that $\boldsymbol \tau(\texttt{v}) = 0$ and $(\curl \boldsymbol \tau)^{\intercal}(\texttt{v}) = 0$. This, combined with the vanishing DoFs~\eqref{eq:3dCrincSfemdofE1}-\eqref{eq:3dCrincSfemdofE2}, leads to $\boldsymbol \tau|_{e} =0$ and $(\curl \boldsymbol \tau)^{\intercal}|_{e} = 0$. Employing integration by parts further yields:
\begin{align*}
\int_f \Pi_f(\inc\boldsymbol{\tau})\boldsymbol{n}\cdot\boldsymbol{q} \dd S= \int_f {\rm rot}_f (\boldsymbol n\times (\curl\boldsymbol \tau)^{\intercal}\Pi_f)\cdot\boldsymbol{q} \dd S = 0, &\quad \boldsymbol{q}\in {\rm RT}(f), \\
\int_f \boldsymbol{n}^{\intercal}(\inc\boldsymbol{\tau})\boldsymbol{n}\,q \dd S=\int_f \div_f \div_f (\boldsymbol n\times \boldsymbol \tau \times \boldsymbol n)\,q \dd S=0, &\quad q\in \mathbb P_1(f).
\end{align*}
More detailed descriptions of edge traces can be found in~\cite[Lemma 4.8]{Chen;Huang:2021Finite}. 
This, combined with the nullified DoFs~\eqref{eq:3dCrincSfemdofV1}-\eqref{eq:3dCrincSfemdofE1},~\eqref{eq:3dCrincSfemdofE3},~\eqref{eq:3dCrincSfemdofF4}-\eqref{eq:3dCrincSfemdofF5}, and~\eqref{eq:3dCrincSfemdofT2},  yields $\inc\boldsymbol{\tau}=\boldsymbol{0}$.

Consequently it implies that $\boldsymbol{\tau}=\defm(\boldsymbol{v})$ for $\boldsymbol{v}\in\mathbb P_{k+2}(T;\mathbb R^3)$.
Furthermore, the nullification of DoFs~\eqref{eq:3dCrincSfemdofV1} and~\eqref{eq:3dCrincSfemdofE1} leads to $(\nabla^i\boldsymbol{v})(\texttt{v})=\boldsymbol{0}$ for $\texttt{v}\in\Delta_0(T)$ and $i=0,\ldots, r_1^{\texttt{v}}+1$, and $(\nabla^i\boldsymbol{v})|_e=\boldsymbol{0}$ for $e\in\Delta_1(T)$ and $i=0,\ldots, r_1^{e}+1$. Similarly, from the nullification of~\eqref{eq:3dCrincSfemdofF1}-\eqref{eq:3dCrincSfemdofF3}, we deduce that $(\nabla^i\boldsymbol{v})|_f=\boldsymbol{0}$ for $f\in\Delta_2(T)$ and $i=0,\ldots, r_1^{f}+1$.

Combining these outcomes, it is apparent that $\boldsymbol{v}\in\mathbb B_{k+2}^{3}(\boldsymbol{r}_0)$. Consequently, $\boldsymbol{v}=\boldsymbol{0}$ is deduced from the vanishing DoF~\eqref{eq:3dCrincSfemdofT1}.
\end{proof}

Next we show the constructed $H(\inc)$-conforming finite element space is indeed the space $\mathbb V_{k+1}^{\inc}(\boldsymbol{r}_1;\mathbb{S})$ defined by~\eqref{eq:VincS} and used in the finite element elasticity complex~\eqref{eq:femelasticitycomplex3d}. 
\begin{lemma}
For $r_1^f=-1, 0$, let
\begin{align*}  
\mathbb V_{k+1}^{\inc}(\boldsymbol{r}_1;\mathbb{S}) :=\{\boldsymbol{\tau}\in{L}^2(\Omega;\mathbb S): &\, \boldsymbol{\tau}|_T\in\mathbb P_{k+1}(T;\mathbb S) \textrm{ for all } T\in\mathcal T_h, \\
&\textrm{ and all the DoFs~\eqref{eq:incdof} are single-valued}\}.
\end{align*}
Then it is equal to the space $\{\boldsymbol{\tau}\in\mathbb V_{k+1}(\boldsymbol r_1)\otimes \mathbb S: \inc\boldsymbol{\tau}\in\mathbb V^{\div}_{k-1}(\boldsymbol{r}_2; \mathbb S)\}$.
\end{lemma}
\begin{proof}
Clearly $\mathbb V_{k+1}^{\inc}(\boldsymbol{r}_1;\mathbb{S})\subseteq \{\boldsymbol{\tau}\in\mathbb V_{k+1}(\boldsymbol r_1)\otimes \mathbb S: \inc\boldsymbol{\tau}\in\mathbb V^{\div}_{k-1}(\boldsymbol{r}_2; \mathbb S)\}$. By~\eqref{eq:incSfemdim} and the proof of Lemma~\ref{lem:incSDoFsnumber}, their dimensions are equal.
\end{proof}

There exist various variations of the finite element elasticity complexes. To illustrate one of these variations, we construct the space $\mathbb V_{k+1}^{\inc^+}(\boldsymbol r_1; \mathbb S)$. In cases where $r_1^f = -1$, we introduce an additional face degree of freedom:
\begin{equation}
\int_f (\boldsymbol{n}^{\intercal}\boldsymbol{\tau}\Pi_f)\cdot\boldsymbol{q} \dd S, \quad \boldsymbol{q}\in \mathbb B_{k+1}^2(f;\boldsymbol r_1).
\label{eq:3dCrinc+SfemdofF}
\end{equation}
Moreover, we modify~\eqref{eq:3dCrincSfemdofT1} to:
\begin{equation}
\int_T \boldsymbol{\tau}:\boldsymbol{q} \dx, \quad \boldsymbol{q}\in \defm(\mathbb B_{k+2}^{\curl}(\boldsymbol{r}_0, (\boldsymbol{r}_1)_+)). \label{eq:3dCrinc+SfemdofT1}
\end{equation}
Recall that $\boldsymbol r_0\geq (2,1,0)^{\intercal}$. For $\boldsymbol u \in \mathbb V_{k+2}^{\curl}(\boldsymbol{r}_0)$, the normal continuity of $\curl \boldsymbol u$ is always maintained. To achieve continuity for all components, we require $\int_f \boldsymbol n\times \curl \boldsymbol u \cdot \boldsymbol q \dd S$ for $\boldsymbol q\in \mathbb B_{k+1}^2(f;\boldsymbol r_1)$. This ensures a balance between the added DoFs in~\eqref{eq:3dCrinc+SfemdofF} and the reduced DoFs from ~\eqref{eq:3dCrincSfemdofT1} to ~\eqref{eq:3dCrinc+SfemdofT1}, maintaining unisolvence in a similar manner.

An advantage of using the space $\mathbb V_{k+1}^{\inc^+}(\boldsymbol{r}_1; \mathbb{S})$ is the reduction in dimension from the space $\mathbb V_{k+1}^{\inc}(\boldsymbol{r}_1; \mathbb{S})$:
$$
\dim \mathbb V_{k+1}^{\inc}(\boldsymbol{r}_1; \mathbb{S}) - \dim \mathbb V_{k+1}^{\inc^+}(\boldsymbol{r}_1; \mathbb{S}) = (4 |\Delta_3(\mathcal T_h)| - |\Delta_2(\mathcal T_h)|) \dim \mathbb B_{k+1}^2(f; \boldsymbol r_1).
$$

A relaxed constraint, $\boldsymbol r_2 \geq \max\{\boldsymbol{r}_1 \ominus 2, (0, -1, -1)^{\intercal}\}$, leads to another variation of the space $\mathbb V_{k+1}^{\inc}(\boldsymbol r_1; \mathbb S)$:
$$
\mathbb V_{k+1}^{\inc}(\boldsymbol r_1, \boldsymbol r_2; \mathbb S) :=  \{ \boldsymbol \tau \in \mathbb V_{k+1}^{\inc}(\boldsymbol r_1; \mathbb S): \inc \boldsymbol \tau \in \mathbb V_{k-1}^{\div}(\boldsymbol r_2; \mathbb S)\}.
$$ 
A finite element description of this space can be derived by first introducing the necessary DoFs to determine $\inc \boldsymbol \tau$, and then eliminating any redundant DoFs. Due to the complexity of these variations, we omit the detailed explanation here.

\subsection{$H(\sym\curl; \mathbb T)$-conforming elements}
We will now provide a comprehensive description of the $H(\sym\curl)$-conforming element space $\mathbb V_{k+1}^{\sym\curl_+}(\boldsymbol{r}_1;\mathbb{T})$. When $r_1^f\geq 1$, it is simply $\mathbb V_{k+1}(\boldsymbol{r}_1)\otimes \mathbb{T}$. So our focus is $r_1^f=-1$ or $r_1^f=0$. 

It is important to recall the trace complexes that were previously established in~\cite{Chen;Huang:2020Finite}:
%, as they serve as the foundation for our description.
\begin{equation*}%\label{eq:tracecomplex1}
\begin{array}{c}
\xymatrix{
{\rm RT} \ar[d]^{} \ar[r]^-{\subset} & \boldsymbol  v \ar[d]^{} \ar[r]^-{\dev \grad}
                & \boldsymbol  \tau \ar[d]^{}   \ar[r]^-{\sym \curl} & \ar[d]^{}\boldsymbol  \sigma \ar[r]^{\div\div} & p \\
\mathbb R \ar[r]^{\subset\;\;} & \boldsymbol  v\cdot \boldsymbol  n \ar[r]^-{-\curl_f}
                & \boldsymbol  n\cdot \boldsymbol  \tau \times \boldsymbol  n   \ar[r]^-{\mathrm{div}_f} &  \boldsymbol  n\cdot \boldsymbol  \sigma \cdot \boldsymbol  n \ar[r]^{}& 0,    }
\end{array}
\end{equation*}
%where $\boldsymbol b_F:=\boldsymbol b\times\boldsymbol n-\boldsymbol a_F(\boldsymbol x\cdot\boldsymbol n)$,
and
\begin{equation*}%\label{eq:tracecomplex2}
\begin{array}{c}
\xymatrix{
{\rm RT} \ar[d]^{} \ar[r]^-{\subset} & \boldsymbol  v \ar[d]^{} \ar[r]^-{\dev \grad}
                & \boldsymbol  \tau \ar[d]^{}   \ar[r]^-{\sym \curl} & \ar[d]^{}\boldsymbol  \sigma \ar[r]^{\div\div} & p \\
{\rm RT}_f \ar[r]^{\quad\subset} & \Pi_f\boldsymbol v \ar[r]^-{-\sym \curl_f}
                & \Pi_f\sym(\boldsymbol \tau\times \boldsymbol n)\Pi_f   \ar[r]^-{\div_f\div_f} &  {\rm tr}_2^{\div\div}(\boldsymbol \sigma) \ar[r]^{}& 0.    }
\end{array}
\end{equation*}
The trace complexes above play a crucial role in guiding the design of edge and face DoFs to ensure the necessary continuity. As shown in~\cite[Lemma 6.1]{Chen;Huang:2020Finite}, the expression
\begin{equation*}
\tr_{2,e}^{\div_f\div_f}(\Pi_f\sym(\boldsymbol \tau\times \boldsymbol n)\Pi_f) =\boldsymbol n_{f,e}^{\intercal}(\curl\boldsymbol \tau)\boldsymbol n_f + \partial_{t}(\boldsymbol t^{\intercal}\bs\tau\boldsymbol t),
\end{equation*}
provides the motivation for introducing DoFs involving terms like $\boldsymbol{n}_2^{\intercal}(\curl\boldsymbol{\tau})\boldsymbol{n}_1+\partial_t(\boldsymbol{t}^{\intercal}\boldsymbol{\tau}\boldsymbol{t})$ on edges. In cases where $r_1^e = 0$, the focus is on enforcing the continuity of $\boldsymbol{n}_2^{\intercal}(\curl\boldsymbol{\tau})\boldsymbol{n}_1$ on edges, which aligns with the requirement $\skw \curl \widehat{\widetilde{\mathbb V}}_{k+1}^{\curl} (\boldsymbol r_1;\mathbb M) \subseteq \mskw\mathbb V^{\curl}_{k}(\boldsymbol{r}_1\ominus1)$. The other edge traces further ensure the continuity of terms like $\boldsymbol n_i^{\intercal}\boldsymbol\tau\boldsymbol t$ on edges.
% These considerations collectively lead to the desired continuity properties within the finite element space.

%{Consider the mapping $$\widehat{\widetilde{\mathbb V}}_{k+1}^{\curl} (\boldsymbol r_1;\mathbb M) \xrightarrow{\curl} \widetilde{\mathbb V}_{k}^{\div}(\boldsymbol r_2;\mathbb M) = \mathbb V^{\div\div^+}_{k}(\boldsymbol{r}_2;\mathbb S)\oplus\mskw\mathbb V^{\curl}_{k}(\boldsymbol{r}_1\ominus1),$$
%which implies 
%\begin{align*}
%\sym \curl \widehat{\widetilde{\mathbb V}}_{k+1}^{\curl} (\boldsymbol r_1;\mathbb M) &\to \mathbb V^{\div\div^+}_{k}(\boldsymbol{r}_2;\mathbb S),\\
%\skw \curl \widehat{\widetilde{\mathbb V}}_{k+1}^{\curl} (\boldsymbol r_1;\mathbb M) &\to \mskw\mathbb V^{\curl}_{k}(\boldsymbol{r}_1\ominus1).
%\end{align*}
%Then $\vskw \curl \boldsymbol \tau \in \mathbb V^{\curl}_{k}(\boldsymbol{r}_1\ominus1)$ for $\boldsymbol \tau \in \widehat{\widetilde{\mathbb V}}_{k+1}^{\curl} (\boldsymbol r_1;\mathbb M)$.  
%Choose an orthonormal frame $\{\boldsymbol n_1, \boldsymbol n_2, \boldsymbol t\}$ and expand the matrix $\curl\boldsymbol \tau$ in this frame. It is straightforward to compute 
%\begin{equation}\label{eq:vskwcurl}
%(\vskw \curl \boldsymbol \tau)\cdot \boldsymbol t = \boldsymbol n_2^{\intercal} (\skw \curl \boldsymbol \tau) \boldsymbol n_1.
%\end{equation}
%Together with the continuity of $\boldsymbol n_i^{\intercal}(\sym \curl \boldsymbol \tau)\boldsymbol n_j$ for $i,j=1,2$, we obtain the continuity of $\boldsymbol n_2^{\intercal} \curl \boldsymbol \tau \boldsymbol n_1$ on edges. 
%}

The shape function space is $\mathbb P_{k+1}(T; \mathbb T)$. The degrees of freedom are 
\begin{subequations}\label{eq:symcurldof}
\begin{align}
\nabla^i\boldsymbol{\tau}(\texttt{v}), & \quad i=0,\ldots, r_1^{\texttt{v}}, \label{eq:3dCrsymcurlTfemdofV1}\\
\sym\curl\boldsymbol{\tau}(\texttt{v}), & \quad \textrm{ if }r_1^{\texttt{v}}=0, \label{eq:3dCrsymcurlTfemdofV2}\\
% e: Dn tau
\int_e \frac{\partial^{j}\boldsymbol{\tau}}{\partial n_1^{i}\partial n_2^{j-i}}:\boldsymbol{q} \dd s, &\quad \boldsymbol{q}\in \mathbb B_{k+1-j}(e; r_1^{\texttt{v}} - j)\otimes  \mathbb T, 0\leq i\leq j\leq r_1^e, \label{eq:3dCrsymcurlTfemdofE1}\\
% e: n \tau t
\int_e (\boldsymbol{n}_i^{\intercal}\boldsymbol{\tau}\boldsymbol{t})\,q \dd s, &\quad q\in \mathbb B_{k+1}(e;  r_1^{\texttt{v}}), \textrm{ if } r_1^{e}=-1, \label{eq:3dCrsymcurlTfemdofE2}\\
% e: tr_e
\int_e (\boldsymbol{n}_2^{\intercal}(\curl\boldsymbol{\tau})\boldsymbol{n}_1+\partial_t(\boldsymbol{t}^{\intercal}\boldsymbol{\tau}\boldsymbol{t}))\,q \dd s, &\quad q\in \mathbb B_{k}(e;  r_1^{\texttt{v}}-1), \textrm{ if } r_1^{e}=-1,0, \label{eq:3dCrsymcurlTfemdofE3}\\
% e: n sym curl n
\int_e (\boldsymbol{n}_i^{\intercal}(\sym\curl\boldsymbol{\tau})\boldsymbol{n}_j)\,q \dd s, &\quad q\in \mathbb B_{k}(e;  r_2^{\texttt{v}}), 1\leq i\leq j\leq 2, \textrm{ if } r_2^{e}=-1, \label{eq:3dCrsymcurlTfemdofE4}\\
% f: tau n
\int_f \boldsymbol \tau \boldsymbol n \cdot \boldsymbol q \dd S, &\quad \boldsymbol{q} \in \mathbb B_{k+1}^3(f; \boldsymbol r_1), \text{ if } r_1^f = 0, 
\label{eq:3dCrsymcurlTfemdofF0}\\
% f: n tau x n
\int_f (\boldsymbol{n}\cdot\boldsymbol{\tau}\times\boldsymbol{n})\cdot\boldsymbol{q} \dd S, &\quad \boldsymbol{q}\in \curl_f\mathbb B_{k+2}(f;\boldsymbol r_0), \label{eq:3dCrsymcurlTfemdofF1}\\
\int_f \div_f(\boldsymbol{n}\cdot\boldsymbol{\tau}\times\boldsymbol{n})\,q \dd S, &\quad q\in \mathbb B_{k}(f;(\boldsymbol r_2)_+)/\mathbb R, \label{eq:3dCrsymcurlTfemdofF2}\\
% f: sym 
\int_f \Pi_f\sym(\boldsymbol{\tau}\times\boldsymbol{n})\Pi_f:\boldsymbol{q} \dd S, &\quad \boldsymbol{q}\in \sym\curl_f\mathbb B_{k+2}^2(f;\boldsymbol r_0), \label{eq:3dCrsymcurlTfemdofF3}\\
% f divdiv 
\int_f \div_f\div_f(\sym(\boldsymbol{\tau}\times\boldsymbol{n}))\,q \dd S, &\quad q\in \mathbb B_{k-1}(f;\boldsymbol r_2\ominus 1)/\mathbb P_1(f), \label{eq:3dCrsymcurlTfemdofF4}\\
\int_f \Pi_f(\sym\curl \boldsymbol \tau)\boldsymbol n\cdot \boldsymbol q \dd S, &\quad \boldsymbol q\in \mathbb B_{k}^{\div}(f;\boldsymbol r_2), \label{eq:3dCrsymcurlTfemdofF5}\\
% f: div tau curl
%\int_f (\div\sym\curl \boldsymbol \tau)\cdot \boldsymbol n\,  q \dd S, &\quad  q\in \mathbb B_{k-1}(f;\boldsymbol r_2\ominus 1), \label{eq:3dCrsymcurlTfemdofF6}\\
\int_T \boldsymbol{\tau}:\boldsymbol{q} \dx, &\quad \boldsymbol{q}\in \dev\grad(\mathbb B_{k+2}^{3}(\boldsymbol{r}_0)), \label{eq:3dCrsymcurlTfemdofT1} \\
\int_T (\sym\curl\boldsymbol{\tau}):\boldsymbol{q} \dx, &\quad \boldsymbol{q}\in \mathbb B_{k}^{\div\div^+}(\boldsymbol{r}_2;\mathbb{S})\cap\ker(\div\div) \label{eq:3dCrsymcurlTfemdofT2}
\end{align}
\end{subequations}
for each $\texttt{v}\in \Delta_{0}(T)$, $e\in \Delta_{1}(T)$ and $f\in \Delta_{2}(T)$.
% Here DoF~\eqref{eq:3dCrdivSfemdofF0} is located on face $f$.

\begin{lemma}
 The following polynomial $\div\div$ complex
\begin{equation}\label{eq:divdivcomplex3dPoly}
%\resizebox{.92\hsize}{!}{$
{\rm RT}\xrightarrow{\subset} \mathbb P_{k+2}(T; \mathbb R^3)\xrightarrow{\dev\grad}\mathbb P_{k+1}(T; \mathbb T)\xrightarrow{\sym\curl} \mathbb P_k(T; \mathbb S) \xrightarrow{\div{\div}} \mathbb P_{k-2}(T)\xrightarrow{}0
%$}
\end{equation}
is exact. For integer $k\geq 2$,
\begin{equation}\label{eq:divdivdim}
-4 + \dim  \mathbb P_{k+2}(T; \mathbb R^3) - \dim \mathbb P_{k+1}(T; \mathbb T) + \mathbb P_k(T; \mathbb S) - \mathbb P_{k-2}(T) = 0.
\end{equation}
And, for all integers $m\geq 0$, 
%i.e., $m\in \mathbb N_0 = \{0,1,2,3 \ldots\}$
\begin{equation}\label{eq:polydim}
-4 +3{m+4\choose3} - 8 {m + 3\choose 3} +6{m+2 \choose 3}-{m \choose 3} = 0.
\end{equation}
\end{lemma}
\begin{proof}
A proof of~\eqref{eq:divdivcomplex3dPoly} can be found in~\cite{Chen;Huang:2020Finite}. A consequence of the exactness of~\eqref{eq:divdivcomplex3dPoly} is the dimension identity~\eqref{eq:divdivdim}. 
Identity~\eqref{eq:polydim} follows from~\eqref{eq:divdivdim} when $m\geq 3$ and can be verified directly for $m = 0,1,2$. 
\end{proof}
 
%We will use~\eqref{eq:polydim} for the dimension count. 
 
\begin{lemma}\label{lem:symcurlTDoFsnumber}
The sum of the number of DoFs~\eqref{eq:symcurldof} equals $\dim\mathbb P_{k+1}(T;\mathbb T)$.
\end{lemma}
\begin{proof}
At each vertex,  when $r_1^{\texttt{v}}\geq 1$, only  DoFs~\eqref{eq:3dCrsymcurlTfemdofV1} exists with dimension $8 {r_1^{\texttt{v}} + 3\choose 3}$. By ~\eqref{eq:polydim}, we have
\begin{equation}\label{eq:symcurldofnumber0}
{\rm DoF}_{k+1}^{\sym\curl_+}(\texttt{v}; \boldsymbol r_1) =  -4 +3{r_0^{\texttt{v}}+3\choose3}+6{r_2^{\texttt{v}}+3\choose3}-{r_3^{\texttt{v}}+3 \choose 3}.
\end{equation}
%For $r_1^{\texttt{v}}\geq 3$,  and $r_3^{\texttt{v}}\geq 0$, thus~\eqref{eq:symcurldofnumber0} is from. For $r_1^{\texttt{v}}=1, 2$,~\eqref{eq:symcurldofnumber0} still holds by direct calculation. 
When $r_1^{\texttt{v}}=0$, additional $6$ DoFs~\eqref{eq:3dCrsymcurlTfemdofV2} are added but now $r_2^{\texttt{v}} = 0, r_3^{\texttt{v}} = -1$ and so~\eqref{eq:symcurldofnumber0} still holds.

On tetrahedron $T$, the number of DoFs~\eqref{eq:3dCrsymcurlTfemdofT1}-\eqref{eq:3dCrsymcurlTfemdofT2} is
\begin{equation}\label{eq:symcurldofnumber3}
4 +\dim\mathbb B_{k+2}^{3}(\boldsymbol{r}_0) + \dim\mathbb B_{k}^{\div\div^+}(\boldsymbol{r}_2;\mathbb S) - \dim\mathbb B_{k-2}(\boldsymbol{r}_3),
\end{equation}
as $\div\div \mathbb B_{k}^{\div\div^+}(\boldsymbol{r}_2;\mathbb T) = \mathbb B_{k-2}(\boldsymbol{r}_3)/{\rm RT}$. 

On each face, we consider $r_1^f = -1$ first. 
Sum of number of DoFs~\eqref{eq:3dCrsymcurlTfemdofF1} and~\eqref{eq:3dCrsymcurlTfemdofF3} is $\dim\mathbb B_{k+2}^3(f; \boldsymbol r_0) =  3|{\rm DoF}_{k+2}^{\grad}(f; \boldsymbol r_0)|$. Sum of~\eqref{eq:3dCrsymcurlTfemdofF2} and~\eqref{eq:3dCrsymcurlTfemdofF4} is $\dim\mathbb B_{k}(f; (\boldsymbol r_2)_+) + \dim\mathbb B_{k-1}(f; \boldsymbol r_2\ominus 1)  -4$. Comparing with the face DoFs~\eqref{eq:divdiv+SdF2}-\eqref{eq:divdiv+SdF3} for $\mathbb V_k^{\div\div^+}(\boldsymbol r_2; \mathbb S)$, plus~\eqref{eq:3dCrsymcurlTfemdofF5}, we conclude the dimension identity
\begin{equation}\label{eq:symcurldofnumber2}
|{\rm DoF}_{k+1}^{\sym\curl_+}(f; \boldsymbol r_1) |= -4 + 3|{\rm DoF}_{k+2}^{\grad}(f; \boldsymbol r_0)| + |{\rm DoF}_{k}^{\div\div^+}(f; \boldsymbol r_2)|.
\end{equation}
When $r_1^f = 0$, $r_2^f=-1$ and thus no change of $|{\rm DoF}_{k}^{\div\div^+}(f; \boldsymbol r_2)|$. But $r_0^f = 1$, one more layer in ${\rm DoF}_{k+1}^{\grad}(f; \boldsymbol r_0)$ is added for $\partial_n \boldsymbol v$: $\dim \mathbb B_{k+2-1}^3(f; \boldsymbol r_0 - 1)$  which matches the number of DoF~\eqref{eq:3dCrsymcurlTfemdofF0} added for $r_1^f = 0$. So~\eqref{eq:symcurldofnumber2} holds for both $r_1^f = -1, 0$. No face DoFs for $\mathbb V_{k-2}^{L^2}(f; \boldsymbol r_3)$ as $r_3^f = -1$.

On each edge, we separate into three cases. \\
\step 1 When $r^e_1\geq 1$, only~\eqref{eq:3dCrsymcurlTfemdofE1} exists. We write out the dimension of edge DoFs for spaces in the divdiv complex 
%\mnote{ check the calculation. mine is different. The complex start from $k+2$. }\eqref{eq:femdivdivcomplex3dvariant1}:
\begin{align}  
\notag 4 +& 3\sum_{j=0}^{r_0^e}(j+1)(k -2r_1^{\texttt{v}}-1+j) - 8\sum_{j=0}^{r_1^e}(j+1)(k-2r_1^{\texttt{v}}+j) \\
& +6\sum_{j=0}^{r_2^e}(j+1)(k-2r_1^{\texttt{v}}+1+j) 
-\sum_{j=0}^{r_3^e}(j+1)(k-2r_1^{\texttt{v}}+3+j).
\label{eq:symcurledgedofnumber}
\end{align}
%\breakline
%{
%\begin{align}  
%\notag 4 +& 3\sum_{j=0}^{r_0^e}(j+1)(k -2r_1^{\texttt{v}}-2+j) - 8\sum_{j=0}^{r_1^e}(j+1)(k-2r_1^{\texttt{v}}-1+j) \\
%& +6\sum_{j=0}^{r_2^e}(j+1)(k-2r_1^{\texttt{v}}+j) 
%-\sum_{j=0}^{r_3^e}(j+1)(k-2r_1^{\texttt{v}}+2+j).
%\label{eq:symcurledgedofnumber}
%\end{align}}
%\breakline
Let $m = k -2r_1^{\texttt{v}} - 1\geq 0$. We split~\eqref{eq:symcurledgedofnumber} into terms containing $m$ or not:
\begin{align*}
% \label{eq:mpart}
{\rm I}_1 &= m\left [ 3{r_1^e+3\choose2}-8{r_1^e+2\choose2}+6{r_1^e+1\choose2} - {r_1^e-1\choose 2} \right ],\\
{\rm I}_2 &= 4 + 3\sum_{j=0}^{r_0^e}(j+1)j - 8\sum_{j=0}^{r_1^e}(j+1)^2+6\sum_{j=0}^{r_2^e}(j+1)(j+2) - \sum_{j=0}^{r_3^e}(j+1)(j+4).
\end{align*}
%{ 
%\begin{align*}
%\label{eq:nompart}
%{\rm I}_2 &= 4 + 3\sum_{j=0}^{r_0^e}(j+1)(j-1) - 8\sum_{j=0}^{r_1^e}(j+1)j+6\sum_{j=0}^{r_2^e}(j+1)^2 - \sum_{j=0}^{r_3^e}(j+1)(j+3).
%\end{align*}
%}
By symbolical calculation, ${\rm I}_1 = 0$ and ${\rm I}_2 = 0$ for all integers $r^e_1\geq 1$ even for the case $(r_0^{e}, r_1^{e}, r_2^{e}, r_3^{e}) = (2,1,0,-1)$. 

%\begin{itemize}
%\item 
%When $r^e_1\geq 3$, $r^f_1\geq 1$ and $(r_0^{e}, r_1^{e}, r_2^{e}, r_3^{e})$ is an arithmetic progression with a common difference of $-1$, $\mathbb V_{k+1}^{\sym\curl_+}(\boldsymbol{r}_1+\boldsymbol{e}_3;\mathbb{T}) = \mathbb V_{k+1}(\boldsymbol{r}_1+\boldsymbol{e}_3) \otimes \mathbb{T}$ with $\boldsymbol{e}_3:=(0,0,1)^{\intercal}$. Then~\eqref{eq:symcurledgedofnumber} follows from the edge DoFs for the smooth finite element divdiv complex.
%
%\item 
%When $r^e_1=1,2$, we have $r_3^e = -1$, then the last term in~\eqref{eq:symcurledgedofnumber} disappears. Since
%\begin{align*}
%&\quad 3\sum_{j=0}^{r_0^e}(j+1) - 8\sum_{j=0}^{r_1^e}(j+1) + 6\sum_{j=0}^{r_2^e}(j+1) \\
%& =3{r_1^e+3\choose2}-8{r_1^e+2\choose2}+6{r_1^e+1\choose2}={r_1^e-1\choose2}=0,
%\end{align*}
%identity~\eqref{eq:symcurledgedofnumber} reduces to 
%\begin{align}  
%\notag 4 +& 3\sum_{j=0}^{r_0^e}(j+1)j - 8\sum_{j=0}^{r_1^e}(j+1)^2+6\sum_{j=0}^{r_2^e}(j+1)(j+2) = 0,
%\end{align}
%which can be verified by direct computation.

\medskip
\step 2
When $r_1^e = 0, r_2^e = -1, r_3^e = -1$, in~\eqref{eq:symcurledgedofnumber} the last two terms disappeared. DoFs 
\eqref{eq:3dCrsymcurlTfemdofE3} and~\eqref{eq:3dCrsymcurlTfemdofE4} are added for ${\rm DoF}_{k+1}^{\sym\curl_+}(e; \boldsymbol r_1)$. The number of DoFs~\eqref{eq:3dCrsymcurlTfemdofE4} is $|{\rm DoF}_{k}^{\div\div^+}(e; \boldsymbol r_2)|$. Recall that  $m = k -2r_1^{\texttt{v}} - 1$. By direct calculation, $ 3|{\rm DoF}_{k+2}^{\grad}(e; \boldsymbol r_0)| = 9m + 6$. Sum of number of DoFs~\eqref{eq:3dCrsymcurlTfemdofE1} and~\eqref{eq:3dCrsymcurlTfemdofE3} is $8\dim\mathbb B_{k+1}(e; r_1^{\texttt{v}})+\dim\mathbb B_{k}(e; r_1^{\texttt{v}}-1) = 9 m + 10$. Therefore we conclude 
\begin{equation}\label{eq:r1e=0symcurl}
|{\rm DoF}_{k+1}^{\sym\curl_+}(e; \boldsymbol r_1)| = 4 + 3|{\rm DoF}_{k+2}^{\grad}(e; \boldsymbol r_0)| + |{\rm DoF}_{k}^{\div\div^+}(e; \boldsymbol r_2)|.
\end{equation}
No $|{\rm DoF}_{k-2}^{L^2}(e; \boldsymbol r_3)|$ is in~\eqref{eq:r1e=0symcurl} as $r_3^e=-1$. 

\medskip
\step 3
When $r_1^e = -1,  r_2^e = -1, r_0^e = 0$,~\eqref{eq:3dCrsymcurlTfemdofE1} is further removed. Now~\eqref{eq:3dCrsymcurlTfemdofE2},~\eqref{eq:3dCrsymcurlTfemdofE3}, and~\eqref{eq:3dCrsymcurlTfemdofE4} are present. The number of DoFs~\eqref{eq:3dCrsymcurlTfemdofE4} is still $|{\rm DoF}_{k}^{\div\div^+}(e; \boldsymbol r_2)|$. The number of DoFs~\eqref{eq:3dCrsymcurlTfemdofE2}-\eqref{eq:3dCrsymcurlTfemdofE3} is
\begin{align*}
&\quad 2\dim\mathbb B_{k+1}(e; r_1^{\texttt{v}})+\dim\mathbb B_{k}(e; r_1^{\texttt{v}}-1) = 3m+1 = 4 + 3|{\rm DoF}_{k+2}^{\grad}(e; \boldsymbol r_0)|.
\end{align*}
So~\eqref{eq:r1e=0symcurl} still holds. 
%\end{itemize}
% \breakline
% {Compute directly}
% When $r_1^e = 2$, $r_3^e = -1$, 
% When $r_1^e = 0, r_2^e = -1$,~\eqref{eq:3dCrsymcurlTfemdofE3} and~\eqref{eq:3dCrsymcurlTfemdofE4} are added.
% When $r_1^e = -1,  r_2^e = -1$, ~\eqref{eq:3dCrsymcurlTfemdofE3} is removed and~\eqref{eq:3dCrsymcurlTfemdofE1} is added. 
% \breakline

In summary, for all cases the number of DoFs~\eqref{eq:3dCrsymcurlTfemdofE1}-\eqref{eq:3dCrsymcurlTfemdofE4} at an edge satisfies
\begin{equation}\label{eq:symcurldofnumber1}
|{\rm DoF}_{k+1}^{\sym\curl_+}(e; \boldsymbol r_1) |= 4 + 3|{\rm DoF}_{k+2}^{\grad}(e; \boldsymbol r_0)| + |{\rm DoF}_{k}^{\div\div^+}(e; \boldsymbol r_2)| - |{\rm DoF}_{k-2}^{L^2}(e; \boldsymbol r_3)|.
\end{equation}

Then combining~\eqref{eq:symcurldofnumber0},~\eqref{eq:symcurldofnumber1},~\eqref{eq:symcurldofnumber2}, and~\eqref{eq:symcurldofnumber3}, 
by the DoFs of spaces $\mathbb V_{k+2}(\boldsymbol{r}_1+1;\mathbb R^3)$, $\mathbb V^{\div\div^+}_{k}(\boldsymbol{r}_2;\mathbb S)$ and $\mathbb V^{L^2}_{k-2}(\boldsymbol{r}_3)$ and the Euler's formulae $|\Delta_0(T)| - |\Delta_1(T)|+|\Delta_2(T)| - |\Delta_3(T)| = 1$, the number of DoFs~\eqref{eq:3dCrsymcurlTfemdofV1}-\eqref{eq:3dCrsymcurlTfemdofT2} is
$$
-4 + \dim\mathbb P_{k+2}(T;\mathbb R^3)+\dim\mathbb P_{k}(T;\mathbb S)-\dim\mathbb P_{k-2}(T),
$$
which equals $\dim\mathbb P_{k+1}(T;\mathbb T)$ in view of~\eqref{eq:divdivdim}.
\end{proof}

\begin{lemma}
The DoFs~\eqref{eq:symcurldof} are uni-solvent for $\mathbb P_{k+1}(T;\mathbb T)$.
\end{lemma}
\begin{proof}
In light of Lemma~\ref{lem:symcurlTDoFsnumber}, we only need to prove $\boldsymbol{\tau}=\boldsymbol{0}$ for $\boldsymbol{\tau}\in\mathbb P_{k+1}(T;\mathbb T)$ that all the specified DoFs~\eqref{eq:symcurldof} are nullified. 

The vanishing DoF~\eqref{eq:3dCrsymcurlTfemdofE1} implies 
$$
\int_e \frac{\partial^{j}(\sym\curl \boldsymbol \tau)}{\partial n_1^{i}\partial n_2^{j-i}}:\boldsymbol{q} \dd s=0,  \quad \boldsymbol{q}\in \mathbb B_{k-j}(e; r_2^{\texttt{v}}-j)\otimes\mathbb S, 0\leq i\leq j\leq r_2^{e}.
$$
The vanishing DoF~\eqref{eq:3dCrsymcurlTfemdofE3} implies 
$$
\int_e \tr_2^{\div_f\div_f}(\Pi_f\sym(\boldsymbol \tau\times \boldsymbol n)\Pi_f)\,q \dd s=0, \quad q\in \mathbb B_{k}(e;  r_1^{\texttt{v}}-1), \textrm{ if } r_1^{e}=-1,0.
$$
Detailed expressions for these formulations are provided in~\cite[Lemma 6.1]{Chen;Huang:2020Finite}.

Applying the integration by parts, it follows from the vanishing DoFs~\eqref{eq:3dCrsymcurlTfemdofE1}-\eqref{eq:3dCrsymcurlTfemdofE3} that
\begin{align*}
\int_f \div_f(\boldsymbol{n}\cdot\boldsymbol{\tau}\times\boldsymbol{n}) \dd S=0, & \\
\int_f \div_f\div_f(\sym(\boldsymbol{\tau}\times\boldsymbol{n}))\,q \dd S=0, &\quad q\in \mathbb P_1(f).
\end{align*}
Using the identities
\begin{align*}
\tr_2^{\div\div}(\sym\curl \boldsymbol \tau) &= \div_f\div_f \sym(\boldsymbol{\tau}\times\boldsymbol{n}),\\
\tr_2^{\div\div}(\sym\curl \boldsymbol \tau) & = (\div \sym\curl \boldsymbol \tau)\cdot \boldsymbol n+ \div_f (\Pi_f (\sym\curl \boldsymbol \tau)\boldsymbol n),
\end{align*}
the linear combination of DoFs~\eqref{eq:3dCrsymcurlTfemdofF4} and~\eqref{eq:3dCrsymcurlTfemdofF5} implies the continuity of $(\div \sym\curl \boldsymbol \tau)\cdot \boldsymbol n$, i.e., $ \sym\curl \boldsymbol \tau\in  H(\div,\Omega;\mathbb S)$.  

This, combined with the vanishing~\eqref{eq:3dCrsymcurlTfemdofV1}-\eqref{eq:3dCrsymcurlTfemdofE1},\eqref{eq:3dCrsymcurlTfemdofE3}-\eqref{eq:3dCrsymcurlTfemdofE4},\eqref{eq:3dCrsymcurlTfemdofF2},\eqref{eq:3dCrsymcurlTfemdofF4}-\eqref{eq:3dCrsymcurlTfemdofF5}, and~\eqref{eq:3dCrsymcurlTfemdofT2}, lead us to the conclusion that $\sym\curl\boldsymbol{\tau}=\boldsymbol{0}$. As a result, we find that $\boldsymbol{\tau}=\dev\grad\boldsymbol{v}$ where $\boldsymbol{v}\in\mathbb P_{k+2}(T;\mathbb R^3)$.

Then the vanishing DoFs~\eqref{eq:3dCrsymcurlTfemdofV1} and~\eqref{eq:3dCrsymcurlTfemdofE1}-\eqref{eq:3dCrsymcurlTfemdofE3} imply $(\nabla^i\boldsymbol{v})(\texttt{v})=\boldsymbol{0}$ for $\texttt{v}\in\Delta_0(T)$ and $i=0,\ldots, r_0^{\texttt{v}}$, and $(\nabla^i\boldsymbol{v})|_e=\boldsymbol{0}$ for $e\in\Delta_1(T)$ and $i=0,\ldots, r_0^{e}$. By the vanishing DoFs~\eqref{eq:3dCrsymcurlTfemdofF1} and~\eqref{eq:3dCrsymcurlTfemdofF3}, we get $(\nabla^i\boldsymbol{v})|_f=\boldsymbol{0}$ for $f\in\Delta_2(T)$ and $i=0,\ldots, r_0^{f}$.
Combining these results indicates $\boldsymbol{v}\in\mathbb B_{k+2}^{3}(\boldsymbol{r}_0)$. Therefore $\boldsymbol{v}=\boldsymbol{0}$ holds from the vanishing DoF~\eqref{eq:3dCrsymcurlTfemdofT1}.
\end{proof}

\begin{lemma}
For $r_1^f=-1,0$, define
\begin{align*}  
\mathbb V_{k+1}^{\sym\curl_+}(\boldsymbol{r}_1;\mathbb{T}):=\{\boldsymbol{\tau}\in{L}^2(\Omega;\mathbb T): &\, \boldsymbol{\tau}|_T\in\mathbb P_{k+1}(T;\mathbb T) \textrm{ for all } T\in\mathcal T_h, \\
&\textrm{ and all the DoFs~\eqref{eq:symcurldof} are single-valued}\}.
\end{align*}
Then it is equal to the space $\{\boldsymbol{\tau}\in\mathbb V_{k+1}(\boldsymbol r_1)\otimes \mathbb T: \sym\curl\boldsymbol{\tau}\in{\mathbb V}_{k}^{\div\div^+}(\boldsymbol r_2;\mathbb S)\}.$
%\begin{equation}
%\mathbb V_{k+1}^{\sym\curl_+}(\boldsymbol{r}_1;\mathbb{T})=\mathbb V.
%\end{equation}
\end{lemma}
\begin{proof}
Clearly $\mathbb V_{k+1}^{\sym\curl_+}(\boldsymbol{r}_1;\mathbb{T})\subseteq \{\boldsymbol{\tau}\in\mathbb V_{k+1}(\boldsymbol r_1)\otimes \mathbb T: \sym\curl\boldsymbol{\tau}\in{\mathbb V}_{k}^{\div\div^+}(\boldsymbol r_2;\mathbb S)\}$. By~\eqref{eq:symcurlTfemdim} and the proof of Lemma~\ref{lem:symcurlTDoFsnumber}, their dimensions are equal. 
\end{proof}

To construct the space $\mathbb V_{k+1}^{\sym\curl}(\boldsymbol{r}_1;\mathbb{T})$, we will modify the DoFs~\eqref{eq:symcurldof}. The changes involve removing the DoF~\eqref{eq:3dCrsymcurlTfemdofF5}, and extending the DoF~\eqref{eq:3dCrsymcurlTfemdofT2} to a more general form:
\begin{align}
\int_T (\sym\curl\boldsymbol{\tau}):\boldsymbol{q} \dx, &\quad \boldsymbol{q}\in \mathbb B_{k}^{\div\div}(\boldsymbol{r}_2;\mathbb{S})\cap\ker(\div\div).
\label{eq:3dCrsymcurlTfemdofT2modified}
\end{align}
These modifications maintain the sum of DoFs unchanged, as determined by the bubble space definition. The unisolvence of the modified DoFs can be proven in a manner similar to the original ones.

For the case when $r_1^f=-1$ or $0$, we need to redefine the space as follows:
\begin{align*}  
 \mathbb V_{k+1}^{\sym\curl}(\boldsymbol{r}_1;\mathbb{T}):= &\{\boldsymbol{\tau}\in{L}^2(\Omega;\mathbb T): \, \boldsymbol{\tau}|_T\in\mathbb P_{k+1}(T;\mathbb T) \textrm{ for all } T\in\mathcal T_h, \text{ and }\\
&\textrm{ all DoFs~\eqref{eq:3dCrsymcurlTfemdofV1}-\eqref{eq:3dCrsymcurlTfemdofT1}, \text{ removing }\eqref{eq:3dCrsymcurlTfemdofF5} \text{ adding}~\eqref{eq:3dCrsymcurlTfemdofT2modified}, are single-valued}\}.
\end{align*}
By this construction, we ensure that $\sym\curl \mathbb V_{k+1}^{\sym\curl}(\boldsymbol{r}_1;\mathbb{T}) \subset \mathbb V_{k+1}^{\div\div}(\boldsymbol{r}_2;\mathbb{S})$. Thus, we obtain the complex~\eqref{eq:femdivdivcomplex3d}. To establish its exactness, we verify the dimension identity:
\begin{align*}
%\label{eq:symcurlTfemdim}
- 4 + \dim\mathbb V_{k+2}^{\grad} (\boldsymbol{r}_0;\mathbb R^3)  & - \dim\mathbb V_{k+1}^{\sym\curl}(\boldsymbol{r}_1;\mathbb{T})\\
& +\dim \mathbb V_{k}^{\div\div}(\boldsymbol{r}_2;\mathbb{S}) -  \dim \mathbb V_{k-2}^{L^2} (\boldsymbol{r}_3) = 0,\notag
\end{align*}
which can be derived from~\eqref{eq:symcurlTfemdim} by notifying 
\begin{align*}
& \dim \mathbb V_{k+1}^{\sym\curl}(\boldsymbol{r}_1;\mathbb{T}) - \dim \mathbb V_{k+1}^{\sym\curl_+}(\boldsymbol{r}_1;\mathbb{T})
\\
 = &\dim \mathbb V_{k}^{\div\div}(\boldsymbol{r}_2;\mathbb{S}) - \dim \mathbb V_{k}^{\div\div^+}(\boldsymbol{r}_2;\mathbb{S}) \\
= &\left (4|\Delta_3(\mathcal T_h)| -  |\Delta_2(\mathcal T_h)|\right ) \dim \mathbb B^{\div}_k(f; \boldsymbol r_2). 
\end{align*}
The removed DoF~\eqref{eq:3dCrsymcurlTfemdofF5} will contribute to the bubble functions. 

\begin{remark}\rm 
When considering the $\div\div^+$ element, one can explore various variants such as the Hu-Zhang type element (cf. Remark~\ref{rm:HZdof}). This may lead to an increase in the number of edge DoFs for $\mathbb V_k^{\div\div^+}(\boldsymbol r_2)$ when $r_2^e=-1$. 
% One can refer to~\eqref{eq:3dHZdivSfemdofE2} in Remark~\ref{rm:HZdof} for details on these additional DoFs.
%
However, the modifications introduced will not affect the dimension count for edge DoFs, as the added edge DoFs will correspond to $|{\rm DoF}_{k}^{\div\div^+}(e; \boldsymbol r_2)|$, thus preserving the relationship stated in~\eqref{eq:r1e=0symcurl}. $\qed$
\end{remark}

%{define $\sym\curl^+$ space and construct $\sym\curl^+ - \div\div^+$ complex.}

%Use formulae
%$$
%\div \sym \curl \boldsymbol \tau = \nabla \times \boldsymbol \tau \cdot \nabla 
%$$
%\begin{remark}\rm 
There are more variants of $\sym\curl$ elements. We can add edge continuity of $\boldsymbol t^{\intercal}\boldsymbol \tau\boldsymbol t$ and face continuity $\int_f \skw (\Pi_f \boldsymbol \tau \times\boldsymbol{n})$ so that $\boldsymbol \tau \times \boldsymbol n$ is continuous. Then we can construct finite element spaces $\mathbb V_{k+1}^{\sym\curl^+}(\boldsymbol{r}_1;\mathbb{T})$ and $\mathbb V_{k+1}^{\sym\curl_+^+}(\boldsymbol{r}_1;\mathbb{T})$.
We can also relax to $\boldsymbol r_3\geq \boldsymbol r_2\ominus 1$ and impose condition
$$
\sym\curl \boldsymbol \tau \in \mathbb V_k^{\div\div^+}(\boldsymbol r_2, \boldsymbol r_3; \mathbb S),
$$
which require additional DoFs for $\div \sym\curl \boldsymbol \tau \in \mathbb V_{k-1}^{\div}(\boldsymbol r_3; \mathbb S)$. The divdiv complex in~\cite{Hu;Liang;Ma;Zhang:2022conforming} belongs to this type of variant. Furthermore $\boldsymbol r_2$ can be relaxed to $\boldsymbol r_2\geq \boldsymbol r_1\ominus 1$ but the modification of DoFs will be more involved and the lengthy detail is skipped here. 

Notice that the space $\mathbb V_{k+1}^{\sym\curl}( (0,-1,-1)^{\intercal};\mathbb{T})$ constructed in~\cite{ChenHuang2024div-div-conforming} is much simpler as $\mathbb V_{k}^{\div\div}(\bs{-1};\mathbb{S})$ is used.
%\end{remark}
%
%
%\input{conformalHessian}
%
%\input{conformalElastcity}

\appendix
\section{Bubble Polynomial Complexes}\label{sec:bubblecomplex}
%%%%%%%%%%%%%%%%%%%%%%%%%%%%
In this appendix we will develop various bubble polynomial complexes. We refer to Section \ref{sec:smoothbubble} for the definition of polynomial  bubble spaces on faces and tetrahedron. 

\subsection{Bubble de Rham complex}
On the dimension of the bubble polynomials, it holds that~\cite[Lemma 3.11]{Chen;Huang:2022FEMcomplex3D}
\begin{equation}\label{eq:dimBr}
\begin{aligned}
\dim\mathbb B_k(T; \boldsymbol{r}) & =-3{k+3\choose3} + 8{r^{\texttt{v}}+3 \choose 3} - 6(k+r^e-2r^{\texttt{v}}-1){r^e+2\choose 2} \\
&\quad + 6{r^e+2\choose 3} +12{k-2r^{\texttt{v}}-1 \choose 3} +12(r^f+1){k-2r^{\texttt{v}}+r^e\choose 2}  \\
&\quad +4{k+2-r^f\choose3} -12{r^{\texttt{v}}+2-r^f\choose3} - 12{k-2r^{\texttt{v}} +r^f\choose 3}.
\end{aligned}
\end{equation}

We first consider the construction of bubble de Rham complexes in two dimensions.
\begin{lemma}
Let smoothness vectors $\boldsymbol{r}_1\geq-1$, $\boldsymbol{r}_0 = \boldsymbol{r}_1+ 1$, and $\boldsymbol{r}_2=\boldsymbol{r}_1\ominus1$. Let $f\in\Delta_2(T)$ and $k\geq\max\{2r_1^{\texttt{v}}-1,0\}$.
Then the bubble de Rham complexes 
\begin{equation}\label{eq:femderhambubblecomplex2d}
0\xrightarrow{\subset}\mathbb B_{k+2}(f;\boldsymbol{r}_0)\xrightarrow{\curl_f}\mathbb B^{\div_f}_{k+1}(f; \boldsymbol{r}_1)\xrightarrow{\div_f}\mathbb B_{k}(f;\boldsymbol{r}_2)/\mathbb R\to0,
\end{equation}
% and
\begin{equation}\label{eq:rotfemderhambubblecomplex2d}
0\xrightarrow{\subset}\mathbb B_{k+2}(f;\boldsymbol{r}_0)\xrightarrow{\grad_f}\mathbb B^{\rot_f}_{k+1}(f; \boldsymbol{r}_1)\xrightarrow{\rot_f}\mathbb B_{k}(f;\boldsymbol{r}_2)/\mathbb R\to0
\end{equation}
are exact.
\end{lemma}
\begin{proof}
Since complex \eqref{eq:rotfemderhambubblecomplex2d} is only the rotation of complex \eqref{eq:femderhambubblecomplex2d}, it suffices to prove the exactness of complex \eqref{eq:femderhambubblecomplex2d}.

We refer to~\cite[Corollary~4.1]{Chen;Huang:2022femcomplex2d} for the case $r_1^e \geq 0$. Now, we assume $r_1^e = -1$. It is straightforward to verify that  
$$
\mathbb B^{\div_f}_{k+1}(f; \boldsymbol{r}_1) \cap \ker(\div_f) = \curl_f \mathbb B_{k+2}(f; \boldsymbol{r}_0).
$$  
To complete the proof of the exactness of the complex \eqref{eq:femderhambubblecomplex2d}, we proceed by verifying the dimension identity
\begin{equation}\label{eq:20240316}
\dim\mathbb B_{k+2}(f;\boldsymbol{r}_0)+\dim\mathbb B_{k}(f;\boldsymbol{r}_2)/\mathbb R=\dim\mathbb B^{\div_f}_{k+1}(f; \boldsymbol{r}_1).
\end{equation}
By Lemma~3.5 in~\cite{Chen;Huang:2022femcomplex2d},
\begin{align*}
\dim\mathbb B_{k+2}(f;\boldsymbol{r}_0)&=\dim\mathbb P_{k+2}(f)-3{r_1^{\texttt{v}}+3\choose2}-3(k-2r_1^{\texttt{v}}-1), \\
\dim\mathbb B_{k}(f;\boldsymbol{r}_2)&=\dim\mathbb P_{k}(f)-3{r_2^{\texttt{v}}+2\choose2}, \\
\dim\mathbb B^{\div_f}_{k+1}(f; \boldsymbol{r}_1)&=2\dim\mathbb P_{k+1}(f)-6{r_1^{\texttt{v}}+2\choose2} -3(k-2r_1^{\texttt{v}}). 
\end{align*}
Noting that $\dim\mathbb P_{k+2}(f)+\dim\mathbb P_{k}(f)=2\dim\mathbb P_{k+1}(f)+1$, we have
\begin{align*}
&\quad \dim\mathbb B_{k+2}(f;\boldsymbol{r}_0)+\dim\mathbb B_{k}(f;\boldsymbol{r}_2)/\mathbb R-\dim\mathbb B^{\div_f}_{k+1}(f; \boldsymbol{r}_1) \\
&=6{r_1^{\texttt{v}}+2\choose2} +3-3{r_1^{\texttt{v}}+3\choose2}-3{r_2^{\texttt{v}}+2\choose2}=0.
\end{align*}
Thus, \eqref{eq:20240316} follows.
\end{proof}

We then move to the three dimensions.
\begin{lemma}
Let $\boldsymbol r_0 \geq 0, \bs r_1 = \bs r_0 - 1, \boldsymbol r_2=\boldsymbol r_1\ominus1, \boldsymbol r_3=\boldsymbol r_2\ominus1$ be smoothness vectors.
Assume $\boldsymbol r_2$ satisfies \eqref{eq:boundr2fordivbubble}, and $k\geq\max\{2r_{2}^{\texttt{v}}, 1\}$.
Then the bubble de Rham complex 
\begin{equation}\label{eq:femderhambubblecomplex}
0\xrightarrow{\subset}\mathbb B_{k+2}(\boldsymbol{r}_0)\xrightarrow{\grad}\mathbb B^{\curl}_{k+1}(\boldsymbol{r}_1)\xrightarrow{\curl}\mathbb B^{\div}_{k}(\boldsymbol{r}_2)\xrightarrow{\div}\mathbb B_{k-1}(\boldsymbol{r}_3)/\mathbb R\to0
\end{equation}
is exact.
\end{lemma}
\begin{proof}
Clearly \eqref{eq:femderhambubblecomplex} is a complex, and $\grad\mathbb B_{k+2}(\boldsymbol{r}_0)=\mathbb B^{\curl}_{k+1}(\boldsymbol{r}_1)\cap\ker(\curl)$. Then
$
\dim\curl\mathbb B^{\curl}_{k+1}(\boldsymbol{r}_1)=\dim\mathbb B^{\curl}_{k+1}(\boldsymbol{r}_1)-\dim\mathbb B_{k+2}(\boldsymbol{r}_0).
$
On the other side, by the div stability \eqref{eq:divbubbleonto},
$
\dim\mathbb B^{\div}_{k}(\boldsymbol{r}_2)\cap\ker(\div)=\dim\mathbb B^{\div}_{k}(\boldsymbol{r}_2)-\dim\mathbb B_{k-1}(\boldsymbol{r}_3)/\mathbb R.
$
Hence %we have
\begin{align*}
\dim\mathbb B^{\div}_{k}(\boldsymbol{r}_2)\cap\ker(\div)-\dim\curl\mathbb B^{\curl}_{k+1}(\boldsymbol{r}_1)&=\dim\mathbb B_{k+2}(\boldsymbol{r}_0)-\dim\mathbb B^{\curl}_{k+1}(\boldsymbol{r}_1)\\
&\quad +\dim\mathbb B^{\div}_{k}(\boldsymbol{r}_2)-\dim\mathbb B_{k-1}(\boldsymbol{r}_3)/\mathbb R.
\end{align*}

\step 1 Consider $r_0^f\geq2$. 
Using the last row of Table 2 in~\cite{Chen;Huang:2022FEMcomplex3D}, we obtain $\mathbb B^{\div}_{k}(\boldsymbol{r}_2)\cap\ker(\div)=\curl\mathbb B^{\curl}_{k+1}(\boldsymbol{r}_1)$.

\step 2 Consider $\boldsymbol{r}_0=(r_0^{\texttt{v}}, r_0^e, 1)^{\intercal}$ with $r_0^{e}\geq4$. 
Set $\widetilde{\boldsymbol{r}}_0=(r_0^{\texttt{v}}, r_0^e, 2)^{\intercal}$, $\widetilde{\boldsymbol{r}}_1=\widetilde{\boldsymbol{r}}_0-1$, $\widetilde{\boldsymbol{r}}_2=\widetilde{\boldsymbol{r}}_1-1$ and $\widetilde{\boldsymbol{r}}_3=\boldsymbol{r}_3$.
By DoFs \eqref{eq:Cr3D},
\begin{align*}
\dim\mathbb B_{k+2}(\boldsymbol{r}_0)-\dim\mathbb B_{k+2}(\widetilde{\boldsymbol{r}}_0)&=\sum_{f\in\Delta_2(T)}\dim\mathbb B_{k}(f;\boldsymbol{r}_0-2),
\end{align*}
\begin{align*}
\dim\mathbb B^{\curl}_{k+1}(\boldsymbol{r}_1)-\dim\mathbb B^{\curl}_{k+1}(\widetilde{\boldsymbol{r}}_1)&=3\sum_{f\in\Delta_2(T)}\dim\mathbb B_{k}(f;\boldsymbol{r}_1-1),
\end{align*}
and by \eqref{eq:divbubbledecomp},
\begin{align*}
\dim\mathbb B^{\div}_{k}(\boldsymbol{r}_2)-\dim\mathbb B^{\div}_{k}(\widetilde{\boldsymbol{r}}_2) & =2\sum_{f\in\Delta_2(T)}\dim\mathbb B_{k}(f;\boldsymbol{r}_2).
\end{align*}
% We have
% \begin{align*}
% \dim\mathbb B_{k+2}(\boldsymbol{r}_0)-\dim\mathbb B_{k+2}(\widetilde{\boldsymbol{r}}_0)&=-12{k+2-2r_0^{\texttt{v}}+r_0^e\choose 2}  +4{k+2\choose2} -12{r_0^{\texttt{v}}\choose2} \\
% &\quad + 12{k-2r_0^{\texttt{v}}+3 \choose 2},
% \end{align*}
% \begin{align*}
% \dim\mathbb B^{\curl}_{k+1}(\boldsymbol{r}_1)-\dim\mathbb B^{\curl}_{k+1}(\widetilde{\boldsymbol{r}}_1)&=-36{k+1-2r_1^{\texttt{v}}+r_1^e\choose 2}  +12{k+2\choose2} -36{r_1^{\texttt{v}}+1\choose2} \\
% &\quad + 36{k-2r_1^{\texttt{v}} +1\choose 2},
% \end{align*}
% \begin{align*}
% \dim\mathbb B^{\div}_{k}(\boldsymbol{r}_2)-\dim\mathbb B^{\div}_{k}(\widetilde{\boldsymbol{r}}_2) & =8{k-3r_2^e-1\choose 2}-24{r_2^{\texttt{v}}-2r_2^e\choose 2}.
% \end{align*}
% It can be directly verified that
Then it follows from $\boldsymbol{r}_0-2=\boldsymbol{r}_1-1=\boldsymbol{r}_2$ that
\begin{align*}
\dim\mathbb B^{\curl}_{k+1}(\boldsymbol{r}_1)-\dim\mathbb B^{\curl}_{k+1}(\widetilde{\boldsymbol{r}}_1)&=\dim\mathbb B_{k+2}(\boldsymbol{r}_0)-\dim\mathbb B_{k+2}(\widetilde{\boldsymbol{r}}_0) \\
&\quad+\dim\mathbb B^{\div}_{k}(\boldsymbol{r}_2)-\dim\mathbb B^{\div}_{k}(\widetilde{\boldsymbol{r}}_2).
\end{align*} 
Since $\mathbb B_{k-1}^{L^2}(\boldsymbol{r}_3)=\mathbb B_{k-1}^{L^2}(\widetilde{\boldsymbol{r}}_3)$,
we conclude the exactness of the bubble complex \eqref{eq:femderhambubblecomplex} beginning with $\boldsymbol{r}_0$ from the bubble complex \eqref{eq:femderhambubblecomplex} beginning with $\widetilde{\boldsymbol{r}}_0$.

\step 3 Consider $\boldsymbol{r}_0=(r_0^{\texttt{v}}, 2, 1)^{\intercal}$ with $r_0^{\texttt{v}}\geq4$. 
We have
\begin{align*}
\dim\mathbb B_{k+2}(\boldsymbol{r}_0) & =-3{k+5\choose3} + 8{r_1^{\texttt{v}}+4 \choose 3} +4{k+3\choose3} -12{r_1^{\texttt{v}}+2\choose3} \\
&\quad + 24(k-2r_1^{\texttt{v}}),
\end{align*}
\begin{align*}
\dim\mathbb B^{\curl}_{k+1}(\boldsymbol{r}_1)&=-9{k+4\choose3} + 24{r_1^{\texttt{v}}+3 \choose 3}+12{k+3\choose3} -36{r_1^{\texttt{v}}+2\choose3} \\
&\quad+18(k-2r_1^{\texttt{v}}),
\end{align*}
\begin{align*}
\dim\mathbb B_k^{\div}(\boldsymbol{r}_2) & =-9{k+3\choose3} + 24{r_1^{\texttt{v}}+2 \choose 3} +12{k+2\choose3} -36{r_1^{\texttt{v}}+1\choose3} \\
&\quad + 18(k-2r_1^{\texttt{v}}+1)  +8{k-1\choose 2}-24{r_2^{\texttt{v}}\choose 2},
\end{align*}
\begin{align*}
\dim\mathbb B_{k-1}^{L^2}(\boldsymbol{r}_3) & ={k+2\choose3} - 4{r_1^{\texttt{v}}+1 \choose 3}-1.
\end{align*}
Then 
\begin{equation*}
\dim\mathbb B_{k+2}(\boldsymbol{r}_0)-\dim\mathbb B^{\curl}_{k+1}(\boldsymbol{r}_1)+\dim\mathbb B_k^{\div}(\boldsymbol{r}_2)-\dim\mathbb B_{k-1}^{L^2}(\boldsymbol{r}_3)=0.
\end{equation*}

\step 4 The exactness of the bubble complex \eqref{eq:femderhambubblecomplex} beginning with  $\boldsymbol{r}_0=(r_0^{\texttt{v}}, 3, 1)^{\intercal}$ with $r_0^{\texttt{v}}\geq6$ can be proved similarly.

\step 5 Consider $\boldsymbol{r}_0=(r_0^{\texttt{v}}, r_0^e, 0)^{\intercal}$ with $r_0^{e}\geq2$.
Set $\widetilde{\boldsymbol{r}}_0=(r_0^{\texttt{v}}, r_0^e, 1)^{\intercal}$, $\widetilde{\boldsymbol{r}}_1=\widetilde{\boldsymbol{r}}_0-1$, $\widetilde{\boldsymbol{r}}_2=\boldsymbol{r}_2$ and $\widetilde{\boldsymbol{r}}_3=\boldsymbol{r}_3$.
By DoFs \eqref{eq:Cr3D} and \eqref{eq:curlbubbledecomp},
\begin{align*}
\dim\mathbb B_{k+2}(\boldsymbol{r}_0)-\dim\mathbb B_{k+2}(\widetilde{\boldsymbol{r}}_0)&=\sum_{f\in\Delta_2(T)}\dim\mathbb B_{k+1}(f;\boldsymbol{r}_0-1), \\
\dim\mathbb B^{\curl}_{k+1}(\boldsymbol{r}_1)-\dim\mathbb B^{\curl}_{k+1}(\widetilde{\boldsymbol{r}}_1)&=\sum_{f\in\Delta_2(T)}\dim \mathbb B_{k+1}(f; \bs r_1).
\end{align*}
% We have
% \begin{align*}
% \dim\mathbb B_{k+2}(\boldsymbol{r}_0)-\dim\mathbb B_{k+2}(\widetilde{\boldsymbol{r}}_0)&=-12{k+2-2r_0^{\texttt{v}}+r_0^e\choose 2}  +4{k+3\choose2} -12{r_0^{\texttt{v}}+1\choose2} \\
% &\quad + 12{k-2r_0^{\texttt{v}}+2 \choose 2},
% \end{align*}
% \begin{align*}
% \dim\mathbb B^{\curl}_{k+1}(\boldsymbol{r}_1)-\dim\mathbb B^{\curl}_{k+1}(\widetilde{\boldsymbol{r}}_1)&=\dim \mathbb B_{k+1}(f; \bs r_+)=4{k-3r_1^e\choose 2}-12{r_1^{\texttt{v}}-2r_1^e\choose 2},
% \end{align*}
% It can be directly verified that
So
\begin{align*}
\dim\mathbb B^{\curl}_{k+1}(\boldsymbol{r}_1)-\dim\mathbb B^{\curl}_{k+1}(\widetilde{\boldsymbol{r}}_1)&=\dim\mathbb B_{k+2}(\boldsymbol{r}_0)-\dim\mathbb B_{k+2}(\widetilde{\boldsymbol{r}}_0).
\end{align*} 
Since $\mathbb B^{\div}_{k}(\boldsymbol{r}_2)=\mathbb B^{\div}_{k}(\widetilde{\boldsymbol{r}}_2)$ and $\mathbb B_{k-1}^{L^2}(\boldsymbol{r}_3)=\mathbb B_{k-1}^{L^2}(\widetilde{\boldsymbol{r}}_3)$,
we conclude the exactness of the bubble complex \eqref{eq:femderhambubblecomplex} beginning with $\boldsymbol{r}_0$ from the bubble complex \eqref{eq:femderhambubblecomplex} beginning with $\widetilde{\boldsymbol{r}}_0$.

\step 6 Consider $\boldsymbol{r}_0=(r_0^{\texttt{v}}, 1, 0)^{\intercal}\geq 0$ with $r_0^{\texttt{v}}\geq2$. 
We have
\begin{align*}
\dim\mathbb B_{k+2}(\boldsymbol{r}_0) & =-3{k+5\choose3} + 8{r_0^{\texttt{v}}+3 \choose 3} +4{k+4\choose3} -12{r_0^{\texttt{v}}+2\choose3} \\
&\quad + 6(k-2r_0^{\texttt{v}}+1),
\end{align*}
\begin{align*}
\dim\mathbb B^{\curl}_{k+1}(\boldsymbol{r}_1) & =-9{k+4\choose3} + 24{r_1^{\texttt{v}}+3 \choose 3} +12{k+3\choose3} -36{r_1^{\texttt{v}}+2\choose3}  \\
&\quad + 18(k-2r_1^{\texttt{v}}) +4{k\choose2}-12{r_1^{\texttt{v}}\choose 2},
\end{align*}
\begin{align*}
\dim\mathbb B_k^{\div}(\boldsymbol{r}_2) & =-9{k+3\choose3} + 24{r_1^{\texttt{v}}+2 \choose 3} +12{k+2\choose3} -36{r_1^{\texttt{v}}+1\choose3} \\
&\quad + 24(k-2r_1^{\texttt{v}}+1)  +8{k-1\choose 2}-24{r_2^{\texttt{v}}\choose 2},
\end{align*}
\begin{align*}
\dim\mathbb B_{k-1}^{L^2}(\boldsymbol{r}_3) & ={k+2\choose3} - 4{r_1^{\texttt{v}}+1 \choose 3}-1.
\end{align*}
Then 
\begin{equation*}
\dim\mathbb B_{k+2}(\boldsymbol{r}_0)-\dim\mathbb B^{\curl}_{k+1}(\boldsymbol{r}_1)+\dim\mathbb B_k^{\div}(\boldsymbol{r}_2)-\dim\mathbb B_{k-1}^{L^2}(\boldsymbol{r}_3)=0.
\end{equation*}

\step 7 The exactness of the bubble complex \eqref{eq:femderhambubblecomplex} beginning with  $\boldsymbol{r}_0=(r_0^{\texttt{v}}, 0, 0)^{\intercal}$ with $r_0^{\texttt{v}}\geq0$ can be proved similarly. 
\end{proof}

\subsection{Bubble spaces for tensors}

For a tensor space $\mathbb X$ with $\mathbb X=\mathbb M, \mathbb S, \mathbb T$,  define bubble spaces
\begin{align*}
\mathbb B^{\div}_{k}(\boldsymbol{r}; \mathbb X):={}&\{\boldsymbol{\tau}\in \mathbb B_{k}(\boldsymbol{r}; \mathbb X): \boldsymbol{\tau}\boldsymbol{n}|_{\partial T}=0\},\\
\mathbb B^{\curl}_{k}(\boldsymbol{r};\mathbb X):={}&\{\boldsymbol{\tau}\in\mathbb B_{k}(\boldsymbol{r}; \mathbb X): \boldsymbol{\tau}\times\boldsymbol{n}|_{\partial T}=0\},\\
\mathbb B^{\div\div}_{k}(\boldsymbol{r}; \mathbb X):={}&\{\boldsymbol{\tau}\in \mathbb B_{k}(\boldsymbol{r}; \mathbb X): (\boldsymbol n^{\intercal}\boldsymbol \tau\boldsymbol n)|_{\partial T}=0, \tr_2^{\div\div}(\boldsymbol \tau)=0, \\
&\qquad\qquad\qquad\quad\;\; (\boldsymbol n_i^{\intercal}\boldsymbol \tau\boldsymbol n_j)|_{e}=0 \textrm{ for } e\in\Delta_{1}(T) \textrm{ and }i,j=1,2\},\\
\mathbb B^{\sym\curl}_{k}(\boldsymbol{r}; \mathbb X):={}&\{\boldsymbol{\tau}\in \mathbb B_{k}(\boldsymbol{r}; \mathbb X): (\boldsymbol{n}\times\sym(\boldsymbol{\tau}\times\boldsymbol{n})\times\boldsymbol{n})|_{\partial T}=0, \\
&\quad\; (\boldsymbol{n}\cdot\boldsymbol{\tau}\times\boldsymbol{n})|_{\partial T}=0, \;\sym\curl\boldsymbol{\tau}\in \mathbb B^{\div\div}_{k-1}((0,-1,-1)^{\intercal}; \mathbb S)\},\\
\mathbb B^{\inc}_{k}(\boldsymbol{r};\mathbb X):={}&\{\boldsymbol{\tau}\in\mathbb B_{k}(\boldsymbol{r}; \mathbb X): (\boldsymbol{n}\times\boldsymbol{\tau}\times\boldsymbol{n})|_{\partial T}=0, \tr_2^{\inc}(\boldsymbol{\tau})=0, \\
&\qquad\qquad\qquad\qquad\qquad \inc\boldsymbol{\tau}\in \mathbb B^{\div}_{k-2}((0,-1,-1)^{\intercal}; \mathbb S)\}.
\end{align*}

\subsection{Bubble Hessian complex}\label{sec:bubblehesscomplex}

\begin{lemma}\label{lem:divbubbleontoT}
Assume the polynomial degree $k\geq \max\{2r^{\texttt{v}}+2,3\}$, and the smoothness vector $\boldsymbol r$ satisfies either: 
\begin{enumerate}
\item $r^{\texttt{v}}\geq 2r^e+1$ and $r^{e}\geq 2(r^f+1)$,  or
% \item $r^{\texttt{v}}\geq 1$, $r^{e}\geq 0$, both $(\boldsymbol r + 1, \boldsymbol r, k+1)$ and $(\boldsymbol r, \boldsymbol r\ominus 1, k)$ satisfy the condition in Lemma \ref{lem:divbubbleonto}, or
\item $r^{\texttt{v}}\geq 0$ and $r^{e}=r^{f}=-1$.
% $$
% \begin{cases}
%   k\geq 2r^{\texttt{v}}+2, & \textrm{ if } r^{\texttt{v}}\geq1,\\
% k\geq2r^{\texttt{v}}+3, & \textrm{ if } r^{\texttt{v}}=0.
%  \end{cases}
% $$
\end{enumerate}
% Assume the smoothness vector $\boldsymbol r$ and polynomial degree $k$ satisfy either: 
% \begin{enumerate}
% \item $r^{\texttt{v}}\geq 1$, $r^{e}\geq 0$, both $(\boldsymbol r + 1, \boldsymbol r, k+1)$ and $(\boldsymbol r, \boldsymbol r\ominus 1, k)$ satisfy the condition in Lemma \ref{lem:divbubbleonto}, or
% \item $r^{\texttt{v}}\geq 0$, $r^{e}=r^{f}=-1$, and $k\geq\max\{2r^{\texttt{v}}+2,3\}$.
% \end{enumerate}
Then we have the $(\div; \mathbb T)$ stability
\begin{equation}\label{eq:divbubbleontoT}
\div\mathbb B^{\div}_{k-1}(\boldsymbol{r};\mathbb T)=\mathbb B_{k-2}(\boldsymbol{r}\ominus1;\mathbb R^3)/{\rm RT}.
\end{equation}
\end{lemma}
\begin{proof}
(1) For case $r^{\texttt{v}}\geq 1$ and $r^{e}\geq 0$,
consider the bubble diagram
\begin{equation*}%\label{eq:BGGdivS}
\begin{tikzcd}[column sep=small, row sep=normal]
% \mathbb R \arrow{r}{\subset}
% &
% \mathbb V^{\grad}_{k+2}(\boldsymbol{r}_0)
% \arrow{r}{\grad}
%  &
% \mathbb V^{\curl}_{k+1}(\boldsymbol{r}_1+1)
% \arrow{r}{\curl}
%  &
&
\mathbb B^{\div}_{k}(\boldsymbol{r}+1)
% \arrow[dl,swap,"\mskw "']  
  \arrow{r}{\div}
 &
\mathbb B_{k-1}(\boldsymbol{r})
\arrow{r}{}
& \mathbb R \\
% \\
% \mathbb V \arrow{r}{\subset}
% &
% \mathbb V^{\grad}_{k+1}(\boldsymbol{r}_1+1;\mathbb R^3)
%  \arrow[ur,swap,"{\rm id}"'] \arrow{r}{\grad}
%  & 
{\mathbb B}^{\curl}_{k}(\boldsymbol{r}+1;\mathbb M)
 \arrow[ur,swap,"- 2\vskw"'] \arrow{r}{\curl}
 & 
\mathbb B^{\div}_{k-1}(\boldsymbol{r};\mathbb M)
 \arrow[ur,swap,"\tr "'] \arrow{r}{\div}
 & 
\mathbb B_{k-2}(\boldsymbol{r}\ominus 1;\mathbb R^3)/\mathbb R^3
\arrow[r] 
&0.
\end{tikzcd}
\end{equation*}
For $\boldsymbol{u}\in\mathbb B_{k-2}(\boldsymbol{r}\ominus1;\mathbb R^3)/{\rm RT}$, since $(\boldsymbol r, \boldsymbol r\ominus 1, k-1)$ is div stable, there exists $\boldsymbol{\tau}\in\mathbb B^{\div}_{k-1}(\boldsymbol{r};\mathbb M)$ such that $\div\boldsymbol{\tau}=\boldsymbol{u}$. Noting that $\tr\boldsymbol{\tau}\in\mathbb B_{k-1}(\boldsymbol{r})/\mathbb R$ and $(\boldsymbol r+1, \boldsymbol r, k)$ is div stable, we have $\tr\boldsymbol{\tau}=\div\boldsymbol{v}$ with $\boldsymbol{v}\in\mathbb B^{\div}_{k}(\boldsymbol{r}+1)=\mathbb B_{k}(\boldsymbol{r}+1;\mathbb R^3)$. Choose $\boldsymbol{\sigma}=\boldsymbol{\tau}-\frac{1}{2}\curl(\mskw\boldsymbol{v})\in\mathbb B^{\div}_{k-1}(\boldsymbol{r};\mathbb M)$. Then $\div\boldsymbol{\sigma}=\div\boldsymbol{\tau}=\boldsymbol{u}$, and by the anticommutativity, $\tr\boldsymbol{\sigma}=\tr\boldsymbol{\tau}-\div\boldsymbol{v}=0$. Thus, $\boldsymbol{\sigma}\in\mathbb B^{\div}_{k-1}(\boldsymbol{r};\mathbb T)$, and $\boldsymbol{u}\in\div\mathbb B^{\div}_{k-1}(\boldsymbol{r};\mathbb T)$.

(2) The case $r^{\texttt{v}}= 0$ and $r^{e}=r^{f}=-1$ can be found in~\cite[Theorem 4.4]{HuLiang2021}. Now consider case $r^{\texttt{v}}\geq 1$ and $r^{e}=r^{f}=-1$. Notice that $$\mathbb B^{\div}_{k-1}((r^{\texttt{v}},0,-1)^{\intercal};\mathbb T)\subseteq \mathbb B^{\div}_{k-1}((r^{\texttt{v}},-1,-1)^{\intercal};\mathbb T).$$ We conclude \eqref{eq:divbubbleontoT} with $\boldsymbol{r}=(r^{\texttt{v}},-1,-1)^{\intercal}$ from \eqref{eq:divbubbleontoT} with $\boldsymbol{r}=(r^{\texttt{v}},0,-1)^{\intercal}$.
\end{proof}

\begin{lemma}
Let $\boldsymbol r_0 \geq (4,2,1)^{\intercal}, \boldsymbol r_1 = \boldsymbol r_0 -2, \boldsymbol r_2=\boldsymbol r_1\ominus1, \boldsymbol r_3= \boldsymbol r_2\ominus1$. 
Assume $\boldsymbol r_2$ satisfies the condition in Lemma \ref{lem:divbubbleontoT}, and $k\geq2r_2^{\texttt{v}}+2$.
Then the bubble Hessian complex
\begin{equation}\label{eq:fembubblehessiancomplex}
\resizebox{.925\hsize}{!}{$
0\xrightarrow{\subset} \mathbb B_{k+2}(\boldsymbol{r}_0)\xrightarrow{\hess}\mathbb B^{\curl}_{k}(\boldsymbol{r}_1;\mathbb S)\xrightarrow{\curl} \mathbb B^{\div}_{k-1}(\boldsymbol{r}_2;\mathbb T) \xrightarrow{\div} \mathbb B_{k-2}(\boldsymbol{r}_3;\mathbb R^3)/{\rm RT}\xrightarrow{}0
$}
\end{equation}
is exact.
\end{lemma}
\begin{proof}
Obviously \eqref{eq:fembubblehessiancomplex} is a complex, $\mathbb B^{\curl}_{k}(\boldsymbol{r}_1;\mathbb S)\cap\ker(\curl)=\hess\mathbb B_{k+2}(\boldsymbol{r}_0)$, and
\begin{align*}
\dim\curl\mathbb B_{k}^{\curl}(\boldsymbol{r}_1;\mathbb S)=3\dim\mathbb B_{k}((\boldsymbol{r}_1)_+)+\dim\mathbb B_{k}^{\curl}(\boldsymbol{r}_1)-\dim\mathbb B_{k+2}(\boldsymbol{r}_0).
\end{align*}
By bubble complex \eqref{eq:femderhambubblecomplex},
\begin{align*}
\dim\mathbb B_{k+2}(\boldsymbol{r}_0)&=3\dim\mathbb B_{k+1}(\boldsymbol{r}_0-1)-\dim\mathbb B_{k}^{\div}(\boldsymbol{r}_1)+\dim\mathbb B_{k-1}(\boldsymbol{r}_2)-1, \\
\dim\mathbb B_{k+1}(\boldsymbol{r}_0-1)&=\dim\mathbb B_{k}^{\curl}(\boldsymbol{r}_1)-\dim\mathbb B_{k-1}^{\div}(\boldsymbol{r}_2)+\dim\mathbb B_{k-2}(\boldsymbol{r}_3)-1.
\end{align*}
Noting that $2\dim\mathbb B_{k}^{\curl}(\boldsymbol{r}_1)=\dim\mathbb B_{k}^{\div}(\boldsymbol{r}_1)+3\dim\mathbb B_{k}((\boldsymbol{r}_1)_+)$, we get from the last two equations that
\begin{align*}
\dim\mathbb B_{k+2}(\boldsymbol{r}_0)&=3\dim\mathbb B_{k}((\boldsymbol{r}_1)_+)+\dim\mathbb B_{k}^{\curl}(\boldsymbol{r}_1)\\
&\quad -3\dim\mathbb B_{k-1}^{\div}(\boldsymbol{r}_2)+\dim\mathbb B_{k-1}(\boldsymbol{r}_2)+3\dim\mathbb B_{k-2}(\boldsymbol{r}_3)-4.
\end{align*}
Then
\begin{equation*}
\dim\curl\mathbb B_{k}^{\curl}(\boldsymbol{r}_1;\mathbb S)=  3\dim\mathbb B_{k-1}^{\div}(\boldsymbol{r}_2)-\dim\mathbb B_{k-1}(\boldsymbol{r}_2)-3\dim\mathbb B_{k-2}(\boldsymbol{r}_3)+4.
\end{equation*}

Thanks to the $(\div; \mathbb T)$ stability \eqref{eq:divbubbleontoT}, we only need to prove $\mathbb B^{\div}_{k-1}(\boldsymbol{r}_2;\mathbb T)\cap\ker(\div)=\curl\mathbb B_{k}^{\curl}(\boldsymbol{r}_1;\mathbb S)$. Since $\curl\mathbb B_{k}^{\curl}(\boldsymbol{r}_1;\mathbb S)\subseteq\mathbb B^{\div}_{k-1}(\boldsymbol{r}_2;\mathbb T)\cap\ker(\div)$, we will finish the proof by dimension count.
According to the $(\div; \mathbb T)$ stability \eqref{eq:divbubbleontoT} and $\dim \mathbb B^{\div}_{k-1}(\boldsymbol{r}_2;\mathbb T)=3\dim \mathbb B^{\div}_{k-1}(\boldsymbol{r}_2)-\dim \mathbb B_{k-1}(\boldsymbol{r}_2)$,
\begin{align*}
\dim \mathbb B^{\div}_{k-1}(\boldsymbol{r}_2;\mathbb T)\cap\ker(\div)&=\dim \mathbb B^{\div}_{k-1}(\boldsymbol{r}_2;\mathbb T)-3\dim\mathbb B_{k-2}(\boldsymbol{r}_3) + 4 \\
&=3\dim \mathbb B^{\div}_{k-1}(\boldsymbol{r}_2)-\dim \mathbb B_{k-1}(\boldsymbol{r}_2)-3\dim\mathbb B_{k-2}(\boldsymbol{r}_3) + 4.
\end{align*}
Hence, $\dim \mathbb B^{\div}_{k-1}(\boldsymbol{r}_2;\mathbb T)\cap\ker(\div)=\dim\curl\mathbb B_{k}^{\curl}(\boldsymbol{r}_1;\mathbb S)$.
\end{proof}

\subsection{Bubble elasticity complex}\label{sec:bubbleelascomplex}

\begin{lemma}\label{lem:divbubbleontoS}
Assume the polynomial degree $k\geq \max\{2r^{\texttt{v}},1\}$, and the smoothness vector $\boldsymbol r$ satisfies either: 
\begin{enumerate}
\item $r^{\texttt{v}}\geq 2r^e+1$ and $r^{e}\geq 2(r^f+1)$,  or
% \item $r^{\texttt{v}}\geq 1$, $r^{e}\geq 0$, both $(\boldsymbol r + 1, \boldsymbol r, k+1)$ and $(\boldsymbol r, \boldsymbol r\ominus 1, k)$ satisfy the condition in Lemma \ref{lem:divbubbleonto}, or
\item $r^{\texttt{v}}\geq 0$ and $r^{e}=r^{f}=-1$.
% $$
% \begin{cases}
%   k\geq 2r^{\texttt{v}}+2, & \textrm{ if } r^{\texttt{v}}\geq1,\\
% k\geq2r^{\texttt{v}}+3, & \textrm{ if } r^{\texttt{v}}=0.
%  \end{cases}
% $$
\end{enumerate}
% Assume the smoothness vector $\boldsymbol r$ and polynomial degree $k$ satisfy either: 
% \begin{enumerate}
% \item $r^{\texttt{v}}\geq 1$, $r^{e}\geq 0$, both $(\boldsymbol r + 1, \boldsymbol r, k+2)$ and $(\boldsymbol r, \boldsymbol r\ominus 1, k+1)$ are div stable, or
% \item $r^{\texttt{v}}\geq 0$, $r^{e}=r^{f}=-1$, and $k\geq\max\{2r^{\texttt{v}},1\}$.
% % $$
% % \begin{cases}
% %   k\geq 2r^{\texttt{v}}+2, & \textrm{ if } r^{\texttt{v}}\geq1,\\
% % k\geq2r^{\texttt{v}}+3, & \textrm{ if } r^{\texttt{v}}=0.
% %  \end{cases}
% % $$
% \end{enumerate}
Then we have the $(\div; \mathbb S)$ stability
\begin{equation}\label{eq:divbubbleontoS}
\div\mathbb B^{\div}_{k+1}(\boldsymbol{r};\mathbb S)=\mathbb B_{k}(\boldsymbol{r}\ominus1;\mathbb R^3)/{\rm RM}.
\end{equation}
\end{lemma}
\begin{proof}
(1) For case $r^{\texttt{v}}\geq 1$ and $r^{e}\geq 0$,
consider the bubble diagram
\begin{equation*}%\label{eq:BGGdivS}
\begin{tikzcd}[column sep=small, row sep=normal]
% \mathbb R \arrow{r}{\subset}
% &
% \mathbb V^{\grad}_{k+2}(\boldsymbol{r}_0)
% \arrow{r}{\grad}
%  &
% \mathbb V^{\curl}_{k+1}(\boldsymbol{r}_1+1)
% \arrow{r}{\curl}
%  &
&
\mathbb B^{\div}_{k+2}(\boldsymbol{r}+1;\mathbb M)
% \arrow[dl,swap,"\mskw "']  
  \arrow{r}{\div}
 &
\mathbb B_{k+1}(\boldsymbol{r};\mathbb R^3)
\arrow{r}{}
& \mathbb R^3 \\
% \\
% \mathbb V \arrow{r}{\subset}
% &
% \mathbb V^{\grad}_{k+1}(\boldsymbol{r}_1+1;\mathbb R^3)
%  \arrow[ur,swap,"{\rm id}"'] \arrow{r}{\grad}
%  & 
{\mathbb B}^{\curl}_{k+2}(\boldsymbol{r}+1;\mathbb M)
 \arrow[ur,swap,"S"'] \arrow{r}{\curl}
 & 
\mathbb B^{\div}_{k+1}(\boldsymbol{r};\mathbb M)
 \arrow[ur,swap,"- 2\vskw"'] \arrow{r}{\div}
 & 
\mathbb B_{k}(\boldsymbol{r}\ominus 1;\mathbb R^3)/\mathbb R^3
\arrow[r] 
&0.
\end{tikzcd}
\end{equation*}
Then follow the proof of the first part of Lemma~\ref{lem:divbubbleontoT} to obtain \eqref{eq:divbubbleontoS}.

(2) The case $r^{\texttt{v}}= 0$ and $r^{e}=r^{f}=-1$ can be found in~\cite[Lemma 3.2]{HuZhang2015}. Now consider case $r^{\texttt{v}}\geq 1$ and $r^{e}=r^{f}=-1$. Notice that $$\mathbb B^{\div}_{k+1}((r^{\texttt{v}},0,-1)^{\intercal};\mathbb S)\subseteq \mathbb B^{\div}_{k+1}((r^{\texttt{v}},-1,-1)^{\intercal};\mathbb S).$$ We conclude \eqref{eq:divbubbleontoS} with $\boldsymbol{r}=(r^{\texttt{v}},-1,-1)^{\intercal}$ from \eqref{eq:divbubbleontoS} with $\boldsymbol{r}=(r^{\texttt{v}},0,-1)^{\intercal}$.
\end{proof}

\begin{lemma}
Let 
\begin{equation*}%\label{eq:relasticity}
\boldsymbol{r}_0 \geq (2, 1, 0)^{\intercal},
\; 
\boldsymbol r_1 =\boldsymbol{r}_0-1, 
%\geq\begin{pmatrix}
%1\\
%0 \\
%-1
%\end{pmatrix},
\;
 \boldsymbol{r}_2=\max\{\boldsymbol{r}_1\ominus2, (0, -1, -1)^{\intercal}\},
\;
 \boldsymbol{r}_3=\boldsymbol{r}_2\ominus1.
\end{equation*}
Assume $\boldsymbol r_2$ satisfies the condition in Lemma \ref{lem:divbubbleontoS}, and $k\geq\max\{2r_2^{\texttt{v}},1\}$. %Let $k\geq 2r_1^{\texttt{v}}-1$. 
Then the bubble elasticity complex
\begin{equation}\label{eq:fembubbleelasticitycomplex}
\resizebox{.925\hsize}{!}{$
0\xrightarrow{\subset} \mathbb B_{k+4}(\boldsymbol{r}_0;\mathbb R^3)\xrightarrow{\defm}\mathbb B^{\inc}_{k+3}(\boldsymbol{r}_1;\mathbb S)\xrightarrow{\inc} \mathbb B^{\div}_{k+1}(\boldsymbol{r}_2;\mathbb S) \xrightarrow{\div} \mathbb B_{k}(\boldsymbol{r}_3;\mathbb R^3)/{\rm RM}\xrightarrow{}0
$}
\end{equation}
is exact.
\end{lemma}
\begin{proof}
Obviously \eqref{eq:fembubbleelasticitycomplex} is a complex, $\mathbb B^{\inc}_{k+3}(\boldsymbol{r}_1;\mathbb S)\cap\ker(\inc)=\defm\mathbb B_{k+4}(\boldsymbol{r}_0;\mathbb R^3)$, and
\begin{align*}
\dim\inc\mathbb B^{\inc}_{k+3}(\boldsymbol{r}_1;\mathbb S)=\dim\mathbb B^{\inc}_{k+3}(\boldsymbol{r}_1;\mathbb S)-3\dim\mathbb B_{k+4}(\boldsymbol{r}_0).
\end{align*}
Thanks to the $(\div; \mathbb S)$ stability \eqref{eq:divbubbleontoS}, we only need to prove $\mathbb B^{\div}_{k+1}(\boldsymbol{r}_2;\mathbb S)\cap\ker(\div)=\inc\mathbb B^{\inc}_{k+3}(\boldsymbol{r}_1;\mathbb S)$. Since $\inc\mathbb B^{\inc}_{k+3}(\boldsymbol{r}_1;\mathbb S)\subseteq\mathbb B^{\div}_{k+1}(\boldsymbol{r}_2;\mathbb S)\cap\ker(\div)$, we will finish the proof by dimension count.

We first consider case $r_1^f=-1,0$.
By DoFs \eqref{eq:incdof}, 
\begin{equation*}
\dim\mathbb B^{\inc}_{k+3}(\boldsymbol{r}_1;\mathbb S)=3\dim\mathbb B_{k+4}(\boldsymbol{r}_0)+\dim\mathbb B^{\div}_{k+1}(\boldsymbol{r}_2;\mathbb S)\cap\ker(\div).
\end{equation*}
This indicates $\dim\inc\mathbb B^{\inc}_{k+3}(\boldsymbol{r}_1;\mathbb S)=\dim\mathbb B^{\div}_{k+1}(\boldsymbol{r}_2;\mathbb S)\cap\ker(\div)$, as required.

Next consider case $r_1^f\geq1$.
By bubble complex \eqref{eq:femderhambubblecomplex},
\begin{align*}
\dim\mathbb B_{k+4}(\boldsymbol{r}_0)&=3\dim\mathbb B_{k+3}(\boldsymbol{r}_1)-3\dim\mathbb B_{k+2}(\boldsymbol{r}_1-1)+\dim\mathbb B_{k+1}(\boldsymbol{r}_2)-1, \\
\dim\mathbb B_{k+3}(\boldsymbol{r}_1)&=3\dim\mathbb B_{k+2}(\boldsymbol{r}_1-1)-\dim\mathbb B_{k+1}^{\div}(\boldsymbol{r}_2)+\dim\mathbb B_{k}(\boldsymbol{r}_3)-1.
\end{align*}
Combining the last two equations yields
\begin{align*}
\dim\inc\mathbb B^{\inc}_{k+3}(\boldsymbol{r}_1;\mathbb S)&=6\dim\mathbb B_{k+3}(\boldsymbol{r}_1)-3\dim\mathbb B_{k+4}(\boldsymbol{r}_0) \\
&=3\dim\mathbb B_{k+1}^{\div}(\boldsymbol{r}_2)-3\dim\mathbb B_{k+1}(\boldsymbol{r}_2)-3\dim\mathbb B_{k}(\boldsymbol{r}_3)+6.
\end{align*}
Noting that $\dim\mathbb B^{\div}_{k+1}(\boldsymbol{r}_2;\mathbb S)=3\dim\mathbb B_{k+1}^{\div}(\boldsymbol{r}_2)-3\dim\mathbb B_{k+1}(\boldsymbol{r}_2)$, we get from \eqref{eq:divbubbleontoS} that
\begin{equation*}
\dim\inc\mathbb B^{\inc}_{k+3}(\boldsymbol{r}_1;\mathbb S)=\dim\mathbb B^{\div}_{k+1}(\boldsymbol{r}_2;\mathbb S)\cap\ker(\div).
\end{equation*}
Therefore, $\inc\mathbb B^{\inc}_{k+3}(\boldsymbol{r}_1;\mathbb S)=\mathbb B^{\div}_{k+1}(\boldsymbol{r}_2;\mathbb S)\cap\ker(\div)$.
\end{proof}

\subsection{Bubble divdiv complex}\label{sec:bubbledivdivcomplex}

\begin{lemma}\label{lem:fembubbledivdiv+complex}
Let %\mnote{ can we enlarge to $r_0 = (1, 0, 0)$?}
\begin{equation*}%\label{eq:divdivrsequence}
\boldsymbol{r}_0\geq (1, 0, 0)^{\intercal}, \quad 
\boldsymbol{r}_1=\boldsymbol{r}_0-1,\quad 
\boldsymbol{r}_2= \max\{\boldsymbol{r}_1\ominus1, (0, -1, -1)^{\intercal}\},\quad 
\boldsymbol{r}_3=\boldsymbol{r}_2\ominus2.
\end{equation*}
Assume $\boldsymbol r_2$ satisfies the condition in Lemma \ref{lem:divbubbleontoS}, $\boldsymbol r_2\ominus1$ satisfies \eqref{eq:boundr2fordivbubble}, and $k\geq \max\{2r_2^{\texttt{v}}+1,3\}$.
% Assume $(\boldsymbol r_2, \boldsymbol r_2\ominus 1, k)$ is $(\div;\mathbb S)$ stable and $(\boldsymbol r_2\ominus 1, \boldsymbol r_3, k-1)$ is div stable.
We have the following exact bubble divdiv complex   
\begin{align}
\label{eq:fembubbledivdiv+complex}
0\xrightarrow{\subset}\mathbb B_{k+2} (\boldsymbol{r}_0;\mathbb R^3) &\xrightarrow{\dev\grad} \mathbb B_{k+1}^{\sym\curl_+}(\boldsymbol{r}_1;\mathbb{T}) \\
&\xrightarrow{\sym\curl}  \mathbb B_{k}^{\div\div^+}(\boldsymbol{r}_2;\mathbb{S}) \xrightarrow{\div\div} \mathbb B_{k-2}(\boldsymbol{r}_3)/\mathbb P_1(T)\rightarrow 0, \notag
\end{align}
where 
\begin{align*}
\mathbb B_{k+1}^{\sym\curl_+}(\boldsymbol{r}_1;\mathbb{T})&:=\{\boldsymbol{\tau}\in\mathbb B_{k+1}^{\sym\curl}(\boldsymbol{r}_1;\mathbb{T}): \sym\curl\boldsymbol{\tau}\in\mathbb B_{k}^{\div\div^+}(\boldsymbol{r}_2;\mathbb{S})\}, \\
\mathbb B_{k}^{\div\div^+}(\boldsymbol{r}_2;\mathbb{S})&:=\{\boldsymbol{\tau}\in\mathbb B_{k}^{\div\div}(\boldsymbol{r}_2;\mathbb{S}): \boldsymbol{\tau}\boldsymbol{n}|_{\partial T}=0\}.
\end{align*}
\end{lemma}
\begin{proof}
By the definition of the bubble spaces, we can see that \eqref{eq:fembubbledivdiv+complex} is a complex, and $\mathbb B_{k+1}^{\sym\curl_+}(\boldsymbol{r}_1;\mathbb{T})\cap\ker(\sym\curl)=\dev\grad\mathbb B_{k+2} (\boldsymbol{r}_0;\mathbb R^3)$. 

We first prove $\div\div\mathbb B_{k}^{\div\div^+}(\boldsymbol{r}_2;\mathbb{S})=\mathbb B_{k-2}(\boldsymbol{r}_3)/\mathbb P_1(T)$. When $r_2^f\geq1$, we have $\mathbb B_{k}^{\div\div^+}(\boldsymbol{r}_2;\mathbb{S})=\mathbb B_{k}(\boldsymbol{r}_2;\mathbb{S})$, then apply the bubble complex~\eqref{eq:femderhambubblecomplex} to get
\begin{align*}
\div\div\mathbb B_{k}^{\div\div^+}(\boldsymbol{r}_2;\mathbb{S})&=\div\div\mathbb B_{k}(\boldsymbol{r}_2;\mathbb{S})=\div\div\mathbb B_{k}(\boldsymbol{r}_2;\mathbb{M}) \\
&=\div(\mathbb B_{k-1}(\boldsymbol{r}_2-1;\mathbb{M})/\mathbb R^3)=\mathbb B_{k-2}(\boldsymbol{r}_3)/\mathbb P_1(T).
\end{align*}
When $r_2^f=-1,0$, we follow the proof of Lemma~\ref{lem:divdivSr2onto} and Corollary~\ref{cor:divdiv+}
to acquire $\div\div\mathbb B_{k}^{\div\div^+}(\boldsymbol{r}_2;\mathbb{S})=\mathbb B_{k-2}(\boldsymbol{r}_3)/\mathbb P_1(T)$.

We next prove $\sym\curl\mathbb B_{k+1}^{\sym\curl_+}(\boldsymbol{r}_1;\mathbb{T})= \mathbb B_{k}^{\div\div^+}(\boldsymbol{r}_2;\mathbb{S})\cap\ker(\div\div)$.

\step1
First consider $\boldsymbol{r}_0\geq (2, 1, 0)^{\intercal}$.
Let $\boldsymbol{\sigma}\in\mathbb B_{k}^{\div\div^+}(\boldsymbol{r}_2;\mathbb{S})\cap\ker(\div\div)$. Then $\div\boldsymbol{\sigma}\in\mathbb B_{k-1}^{\div}(\boldsymbol{r}_2\ominus1;\mathbb{R}^3)\cap\ker(\div)$. By the bubble complex~\eqref{eq:femderhambubblecomplex}, $\div\boldsymbol{\sigma}=\curl\boldsymbol{v}$ with $\boldsymbol{v}\in\mathbb B_{k}^{\curl}(\boldsymbol{r}_2;\mathbb{R}^3)$. That is $\boldsymbol{\sigma}-\mskw\boldsymbol{v}\in\mathbb B_{k}^{\div}(\boldsymbol{r}_2;\mathbb{M})\cap\ker(\div)$. Apply the bubble complex~\eqref{eq:femderhambubblecomplex} again to get $\boldsymbol{\sigma}-\mskw\boldsymbol{v}=\curl\boldsymbol{\tau}$ with $\boldsymbol{\tau}\in\mathbb B_{k+1}^{\curl}(\boldsymbol{r}_1;\mathbb{M})$. By the symmetry of $\boldsymbol{\sigma}$, we have $\boldsymbol{\sigma}=\sym\curl\boldsymbol{\tau}=\sym\curl(\dev\boldsymbol{\tau})$, where $\dev\boldsymbol{\tau}\in\mathbb B_{k+1}^{\sym\curl_+}(\boldsymbol{r}_1;\mathbb{T})$.

\step2 Then consider $\boldsymbol{r}_0=(r_0^{\texttt{v}}, 0, 0)^{\intercal}$ with $r_0^{\texttt{v}}\geq1$. Let $\hat{\bs r}_1 = (\max\{r_1^{\texttt{v}},1\}, 0, -1)^{\intercal}\geq(1,0,-1)^{\intercal}$. We have $\boldsymbol{r}_1\leq \hat{\bs r}_1$ and $\hat{\bs r}_1\ominus1=\boldsymbol{r}_2$. Thus, we conclude the result from 
$$
\sym\curl\mathbb B_{k+1}^{\sym\curl_+}(\hat{\bs r}_1;\mathbb{T})\subseteq\sym\curl\mathbb B_{k+1}^{\sym\curl_+}(\boldsymbol{r}_1;\mathbb{T})\subseteq \mathbb B_{k}^{\div\div^+}(\boldsymbol{r}_2;\mathbb{S})\cap\ker(\div\div)
$$
and $\sym\curl\mathbb B_{k+1}^{\sym\curl_+}(\hat{\bs r}_1;\mathbb{T})=\mathbb B_{k}^{\div\div^+}(\boldsymbol{r}_2;\mathbb{S})\cap\ker(\div\div)$.
\end{proof}

\begin{lemma}
Under the same assumption of Lemma~\ref{lem:fembubbledivdiv+complex}, we have the following exact bubble divdiv complex   
\begin{align}
\label{eq:fembubbledivdivcomplex}
0\xrightarrow{\subset}\mathbb B_{k+2} (\boldsymbol{r}_0;\mathbb R^3) &\xrightarrow{\dev\grad} \mathbb B_{k+1}^{\sym\curl}(\boldsymbol{r}_1;\mathbb{T}) \\
&\xrightarrow{\sym\curl}  \mathbb B_{k}^{\div\div}(\boldsymbol{r}_2;\mathbb{S}) \xrightarrow{\div\div} \mathbb B_{k-2}(\boldsymbol{r}_3)/\mathbb P_1(T)\rightarrow 0. \notag
\end{align}
\end{lemma}
\begin{proof}
When $r_2^f\geq0$, the complex \eqref{eq:fembubbledivdivcomplex} is exactly the complex \eqref{eq:fembubbledivdiv+complex}. Then we focus on the proof for the case $r_2^f=-1$.
By the definition of the bubble spaces, we can see that \eqref{eq:fembubbledivdivcomplex} is a complex, and $\mathbb B_{k+1}^{\sym\curl}(\boldsymbol{r}_1;\mathbb{T})\cap\ker(\sym\curl)=\dev\grad\mathbb B_{k+2} (\boldsymbol{r}_0;\mathbb R^3)$. 

Since $\div\div\mathbb B_{k}^{\div\div^+}(\boldsymbol{r}_2;\mathbb{S})\subseteq\div\div\mathbb B_{k}^{\div\div}(\boldsymbol{r}_2;\mathbb{S})\subseteq \mathbb B_{k-2}(\boldsymbol{r}_3)/\mathbb P_1(T)$, by complex \eqref{eq:fembubbledivdiv+complex}, we have
\begin{equation*}
\div\div\mathbb B_{k}^{\div\div}(\boldsymbol{r}_2;\mathbb{S})=\div\div\mathbb B_{k}^{\div\div^+}(\boldsymbol{r}_2;\mathbb{S}) = \mathbb B_{k-2}(\boldsymbol{r}_3)/\mathbb P_1(T).
\end{equation*}

Thanks to \eqref{eq:divdivSdF2} and \eqref{eq:3dCrsymcurlTfemdofF5}, 
we have
\begin{equation*}
% $
\dim\mathbb B_{k}^{\div\div}(\boldsymbol{r}_2;\mathbb{S})-\dim \mathbb B_{k+1}^{\sym\curl}(\boldsymbol{r}_1;\mathbb{T}) = \dim\mathbb B_{k}^{\div\div^+}(\boldsymbol{r}_2;\mathbb{S})-\dim \mathbb B_{k+1}^{\sym\curl_+}(\boldsymbol{r}_1;\mathbb{T}).
% $
\end{equation*}
Finally, apply the exactness of complex \eqref{eq:fembubbledivdiv+complex} to get
$$\mathbb B_{k}^{\div\div}(\boldsymbol{r}_2;\mathbb{S})\cap\ker(\div\div)=\sym\curl\mathbb B_{k+1}^{\sym\curl}(\boldsymbol{r}_1;\mathbb{T}).$$
\end{proof}

% When $\div\div\mathbb B_{k}^{\div\div}(\boldsymbol{r}_2;\mathbb{S})=\mathbb B_{k-2}(\boldsymbol{r}_3)/\mathbb P_1(T)$, we shall call $(\boldsymbol{r}_2,\boldsymbol{r}_3,k)$ is $(\div\div;\mathbb S)$ stable.

\bibliographystyle{abbrv}
\bibliography{./paper,refgeodecomp}
\end{document}